\documentclass[11pt]{amsart}
\usepackage{amssymb,bm,mathrsfs}
\usepackage[all]{xy}
\usepackage{hyperref}
\usepackage{MnSymbol}

\newcommand{\map}[1]{\xrightarrow{#1}}
\newcommand{\iso}{\xrightarrow{\simeq}}
\newcommand{\define}{\stackrel{\mathrm{def}}{=}}

\newcommand{\Gal}{\mathrm{Gal}}
\newcommand{\Hom}{\mathrm{Hom}}
\newcommand{\Aut}{\mathrm{Aut}}
\newcommand{\End}{\mathrm{End}}
\newcommand{\Spec}{\mathrm{Spec}}

\newcommand{\Q}{\mathbb Q}
\newcommand{\Z}{\mathbb Z}
\newcommand{\R}{\mathbb R}
\newcommand{\C}{\mathbb C}
\newcommand{\F}{\mathbb F}
\newcommand{\A}{\mathbb A}
\newcommand{\co}{\mathcal O}
\newcommand{\dR}{\mathrm{dR}}
\newcommand{\et}{\mathrm{et}}
\newcommand{\cris}{\mathrm{cris}}

\newcommand{\alg}{\mathrm{alg}}
\newcommand{\ord}{\mathrm{ord}}
\newcommand{\length}{\mathrm {length}}
\newcommand{\Lie}{\mathrm{Lie}}
\newcommand{\Fil}{\mathrm{Fil}}
\newcommand{\Pic}{\mathrm {Pic}}
\newcommand{\SL}{\mathrm {SL}}

\newcommand{\Falt}{\mathrm{Falt}}
\newcommand{\Colmez}{\mathrm{Col}}

\newcommand{\GSpin}{\mathrm{GSpin}}
\newcommand{\SO}{\mathrm{SO}}

\newcommand{\GL}{\mathrm{GL}}

\newcommand{\near}{ { \empty^{ \mathfrak{p} }} }
\newcommand{\nearp}{ { \empty^{ (\mathfrak{p}) }} }
\newcommand{\action}{\bullet}
\newcommand{\kk}{{\bm{k}}}
\newcommand{\beef}{ \diamond}
\newcommand{\cofu}{a}

\newcommand{\pow}[1]{[\vert#1\vert]}

\begin{document}

\author[F.~Andreatta, E. Z.~Goren, B.~Howard, K.~Madapusi Pera]{Fabrizio Andreatta, Eyal Z. Goren, \\ Benjamin Howard, Keerthi Madapusi Pera}
\title[Faltings heights of abelian varieties]{Faltings heights of abelian varieties with complex multiplication}

\address{Dipartimento di Matematica ``Federigo Enriques", Universit\`a di Milano, via C.~Saldini 50, Milano, Italia}
\email{fabrizio.andreatta@unimi.it}

\address{Department of Mathematics and Statistics, McGill University, 805 Sherbrooke St. West, Montreal, QC, Canada}
\email{eyal.goren@mcgill.ca}

\address{Department of Mathematics, Boston College, 140 Commonwealth Ave,
Chestnut Hill, MA, USA}
\email{howardbe@bc.edu}

\address{Department of Mathematics, University of Chicago, 5734 S University Ave, Chicago, IL, USA}
\email{keerthi@math.uchicago.edu}

\thanks{F.~Andreatta is supported by the Italian grant Prin 2010/2011. E. Z.~Goren is supported by an NSERC discovery grant, B.~Howard is supported by NSF grants DMS-1201480 and DMS-1501583. K.~Madapusi Pera is supported by NSF Postdoctoral Research Fellowship DMS-1204165 and NSF grant DMS-1502142.}

%\date{\today}

\begin{abstract}
Let $M$ be the Shimura variety associated with the group of spinor similitudes of a  quadratic space over $\Q$ of signature $(n,2)$.
We prove a conjecture of Bruinier-Kudla-Yang, relating the arithmetic intersection  multiplicities of special divisors
and  big CM points on $M$  to the central derivatives of  certain  $L$-functions.

As an application of this result, we prove an averaged version of Colmez's conjecture on the Faltings heights of CM abelian varieties. 
\end{abstract}

\maketitle

\setcounter{tocdepth}{1}
\tableofcontents

\theoremstyle{plain}

\newtheorem{theorem}{Theorem}[subsection]
\newtheorem{proposition}[theorem]{Proposition}
\newtheorem{lemma}[theorem]{Lemma}
\newtheorem{corollary}[theorem]{Corollary}
\newtheorem{conjecture}[theorem]{Conjecture}
\newtheorem{BigThm}{Theorem}

\theoremstyle{definition}
\newtheorem{definition}[theorem]{Definition}
\newtheorem{hypothesis}[theorem]{Hypothesis}

\theoremstyle{remark}
\newtheorem{remark}[theorem]{Remark}
\newtheorem{question}[theorem]{Question}
\newtheorem{example}[theorem]{Example}

\numberwithin{equation}{subsection}

\renewcommand{\theBigThm}{\Alph{BigThm}}
%\renewcommand{\theBigCor}{\Alph{BigCor}}
%\renewcommand{\theBigHyp}{\Alph{BigHyp}}
%\renewcommand{\theproblemma}{\Alph{problemma}}

%%%%%%%%%%%%%%%%%%%%%%%%%%%%%%%%%%

\section{Introduction}

%%%%%%%%%%%%%%%%%%%%%%%%%%%%%%%%%%%

%%%%%%%%%%%%%%%%%%%%%%%%%%%%%%%%%%%

\subsection{The average Colmez conjecture}

%%%%%%%%%%%%%%%%%%%%%%%%%%%%%%%%%%%

Let  $E$ be  a CM field of degree $2d$ with maximal totally real subfield $F$.  Let
$A$ be an abelian variety over $\C$ of dimension $d$ with   complex multiplication by the maximal order $\co_E\subset E$
and having CM type $\Phi\subset \Hom(E,\C)$.  In this situation, Colmez \cite{Colmez} has proved that the Faltings height $h^\mathrm{Falt}(A)$
of $A$ depends only on the pair $(E,\Phi)$, and not on $A$ itself.   We denote it by 
\[
h^\mathrm{Falt}_{(E,\Phi)} = h^\mathrm{Falt}(A)
\]

Colmez stated in  [\emph{loc.~cit.}] a conjectural 
formula for $h^\mathrm{Falt}_{(E,\Phi)}$ in terms of the logarithmic derivatives at $s=0$ of certain Artin $L$-functions, constructed in terms of the 
purely Galois-theoretic input $(E,\Phi)$.  The precise conjecture is recalled in \S \ref{ss:colmez statement},
where the reader may also find our  precise normalization  of the Faltings height.

When $d=1$, so $E$ is a quadratic imaginary field, Colmez's conjecture is a form of the famous Chowla-Selberg formula.
When $E/\Q$ is an abelian extension, Colmez proved his conjecture in [\emph{loc.~cit.}], up to a rational multiple of $\log(2)$.  
This extra error term was subsequently removed by Obus \cite{Obus}.  When $d=2$, Yang \cite{Yang:colmez} 
was able to prove Colmez's conjecture in many cases, including the first known cases of non-abelian extensions.

% If one holds $E$ fixed, but averages both sides of Colmez's conjecture over all CM types, the $L$-function side of the conjecture simplifies, leaving 
% \[
% \frac{1}{2^d} \sum_\Phi h^\mathrm{Falt}_{(E,\Phi)} 
% = - \frac{1}{2} \cdot \frac{ L'(0,\chi) }{ L(0,\chi ) } - \frac{1}{4} \cdot \log \left|  \frac{D_E}{D_F} \right| - \frac{d}{2}\cdot \log(2\pi ).
% \]
% Here $\chi : \A_F^\times \to \{\pm 1\}$ is the quadratic Hecke character determined by the extension $E/F$, and $L(s,\chi)$
% is the usual $L$-function, without the local factors at archimedean places.

Our first main result, stated in the text as Theorem \ref{thm:average colmez},  
is the proof of an averaged form of Colmez's conjecture for a fixed $E$, obtained by  averaging both sides of the conjectural formula  over all CM types.

\begin{BigThm}\label{bigthm:average colmez}
\[
\frac{1}{2^d} \sum_\Phi h^\mathrm{Falt}_{(E,\Phi)} 
= - \frac{1}{2} \cdot \frac{ L'(0,\chi) }{ L(0,\chi ) } - \frac{1}{4} \cdot \log \left|  \frac{D_E}{D_F} \right| - \frac{d}{2}\cdot \log(2\pi ).
\]
\end{BigThm}
Here $\chi : \A_F^\times \to \{\pm 1\}$ is the quadratic Hecke character determined by the extension $E/F$, and $L(s,\chi)$
is the usual $L$-function  without the local factors at archimedean places. The sum on the left is over all CM types of $E$, and $D_E$ and $D_F$ are the discriminants of $E$ and $F$, respectively.

\begin{remark}
Very shortly after this theorem was announced, Yuan-Zhang  also announced a proof; see \cite{YZ}.   The proofs are very different.  
The proof of Yuan-Zhang  is based on the Gross-Zagier style results of \cite{YZZ} for Shimura curves over totally real fields.
 Our proof, which is inspired by the $d=2$ case found in \cite{Yang:colmez}, 
 revolves around the calculation of arithmetic intersection multiplicities
on Shimura varieties of type $\GSpin(n,2)$, and makes essential use of the theory of Borcherds products, as well as certain Green function calculations 
of Bruinier-Kudla-Yang \cite{BKY}.   
\end{remark}

\begin{remark}
Tsimerman \cite{Tsimerman} has proved that Theorem  \ref{bigthm:average colmez} implies the Andr\'e-Oort conjecture for all Siegel (and hence all abelian type) Shimura varieties.
\end{remark}

%%%%%%%%%%%%%%%%%%%%%%%%%%%%%%%%%%

\subsection{GSpin Shimura varieties and special divisors}
\label{ss:intro shimura variety}

%%%%%%%%%%%%%%%%%%%%%%%%%%%%%%%%%%%

Let $(V,Q)$ be a quadratic space over $\Q$ of signature $(n,2)$ with $n\ge 1$, and let $L\subset V$ be a maximal lattice; that is, we assume that 
$Q(L) \subset \Z$, but that no lattice properly containing $L$ has this property.  
Let $L^\vee \subset V$ be the dual lattice of $L$ with respect to the bilinear form
\[
[x,y] = Q(x+y) - Q(x) -Q(y),
\]
and abbreviate $D_L = [L^\vee : L]$ for the discriminant of $L$.

To this  data one can associate a reductive group  $G=\GSpin(V)$ over $\Q$, a particular compact open subgroup $K \subset G(\A_f)$, 
and a hermitian domain
\[
\mathcal{D} = \{ z \in V_\C : [z,z]=0,\, [z, \overline{z}] <0 \} /\C^\times
\]
with an action of $G(\R)$ by holomorphic automorphisms.  
The $n$-dimensional complex orbifold
\[
M(\C) = G(\Q) \backslash \mathcal{D} \times G(\A_f) / K
\]
is the space of complex points of a  smooth algebraic stack $M$ over $\Q$, called  the  \emph{GSpin Shimura variety}.
It admits, as we explain in \S \ref{s:orthogonal shimura},   a flat and normal integral model $\mathcal{M}$  over $\Z$,
which  is smooth after inverting $D_L$. For any prime $p>2$ the special fiber $\mathcal{M}_{\F_p}$ is normal and Cohen-Macaulay.

The Weil representation 
\[
\omega_L: \widetilde{\SL}_2(\Z) \to \Aut_\C(S_L)
\]
defines an action of the metaplectic double cover $\widetilde{\SL}_2(\Z) \to \SL_2(\Z)$ on the space 
$S_L = \C[L^\vee / L]$ of complex valued functions on $L^\vee /L$.  
Associated with it are, for any half-integer $k$,  several spaces of vector-valued modular forms:
the space of  cusp forms $S_k(\omega_L)$, the space of weakly holomorphic forms $M^!_k(\omega_L)$, 
and the space of harmonic weak Maass forms  $H_k(\omega_L)$.  
There are similar spaces for the complex-conjugate representation $\overline{\omega}_L$.  
By a theorem of Bruinier-Funke \cite{BF}, these are related by an exact sequence 
\begin{equation}\label{intro BF}
0 \to M^!_{1-\frac{n}{2}}(\omega_L) \to  H_{1-\frac{n}{2}}(\omega_L) \map{\xi} S_{1+ \frac{n}{2}}  (\overline{\omega}_L) \to 0,
\end{equation}
where $\xi$ is an explicit conjugate-linear differential operator.

Let $\varphi_\mu \in S_L$ be the characteristic function of the coset $\mu \in L$.
Each form $f\in H_{1-\frac{n}{2}}(\omega_L)$ has a \emph{holomorphic part}, which is a formal $q$-expansion
\[
f^+ = \sum_{  \substack{     m  \gg -\infty \\ \mu \in L^\vee /L    }    } c_f^+(m,\mu) \varphi_\mu  \cdot q^m
\]
valued in $S_L$.  The sum is over all   $m\in D_L^{-1} \Z$, but there are only finitely many  nonzero terms with $m<0$.

The Shimura variety $\mathcal{M}$ comes with a  family of effective Cartier divisors $\mathcal{Z}(m,\mu)$, 
indexed by  positive  $m\in D_L^{-1} \Z$ and $\mu \in L^\vee /L$.  If the harmonic weak Maass form $f$
 has \emph{integral principal part}, in the sense that  $c_f^+(m,\mu) \in \Z$ for all $m\le 0$ and $\mu \in L^\vee /L$, then 
we may form the Cartier divisor
\[
\mathcal{Z}(f) =  \sum_{  \substack{     m > 0  \\ \mu \in L^\vee /L    }    } c_f^+(-m,\mu) \cdot \mathcal{Z}(m,\mu)
\]
on $\mathcal{M}$.  A construction of Bruinier \cite{Bru} endows this divisor with a Green function $\Phi(f)$, constructed as a
regularized theta lift of $f$.   From this  divisor and its Green function, we obtain a metrized line bundle
\[
\widehat{\mathcal{Z}}(f) = \big( \mathcal{Z}(f) , \Phi(f) \big) \in \widehat{\mathrm{Pic}} ( \mathcal{M} ).
\]

%%%%%%%%%%%%%%%%%%%%%%%%%%%%%%%%%%

\subsection{The arithmetic Bruinier-Kudla-Yang theorem}

%%%%%%%%%%%%%%%%%%%%%%%%%%%%%%%%%%%

We now explain how to  construct certain \emph{big CM cycles} on  $\GSpin$ Shimura varieties,  as in  \cite{BKY}.

Start with a totally real field $F$ of degree $d$, and a quadratic space $(\mathscr{V} , \mathscr{Q})$ over $F$ of 
dimension $2$ and signature $( ( 0 ,2) , (2,0) , \ldots , (2,0) )$.  In other words, $\mathscr{V}$ is negative definite at one 
archimedean place, and positive definite at the rest.    The even Clifford algebra $E=C^+(\mathscr{V})$   is a CM field  of degree $2d$ 
with $F$ as its maximal totally real subfield.

Now define a quadratic space 
\begin{equation}\label{intro trace}
( V,Q) = (\mathscr{V} , \mathrm{Tr}_{F/\Q} \circ \mathscr{Q} )
\end{equation}
over $\Q$ of signature $(n,2) = (2d-2,2)$, and fix a maximal lattice $L\subset V$.  
As described above in \S \ref{ss:intro shimura variety}, we obtain from this data a $\GSpin$ Shimura variety
$M\to \Spec(\Q)$, but now endowed with the additional structure of a distinguished $0$-cycle.   Indeed, the relation  (\ref{intro trace})
induces a morphism $T\to G$, where $T$ is the torus over $\Q$ with  points 
\[
T(\Q) = E^\times /   \mathrm{ker} \left(  \mathrm{Nm}: F^\times \to \Q^\times \right).
\]
From the morphism $T \to G$ one can construct a $0$-dimensional Shimura variety $Y$ over $E$, together with a  morphism 
$Y \to M$  of $\Q$-stacks.  The image of this morphism consists of  special points  (in the sense of Deligne),
and  are the \emph{big CM points} of \cite{BKY}.

In \S \ref{ss:integral_model_Y} we  define an integral model $\mathcal{Y}$ of $Y$, regular and flat over  $\co_E$, along with  a morphism 
$\mathcal{Y}  \to \mathcal{M}$ of $\Z$-stacks.    
Composing the pullback of metrized line bundles with  the arithmetic degree on the arithmetic curve $\mathcal{Y}$ 
defines a linear functional
\[
\widehat{\mathrm{Pic}} ( \mathcal{M} ) \to \widehat{\mathrm{Pic}} ( \mathcal{Y}) \map{\widehat{\deg}} \R.
\]
We call this linear function \emph{arithmetic degree along $\mathcal{Y}$},   and denote it by 
\[
\widehat{\mathcal{Z}} \mapsto [ \widehat{\mathcal{Z}} : \mathcal{Y} ].
\]

To state our second main theorem, we need to introduce one more actor to our drama.  This is a certain Hilbert modular Eisenstein series
$E(\vec{\tau} , s) $ of parallel weight $1$, valued in the dual representation $S_L^\vee$.   Starting from  any $f\in H_{ 2-d }(\omega_L)$,
we may apply the differential operator of (\ref{intro BF}) to obtain a vector-valued cusp form
\[
\xi(f) \in S_d( \overline{\omega}_L),
\]
and then form the Petersson inner product $\mathcal{L}( s, \xi(f) )$ of $\xi(f)$ against  the diagonal restriction of $E(\vec{\tau} , s) $ 
to the upper-half plane. This rather mysterious function inherits analytic continuation and a functional equation from the Eisenstein series, and the functional
equation forces $\mathcal{L}( s, \xi(f) )$ to vanish at $s=0$.  Our second main result, stated in the text as Theorem \ref{thm:arithmetic BKY},  is a formula for its derivative.

\begin{BigThm}\label{bigthm:arithmetic BKY}
For any $f\in H_{ 2-d }(\omega_L)$ with integral principal part, the equality 
\[
 \frac{  [ \widehat{\mathcal{Z}}(f) : \mathcal{Y} ]   }{   \deg_\C (Y)  }   =
-  \frac{ \mathcal{L}'(0 , \xi(f) )  }  {  \Lambda( 0 , \chi )  }  
+     \frac{  a(0,0)   \cdot  c_f^+( 0,0) } {  \Lambda( 0 , \chi )  } 
\]
holds up to a $\Q$-linear combination of $\{ \log(p) : p\mid D_{bad, L} \}$.  
\end{BigThm}

The unexplained notation in the theorem is as follows:  
$D_{bad,L}$ is the product of   certain ``bad" primes, depending on the lattice $L\subset V=\mathscr{V}$, specified in Definition \ref{defn:D bad};
\[
 \deg_\C (Y) = \sum_{y \in Y(\C) } \frac{1}{ \vert \Aut(y) \vert }
 \]
 is the number of $\C$-points of the $E$-stack $Y$, counted with multiplicities;  $\Lambda(s,\chi)$ is the completed $L$-function 
of   (\ref{completed L});
 and the constant $a(0,0)=a_F(0,\varphi_0)$ is defined in  Proposition \ref{prop:coarse constant}.  
In fact,  $a(0,0)$ is essentially the derivative at $s=0$ of the constant term of $E(\vec{\tau} , s) $.  By  Proposition \ref{prop:constant term eval},
it satisfies
\begin{equation}\label{intro constant}
\frac{  a(0,0)   } {  \Lambda( 0 , \chi )  } =  - \frac{  2\Lambda'(0,\chi)   } {  \Lambda( 0 , \chi )  }
\end{equation}
up to a $\Q$-linear combination of $\{ \log(p) : p\mid D_{bad, L} \}$.

A key component of the proof of Theorem \ref{thm:arithmetic BKY}  is the Bruinier-Kudla-Yang \cite{BKY} 
calculation of the values of the Green function $\Phi(f)$ at the 
points of $\mathcal{Y}$, which we recall in Theorem \ref{thm:BKY}.  In fact, a form of Theorem \ref{bigthm:arithmetic BKY} was
conjectured in \cite{BKY} based on these Green function calculations.

The bulk of this paper is devoted to computing the finite  intersection multiplicities that comprise the remaining 
contributions to the arithmetic intersection $ [ \widehat{\mathcal{Z}}(f) : \mathcal{Y} ] $.   More concretely, most of the paper consists
of the calculation of the degrees of the $0$-cycles $\mathcal{Y} \times_\mathcal{M} \mathcal{Z}(m,\mu)$ on $\mathcal{Y}$, and the comparision of 
these degrees with  the Fourier coefficients of the derivative $E'(\vec{\tau} ,0)$.

The first main new ingredient for the calculation, found in \S~\ref{ss:lubin-tate deformation}, is the computation of the deformation theory of certain `special' endomorphisms of Lubin-Tate formal groups, which, using Breuil-Kisin theory, we are able to do without any restriction on the ramification degree of the fields involved. This is a direct generalization of the seminal computations of Gross~\cite{Gross1986-ia} for Lubin-Tate groups associated with quadratic extensions of $\Q_p$.

The second new ingredient is the computation of certain $2$-adic Whittaker functions, which forms the bulk of \S~\ref{ss:whittaker functions}. 

The introduction to each section has some further explanation of its role in the proof of the main theorem.

\begin{remark}
The authors' earlier paper \cite{AGHMP} proves a result similar to Theorem \ref{bigthm:arithmetic BKY},  but for a 
cycle of \emph{small CM points}  $\mathcal{Y} \to \mathcal{M}$ defined by the inclusion of a rank $2$ torus into $G$.
In the present work the cycle of \emph{big CM points} $\mathcal{Y} \to \mathcal{M}$  is determined by a torus of maximal rank. 
One essential difference between these cases is that the big CM points  always have proper intersection (on the whole integral model $\mathcal{M}$)
with the special divisors $\mathcal{Z}(f)$.  Thus, unlike in \cite{AGHMP}, we do not have to deal with improper intersection.
\end{remark}

\begin{remark}
In the special case of $d=2$, results similar to Theorem \ref{bigthm:arithmetic BKY} can be found in the work of Yang 
\cite{Yang:colmez},  and of Yang and the third named author \cite{HY}.  
Note that when $d=2$ we are working on a Shimura variety of type $\GSpin(2,2)$, 
and this class of Shimura varieties includes the  classical Hilbert modular surfaces.  

The paper \cite{HowUnitaryCM} contains  results similar to  Theorem \ref{bigthm:arithmetic BKY}, but
on  the Shimura varieties associated with  unitary similitude groups instead of $\GSpin$.  
Of course the unitary case is easier, as those Shimura varieties can be realized as moduli spaces of abelian varieties. 
\end{remark}

%%%%%%%%%%%%%%%%%%%%%%%%%%%%%%%%%%

\subsection{From arithmetic intersection to Colmez's conjecture}\label{ss:intro colmez}

%%%%%%%%%%%%%%%%%%%%%%%%%%%%%%%%%%%

We explain how to deduce Theorem \ref{bigthm:average colmez} from Theorem \ref{bigthm:arithmetic BKY}, following  roughly the strategy of Yang \cite{Yang:colmez}.     
First, we choose the harmonic weak Maass form $f$ of Theorem \ref{bigthm:arithmetic BKY} so that $f$ is actually weakly holomorphic. In other words, 
we assume that
\[
f  = \sum _{  \substack{    m \gg -\infty \\ \mu \in L^\vee / L     }   } c_f(m,\mu) \varphi_\mu \cdot q^m  \in M^!_{2-d}(\omega_L) ,
\]
and so $\xi(f) = 0$ by the exact sequence (\ref{intro BF}).  
Combining Theorem \ref{bigthm:arithmetic BKY} with (\ref{intro constant})  gives 
\begin{equation}\label{intro holomorphic intersection}
 \frac{  [ \widehat{\mathcal{Z}}(f) : \mathcal{Y} ]   }{   \deg_\C (Y)  }  \approx_L 
  -  c_f ( 0,0)  \cdot \frac{  2\Lambda'(0,\chi)   } {  \Lambda( 0 , \chi )  } ,
\end{equation}
where $\approx_L$ means equality up to a $\Q$-linear combination of $\log(p)$ with $p\mid D_{bad,L}$.

The integral model $\mathcal{M}$ carries over it a line bundle 
$\bm{\omega}$ called the \emph{tautological bundle}, or the \emph{line bundle of weight one modular forms}.
Any $g\in G(\A_f)$ determines a uniformization 
\[
\mathcal{D} \map{z\mapsto (z,g) } M(\C)
\] 
of a connected component of the complex fiber of $\mathcal{M}$, and the line bundle $\bm{\omega}$ 
pulls back to the tautological bundle on $\mathcal{D}$,  whose fiber at $z$ is the isotropic line $\C z \subset V_\C$.  
If we now endow $\bm{\omega}$ with the metric $|| z ||^2 = -[z,\overline{z}]$, we obtain 
 a metrized line bundle
\[
\widehat{\bm{\omega} } \in \widehat{\mathrm{Pic}}(\mathcal{M}).
\]

For simplicity, assume that $d\geq 4$ (this guarantees that $V$ contains an isotropic line; throughout the body of the paper, we only require $d\geq 2$). 
After possibly replacing $f$ by a positive integer multiple, the theory of Borcherds products \cite{Bor98, Hormann, HMP}
gives us a rational section $\Psi(f)$ of the line bundle
$\bm{\omega}^{ \otimes c_f(0,0)}$, satisfying
\[
- \log|| \Psi(f)||^2 = \Phi(f)  -  c_f(0,0) \log(4\pi e^\gamma),
\]
and  satisfying $\mathrm{div}( \Psi(f) ) = \mathcal{Z}(f)$ \emph{up to a linear combination of irreducible components of the special fiber $\mathcal{M}_{\F_2}$}.

% This is the ultimate source of our $\log(2)$ error term: If the  special fiber $\mathcal{M}_{\F_2}$ has irreducible components 
% that do not meet the boundary of a toroidal compactification, then  one cannot detect the multiplicities of these components  in 
% $\mathrm{div}(\Psi(f))$ by examining the $q$-expansion of $\Psi(f)$.  

We define a Cartier divisor 
\[
\mathcal{E}_2(f) = \mathrm{div}(\Psi(f)) - \mathcal{Z}(f),
\]
on $\mathcal{M}$,  supported entirely in characteristic $2$, which should be viewed as an unwanted error term.  
Endowing this divisor with the trivial Green function, we obtain a metrized line bundle 
$\widehat{\mathcal{E}}_2(f)\in \widehat{\mathrm{Pic}}(\mathcal{M})$
satisfying 
\[
[ \widehat{\bm{\omega}}^{\otimes c_f(0,0) }  : \mathcal{Y} ]
= [ \widehat{\mathcal{Z}}(f) : \mathcal{Y}] -  c_f(0,0) \log(4\pi e^\gamma)  \cdot d \deg_\C(Y)
+ [ \widehat{\mathcal{E}}_2(f) : \mathcal{Y}].
\]
If we choose $f$ such that $c_f(0,0)\neq 0$, then combining this with (\ref{intro holomorphic intersection}) and dividing by $c_f(0,0)$ leaves
 \[
 \frac{  [ \widehat{\bm{\omega}} : \mathcal{Y} ]   }{   \deg_\C (Y)  }  
 + d\cdot  \log(4\pi e^\gamma) 
 \approx_L  -   \frac{  2  \Lambda'( 0 , \chi )   } {  \Lambda( 0 , \chi )  } 
 + \frac{1}{c_f(0,0)} \frac{[ \widehat{\mathcal{E}}_2(f) : \mathcal{Y}]}{ \deg_\C(Y)   }.
 \]

The pullback to $\mathcal{Y}$ of the metrized line bundle   $\widehat{\bm{\omega}}$  computes the averaged Faltings heights of abelian varieties with CM by $E$. More precisely, the cycle $\mathcal{Y}$ carries a canonical metrized line bundle $\widehat{\bm{\omega}}_0$ with two important properties: First, we show in Theorem~\ref{thm:Faltings height line bundles} that the arithmetic degree of $\widehat{\bm{\omega}}_0$ computes the averaged Faltings height:
\[
\frac{1}{2^{d-2}} \sum_\Phi h^\mathrm{Falt}_{(E,\Phi)} = \frac{\widehat{\deg}_{\mathcal{Y}}(\widehat{\bm{\omega}}_0)}{\deg_{\C}(Y)} + \log|D_F| - 2d\cdot\log(2\pi).
\]
Second, in Proposition~\ref{prop:colmez prelim bound}, we prove the approximate equality
\[
\frac{ [ \widehat{\bm{\omega}} : \mathcal{Y} ] } { \deg_\C(Y) }    
\approx_L  \frac{\widehat{\deg}_{\mathcal{Y}}(\widehat{\bm{\omega}}_0)}{\deg_{\C}(Y)} + \log|D_F|.
\]

Putting all this together, we find that
 \[
 \frac{1}{2^d} \sum_\Phi h^\mathrm{Falt}_{(E,\Phi)}
=    -  \frac{1}{2} \cdot \frac{  L'(0,\chi)   } { L ( 0 , \chi )  } - \frac{1}{4} \cdot \log\left| \frac{D_E}{D_F} \right| - 
    \frac{d}{2} \log(2\pi ) + \sum_p  b_E(p)  \log(p) 
 \]
for some rational numbers $b_E(p)$, with   $b_E(p)=0$ for all $p\nmid 2 D_{bad,L}$.

 The integer  $D_{bad ,L}$  depends on the choice of auxiliary $F$-quadratic space $(\mathscr{V} , \mathscr{Q})$
 and lattice $L$, and to show that $b_E(p)=0$, one only has to find \emph{some} choice of the auxiliary data for which $p\nmid  2 D_{bad,L}$.
 We show that for any prime $p$ the auxiliary data can be chosen so that $p\nmid D_{bad,L}$, and hence 
 $b_E(p)=0$ for all $p>2$.   This proves Theorem \ref{bigthm:average colmez},  except that we have not shown that $b_E(2) = 0$. 

 For this, we embed $L$ in a larger lattice $L^\beef$ that has rank $2d+2$, and which is self-dual at $2$. The integral model $\mathcal{M}^\beef$ of the Shimura variety associated with $L^\beef$ is a smooth integral canonical model in the sense of~\cite{KisinJAMS}.

 Using a result of Bruinier~\cite{Bru:prescribed}, we now pick a Borcherds lift $\Psi^\beef(f)$ over $\mathcal{M}^\beef$, whose divisor intersects $\mathcal{Y}$ properly, and which allows us to compute the height of the canonical bundle $\widehat{\omega}$ along $\mathcal{Y}$ even at the prime $2$. This enables us to prove that the constant $b_E(2)$ does indeed vanish.

%%%%%%%%%%%%%%%%%%%%%%%%%%%%%%%%%%

\subsection{Acknowledgements}

%%%%%%%%%%%%%%%%%%%%%%%%%%%%%%%%%%%

The authors thank Jan Bruinier, Pierre Colmez, Steve Kudla, Jacob Tsimerman, and Tonghai Yang for helpful conversations.

Parts of this research were carried out during the  various authors' visits to the 
Erwin Schr\"{o}dinger International Institute for Mathematical Physics, 
the Mathematisches Forschungsinstitut Oberwolfach,
the Banff International Research Station,  
the Centre de Researches Math\'ematiques, and   the Universit\`{a} Statale di Milano.
The authors thank these institutions for their hospitality.

%%%%%%%%%%%%%%%%%%%%%%%%%%%%%%%%%%

\section{Special endomorphisms of Lubin-Tate groups}\label{s:lubin_tate}

%%%%%%%%%%%%%%%%%%%%%%%%%%%%%%%%%%%

In this section, we generalize the results of Gross~\cite{Gross1986-ia} to Lubin-Tate groups over arbitrary finite extensions  $K/\Q_p$. Namely, we study the deformation theory of certain `special' endomorphisms of such groups.

This generalization, which appears as Theorem~\ref{thm:deformation special}, is the basis for the local intersection theory calculations underlying the proof of our main technical result, Theorem~\ref{thm:arithmetic BKY}. To be able to avoid restrictions on the ramification index of $K$, we are compelled to employ the theory of Breuil-Kisin modules~\cite{kisin:f_crystals}. This allows us to give a uniform treatment of all relevant cases.

The reader uninterested in the nitty gritty of $p$-adic Hodge theory, wanting only to understand the statement of Theorem~\ref{thm:deformation special}, can find the relevant definitions in the first paragraphs of \S~\ref{ss:lubin-tate_special_endomorphisms},~\ref{ss:denominators}, and~\ref{ss:lubin-tate deformation}.

%%%%%%%%%%%%%%%%%%%%%%%%%%%%%%%%%%

\subsection{Breuil-Kisin modules and $p$-divisible groups}\label{ss:breuil_kisin}

%%%%%%%%%%%%%%%%%%%%%%%%%%%%%%%%%%%

Fix a prime $p$. Let $\Q_p^{\alg}$ be an algebraic closure of $\Q_p$, and let $\C_p$ be its completion. Set $W=W(\F_p^{\alg})$ and let $\mathrm{Frac}(W)^{\alg}\subset\C_p$ be the algebraic closure of its fraction field $\mathrm{Frac}(W)$. Let $K\subset\mathrm{Frac}(W)^{\alg}$ be a finite extension of $\mathrm{Frac}(W)$, and let 
\[
\Gamma_K = \Gal(\mathrm{Frac}(W)^{\alg}/K)
\]
 be its absolute Galois group.

Set $\mathfrak{S}=W\pow{u}$, the power series ring over $W$ in the variable $u$. Fix a uniformizer $\varpi\in\co_K$ and let $\mathcal{E}(u)\in W[u]$ be the associated Eisenstein polynomial satisfying $\mathcal{E}(0) = p$. A \emph{Breuil-Kisin module} over $\co_K$ (with respect to $\varpi$) is a pair $(\mathfrak{M},\varphi_{\mathfrak{M}})$, where $\mathfrak{M}$ is a finite free $\mathfrak{S}$-module and 
\[
\varphi_{\mathfrak{M}}:\varphi^*\mathfrak{M}[\mathcal{E}^{-1}]\xrightarrow{\simeq}\mathfrak{M}[\mathcal{E}^{-1}]
\]
 is an isomorphism of $\mathfrak{S}$-modules. Here, $\varphi:\mathfrak{S}\to\mathfrak{S}$ is the Frobenius lift that extends the canonical Frobenius automorphism $\mathrm{Fr}:W\to W$ and satisfies $\varphi(u)=u^p$. 

Usually, the map $\varphi_{\mathfrak{M}}$ will be clear from context and we will denote the Breuil-Kisin module by its underlying $\mathfrak{S}$-module $\mathfrak{M}$.

We will write $\mathbf{1}$ for the Breuil-Kisin module whose underlying $\mathfrak{S}$-module is just $\mathfrak{S}$ equipped with the canonical identification $\varphi^*\mathfrak{S} = \mathfrak{S}$.

By~\cite{KisinJAMS}, there is a fully faithful tensor functor $\mathfrak{M}$ from the category of $\Z_p$-lattices in crystalline $\Gamma_K$-representations to the category of Breuil-Kisin modules over $\co_K$. It has various useful properties. To describe them, fix a crystalline $\Z_p$-representation $\Lambda$. Then:

\begin{itemize}
% \item If $\Lambda = \Z_p$ is the trivial rank $1$-representation of $\Gamma_K$, then there is a natural identification
% \begin{equation}\label{bk:trivial}
% \mathfrak{M}(\Z_p) = \mathbf{1}.
% \end{equation}

% \item There is a canonical isomorphism of Breuil-Kisin modules:
% \begin{equation}\label{bk:duality}
% \mathfrak{M}(\Lambda^\vee)\xrightarrow{\simeq} \mathfrak{M}(\Lambda)^\vee.
% \end{equation}

% \item For any $i\in\Z_{\geq 0}$, there are canonical isomorphisms:
% \begin{align}\label{bk:symm_ext}
% \mathfrak{M}(\mathrm{Sym}^i\Lambda) \xrightarrow{\simeq} \mathrm{Sym}^i\mathfrak{M}(\Lambda);\\
% \mathfrak{M}(\wedge^i\Lambda) \xrightarrow{\simeq} \wedge^i\mathfrak{M}(\Lambda)\nonumber
% \end{align}
% of Breuil-Kisin modules.

\item There is a canonical isomorphism of $F$-isocrystals over $\mathrm{Frac}(W)$:
	\begin{equation}\label{bk:dcris}
     \mathfrak{M}(\Lambda)/u\mathfrak{M}(\Lambda)[p^{-1}]\xrightarrow{\simeq}D_{\cris}(\Lambda)=(\Lambda\otimes_{\Z_p}B_{\cris})^{\Gamma_K}.
	\end{equation}
\item If we equip $\varphi^*\mathfrak{M}(\Lambda)$ with the descending filtration $\Fil^\bullet\varphi^*\mathfrak{M}$ given by
\[
\Fil^i\varphi^*\mathfrak{M}(\Lambda) = \{x\in\varphi^*\mathfrak{M}(\Lambda):\;\varphi_{\mathfrak{M}(\Lambda)}(x)\in\mathcal{E}(u)^i\mathfrak{M}(\Lambda)\},
\] 
then there is a canonical isomorphism of filtered $K$-vector spaces
\begin{equation}\label{bk:ddr}
(\varphi^*\mathfrak{M}(\Lambda)/\mathcal{E}(u)\varphi^*\mathfrak{M}(\Lambda))[p^{-1}]\xrightarrow{\simeq}K\otimes_{\mathrm{Frac}(W)}D_{\cris}(\Lambda).
\end{equation}
Here, the left hand side is equipped with the filtration induced from $\Fil^\bullet\varphi^*\mathfrak{M}(\Lambda)$.  
\end{itemize}

Kisin's functor can be used to classify $p$-divisible groups over $\co_K$. This was done by Kisin himself~\cite{kisin:f_crystals} when $p>2$, and the case $p=2$ was dealt with by W. Kim~\cite{kim:2-adic}. We now present a summary of their results.

%First, by a theorem of Tate, the composite functor
%\[
%\mathcal{H}\mapsto T_p(\mathcal{H})\mapsto\mathfrak{M}(T_p(\mathcal{H}))
%\]
%is a fully faithful functor from the category of $p$-divisible groups over $\co_{E}$ to the category of Breuil-Kisin modules. 

We will say that $\mathfrak{M}$ has \emph{$\mathcal{E}$-height $1$} if the isomorphism $\varphi_{\mathfrak{M}}$ arises from a map $\varphi^*\mathfrak{M}\to\mathfrak{M}$ whose cokernel is killed by $\mathcal{E}(u)$.

Let $S\to\co_K$ be the $p$-adic completion of the divided power envelope of the surjection 
\[
W[u]\xrightarrow{u\mapsto\varpi}\co_K.
\] 
The natural map $W[u]\to S$ extends to an embedding $\mathfrak{S}\hookrightarrow S$, and the Frobenius lift $\varphi:\mathfrak{S}\to\mathfrak{S}$ extends continuously to an endomorphism $\varphi:S\to S$. 

Write $\Fil^1S\subset S$ for the kernel of the map $S\to\co_K$. If $\mathfrak{M}$ is a Breuil-Kisin module of $\mathcal{E}$-height $1$, and $\mathcal{M}=S\otimes_{\varphi,\mathfrak{S}}\mathfrak{M}$, we will set
\begin{equation}\label{BK filtration}
\Fil^1\mathcal{M} =\{x\in \mathcal{M}=S\otimes_{\mathfrak{S}}\varphi^*\mathfrak{M}:\;(1\otimes\varphi_{\mathfrak{M}})(x)\in \Fil^1S\otimes_{\mathfrak{S}}\mathfrak{M}\subset S\otimes_{\mathfrak{S}}\mathfrak{M}\}.
\end{equation}
The image of $\Fil^1\mathcal{M}$ in $\co_K\otimes_S\mathcal{M}=\co_K\otimes_{\mathfrak{S}}\varphi^*\mathfrak{M}$ is a $\co_K$-linear direct summand, and so equips the ambient space with a two-step descending filtration.

For any $p$-divisible group $\mathcal{H}$ over a $p$-adically complete ring $R$, we will consider the contravariant Dieudonn\'e $F$-crystal $\mathbb{D}(\mathcal{H})$ associated with $\mathcal{H}$ (see for instance~\cite{BerthelotBreenMessingII}). 

Given any nilpotent thickening $R'\to R$, whose kernel is equipped with divided powers, we can evaluate $\mathbb{D}(\mathcal{H})$ on $R'$ to obtain a finite projective $R'$-module $\mathbb{D}(\mathcal{H})(R')$ (this construction depends on the choice of divided power structure, which will be specified or evident from context). If $R'$ admits a Frobenius lift $\varphi:R'\to R'$, then we get a canonical map
\[
\varphi:\varphi^*\mathbb{D}(\mathcal{H})(R')\to \mathbb{D}(\mathcal{H})(R')
\]
obtained from the $F$-crystal structure on $\mathbb{D}(\mathcal{H})$. 

An example of a (formal) divided power thickening  is any surjection of the form $R'\to R'/pR'$, where we equip $pR'$ with the canonical divided power structure induced from that on $p\Z_p$. Another example is the surjection $S\to \co_K$ considered above.

The evaluation on the trivial thickening $R\to R$ gives us a projective $R$-module $\mathbb{D}(\mathcal{H})(R)$ of finite rank equipped with a short exact sequence of projective $R$-modules:
\[
0 \to \Lie (\mathcal{H})^\vee\to \mathbb{D}(\mathcal{H})(R) \to \Lie(\mathcal{H}^\vee) \to 0.
\]
We will set
\[
\Fil^1\mathbb{D}(\mathcal{H})(R)\define \Lie ( \mathcal{H})^\vee\subset \mathbb{D}(\mathcal{H})(R), 
\]
and term it the \emph{Hodge filtration}.

\begin{theorem}\label{thm:kisin_p_divisible}
For any $p$-divisible group $\mathcal{H}$ over $\co_K$, write $\mathcal{H}^\vee$ for its Cartier dual. Then the functor $\mathcal{H}\mapsto\mathfrak{M}(T_p(\mathcal{H}^\vee))$ 
is an exact contravariant equivalence of categories from the category of $p$-divisible groups over $\co_K$ 
to the category of Breuil-Kisin modules of $\mathcal{E}$-height $1$. Moreover, if we abbreviate
\[
\mathfrak{M}(\mathcal{H}) \define \mathfrak{M}(T_p(\mathcal{H})^\vee),
\] 
then the functor has the following properties:
\begin{enumerate}
	\item\label{kisin:cris}The $\varphi$-equivariant composition
	\begin{align*}
    \varphi^*\mathfrak{M}(\mathcal{H})/u\varphi^*\mathfrak{M}(\mathcal{H}) &\xrightarrow{\varphi_{\mathfrak{M}(\mathcal{H})}}    \mathfrak{M}(\mathcal{H})/u\mathfrak{M}(\mathcal{H})[p^{-1}]\\
    &\xrightarrow[\simeq]{\eqref{bk:dcris}}	D_{\cris}(T_p(\mathcal{H})^\vee)	\xrightarrow{\simeq}	\mathbb{D}(\mathcal{H})(W)[p^{-1}]
	\end{align*}
	maps $\varphi^*\mathfrak{M}(\mathcal{H})/u\varphi^*\mathfrak{M}(\mathcal{H})$ isomorphically onto $\mathbb{D}(\mathcal{H})(W)$. Here, in a slight abuse of notation, we write $\mathbb{D}(\mathcal{H})(W)$ for the evaluation on $W$ of the Dieudonn\'e $F$-crystal associated with the reduction of $\mathcal{H}$ over $\F_p^\alg$.
	\item\label{kisin:derham}The filtered isomorphism
	\begin{align*}
    \varphi^*\mathfrak{M}(\mathcal{H})/\mathcal{E}(u)\varphi^*\mathfrak{M}(\mathcal{H})[p^{-1}]&\xrightarrow[\simeq]{\eqref{bk:ddr}}K\otimes_{\mathrm{Frac}(W)}D_{\cris}(T_p(\mathcal{H})^\vee)\\
    &\xrightarrow{\simeq}\mathbb{D}(\mathcal{H})(\co_K)[p^{-1}]
	\end{align*}
	maps $\varphi^*\mathfrak{M}(\mathcal{H})/\mathcal{E}(u)\varphi^*\mathfrak{M}(\mathcal{H})$ isomorphically onto $\mathbb{D}(\mathcal{H})(\co_K)$.
	\item\label{kisin:breuil}There is a canonical $\varphi$-equivariant isomorphism
	\[
    S\otimes_{\varphi,\mathfrak{S}}\mathfrak{M}(\mathcal{H})\xrightarrow{\simeq}\mathbb{D}(\mathcal{H})(S)
	\] 
	whose reduction along the map $S\to\co_K$ gives the filtration preserving isomorphism in~\eqref{kisin:derham}.
\end{enumerate}
\end{theorem}
\begin{proof}
This follows from~\cite[Theorem 1.1.6]{KisinModp}, using the work of Kim~\cite{kim:2-adic} when $p=2$. Note that this corrects an error in the statement of~\cite[Theorem 1.4.2]{KisinJAMS}, which is off by a Tate twist.
\end{proof}

\subsection{Lubin-Tate groups}\label{ss:lubin_tate}

%%%%%%%%%%%%%%%%%%%%%%%%%%%%%%%%%%%

Fix a finite extension $E$ of $\Q_p$, and a uniformizer $\pi_E\in E$. Let $e(X)\in\co_E[X]$ be a Lubin-Tate polynomial associated with $\pi_E$, so that 
\begin{align*}
e(X)&\equiv\pi_EX\pmod{X^2}, \\
e(X)&\equiv X^q\pmod{\pi_E}.
\end{align*}
Here, $q=\# k_E$ is the size of the residue field $k_E$ of $E$.

Let $\mathcal{G}=\mathrm{Spf}( \co_E\pow{X})$ be the unique formal $\co_E$-module in one variable over $\co_E$ with multiplication by $\pi_E$ given by the polynomial $[\pi_E](X)=e(X)$. For each $n\in\Z_{>0}$, write $\mathcal{G}[\pi_E^n]$ for the $\pi_E^n$-torsion $\co_E$-submodule of $\mathcal{G}$. These fit into a $\pi_E$-divisible group $\mathcal{G}[\pi_E^\infty]$ over $\co_E$.

Assume now that we have an embedding $E\hookrightarrow K$. We  obtain a formal $\co_E$-module $\mathcal{G}_{\co_K}$ over $\co_K$ and a $\pi_E$-divisible group $\mathcal{G}[\pi_E^\infty]_{\co_K}$. For simplicity, we will omit $\co_K$ from the subscripts in what follows, and so will be viewing both $\mathcal{G}$ and $\mathcal{G}[\pi_E^\infty]$ as objects over $\co_K$.

Let
\[
T_{\pi_E}(\mathcal{G}) \define \varprojlim_{n}\mathcal{G}[\pi_E^n](\mathrm{Frac}(W)^{\alg}).
\]
be the $\pi_E$-adic Tate module associated with $\mathcal{G}[\pi_E^\infty]$.  This a crystalline $\Z_p$-representation of $\Gamma_K$ equipped with an $\co_E$-action, making it an $\co_E$-module of rank $1$. We will now describe the associated  Breuil-Kisin module 
\[
\mathfrak{M}(\mathcal{G}) \define \mathfrak{M}\bigl(T_{\pi_E}(\mathcal{G})^\vee\bigr)
\]
with its $\co_E$-action.
This will involve constructing an explicit candidate $\mathfrak{M}$ for such a module, and then showing that this candidate is indeed isomorphic to $\mathfrak{M}(\mathcal{G})$.

Let $E_0\subset E$ be the maximal unramified subextension, and let $\mathrm{Emb}(E_0)$ be the set of embeddings $E_0\hookrightarrow \mathrm{Frac}(W)$. Let $\iota_0\in\mathrm{Emb}(E_0)$ be the distinguished element induced by the embedding $E\hookrightarrow K$. The Frobenius automorphism $\mathrm{Fr}$ of $W$ acts on $\mathrm{Emb}(E_0)$, and every $\iota\in\mathrm{Emb}(E_0)$ is of the form $\mathrm{Fr}^i(\iota_0)$ for a unique $i\in\{0,1,\ldots,d_0-1\}$, where $d_0 = [E_0:\Q_p]$.

The underlying $\co_E$-equivariant $\mathfrak{S}$-module for our candidate is
\[
\mathfrak{M} = \mathfrak{S}\otimes_{\Z_p}\co_E = \bigoplus_{\iota\in\mathrm{Emb}(E_0)}\mathfrak{S}\otimes_{\iota,\co_{E_0}}\co_E = \bigoplus_{\iota}\mathfrak{S}_{\iota},
\]
where, for $\iota\in\mathrm{Emb}(E_0)$, we have set $W_{\iota} = W\otimes_{\iota,\co_{E_0}}\co_E$ and $\mathfrak{S}_{\iota}=W_{\iota}\pow{u}$. There is a canonical $\co_E$-equivariant identification of $\mathfrak{S}$-modules 
\[
\varphi^*\mathfrak{M} = \bigoplus_{\iota}\varphi^*\mathfrak{S}_{\mathrm{Fr}^{-1}(\iota)} = \bigoplus_{\iota}\mathfrak{S}_{\iota} = \mathfrak{M}.
\]

The $\mathfrak{S}\otimes_{\Z_p}\co_E$-equivariant isomorphism $\varphi_{\mathfrak{M}}$ will now arise from a map
\[
\varphi_{\mathfrak{M}}:\varphi^*\mathfrak{M}=\mathfrak{M}\xrightarrow[\beta]{\simeq}\mathfrak{M},
\]
for some
\[
\beta\in (\mathfrak{S}\otimes_{\Z_p}\co_E)\cap (\mathfrak{S}[\mathcal{E}^{-1}]\otimes_{\Z_p}\co_E)^\times = \prod_{\iota\in \mathrm{Emb}(E_0)}\mathfrak{S}_\iota\cap\mathfrak{S}_{\iota}[\mathcal{E}^{-1}]^\times.
\]
To describe $\beta$ explicitly, we have to specify each of its components 
\[
\beta_{\iota}\in \mathfrak{S}_{\iota}\cap \mathfrak{S}_{\iota}[\mathcal{E}^{-1}]^\times\subset \mathfrak{S}_\iota[\mathcal{E}^{-1}].
\]
Let $\mathcal{E}_{\iota_0}(u)\in W_{\iota_0}[u]$ be the Eisenstein polynomial for $\varpi$ over $W_{\iota_0}$ satisfying $\mathcal{E}_{\iota_0}(0) = \iota_0(\pi_E)$, and set
\begin{align*}
\beta_{\iota}&=\begin{cases}
	\mathcal{E}_{\iota_0}(u) & \text{ if $\iota=\iota_0$} \\
	1& \text{otherwise.}
\end{cases}
\end{align*}

From $\mathfrak{M}$, we obtain an abstract `crystalline' realization
\begin{align}\label{eqn:mfM_cris}
M_{\cris}&\define W\otimes_{\mathfrak{S}}\varphi^*\mathfrak{M} = W\otimes_{\Z_p}\co_E = \bigoplus_{\iota}W_{\iota},
\end{align}
where we view $W$ as an $\mathfrak{S}$-algebra via $u\mapsto 0$. This also identifes $\mathrm{Fr}^*M_{\cris}$ with $\bigoplus_{\iota}W_{\iota}$. Under these identifications, the $F$-crystal structure on $M_{\cris}$ is given by multiplication by the image of $\beta$ under
 \[
\mathfrak{S}\otimes_{\Z_p}\co_E\xrightarrow{\varphi\otimes 1} \mathfrak{S} \otimes_{\Z_p} \co_E  \xrightarrow{u\otimes 1 \mapsto 0} W\otimes_{\Z_p}\co_E.
 \] 
 This image is easy to describe: Its $\iota$-component is $1$ when $\iota\neq \mathrm{Fr}(\iota_0)$, while its $\mathrm{Fr}(\iota_0)$-component is $\mathrm{Fr}(\iota_0(\pi_E))\in W_{\mathrm{Fr}(\iota_0)}$.

Similarly, we  obtain an abstract `de Rham' realization
\begin{align}\label{eqn:mfM_dR}
M_{\dR}&\define \co_K\otimes_{\mathfrak{S}}\varphi^*\mathfrak{M},
\end{align}
where we view $\co_K$ as an $\mathfrak{S}$-algebra via $u\mapsto \varpi$.
Write $M_{\iota}$ for the $\iota$-isotypic component of $M_{\dR}$;  this is simply 
\[
\co_{K,\iota}\define\co_K\otimes_{\iota,\co_{E_0}}\co_E
\]
 viewed as a module over itself. 

The recipe in (\ref{bk:ddr}) also gives us a direct summand $\Fil^1M_{\dR}\subset M_{\dR}$. 
This is an $\co_K\otimes_{\Z_p}\co_E$-stable submodule, and so it suffices to specify its $\iota$-isotypic component $\Fil^1M_{\iota}\subset M_{\iota}$ for each $\iota$. 
To do this, we first need to describe the subspace $\Fil^1\varphi^*\mathfrak{M}$. 
By definition, we have
\[
\Fil^1\varphi^*\mathfrak{M} = \{x\in\varphi^*\mathfrak{M}:\varphi_{\mathfrak{M}}(x)\in\mathcal{E}(u)\mathfrak{M}\}.
\]
From this, we deduce
\begin{align}\label{eqn:fil_mfM}
\bigl(\Fil^1\varphi^*\mathfrak{M}\bigr)_{\iota}&=\begin{cases}
\{x\in\mathfrak{S}_{\iota}:   \mathcal{E}_{\iota_0}(u)x\in\mathcal{E}(u)\mathfrak{S}_{\iota}\}  &
\text{ if $\iota=\iota_0$} \\
\mathcal{E}(u)\mathfrak{S}_{\iota}  & \text{otherwise}.
\end{cases}
\end{align}
Reducing mod $\mathcal{E}(u)$, we now find
\begin{align}\label{eqn:fil1_H}
\Fil^1M_{\iota} &= \begin{cases}
\{x\in M_{\iota}:\mathcal{E}_{\iota_0}(\varpi\otimes 1)x=0\} & \text{ if $\iota=\iota_0$}  \\
0 & \text{ otherwise}.
\end{cases}
\end{align}

Clearly, $\mathfrak{M}$ has $\mathcal{E}$-height $1$. Therefore, by Theorem~\ref{thm:kisin_p_divisible}, there exists a $p$-divisible group $\mathcal{H}$ over $\co_K$, equipped with an $\co_E$-module structure, such that $\mathfrak{M} = \mathfrak{M}(\mathcal{H})$. Moreover, the $\co_E$-action on $\mathfrak{M}$ translates to an $\co_E$-action on $\mathcal{H}$. 

Let $T_{\pi_E}(\mathcal{H})$ be the $\pi_E$-adic Tate module over $K$ associated with $\mathcal{H}$. We can use the explicit descriptions of $M_{\cris}$ and $M_{\dR}$ above to obtain a description of the associated $E$-equivariant, filtered $\varphi$-module 
\[
D_{\cris}\define D_{\cris}(T_{\pi_E}(\mathcal{H})^\vee).
\]
The underlying $\mathrm{Frac}(W)$-vector space is
\[
D_{\cris} = M_{\cris}[p^{-1}] = \mathrm{Frac}(W)\otimes_{\Q_p}E = \bigoplus_{\iota}\mathrm{Frac}(W)_{\iota}.
\]
As before, this description also identifies $\mathrm{Fr}^*D_{\cris}$ with $\bigoplus_{\iota}\mathrm{Frac}(W)_{\iota}$, and under these identifications, the $F$-isocrystal structure on $D_{\cris}$ is given simply by multiplication by $\pi_{\iota_0}$ on the $\iota_0$-factor, and the identity on the remaining factors.

To complete our description, we need to know the subspace 
\[\Fil^1 D_{\dR}\subset D_{\dR} = K\otimes_{\mathrm{Frac}(W)}D_{\cris}.\] Let $D_{\iota}\subset D_{\dR}$ be the $\iota$-isotypic component. This is a rank $1$ free module over $K_{\iota} = K\otimes_WW_\iota$. Note that we have a quotient map
\begin{equation}\label{eqn:iota_0_quotient}
K_{\iota_0} = K\otimes_{\iota_0,E_0}E\to K
\end{equation}
induced by the distinguished embedding $E\hookrightarrow K$. This gives us an idempotent projector $\bm{e}_0:K_{\iota_0}\to K_{\iota_0}$ such that (\ref{eqn:iota_0_quotient})  identifies $\bm{e}_0 K_{\iota_0} \iso K$. From~\eqref{eqn:fil1_H}, we now have
\[
\Fil^1 D_{\dR} = \bigoplus_{\iota}\Fil^1 D_{\iota}\subset \bigoplus_{\iota}D_{\iota},
\]
where
\begin{align*}
%\label{eqn:fil1_Ddr}
\Fil^1D_{\iota} &= \begin{cases}
\bm{e}_0D_{\iota} & \text{ if $\iota=\iota_0$} \\
0 & \text{otherwise}.
\end{cases}
\end{align*}

\begin{proposition}\label{prp:bk_lubin-tate}
There is a $\co_E$-equivariant isomorphism
\[
\mathfrak{M}(\mathcal{G}) \xrightarrow{\simeq}  \mathfrak{M}
\]
of Breuil-Kisin modules.
In particular, we have $\co_E$-equivariant isomorphisms
\begin{align*}
\mathbb{D}(\mathcal{G})(W)\xrightarrow{\simeq}M_{\cris}, \qquad
\mathbb{D}(\mathcal{G})(\co_K)&\xrightarrow{\simeq}M_{\dR}
\end{align*}
of $F$-crystals over $W$ and filtered $\co_K$-modules, respectively.
\end{proposition}
\begin{proof}

The first assertion of the proposition amounts to showing that we have an $\co_E$-equivariant isomorphism
\[
T_{\pi_E}(\mathcal{H})\xrightarrow{\simeq} T_{\pi_E}(\mathcal{G})
\]
of $\pi_E$-adic Tate modules over $K$. In fact, since all $\co_E$-lattices in $T_{\pi_E}(\mathcal{G})[p^{-1}]$ are simply dilations of $T_{\pi_E}(\mathcal{G})$ by powers of $\pi_E$, it is enough to show that we have an $\co_E$-equivariant isomorphism
\[
T_{\pi_E}(\mathcal{H})[p^{-1}]\xrightarrow{\simeq} T_{\pi_E}(\mathcal{G})[p^{-1}].
\]

To do this, we will show that the admissible filtered $\varphi$-modules associated with the two representations are $\co_E$-equivariantly isomorphic. We have already computed the admissible $\varphi$-module $D_{\cris}$ associated with $T_{\pi_E}(\mathcal{H})^\vee$. So we have to check that it agrees with that obtained from $T_{\pi_E}(\mathcal{G})^\vee$. This follows from   \cite[Lemma 1.22]{rapoport_zink}.

The last assertion follows from the first via~\eqref{kisin:cris}, and~\eqref{kisin:derham} of Theorem~\ref{thm:kisin_p_divisible}.
\end{proof}

\subsection{Special endomorphisms}
\label{ss:lubin-tate_special_endomorphisms}

%%%%%%%%%%%%%%%%%%%%%%%%%%%%%%%%%%%%%%%%%%%%%%%

We will assume that $E$ is equipped with a non-trivial involution $\tau$. Let $F\subset E$ be the fixed field of $\tau$. If $E/F$ is unramified, then we will further assume that the uniformizer $\pi_E$ is in fact a uniformizer in $F$. Given a $p$-adically complete $\co_{E}$-algebra $R$, a \emph{special endomorphism} of $\mathcal{G}_R$ will be an element $f\in\End(\mathcal{G}_R)$ such that 
\[
f([a](X))=[\tau(a)](f(X)),
\] 
for any $a\in\co_{E}$. Write $V(\mathcal{G}_R)$ for the space of special endomorphisms of $\mathcal{G}$. 

The following proposition is clear.

\begin{proposition}\label{prop:lubin-tate_special_end}
\mbox{}
\begin{enumerate}
	\item The subspace $V(\mathcal{G}_R)\subset\End(\mathcal{G}_R)$ is $\co_{E}$-stable. If it is non-zero, then it is a finite free $\co_{E}$-module of rank $1$.
	\item For any $x_1,x_2\in V(\mathcal{G}_R)$, there exists a unique $\langle x_1,x_2\rangle\in\co_{E}$ such that 
\[
	x_1\circ x_2=[\langle x_1,x_2\rangle]\in\End_{\co_{E}}(\mathcal{G}_R).
\]
	\item The pairing $(x_1,x_2)\mapsto \langle x_1,x_2\rangle$ is a Hermitian pairing on $V(\mathcal{G}_R)$.
\end{enumerate}
\end{proposition}

It will be useful to have the following notation: Let $R$ be a commutative ring with a non-trivial involution $\tau$. For any $R$-module $M$, we will set:
\begin{equation*}
%\label{eqn:V M tau}
V(M,\tau) = \{f\in\End(M):f(a\cdot m)=\tau(a)f(m)\text{, for all $a\in R$}\}.
\end{equation*}
This is an $R$-submodule of $\End(M)$, where we equip the latter with the $R$-module structure obtained from post-composition with scalar multiplication by $R$. 

The embedding $\iota_0\in \mathrm{Emb}(E_0)$ induces an embedding $k_E\hookrightarrow\F_p^{\alg}$. Set
\[
V(\mathcal{G}_1) = V(\mathcal{G}_{\F^{\alg}_p})
\]
and
\[
V_{\cris}(\mathcal{G}) = V(M_\cris,\tau),
\]
where we view $M_\cris$ as an $\co_E\otimes_{\Z_p}W$-module.

We now have the following easy lemma, whose proof we omit.
\begin{lemma}\label{lem:V_cris_G}\
\begin{enumerate}
\item $V_{\cris}(\mathcal{G})$ is an $\co_E$-stable subspace of $\End(M_{\cris})$, which is free of rank $1$ over $W\otimes_{\Z_p}\co_{E}$. Conjugation by $\varphi_0:\mathrm{Fr}^*M_\cris\to M_\cris$ induces a $\co_E\otimes_{\Z_p}W$-linear automorphism
\[
\varphi: \mathrm{Fr}^*V_\cris(\mathcal{G})[p^{-1}]\xrightarrow{\simeq} V_\cris(\mathcal{G})[p^{-1}].
\]
\item There is a canonical identification
\[
V(\mathcal{G}_1) = V_{\cris}(\mathcal{G})^{\varphi=1}
\]
of $V(\mathcal{G}_1)$ with the $\varphi$-equivariant elements in $V_{\cris}(\mathcal{G})$.
\item For $x,y\in V_{\cris}(\mathcal{G})$, $x\circ y\in \End(M_{\cris})$ corresponds to multiplication by an element $\langle x,y\rangle\in W\otimes_{\Z_p}\co_E$. The assignment
\[
(x,y)\mapsto \langle x,y\rangle \in W\otimes_{\Z_p}\co_E
\]
is a $\tau$-Hermitian form on $V_{\cris}(\mathcal{G})$, which restricts to the canonical $\co_E$-valued Hermitian form on $V(\mathcal{G}_1)$ from Proposition~\ref{prop:lubin-tate_special_end}.
\end{enumerate}
\end{lemma}

It will be useful to have an explicit description of $V_{\cris}(\mathcal{G})$ along with that of the conjugation action of the semi-linear endomorphism $\varphi_0$ of $M_{\cris}$. This is easily deduced from the explicit description of $M_{\cris}$ from~\eqref{eqn:mfM_cris}. 

For each $\iota\in \mathrm{Emb}(E_0)$, set $V_{\iota} = W\otimes_{\iota,\co_{E_0}}V(\co_{E},\tau)$.
This is a rank $1$ free module over $W_{\iota}$. Using~\eqref{eqn:mfM_cris}, we now obtain a canonical $\co_{E}$-equivariant identification
\[
V_{\cris}(\mathcal{G}) = W\otimes_{\Z_p}V(\co_{E},\tau) = \bigoplus_{\iota\in \mathrm{Emb}(E_0)}V_{\iota}.
\]
This also identifes $\mathrm{Fr}^*V_{\cris}(\mathcal{G})$ with $\bigoplus_{\iota\in \mathrm{Emb}(E_0)}V_{\iota}$.

As before, set $d_0=[E_0:\Q_p]$. Any element of $V_{\cris}(\mathcal{G})$ is a tuple of the form $f=(f_i)_{0\leq i\leq d_0-1}$ for some $a_i\in V_{\mathrm{Fr}^i(\iota_0)}$, and 
\[
\varphi(f)_i = \eta_i\varphi(f_{i-1})\in V_{\mathrm{Fr}^i(\iota_0)}[p^{-1}], 
\]
for certain $\eta_i\in \mathrm{Frac}(W_{\mathrm{Fr}^i(\iota_0)})$.

To pin the $\eta_i$ down, first consider the case where $E$ is unramified over $F$. In this case, $\pi_E$ is a uniformizer for $F$ by hypothesis, and hence satisfies $\tau(\pi_E) = \pi_E$. Also, $\tau$ acts non-trivially on $\mathrm{Emb}(E_0)$: If $r\in\Z_{\geq 1}$ is such that $2r=d_0$, we have, for any $\iota\in \mathrm{Emb}(E_0)$, 
\[
\mathrm{Fr}^r(\iota)=\tau(\iota) \define \iota\circ\tau.
\]
We can now identify 
\[
V_{\iota}=\Hom_{W_{\iota}}(W_{\tau(\iota)},W_{\iota}),
\]
as $W_{\iota}$-modules. Here, we view $W_{\iota}$ as acting on $W_{\tau(\iota)}$ via the isomorphism $W_{\iota}\xrightarrow{\simeq}W_{\tau(\iota)}$ induced by $\tau$. 

Now, as seen in~\eqref{ss:lubin_tate}, the $F$-crystal structure on $M_\cris$ corresponds under the identification~\eqref{eqn:mfM_cris} to multiplication by the element $\beta_0\in W\otimes_{\Z_p}\co_E$, whose $\iota_0$-isotypic component is $1\otimes\pi_E$, and whose $\iota$-isotypic component for $\iota\neq \iota_0$ is $1$. From this we deduce:
\begin{align*}
	\eta_i&=\begin{cases}
		1\otimes \pi_E  & \text{ if $i=1$} \\
		1\otimes \pi_E^{-1}  &  \text{ if $i=r+1$}\\
		1 & \text{otherwise}.
	\end{cases}
\end{align*}

Using this, we easily obtain the following explicit description of the space $V(\mathcal{G}_1)\subset V_{\cris}(\mathcal{G})$.

\begin{proposition}\label{prp:unramified_vcris}
When $E/F$ is unramified, $V(\mathcal{G}_1)\subset V_{\cris}(\mathcal{G})$ consists precisely of the elements $f=(f_i)$ such that:
\begin{itemize}
	\item $f_{0}\in V_{\iota_0}^{\mathrm{Fr}^{d_0}=1} = V(\co_E,\tau)$;
	\item $f_i = \mathrm{Fr}^i((1\otimes\pi_E)a_0)$\text{, for $1\leq i\leq r$};
	\item $f_i = \mathrm{Fr}^i(a_0)\text{, for $r+1\leq i\leq 2r-1$}$.
\end{itemize}
In particular, we have an isometry
\begin{align*}
%\label{eqn:v0_expl_unramified}
(V(\mathcal{G}_1),\langle\cdot,\cdot\rangle)&\xrightarrow{\simeq} (V(\co_E,\tau),\pi_E\langle \cdot,\cdot\rangle)\\
f&\mapsto f_0\nonumber
\end{align*}
of Hermitian $\co_E$-modules, where, for $x,y\in V(\co_E,\tau)$, $\langle x,y\rangle\in\co_E$ is the element such that $x\circ y\in\End(\co_E)$ is multiplication by $\langle x,y\rangle$.
\end{proposition}

Let us now consider the case where $E/F$ is ramified. In this case, $\tau$ fixes every element in $\mathrm{Emb}(E_0)$ and so induces involutions $\tau:W_{\iota}\to W_{\iota}$ for each $\iota\in \mathrm{Emb}(E_0)$. Once again, as in the ramified case, from the explicit description of the $F$-crystal structure on $M_{\cris}$ under the identification~\eqref{eqn:mfM_cris}, we have $V_{\iota} = V(W_{\iota},\tau)$, and also
\begin{align*}
	\eta_i&=\begin{cases}
		\frac{1\otimes\pi_E}{1\otimes\tau(\pi_E)}   & \text{if $i=1$}\\
		1& \text{ otherwise}.
	\end{cases}
\end{align*}

So we obtain:
\begin{proposition}\label{prp:ramified_vcris}
When $E/F$ is ramified, $V(\mathcal{G}_1)\subset V_{\cris}(\mathcal{G})$ consists precisely of the elements $f=(f_i)$ such that:
\begin{itemize}
	\item $f_0\in V_{\iota_0}$ satisfies $(1\otimes\pi_E)\mathrm{Fr}^{d_0}(f_0)=(1\otimes\tau(\pi_E))f_0$;
	\item $f_i =(1\otimes \frac{\pi_E}{\tau(\pi_E)})\cdot\mathrm{Fr}^i(f_0)$\text{, for $i=1,\ldots,d_0-1$}
\end{itemize}
In particular, the map
\[
W\otimes_{\Z_p}V(\mathcal{G}_1)\to V_{\cris}(\mathcal{G})
\]
is an isomorphism. Moreover, if $\gamma\in W^\times_{\iota_0}$ is such that $(1\otimes\pi_E)\mathrm{Fr}^{d_0}(\gamma)=(1\otimes\tau(\pi_E))\gamma$, then we have an isometry
\begin{equation*}
%\label{eqn:v0_expl_ramified}
\big(V(\mathcal{G}_1),\langle\cdot,\cdot\rangle\big) \xrightarrow{\simeq}   \big(V(\co_E,\tau),\gamma\tau(\gamma)\langle\cdot,\cdot\rangle\big)
\end{equation*}
of Hermitian $\co_E$-modules  defined by  $f\mapsto \gamma^{-1}f_0$.
\end{proposition}

%%%%%%%%%%%%%%%%%%%%%%%%%%%%%%%%%%%%%%%%%%%%%%%%%%

\subsection{Special endomorphisms with denominators}
\label{ss:denominators}

%%%%%%%%%%%%%%%%%%%%%%%%%%%%%%%%%%%%%%%%%%%%%%%

Let $R$ be a $p$-adically complete $\co_{E}$-algebra. Fix an element $\mu\in E/\co_E$, and choose any representative $\tilde{\mu}\in E$ for it. If $\mu\neq 0$, the positive integer $r(\mu) = -\ord_{\mathfrak{p}}(\tilde{\mu})$ depends only on $\mu$; if $\mu = 0$, set $r(\mu) = 0$. Let $[\tilde{\mu}]\in \pi_E^{-r(\mu)}\End(\mathcal{G}_R)$ be the corresponding quasi-isogeny from $\mathcal{G}_R$ to itself. Set:
\[
V_\mu(\mathcal{G}_R) = \bigl\{f\in V(\mathcal{G}_R)[\pi_E^{-1}]: f - [\tilde{\mu}]\in \End_{\co_F}(\mathcal{G}_R)\}.
\]
This does not depend on the choice of representative $\tilde{\mu}$.

\begin{proposition}
\label{prop:lubin-tate special denom}
Suppose that $E/F$ is ramified. 

\begin{enumerate}
\item If $\mu = 0$, then $V_\mu(\mathcal{G}_R) = V(\mathcal{G}_R)$. 
\item If $\mu\neq 0$, then $V_\mu(\mathcal{G}_{\co_K/\varpi})$ is non-empty if and only if 
\[
r(\mu)\leq \ord_E(\mathfrak{d}_{E/F})-1, 
\]
in which case it is a torsor under translation by $V(\mathcal{G}_{\co_K/\varpi})$. Here, $\mathfrak{d}_{E/F}$ is the relative different of $E$ over $F$.
\item If $p\neq 2$ and $\mu\neq 0$, then $V_\mu(\mathcal{G}_{\co_K/\varpi}) = \emptyset$. 
\end{enumerate}
\end{proposition}
\begin{proof}
For simplicity, set $m = \ord_E(\mathfrak{d}_{E/F})$.

The first assertion is clear. Suppose therefore that $\mu\neq 0$, and that we have $f \in V_\mu(\mathcal{G}_{\co_K/\varpi})$. In the notation of Proposition~\ref{prp:ramified_vcris}, $f$ corresponds to a tuple $(f_i)$ with $f_i \in V_{\iota_i}[\pi_E^{-1}] = V(W_{\iota_i},\tau)[\pi_E^{-1}]$, where $f_i = \mathrm{Fr}^i(f_0)$, and where $f_0$ satisfies:
\begin{equation}\label{eqn:f_0 condition}
\mathrm{Fr}^{d_0}(f_0) = \frac{\tau(\pi_E)}{\pi_E}f_0.
\end{equation}
Here, we are identifying $\pi_E$ with the element $1\otimes\pi_E\in W_{\iota_0}$.

Moreover, by hypothesis, $f_0 - [\tilde{\mu}] \in \End(W_{\iota_0})$. Now, $[\tilde{\mu}]$ is invariant under the action of $\mathrm{Fr}^{d_0}$. Therefore,~\eqref{eqn:f_0 condition} implies:
\[
\biggl(1-\frac{\tau(\pi_E)}{\pi_E} \biggr)\cdot f_0  = f_0 - \mathrm{Fr}^{d_0}(f_0) \in V_{\iota_0}.
\]
This implies 
\[
{m} - 1 = \ord_{E}\left( 1-\frac{\tau(\pi_E)}{\pi_E} \right)  \geq r(\mu). 
\]
Hence, we find that $V_\mu(\mathcal{G}_{\co_K/\varpi}) = \emptyset$ whenever $r(\mu) > m - 1.$

Assume now that $r(\mu) \leq m - 1$. To finish the proof of assertion (2), we have to show that we can always find $f_0$ as above satisfying~\eqref{eqn:f_0 condition} and with $f_0 - [\tilde{\mu}]\in\End(W_{\iota_0})$.

For this, choose any $\tilde{f}_0\in V_{\iota_0}[\pi_E^{-1}]$ such that $\tilde{f}_0 - [\tilde{\mu}]$ lies in $\End(W_{\iota_0})$. We now have
\begin{align*}
\frac{\pi_E}{\tau(\pi_E)}\mathrm{Fr}^{d_0}(\tilde{f}_0) - \tilde{f}_0& = \left(\frac{\pi_E}{\tau(\pi_E)} - 1\right)\mathrm{Fr}^{d_0}(\tilde{f}_0) + \mathrm{Fr}^{d_0}(\tilde{f}_0 - [\tilde{\mu}]) - (\tilde{f}_0 - [\tilde{\mu}]).
\end{align*}
Since $r(\mu)\leq m - 1$, we see that this belongs to $V_{\iota_0}$. 

Now, notice that the endomorphism
\[
V_{\iota_0} \xrightarrow{\frac{\pi_E}{\tau(\pi_E)}\cdot\mathrm{Fr}^{d_0} - \mathrm{id}} V_{\iota_0}
\]
is surjective: Indeed, mod $\pi_E$, this is immediate from the fact $\overline{\F}_p$ is algebraically closed. A simple lifting argument, using the completeness of $W_{\iota_0}$ now does the rest.

Therefore, there exists $f'_0\in V_{\iota_0}$ with
\[
\frac{\pi_E}{\tau(\pi_E)}\mathrm{Fr}^{d_0}(f'_0) - f'_0 = \frac{\pi_E}{\tau(\pi_E)}\mathrm{Fr}^{d_0}(\tilde{f}_0) - \tilde{f}_0.
\]
It is an immediate check that we can now take $f_0 = \tilde{f}_0 - f'_0$.

Assertion (3) is clear from (2), since, when $p\neq 2$, $m - 1 = 0$.

\end{proof}

% \end{proof}

%  The following corollary is easily deduced by combining Theorem~\ref{thm:deformation_special} and Proposition~\ref{prop:lubin-tate special denom deform} with the descriptions of the Hermitian form on $V(\mathcal{G}_1)$ in Propositions~\ref{prp:unramified_vcris} and~\ref{prp:ramified_vcris}. 

% \begin{corollary}
% \label{cor:deformation_special}
% Suppose that $\mu\in E/\co_E$ and that $f\in V_{\mu}(\mathcal{G}_{\co_K/\varpi})$.\footnote{Here, we set $V_{\mu}(\mathcal{G})=\emptyset$ whenever $\mu\neq 0$ and $E/F$ is unramified.} Set $\alpha = \langle f,f\rangle\in E$, and
% \[
% k(f) = \max\{k\in\Z_{\geq 1}: f\in V_\mu(\mathcal{G}_{\co_K/\varpi^k})\}.
% \]
% Then we have
% \[
% k(f) = \begin{cases}
% \frac{\ord_{F}(\alpha) + 1}{2} &\text{if $E/F$ is unramified (in which case $\mu=0$)}  \\
% \ord_{F}(\alpha) + \ord_E(\mathfrak{d}_{E/F}) &\text{if $E/F$ is ramified}.
% \end{cases}
% \]
% \end{corollary}

%%%%%%%%%%%%%%%%%%%%%%%%%%%%%%%%%%%%%%%%%%%%%%%%%%

\subsection{Deformation theory}
\label{ss:lubin-tate deformation}

%%%%%%%%%%%%%%%%%%%%%%%%%%%%%%%%%%%%%%%%%%%%%%%

% To study the deformation theory of special endomorphisms, we will need an alternate description of $V_\mu(\mathcal{G}_R)$ when $\mu\in \mathfrak{d}_{E/F}^{-1}/\co_E$. Consider the Serre tensor product $\Hom_{\co_E}(\mathfrak{d}_{E/F},\mathcal{G})$: This is once again a formal $\co_E$-module over $\co_E$.

% % , admitting a canonical $\co_E$-linear isogeny
% % \begin{equation}\label{eqn:d E/F isog}
% % \mathcal{G} = \Hom_{\co_E}(\co_E,\mathcal{G}) \xrightarrow{\eta} \Hom_{\co_E}(\mathfrak{d}_{E/F},\mathcal{G}).
% % \end{equation}
% Multiplication by $a\in \mathfrak{d}_{E/F}^{-1}$ gives a $\co_E$-linear map $\mathfrak{d}_{E/F}\xrightarrow{a} \co_E$, and thus an isogeny
% \[
% [a]_{\mathfrak{d}}: \mathcal{G} = \Hom_{\co_E}(\co_E,\mathcal{G})\to \Hom_{\co_E}(\mathfrak{d}_{E/F},\mathcal{G})
% \]
% obtained via pre-composition by $a$. 

% As above, fix a lift $\tilde{\mu}\in \mathfrak{d}^{-1}_{E/F}$ of $\mu$. 
% \begin{lemma}
% \label{lem:Vmu G alt descp}
% $V_\mu(\mathcal{G}_R)$ can be canonically identified with the space of $\tau$-semilinear homomorphisms
% \[
% f:\mathcal{G}_R\to \Hom_{\co_E}(\mathfrak{d}_{E/F},\mathcal{G})_R
% \]
% such that $f-[\tilde{\mu}]_{\mathfrak{d}}$ factors through the isogeny $[1]_{\mathfrak{d}}$: That is, there exists $a\in\End(\mathcal{G}_R)$ such that $f-[\tilde{\mu}]_{\mathfrak{d}} = [1]_{\mathfrak{d}}\circ a$.
% \end{lemma}

Assume now that $K$ is generated over $\mathrm{Frac}(W)$ by the image of $\iota_0:E \to \mathrm{Frac}(W)^{\alg}$. Set $\varpi = \iota_0(\pi_E)$; this is a uniformizer for $K$. 
For any $k\in\Z_{\geq 1}$, set $\mathcal{G}_k = \mathcal{G}_{\co_{K}/\varpi^{k+1}}$, and for each $\mu\in \mathfrak{d}_{E/F}^{-1}/\co_E$, set
\[
V_\mu(\mathcal{G}_k)\define  V_\mu\bigl(\mathcal{G}_{\co_{K}/(\varpi^{k})}\bigr).
\]

Let $M_{\dR}$ be the de Rham realization of $\mathcal{G}_{\co_K}$ as in~\eqref{eqn:mfM_dR}: It is a free $\co_K\otimes_{\Z_p}\co_E$-module of rank $1$ equipped with the $\co_K$-linear direct summand $\Fil^1M_{\dR}$, described in~\eqref{eqn:fil1_H}. 

In the notation of \S~\ref{ss:lubin-tate_special_endomorphisms}, let
\[
V_{\dR} \define V(M_{\dR},\tau)\subset \End_{\co_K}(M_{\dR})
\]
be the space of $\tau$-semilinear endomorphisms of $M_{\dR}$. 

Given $f_1,f_2\in V_{\dR}$, there is a canonical element $\langle f_1,f_2\rangle\in \co_K\otimes_{\Z_p}\co_E$ such that, for every $m\in M_{\dR}$, $(f_1\circ f_2)(m) = \langle f_1,f_2\rangle\cdot m$. Set $\breve{V}_{\dR} = V_{\dR}\otimes_{\co_E}\mathfrak{d}^{-1}_{E/F}$.

Similarly, for each $k\in\Z_{\geq 1}$, let $M_{\dR,k} = M_{\dR}\otimes_{\co_K}\co_K/\varpi^{k}$ be the induced filtered free module over $\co_K/\varpi^{k}$, and let $V_{\dR,k} = V(M_{\dR,k},\tau)$. We have $V_{\dR,k} = V_{\dR}\otimes_{\co_K}\co_K/\varpi^{k+1}$. Set $\breve{V}_{\dR,k} = \breve{V}_{\dR}\otimes_{\co_K}\co_K/\varpi^k$.

For each $k\in \Z_{\geq 1}$, 
\[
\widetilde{\mathrm{Ob}}_k =\breve{V}_{\dR}\otimes_{\co_K\otimes_{\Z_p}\co_E,1\otimes\tau(\iota_0)}\co_K/\varpi^{k}.
\]
This is a rank $1$ free module over $\co_K/\varpi^k$. 

Now set $\mathrm{Ob}_k = \varpi^{k-1}\cdot \widetilde{\mathrm{Ob}}_k$: This is a $1$-dimensional vector space over $\F_p^\alg$.

\begin{proposition}
\label{prop:lifting end}
For each $k\in\Z_{\geq 1}$, there is a canonical map
	\[
     \mathrm{ob}_{k+1}:V_\mu(\mathcal{G}_k)\to  \mathrm{Ob}_{k+1}
	\]
with the following properties:
\begin{enumerate}
	\item  An element $f\in V_\mu(\mathcal{G}_k)$ lifts to $V_\mu(\mathcal{G}_{k+1})$ if and only if $\mathrm{ob}_{k+1}(f) = 0$.
% 	\item The diagram
% 	\[
% \xymatrix{
%  { V_\mu(\mathcal{G}_{k+1}) }  \ar[rr]^{\mathrm{ob}_{k+2}}  \ar[d] & & {\mathrm{Ob}_{k+2}}  \ar[d]  \\
%  {V_\mu(\mathcal{G}_{k})} \ar[rr]_{\mathrm{ob}_{k+1}}   &&  {  \mathrm{Ob}_{k+1} }
% }
% \]
% commutes. Here, the vertical map on the right is the `reduction mod $\varpi^{k}$' map.
 \item If $a\in \co_E$, then the diagram
 \[
\xymatrix{
 { V_\mu(\mathcal{G}_{k}) }  \ar[rr]^{\mathrm{ob}_{k+1}}  \ar[d]_{f\mapsto a\cdot f} & & {\mathrm{Ob}_{k+1}}  \ar[d]^{x\mapsto \iota_0(\tau(a))\cdot x}  \\
 {V_{a\cdot\mu}(\mathcal{G}_{k})} \ar[rr]_{\mathrm{ob}_{k+1}}   &&  {  \mathrm{Ob}_{k+1} }
}
\]
commutes.
\end{enumerate}
\end{proposition}
\begin{proof}
For any $p$-adicaly complete $\co_E$-algebra $R$, an element $f\in V_\mu(\mathcal{G}_k)$ can be viewed as a $\tau$-semilinear homomorphism
\[
f:\mathcal{G}_R \to \underline{\Hom}_{\co_E}(\mathfrak{d}_{E/F},\mathcal{G}_R)
\]
of formal $\co_E$-modules over $R$.

For each $k\in\Z_{\geq 1}$, $\co_K/\varpi^{k+1}\to \co_K/\varpi^{k}$ is a divided power thickening, and so every $f\in V_\mu(\mathcal{G}_k)$ has a canonical crystalline realization\footnote{Recall that we are using the contravariant Dieudonn\'e $F$-crystal.}
\[
f_{k+1}:M_{\dR,k+1}\otimes_{\co_E}\mathfrak{d}_{E/F}\to M_{\dR,k+1},
\]
which is a $\tau$-semilinear homomorphism of $\co_K/\varpi^{k+1}\otimes_{\Z_p}\co_E$-modules, and thus can be viewed as an element $f_{k+1}\in \breve{V}_{\dR,k+1}$. 

We claim that the map $\mathrm{ob}_{k+1}$ which takes $f\in V_\mu(\mathcal{G}_k)$ to the image of $f_{k+1}$ in $\mathrm{Ob}_{k+1}$ answers to the requirements of the lemma. 

For this, set
\[
a = f - [\tilde{\mu}]\in \End_{\co_F}(\mathcal{G}_k).
\]
It is easily checked that $f$ lifts to an element of $V_\mu(\mathcal{G}_{k+1})$ if and only if $a$ lifts to $\End_{\co_F}(\mathcal{G}_{k+1})$. 

The crystalline realization of $a$ gives a homomorphism
\[
a_{k+1}: M_{\dR,k+1}\to M_{\dR,k+1}
\]
of $\co_K/\varpi^{k+1}\otimes_{\Z_p}\co_F$-modules. 

By Grothendieck-Messing theory~\cite{MessingBT}---which applies even when $p=2$, by the theory of Zink~\cite{ZinkWindows}, because $\mathcal{G}$ is connected---$a$ lifts to $\End(\mathcal{G}_{k+1})$ if and only if $a_{k+1}$ preserves the direct summand $\Fil^1M_{\dR,k+1}\subset M_{\dR,k+1}$.

We now use the explicit description of the filtration from~\eqref{eqn:fil1_H}. Since $\mathcal{E}_{\iota_0}(u) = -u + 1\otimes\pi_E$, we find that, in terms of the natural isotypic decomposition $M_{\dR,k+1} = \oplus_{\iota}M_{\dR,k+1,\iota}$, we have:
\begin{align*}
\Fil^1M_{\dR,k+1,\iota} &= \begin{cases}
\{x\in M_{\dR,k+1,\iota}:(1\otimes \pi_E - \varpi\otimes 1)x=0\} & \text{ if $\iota=\iota_0$}  \\
0 & \text{ otherwise}.
\end{cases}
\end{align*}
Here, we are using the fact that the cokernel of the map
\[
M_{\dR,\iota_0}\xrightarrow{1\otimes\pi_E - \varpi\otimes 1} M_{\dR,\iota_0}
\]
of $\co_K$-modules is free of rank $1$ over $\co_K$, and hence the formation of its kernel is compatible with arbitrary base change.

For this, choose an $\co_{K,\tau(\iota_0)}$-module generator $u\in \breve{V}_{\dR,k,\tau(\iota_0)}$, and let $c\in \co_{K,\tau(\iota_0)}/(\varpi^{k+1}\otimes 1)$ be such that $c\cdot u = f_{k+1}$. The proposition will follow once we show that $f$ lifts to $V_\mu(\mathcal{G}_{k+1})$ if and only if $c$ maps to $0$ under
\[
\co_{K,\tau(\iota_0)}/(\varpi^{k+1}\otimes 1) = (\co_K/\varpi^{k+1})\otimes_{\tau(\iota_0),\co_{E_0}}\co_E\xrightarrow{1\otimes \tau(\iota_0)}\co_K/\varpi^{k+1}.
\]
Equivalently, if and only if $c\in (1\otimes \tau(\pi_E) - \varpi\otimes 1)\cdot \co_{K,\tau(\iota_0)}/(\varpi^{k+1}\otimes 1)$.

First, suppose that $E/F$ is unramified. In this case $V_{\dR,k+1} = \breve{V}_{\dR,k+1}$, we can take $\tilde{\mu} = 0$, and $f$ lifts precisely when we have
\[
f_{k+1}(\Fil^1M_{\dR,k+1})\subset \Fil^1M_{\dR,k+1}.
\]
Now, we have
\[
u(\Fil^1M_{\dR,k+1}) = \overline{\Fil}^1M_{\dR,k+1}\define \{x\in M_{\dR,k+1,\tau(\iota_0)}: (1\otimes\tau(\pi_E) - \varpi\otimes 1)x =0\}.
\]
Therefore, we must have $c\cdot(\overline{\Fil}^1M_{\dR,k+1}) = 0$, which is precisely equivalent to $c\in (1\otimes \tau(\pi_E) - \varpi\otimes 1)$.

Suppose now that $E/F$ is ramified, so that $\iota_0 = \tau(\iota_0)\in \mathrm{Emb}(E_0)$. In this case, the homomorphism
\[
f_{k+1,\iota_0} - 1\otimes\tilde{\mu}: M_{\dR,k+1,\iota_0}\otimes_{\co_E}\mathfrak{d}_{E/F}\to M_{\dR,k+1}
\]
lifts to the endomorphism $a_{k+1,\iota_0}\in \End_{\co_F}(M_{\dR,k+1,\iota_0})$. The proposition now reduces to the following easy observation: Suppose that $f_\infty\in \breve{V}_{\dR}$ is such that $f_\infty - 1\otimes\tilde{\mu}\in \End_{\co_F}(M_{\dR})$. Then we have
\[
(f_\infty - 1\otimes\tilde{\mu})(\Fil^1M_{\dR,k+1}) \subset \Fil^1M_{\dR,k+1}
\]
if and only if $f_{\infty,\iota_0}\in (1\otimes \tau(\pi_E) - \varpi\otimes 1)\cdot \breve{V}_{\dR,\iota_0}$.
\end{proof}

Suppose that $k\leq e$, so that  the surjection 
\[
W[u]/(u^k)\xrightarrow{u\mapsto\varpi}\co_{K}/(\varpi^{k})
\] 
is a divided power thickening (its kernel is generated by $p$). Upon evaluating the crystal $\mathbb{D}(\mathcal{G})$ on this thickening, we obtain a free $W[u]/(u^{k})\otimes_{\Z_p}\co_E$-module $\mathcal{M}_k$ of rank $1$. Using the Frobenius lift $\varphi$ on $W[u]/(u^{k})$ satisfying $\varphi(u) = u^p$ and the $F$-crystal structure on $\mathbb{D}(\mathcal{G})$, we also obtain a canonical $W[u]/(u^{k})\otimes_{\Z_p}\co_E$-linear map
\[
\varphi_k:\varphi^*\mathcal{M}_k\to \mathcal{M}_k.
\]

% By Theorem~\ref{thm:kisin_p_divisible}, we have
% \[
% \mathcal{M}_i = \varphi^*\mathfrak{M}(\mathcal{G})\otimes_{\mathfrak{S}}W[u]/(u^{i+1})
% \]
% as $\varphi$-modules over $W[u]/(u^{i+1})$. Using Proposition~\ref{prp:bk_lubin-tate}, we therefore obtain an explicit description of the $\mathcal{M}_i$ as a $\varphi$-module over 
% \[
% W[u]/(u^{i+1})\otimes_{\Z_p}\co_E = \bigoplus_{\iota\in H_0}W_\iota[u]/(u^{i+1}).
% \]
% For each $\iota\in H_0$, its $\iota$-isotypic component $\mathcal{M}_{i,\iota}$ of $\mathcal{M}_i$ is identified with $W_\iota[u]/(u^{i+1})$. Let $\beta\in \mathfrak{S}\otimes_{\Z_p}\co_E$ be as in \S~\ref{ss:lubin_tate}, and write $\gamma_i\in W[u]/(u^{i+1})$ for the image of $\varphi(\beta)\in \mathfrak{S}\otimes_{\Z_p}\co_E$. Its isotypic decomposition $(\gamma_{i,\iota})\in \prod_{\iota}W_\iota[u]/(u^{i+1})$ has a simple description:
% \begin{align}
% \label{eqn:gamma eqn}
% \gamma_{i,\iota}& = \begin{cases}
% \varphi(\mathcal{E}_{\iota_0}(u)) = -u^p + (1\otimes\pi_E),&\text{ if $\iota = \mathrm{Fr}(\iota_0)$};\\
% 1,&\text{ otherwise}.
% \end{cases}
% \end{align}

Let $\mathcal{V}_k = V(\mathcal{M}_k,\tau)$ be the space of $1\otimes\tau$-semilinear endomorphisms of the $W[u]/(u^{k})\otimes_{\Z_p}\co_E$-module $\mathcal{M}_k$. Conjugation by $\varphi_k$ induces an isomorphism
\[
\varphi_k:\varphi^*\mathcal{V}_k[p^{-1}]\xrightarrow{\simeq}\mathcal{V}_k[p^{-1}].
\]
Set $\breve{\mathcal{V}}_k = \mathcal{V}_k\otimes_{\co_E}\mathfrak{d}^{-1}_{E/F}$. 

The $\co_E$-module structures on $\mathcal{M}_k$, $\mathcal{V}_k$ and $\breve{\mathcal{V}}_k$ equips them with isotypic decompositions 
\[
\mathcal{M}_k = \bigoplus_{\iota}\mathcal{M}_{k,\iota}\;;\;\mathcal{V}_k = \bigoplus_{\iota}\mathcal{V}_{k,\iota}\;;\; \breve{\mathcal{V}}_k = \bigoplus_{\iota}\breve{\mathcal{V}}_{k,\iota}.
\]
% The pair $(\mathcal{V}_i,\varphi_i)$ can also be described explicitly, using~\eqref{eqn:gamma eqn}. To begin, we have an isotypic decomposition $\mathcal{V}_i = \oplus_{\iota}\mathcal{V}_{i,\iota}$, where
% \[
% \mathcal{V}_{i,\iota} = W[u]/(u^{i+1})\otimes_{\iota,\co_{E_0}}V(\co_E,\tau).
% \]
% The map $\varphi_i$ can now be identified with:
% \[
% \varphi^*\mathcal{V}_{i,\iota}[p^{-1}] = W[u]/(u^{i+1})\otimes_{\iota,\co_{E_0}}V(E,\tau) \xrightarrow{\delta_i}W[u]/(u^{i+1})\otimes_{\iota,\co_{E_0}}V(E,\tau) = \mathcal{V}_{i,\iota}[p^{-1}],
% \]
% for an element $\delta_i = (\delta_{i,\iota})$, with $\delta_{i,\iota}\in W_\iota[u]/(u^{i+1})[p^{-1}]^\times$.

% As in \S~\ref{ss:lubin-tate_special_endomorphisms}, to describe $\delta_i$, we need to consider the unramified and ramified cases separately. When $E/F$ is unramified, we have
% \[
% \delta_{i,\iota} = \begin{cases}
%  (1\otimes\pi_E) - u^p,&\text{ if $\iota = \mathrm{Fr}(\iota_0)$};\\
% ( (1\otimes\pi_E) - u^p )^{-1},&\text{ if $\iota = \mathrm{Fr}(\tau(\iota_0))$};\\
% 1,&\text{ otherwise}.
% \end{cases}
% \]
% In the ramified case, we have
% \[
% \delta_{i,\iota} = \begin{cases}
% \frac{(1\otimes\pi_E)-u^p}{(1\otimes\tau(\pi_E))-u^p},&\text{ if $\iota = \mathrm{Fr}(\iota_0)$};\\
% 1,&\text{ otherwise}.
% \end{cases}
% \]

\begin{lemma}
\label{lem:dwork trick}
\mbox{}
\begin{enumerate}
	\item  For each $k\leq e$, the reduction map $\mathcal{V}_k\to V_{\cris}(\mathcal{G})$ induces an isomorphism $\mathcal{V}_k[p^{-1}]^{\varphi_k = 1}\xrightarrow{\simeq}V(\mathcal{G}_1)[p^{-1}]$.
	% \item Suppose that $f\in V(\mathcal{G}_1)$ with unique $\varphi_i$-invariant lift $\tilde{f}_i\in \mathcal{V}_i[p^{-1}]$. If $f$ lifts to $V^{i-1}$, then $\tilde{f}_i\in \mathcal{V}_i$, and $\mathrm{ob}_{i-1}(f)$ is the image of $\tilde{f}_i$ under the composition
	% \begin{equation}\label{eqn:ob i explicit}
 %     \mathcal{V}_i \to V_{\dR,i} \to \mathrm{Ob}_{i-1}.
 % \end{equation}
	% Here, the first map is obtained via reduction modulo $\mathcal{E}(u)$ (equivalently, since $i\leq e-1$, reduction modulo $p$).  
	\item Suppose that $k<e$ and that $f\in V_\mu(\mathcal{G}_k)$. Set
	\[
     \beta_{k+1} = \sum_{i=0}^{k}u^i\otimes\tau(\pi_E)^{k-i}\in W[u]/(u^{k+1})\otimes_{\tau(\iota_0),\co_{E_0}}\co_E.
	\]
	Then $\mathrm{ob}_{k+1}(f) = 0$ if and only if
	\begin{align*}
    \beta_{k+1} \cdot \tilde{f}_{k+1,\tau(\iota_0)}\in (1\otimes\tau(\pi_E)^{k+1})\cdot \breve{\mathcal{V}}_{k+1,\tau(\iota_0)}.
    % ,&\text{ {if $E/F$ is unramified}};\\
    % \beta_{k+1} \cdot (1\otimes\pi_E - u\otimes 1)\cdot \tilde{f}_{k+1,\tau(\iota_0)}\in (1\otimes\tau(\pi_E)^{k+1+n_E})\cdot \breve{\mathcal{V}}_{k+1,\tau(\iota_0)},&\text{ {if $E/F$ is ramified}}.
	\end{align*}
\end{enumerate}
\end{lemma}
\begin{proof}
The first assertion is well-known, and is essentially Dwork's trick: Given an element $f_0\in V(\mathcal{G}_1)[p^{-1}]$, and any lift $\tilde{f}\in \mathcal{V}_k[p^{-1}]$, $\varphi_k^{k-1}(\tilde{f})\in \mathcal{V}_k[p^{-1}]$ will be the unique $\varphi_k$-invariant lift of $f_0$. 

For the second assertion, by multiplying both sides of the condition by $(1\otimes \tau(\pi_E) - u\otimes 1)$, we see that it is equivalent to:
\begin{equation}
\label{eqn:reform condition}
\tilde{f}_{k+1,\tau(\iota_0)}\in (1\otimes\tau(\pi_E) - u\otimes 1)\cdot  \breve{\mathcal{V}}_{k+1,\tau(\iota_0)}.
\end{equation}

The pre-image $\Fil^1 \mathcal{M}_{k+1}$ of $\Fil^1M_{\dR,k+1}$ in $\mathcal{M}_{k+1}$ can be explicitly described using~\eqref{eqn:fil_mfM} and the canonical isomorphism
\[
\varphi^*\mathfrak{M}\otimes_{\mathfrak{S}}W[u]/(u^{k+1}) \xrightarrow{\simeq}\mathcal{M}_{k+1}
\]
obtained from assertion (3) of Theorem~\ref{thm:kisin_p_divisible}. We find:
\[
\Fil^1 \mathcal{M}_{k+1} = \frac{\mathcal{E}(u)\otimes 1}{1\otimes\pi_E - u\otimes 1}\cdot \mathcal{M}_{k+1,\iota_0} + \mathcal{E}(u)\cdot \mathcal{M}_{k+1}
\]

Now, $W[u]/(u^{k})\to \co_K/\varpi^k$ is a divided power thickening, its kernel being the ideal $(p,u^{k-1})$, and $\tilde{f}_{k+1}$, by virtue of being characterized by its $\varphi_k$-invariance, is the evaluation of the crystalline realization of $f$ on this thickening. Therefore, $\mathrm{ob}_{k+1}(f)$ vanishes if and only if $\tilde{f}_{k+1} - 1\otimes\tilde{\mu}$ preserves the submodule $\Fil^1 \mathcal{M}_{k+1}\subset \mathcal{M}_{k+1}$. 

If $E/F$ is unramified, then we can take $\tilde{\mu} = 0$, and the condition translates to:
\[
\frac{\mathcal{E}(u)}{1\otimes\tau(\pi_E) - u\otimes 1}\cdot \tilde{f}_{k+1,\tau(\iota_0)}\in \mathcal{E}(u)\cdot \mathcal{V}_{k+1,\iota_0}.
\]
This is easily seen to be equivalent to~\eqref{eqn:reform condition}.

Now, suppose that $E/F$ is ramified. Fix a $W_{\iota_0}[u]/(u^{k+1})$-module generator $x\in \mathcal{M}_{k+1,\iota_0}$. Then the image of $\Fil^1 \mathcal{M}_{\dR,k+1,\iota_0}$ under $\tilde{f}_{k+1,\iota_0} - 1\otimes\tilde{\mu}$ is generated by
\begin{align*}
\frac{\mathcal{E}(u)\otimes 1}{1\otimes\tau(\pi_E) - u\otimes 1}\cdot \tilde{f}_{k+1,\iota_0}(x) - \frac{\mathcal{E}(u)\otimes 1}{1\otimes\pi_E - u\otimes 1}\cdot (1\otimes\tilde{\mu})\cdot x.
 \end{align*}
This lies in $\Fil^1 \mathcal{M}_{k+1,\iota_0}$ if and only if we have
\[
(1\otimes\pi_E - u\otimes 1)\tilde{f}_{k+1,\iota_0}(x) - (1\otimes\tau(\pi_E) - u\otimes 1)(1\otimes\tilde{\mu})\cdot x \in (1\otimes\tau(\pi_E) - u\otimes 1)\mathcal{M}_{k+1,\iota_0}.
\]
Equivalently, if and only if
\[
(1\otimes(\pi_E - \tau(\pi_E)))\tilde{f}_{k+1,\iota_0}(x) \in (1\otimes\tau(\pi_E) - u\otimes 1)\mathcal{M}_{k+1,\iota_0},
\]
which, as is once again easily verified, is equivalent to~\eqref{eqn:reform condition}.
\end{proof}

\begin{lemma}
\label{lem:non vanish}
For any $k\in \Z_{\geq 1}$, suppose that $f\in V_\mu(\mathcal{G}_k)$ is such that $f$ does not lift to $V_\mu(\mathcal{G}_{k+1})$. Then, for every $a\in \co_E$ with $\ord_E(a) = 1$, $a\cdot f\in V_{a\cdot\mu}(\mathcal{G}_{k})$ lifts to $V_{a\cdot\mu}(\mathcal{G}_{k+1})$ but not to $V_{a\cdot\mu}(\mathcal{G}_{k+2})$.
\end{lemma}
\begin{proof}
It is immediate from Proposition~\ref{prop:lifting end} that 
\[
\mathrm{ob}_{k+1}(a\cdot f) = \iota_0(\tau(a))\mathrm{ob}_{k+1}(f) 
\]
vanishes. Therefore, $a\cdot f$ lifts to $V_{a\cdot\mu}(\mathcal{G}_{k+1})$. It remains to show that it does not lift to $V_{a\cdot \mu}(\mathcal{G}_{k+2})$; that is, we must show that $\mathrm{ob}_{k+2}(a\cdot f)\neq 0$.

Suppose first that $k < e$, where $e=e_E$ is the absolute ramification index of $E$, and suppose that $\mathrm{ob}_{k+2}(a\cdot f) = 0$. We then claim that $\mathrm{ob}_{k+1}(f) = 0$. Indeed, this follows easily from assertion (2) of Lemma~\ref{lem:dwork trick} and the identity
\[
\beta_{k+2}\equiv (1\otimes\tau(\pi_E))\cdot\beta_{k+1}\pmod{u^{k+1}}.
\]
This shows the lemma when $k<e$.

Now suppose that $k\geq e$. Then the map $\co_K\to \co_K/\varpi^{k}$ is a divided power thickening. Therefore, $f\in V_\mu(\mathcal{G}_k)$ has a crystalline realization $f_{\cris}\in \breve{V}_{\dR}$ whose reduction mod $\varpi^{k+1}$ is the crystalline realization $f_{k+1}\in \breve{V}_{\dR,k+1}$. Set
\[
\widetilde{\mathrm{Ob}} \define \breve{V}_{\dR}\otimes_{\co_K\otimes_{\Z_p}\co_E,1\otimes\tau(\iota_0)}\co_K.
\]
Then, $\widetilde{\mathrm{Ob}}$ is a rank $1$ finite free $\co_K$-module, and, for each $i\in \Z_{\geq 1}$, we have $\widetilde{\mathrm{Ob}}_{i+1} = \widetilde{\mathrm{Ob}}\otimes_{\co_K}\co_K/\varpi^{i+1}$.

The hypothesis $\mathrm{ob}_{k+1}(f)\neq 0$ means that that the image of $f_{\cris}$ in $\widetilde{\mathrm{Ob}}$ does not lie in $\varpi^{k+1}\cdot\widetilde{\mathrm{Ob}}$. In turn, this implies that, the image of $(a\cdot f)_{\cris} = (1\otimes a)\cdot f_{\cris}$ in $\widetilde{\mathrm{Ob}}$ does not lie in $\varpi^{k+2}\cdot\widetilde{\mathrm{Ob}}$, and thus that $\mathrm{ob}_{k+2}(a\cdot f)\neq 0$.
\end{proof}

Define a function
\[
\mathrm{ord}_{E}:\;V(\mathcal{G}_1)_\Q \to \Z,
\]
given by two defining properties:
\begin{itemize}
	\item If $a\in E$, and $f\in V(\mathcal{G}_1)_\Q$, then
	\[
     \mathrm{ord}_{E}(a\cdot f) = \mathrm{ord}_{E}(a) + \mathrm{ord}_{E}(f).
	\]
	\item If $f\in V(\mathcal{G}_1)$ is an $\co_{E}$-module generator, then 
    \[
	\mathrm{ord}_{E}(f) = \begin{cases}
	1,&\text{ if $E$ is unramified over $F$};\\
	\ord_{E}(\mathfrak{d}_{E/F}),&\text{ if $E$ is ramified over $F$}.
	\end{cases}
	\]
\end{itemize}

\begin{lemma}
\label{lem:first non vanish}
If $f\in V_\mu(\mathcal{G}_1)$ is such that $\ord_E(f) = 1$, then $f$ does not lift to $V_\mu(\mathcal{G}_2)$.
\end{lemma}
\begin{proof}
Let $f_\cris\in \breve{V}_{\cris}(\mathcal{G}) \define V_{\cris}(\mathcal{G})\otimes_{\co_E}\mathfrak{d}_{E/F}^{-1}$ be the crystalline realization of $f$. Observe that, by Propositions~\ref{prp:unramified_vcris} and~\ref{prp:ramified_vcris}, the hypothesis $\ord_E(f) = 1$ implies:
\[
f_{\cris,\tau(\iota_0)}\in (1\otimes\pi_E)\cdot \breve{V}_{\cris}(\mathcal{G})\backslash (1\otimes\pi_E^2)\cdot\breve{V}_{\cris}(\mathcal{G}).
\]

Now, one only needs to observe that $\co_K/\varpi^2$ is either $W/p^2W$ or the ring $\F_p^\alg[u]/(u^2)$ of dual numbers, and that, in either case, there exists an isomorphism
\[
\breve{V}_{\cris}(\mathcal{G})\otimes_W\co_K/\varpi^2\xrightarrow{\simeq}\breve{V}_{\dR,2}
\]
of $\co_K/\varpi^2\otimes_{\Z_p}\co_E$-modules carrying $f_{\cris}$ to the crystalline lift $f_2\in \breve{V}_{\dR,2}$. 

This immediately implies $\mathrm{ob}_2(f)\neq 0$, and thus gives us the lemma.

\end{proof}

\begin{theorem}
\label{thm:deformation special}
Suppose that $f\in V_\mu(\mathcal{G}_1)$. Then $f$ lifts to $V_\mu(\mathcal{G}_k)$ if and only if $\ord_E(f)\geq k$.
\end{theorem}
\begin{proof}
Immediate from Lemmas~\ref{lem:first non vanish} and~\ref{lem:non vanish}.
\end{proof}

\section{CM Shimura varieties}\label{s:cm shimura}

%%%%%%%%%%%%%%%%%%%%%%%%%%%%%%%%%%%

Our goal in this section is to study a certain zero dimensional Shimura variety. The proof of the main Theorem~\ref{thm:arithmetic BKY} will proceed by embedding this zero dimensional variety into the higher dimensional GSpin Shimura varieties studied in the next section, and then computing the degree of the special divisors on the ambient Shimura variety along the resulting arithmetic curve.

As with the GSpin Shimura varieties themselves, the zero dimensional variety studied here does not admit any obvious moduli interpretation. Instead, we have to resort to abstract existence theorems, working consistently with the various realizations of the putative motives that live over the variety, and exploiting properties arising from the comparison isomorphisms among them. As a result, the exposition is necessarily somewhat technical. 

Now for some notational conventions that will be in force throughout the section: We will fix a CM field $E$ with totally real subfield $F$. We will also take $\Q^{\alg}$ to be the algebraic closure in $\C$ of $\Q$ and write $\Gamma_{\Q}$ for the absolute Galois group $\mathrm{Gal}(\Q^{\alg}/\Q)$. For any algebraic torus $A$ over $\Q$, we will write $X^*(A)$ (resp. $X_*(A)$) for the $\Gamma_\Q$-module of characters (resp. cocharacters) of $A$.

If $\mu:\mathbb{G}_m\to A$ is a cocharacter with field of definition $F\subset \Q^{\alg}$, then its \emph{reflex norm} is given by
\[
r(A,\mu): \mathrm{Res}_{F/\Q}\mathbb{G}_m\xrightarrow{\mathrm{Res}~\mu}\mathrm{Res}_{F/\Q}A\xrightarrow{\mathrm{Nm}_{F/\Q}}A.
\]
Here, $\mathrm{Res}_{F/\Q}A$ is the Weil restriction of the base change of $A$ over $F$, $\mathrm{Res}~\mu$ is the Weil restriction of $\mu$, and $\mathrm{Nm}_{F/\Q}$ is the usual norm map.

%%%%%%%%%%%%%%%%%%%%%%%%%%%%%%%%%%

\subsection{A zero dimensional Shimura variety}
\label{ss:zero dimensional}

%%%%%%%%%%%%%%%%%%%%%%%%%%%%%%%%%%

For any $\Q$-algebra $R$, abbreviate $T_R  =\mathrm{Res}_{R/\Q} \mathbb{G}_m$.   Set  
\[
T_F^1  = \mathrm{ker} ( \mathrm{Nm}: T_F \to \mathbb{G}_m ),
\] 
and $T = T_E / T_F^1.$ 
The natural diagonal embedding $\mathbb{G}_m\hookrightarrow T_E$ induces an embedding $\mathbb{G}_m\hookrightarrow T$.  Set 
\[
T_{so} = \mathrm{ker} ( \mathrm{Nm}_{E/F} : T_E \to T_F).
\]
The rule $x\mapsto x/\overline{x}$ defines a surjection
\begin{equation*}
%\label{eqn:theta}
\theta: 	T_E \to T_{so} ,
\end{equation*}
inducing isomorphisms $T_E/T_F \xrightarrow{\simeq} T/\mathbb{G}_m \xrightarrow{\simeq}T_{so}$. 
The character groups of these tori can be described explicitly. If for any number field $M/\Q$ we  set
\[
\mathrm{Emb}(M)= \{ \mbox{Embeddings $M\hookrightarrow\Q^{\alg}$} \},
\]
then we have natural identifications  
\begin{align}
\label{character groups}
X^*(T_E) &= \bigoplus_{\iota\in\mathrm{Emb}(E)}\Z\cdot[\iota], \\\nonumber
X^*(T)   &= \biggl\{\sum_{\iota}a_{\iota}[\iota] \in X^*(T_E): a_{\iota}+a_{\overline{\iota}}\text{ is independent of $\iota$}\biggr\}, \\
X^*(T_{so}) &= \biggl\{\sum_{\iota}a_{\iota}[\iota] \in X^*(T_E): a_{\iota}+a_{\overline{\iota}}=0\text{, for all $\iota$}\biggr\} \nonumber
\end{align}
of $\Gamma_\Q$-modules.

Identify $\mathrm{Emb}(E)$ with the set of embeddings $E \hookrightarrow \C$, and enumerate the real embeddings $F\hookrightarrow \R$ as $\iota_0,\ldots, \iota_{d-1}$. We will declare $\iota_0$ a distinguished embedding, and we will fix, once and for all, an extension $\iota_0\in\mathrm{Emb}(E)$ of this embedding.

Define a cocharacter $\mu_0\in X_*(T_E)$ by the formula
\begin{equation}\label{eqn:mu_0}
	\langle \mu_0, [\iota]\rangle  = 	\begin{cases}
1 & \text{ if $\iota=\iota_0$} \\
0& \text{ otherwise.}
\end{cases}
\end{equation}
We will also denote the induced cocharacter of $T$ by $\mu_0$. The field of definition for $\mu_0$ is $\iota_0(E)\subset\Q^{\alg}\subset\C$. If no confusion can arise, we will identify $E$ with this subfield of $\C$. 

In our situation, the reflex norm $r(T_E,\mu_0)$ simplifies considerably: It is simply the identity $T_E\xrightarrow{\mathrm{id}}T_E$. In particular, the reflex norm $r(T,\mu_0):T_E\to T$ associated with $\mu_0\in X_*(T)$ is just the natural surjection from $T_E$ to $T$.

Let $r(T,\mu_0)_{\ell}:(\Q_\ell\otimes_\Q E)^\times \to T(\Q_\ell)$ be the evaluation of this map on $\Q_\ell$. We will write $r(T,\mu_0)_{\lambda}:E_{\lambda}^\times\to T(\Q_\ell)$ for its restriction to $E_{\lambda}^\times$. We also have the ad\'elic version
\begin{align*}
 r(T,\mu_0)(\A_f)  & : \A_{f,E}^\times\to T(\A_f).
\end{align*}

Fix a compact open subgroup $K\subset T(\A_f)$. To the triple $(T,\mu_0,K)$ we can attach a finite \'etale algebraic stack $Y_{K}$ over $E$, which we will call a \emph{CM Shimura variety}. This is constructed as follows. Consider the composition
\[
\A_{f,E}^\times\xrightarrow{r(T,\mu_0)(\A_f)} T(\A_f)\to T(\A_f)/T(\Q)K.
\]
This factors via the global reciprocity map through the abelianization of the Galois group $\Gamma_{E}=\Gal(\Q^{\alg}/E)$. Therefore, we obtain a homomorphism
\begin{equation*}
%\label{eqn:rkTmu0}
r_K(T,\mu_0):\Gamma_{E}\to T(\A_f)/T(\Q)K.
\end{equation*}

Now, suppose that $K$ is neat: This is equivalent to requiring that $K\cap T(\Q)$ be torsion-free.\footnote{For instance, one can take $K$ to be the image of the elements in $(\co_E\otimes\widehat{\Z})^\times$ that are congruent to $1$ mod $3$.} Then $Y_K$ will be a finite \'etale \emph{scheme} over $E$ corresponding to the $\Gamma_{E}$-set 
\[
Y_K(\Q^{\alg}) = T(\Q)\backslash\{\mu_0\}\times T(\A_f)/K,
\]
equipped with the Galois action obtained from the map $r_K(T,\mu_0)$. More precisely, there is a natural action of $T(\A_f)/T(\Q)K$ on $Y_K(\Q^{\alg})$ obtained via the right multiplication on $T(\A_f)$ on itself. The action of $\Gamma_E$ on $Y_K(\Q^{\alg})$ is the one induced from that of $T(\A_f)/T(\Q)K$ via the map $r_K(T,\mu_0)$. For general $K$, choose a neat compact open subgroup $K'\subset K$. Then $Y_K$ will be the stack quotient of $Y_{K'}$ by the action of the finite group $K/K'$.

%%%%%%%%%%%%%%%%%%%%%%%%%%%%%%%%%%

\subsection{The integral model}\label{ss:integral_model_Y}

%%%%%%%%%%%%%%%%%%%%%%%%%%%%%%%%%%

We will now choose a particular maximal compact $K_0\subset T(\A_f)$. 
Using the identification
\[
T_{so}(\Q_p) = \frac{(\Q_p\otimes_{\Q}E)^\times}{(\Q_p\otimes_{\Q}F)^\times},
\]
we define $K_{0,p,so}$ to be the subgroup
\begin{equation}\label{eqn:K0p so}
K_{0,p,so} = \frac{(\Z_p\otimes_{\Z}\co_E)^\times}{(\Z_p\otimes_{\Z}\co_F)^\times}\subset T_{so}(\Q_p).
\end{equation}
This will be the image of $K_0$ in $T_{so}(\Q_p)$.

The long exact sequence of $\Gamma_{\Q_p}$-cohomology associated with the short exact sequence 
\[
1\to T_F^1 \to T_E\to T\to 1
\]
 gives us a short exact sequence
\begin{align}\label{Eqn:maximal_compact}
0\to\frac{(\Q_p\otimes_{\Q}E)^\times}{(\Q_p\otimes_{\Q}F)^{\mathrm{Nm}=1}}\to T(\Q_p) \to \bigoplus_{\mathfrak{p}\vert p}\Q_p^\times/\mathrm{Nm}(F_{\mathfrak{p}}^\times)\to 0.
\end{align}
Here, $\mathfrak{p}$ varies over the $p$-adic places of $F$. 
Now define $K_{0,p}\subset T(\Q_p)$ to be the largest subgroup mapping to $K_{0,p,so}\subset T_{so}(\Q_p)$, and to $\bigoplus_{\mathfrak{p}\vert p}\Z_p^\times/\mathrm{Nm}(\co_{F_{\mathfrak{p}}}^\times)$ under~\eqref{Eqn:maximal_compact}. It sits in a short exact sequence
\[
1\to\Z_p^\times\to K_{0,p}\to K_{0,p,so}\to 1.
\]
Finally, set $K_0=\prod_p K_{0,p}$.

For any compact open subgroup $K\subset T(\A_f)$, let $\mathcal{Y}_K$ be the normalization of $\Spec(\co_E)$ in $Y_K$ (see Definition~\ref{defn:normalization} below). Set $Y_0=Y_{K_0}$, and $\mathcal{Y}_0 = \mathcal{Y}_{K_0}$. 

\begin{proposition}\label{prop:Integral_Model_0_cycle}
Let $K\subset T(\A_f)$ be a compact open subgroup. Suppose that $p$ is a prime such that $K_p = K_{0,p}$, and set $\co_{E,p} = \co_E\otimes_{\Z}\Z_p$. Then $\mathcal{Y}_{K}\otimes_{\co_E}\co_{E,p}$ is finite \'etale over $\co_{E,p}$. In particular, $\mathcal{Y}_0$ is a finite \'etale algebraic stack over $\co_E$.
\end{proposition}
\begin{proof}
	By construction, $\mathcal{Y}_K$ is normal and finite flat over $\co_E$. We need to study the ramification of $Y_K$. This is easily done from its explicit description. 

	Fix a prime $\mathfrak{q}\subset\co_E$ lying above a rational prime $p$. Fix an algebraic closure $\Q_p^{\alg}$ of $\Q_p$ and an embedding $\eta_p:\Q^{\alg}\hookrightarrow\Q_p^{\alg}$ such that the closure of $\eta_p(E)\subset\Q_p^{\alg}$ is $E_{\mathfrak{q}}$. This allows us to identify $\Gamma_{E_{\mathfrak{q}}} = \Gal(\Q_p^{\alg}/E_{\mathfrak{q}})$ with a subgroup of $\Gamma_E$.

	Set
\begin{align*}
	Y_{K,\mathfrak{q}} &=  Y_K\times_{\Spec(E)}\Spec(E_{\mathfrak{q}}) \\
	\mathcal{Y}_{K,\mathfrak{q}}  & = \mathcal{Y}_K\times_{\Spec(\co_E)}\Spec(\co_{E,\mathfrak{q}}),
\end{align*}
so that   $Y_{K,\mathfrak{q}}$ is a finite \'etale algebraic stack over $E_{\mathfrak{q}}$. 
Assume that $K$ is neat. Then $Y_{K,\mathfrak{q}}$ is the finite \'etale \emph{scheme} over $E_{\mathfrak{q}}$ 
associated, via the local reciprocity map, with the composition
\[
E_{\mathfrak{q}}^\times\xrightarrow{r(T,\mu_0)_{\mathfrak{q}}}T(\Q_p)\to T(\A_f)/T(\Q)K.
\]
Therefore, the ramification of $Y_{K,\mathfrak{p}}$ over $E_{\mathfrak{p}}$ is controlled by
\[
\ker(\co_{E,\mathfrak{q}}^\times\xrightarrow{r(T,\mu_0)_{\mathfrak{q}}}T(\Q_p)/K_p).
\]
More precisely, the completed \'etale local ring of $\mathcal{Y}_{K,\mathfrak{q}}$ at any $\F^{\alg}_p$-valued point will be the ring of integers in the finite abelian extension of the compositum $W(\F^{\alg}_p)\co_{E,\mathfrak{q}}$ classified by the above compact open subgroup of $\co_{E,\mathfrak{q}}^\times$.

From this, we conclude that to show that $\mathcal{Y}_{K,\mathfrak{q}}$ is finite \'etale over $\co_{E,\mathfrak{q}}$, it is enough to show that 
\[
r(T,\mu_0)_{\mathfrak{q}}(\co_{E,\mathfrak{q}}^\times)\subset K_{0,p}.
\]
From the definition of $r(T,\mu_0)$, this subgroup is exactly the image of $\co_{E,\mathfrak{q}}^\times$ under the map $T_E(\Q_p)\to T(\Q_p)$. It follows easily from the definition of $K_{0,p}$ that it must contain this image.
\end{proof}

%%%%%%%%%%%%%%%%%%%%%%%%%%%%%%%%%%

\subsection{Automorphic sheaves I}
\label{ss:sheaves_i}

%%%%%%%%%%%%%%%%%%%%%%%%%%%%%%%%%%

Fix a compact open subgroup $K\subset T(\A_f)$. We will now construct some natural sheaves on $\mathcal{Y}_K$.

First, suppose that we have an algebraic $\Q$-representation $N$ of $T$. Then we obtain a local system of $\Q$-vector spaces
\[
T(\Q)\backslash\left(\{\mu_0\}\times N \times T(\A_f)/K\right)
\]
over $\mathcal{Y}_K(\C)$. If we fix a $K$-stable lattice $N_{\widehat{\Z}}\subset N_{\A_f}$, we get a local system $\bm{N}_B$ of $\Z$-vector spaces underlying this local system, where the fiber of $\bm{N}_B$ over a point $[(\mu_0,t)]$ of $\mathcal{Y}_K(\A_f)$ is $tN_{\widehat{\Z}}\cap N$. 

We can also associate with $N$ the vector bundle
\[
\bm{N}_{\dR,\C} = T(\Q)\backslash\left(\{\mu_0\}\times N_{\C} \times T(\A_f)/K\right)
\]
over $\mathcal{Y}_K(\C)$. Here, we have equipped $N_{\C}$ with its topological structure as a $\C$-vector space. There is a canonical isomorphism
\begin{equation*}
%\label{eqn:betti de rham zero}
\co_{\mathcal{Y}_K(\C)}\otimes_{\Z}\bm{N}_B\xrightarrow{\simeq} \bm{N}_{\dR,\C}
\end{equation*}
of vector bundles over $\mathcal{Y}_K(\C)$.

Let $v_0$ be the infinite place of $E$ underlying the distinguished embedding $\iota_0$. Via the identification $E_{v_0} = \C$ obtained from $\iota_0$, we have a homomorphism
\begin{equation}\label{eqn:hodge structure zero}
\C^\times = E_{v_0}^\times\xrightarrow{r(T,\mu_0)_{v_0}} T(\R)
\end{equation}
where $r(T,\mu_0)_{v_0}$ is as defined in~\S\ref{ss:zero dimensional}. This homomorphism equips the $T_\C$-representation $N_{\C}$ with a Hodge structure whose Hodge filtration is split by the cocharacter $\mu_0^{-1}$. Since this Hodge structure is $T(\Q)$-equivariant, it descends to one on the vector bundle $\bm{N}_{\dR,\C}$, and so we obtain a Hodge filtration $\Fil^\bullet\bm{N}_{\dR,\C}$ on $\bm{N}_{\dR,\C}$. Therefore, the tuple
\[
(\bm{N}_B,\bm{N}_{\dR,\C},\Fil^\bullet\bm{N}_{\dR,\C},\co_{\mathcal{Y}_K(\C)}\otimes_{\Z}\bm{N}_B\xrightarrow{\simeq} \bm{N}_{\dR,\C})
\]
corresponds to a variation of $\Z$-Hodge structures on $\mathcal{Y}_K(\C)$, which we denote by $\bm{N}_{\mathrm{Hdg}}$. The assignment of $\bm{N}_{\mathrm{Hdg}}$ to the pair $(N,N_{\widehat{\Z}})$ is clearly functorial.

Fix a prime $\ell$. Suppose that $K'_\ell\subset K_{\ell}$ is a compact open subgroup, and set $K'=K'_\ell K^\ell\subset T(\A_f)$. 
Then the proof of Proposition \ref{prop:Integral_Model_0_cycle} shows that the map of integral models $\mathcal{Y}_{K'}\to \mathcal{Y}_K$ is finite \'etale over $\co_E[\ell^{-1}]$. In particular, the pro-finite $\co_E$-scheme
\[
\mathcal{Y}_\ell[\ell^{-1}] = \varprojlim_{K'_\ell\subset K_{\ell}}\mathcal{Y}_{K'}[\ell^{-1}]
\]
is a pro-finite Galois cover of $\mathcal{Y}_K[\ell^{-1}]$ with Galois group $K_\ell$. Therefore, we obtain a functor from continuous $\ell$-adic representations of $K_\ell$ to locally constant $\ell$-adic sheaves on $\mathcal{Y}_K[\ell^{-1}]$ via the contraction product
\begin{equation*}
%\label{eqn:etale_realization}
N_{\ell}\mapsto \bm{N}_{\ell} = (\mathcal{Y}_{\ell}[\ell^{-1}]\times N_{\ell})/K_\ell.
\end{equation*}

The next result is easily checked from the definitions.
\begin{proposition}
\label{prop:betti etale zero}
Suppose that $N$ is a $\Q$-linear algebraic representation of $T$ and that $N_{\widehat{\Z}}\subset N_{\A_f}$ is a $K$-stable $\widehat{\Z}$-lattice. Then, for each $\ell$, there is a canonical isomorphism
\[
\Z_\ell\otimes\bm{N}_B \xrightarrow{\simeq}\bm{N}_{\Z_\ell}\vert_{\mathcal{Y}_K(\C)}
\]
of $\ell$-adic local systems on $\mathcal{Y}_K(\C)$.
\end{proposition}

%%%%%%%%%%%%%%%%%%%%%%%%%%%%%%%%%%

\subsection{Abelian schemes}
\label{ss:abelian schemes}

%%%%%%%%%%%%%%%%%%%%%%%%%%%%%%%%%%

The norm character $\mathrm{Nm}_{E/\Q}:T_E \to \Q$ factors through a homomorphism $\mathrm{Nm}:T \to \mathbb{G}_m$. Suppose that $H$ is a faithful $\Q$-representation of $T$ that admits a $T$-invariant symplectic pairing
\[
\psi: H\times H \to \Q(\mathrm{Nm})
\]
such that the Hodge structure on $H$ arising from the map~\eqref{eqn:hodge structure zero} has weights $(0,-1),(-1,0)$ and is polarized by $\psi$.

For any $K$-stable lattice $H_{\widehat{\Z}}\subset H_{\A_f}$ on which $\psi$ takes values in $\widehat{\Z}$, the associated variation of $\Z$-Hodge structures $\bm{H}_{\mathrm{Hdg}}$ over $\mathcal{Y}_K(\C)$ is the homology of a polarized abelian scheme over $\mathcal{Y}_{K,\C}$. This abelian scheme is associated with a map of Shimura varieties
\[
\mathcal{Y}_{K,\C} \to \mathcal{X}_{r,m,\C},
\]
where $2r = \dim_\Q( H)$, $m^2$ is the discriminant of $\psi$ restricted to $H_{\widehat{\Z}}$, and $\mathcal{X}_{r,m}$ is the Siegel modular scheme over $\Z$ parameterizing polarized abelian varieties of dimension $r$ and degree $m$. 
By the theory of canonical models, this descends to a map $Y_K \to \mathcal{X}_{r,m,E}$ over $E$, and so we obtain a polarized abelian scheme $\mathcal{A}_{H,\Q} \to Y_K$. 

\begin{proposition}
\label{prop:abelian schemes realization zero}
The abelian scheme $\mathcal{A}_{H,\Q}$ extends canonically to an abelian scheme $\mathcal{A}_H \to \mathcal{Y}_K$. Moreover, for any prime $\ell$, the $\ell$-adic Tate module of $\mathcal{A}_H$, viewed as an $\ell$-adic sheaf over $\mathcal{Y}_K[\ell^{-1}]$ is canonically isomorphic to $\bm{H}_{\Z_\ell}$.
	% \item The first relative de Rham homology of $\mathcal{A}_H$ is canonically isomorphic to $\bm{H}_{\dR}$ as a filtered vector bundle over $\mathcal{Y}_K$.
	% \item For any prime $\mathfrak{q}\subset\co_E$, and any point $y\in\mathcal{Y}_K(\F_{\mathfrak{q}}^{\alg})$, the contravariant Dieudonn\'e $F$-crystal of $\mathcal{A}_{H,y}$ is canonically isomorphic to $\bm{H}_{\cris,y}$, and the relative de Rham homology of $\mathcal{A}_{H,\co_y}$ over $\co_y$ is canonically isomorphic to $\bm{H}_{\dR,\co_y}$.
\end{proposition}
\begin{proof}
For any prime $\ell$, it is immediate from Proposition~\ref{prop:betti etale zero} and the theory of canonical models that the $\ell$-adic Tate module of $\mathcal{A}_{H,\Q}$ is canonically isomorphic to the restriction of $\bm{H}_{\Z_\ell}$ to $Y_K$. Since this sheaf extends to a lisse sheaf over $\mathcal{Y}_K[\ell^{-1}]$, it follows from the N\'eron-Ogg-Shafarevich criterion for good reduction that $\mathcal{A}_{H,\Q}$ extends to an abelian scheme over $\mathcal{Y}_K[\ell^{-1}]$ for each prime $\ell$, and hence to an abelian scheme $\mathcal{A}_H\to \mathcal{Y}_K$, whose $\ell$-adic Tate module is canonically isomorphic to $\bm{H}_{\Z_\ell}$.
\end{proof}

We will now give an explicit construction of such a symplectic representation. Write $\mathrm{CM}(E)$ for the set of CM types $\Phi$ for $E$; these are the subsets $\Phi\subset\mathrm{Emb}(E)$ satisfying 
\[
\Phi\sqcup\overline{\Phi} = \mathrm{Emb}(E).
\] 
The \emph{total reflex algebra} of $E$ is the \'etale $\Q$-algebra $E^\sharp$ equipped with an isomorphism 
\[
\Hom(E^\sharp , \Q^\alg) \xrightarrow{\simeq} \mathrm{CM}(E)
\]
as sets with $\Gamma_\Q$-actions. It is easily checked that $E^\sharp$ is a product of CM fields, and is in particular equipped with a canonical complex conjugation $x\mapsto \overline{x}$, corresponding to the involution $\Phi\mapsto \overline{\Phi}$ on $\mathrm{CM}(E)$.

There is a \emph{total reflex norm}
$\mathrm{Nm}^\sharp:T_E \to T_{E^\sharp}$, which factors through an embedding 
\begin{equation}\label{factor reflex}
\mathrm{Nm}^\sharp:T\hookrightarrow T_{E^\sharp}.
\end{equation}
This map can be described explicitly on the level of the associated character groups. 
Using the natural identification of $\Gamma_{\Q}$-modules  of (\ref{character groups}), along with 
\[
X^*(T_{E^\sharp}) = \bigoplus_{\Phi\in\mathrm{CM}(E)}\Z\cdot[\Phi],
\]
it is given by 
\[
X^*(\mathrm{Nm}^\sharp):[\Phi] \mapsto \sum_{\iota\in\Phi}[\iota].
\]

Write $H^\sharp$ for $E^\sharp$ viewed as a representation of $T_{E^\sharp}$ via multiplication. 
Via the map $\mathrm{Nm}^\sharp:T\to T_{E^\sharp}$ of (\ref{factor reflex}), we can consider $H^\sharp$ also as a representation of $T$. 

For $\Phi\in \mathrm{CM}(E)$, write $\iota_{\Phi}$ for the corresponding element in 
\[
\Hom(E^\sharp,\Q^\alg) = \Hom(E^\sharp,\C).
\]
Fix a non-zero element $\xi\in E^\sharp$ such that, for any $\Phi\in \mathrm{CM}(E)$ with $\iota_0\in \Phi$, we have $\iota_{\Phi}(\xi)\in \R_{>0}\cdot i$.

The following proposition is an easy check from the definitions.

\begin{proposition}\label{prop:h sharp repn}
The pairing $(x,y)\mapsto \mathrm{Tr}_{E^\sharp/\Q}(\xi x\overline{y})$ gives rise to a $T$-equivariant symplectic pairing
\[
\psi^\sharp:H^\sharp \times H^\sharp \to \Q(\mathrm{Nm})
\]
such that the Hodge structure on $H^\sharp$ arising from the map~\eqref{eqn:hodge structure zero} has weights $(0,-1),(-1,0)$ and is polarized by $\psi^\sharp$.
\end{proposition}

%%%%%%%%%%%%%%%%%%%%%%%%%%%%%%%%%%

\subsection{Automorphic sheaves II}
\label{ss:sheaves_ii}

%%%%%%%%%%%%%%%%%%%%%%%%%%%%%%%%%%

Recall from \S\ref{ss:sheaves_i} that we have a canonical functor $N\mapsto (\bm{N}_{\dR,\C},\Fil^\bullet\bm{N}_{\dR,\C})$ from algebraic $\Q$-representations $N$ of $T$ to filtered vector bundles over $\mathcal{Y}_K(\C)$. We can interpret this a bit differently. Given any $E$-linear algebraic representation $M$ of $E\otimes_\Q T$, the constant vector bundle
\[
\{\mu_0\}\times M_\C \times T(\A_f)/K
\]
over $\{\mu_0\}\times T(\A_f)/K$ is $T(\Q)$-equivariant, and so descends to a vector bundle $\bm{M}_{\dR,\C}$ over $\mathcal{Y}_K(\C)$. When restricted to a $\Q$-linear representation $N$ equipped with the filration $\Fil^\bullet N_E$ split by the cocharacter $\mu_0$, this functor recovers the filtered vector bundle associated with $N$. 

\begin{proposition}\label{prop:zero dim derham}
For any $E$-linear representation $M$ of $E\otimes_\Q T$, the vector bundle $\bm{M}_{\dR,\C}$ has a canonical and functorial descent to a vector bundle $\bm{M}_{\dR,\Q}$ over $Y_K$. In particular, for any $\Q$-linear representation $N$, the filtered vector bundle $(\bm{N}_{\dR,\C},\Fil^\bullet\bm{N}_{\dR,\C})$ has a functorial descent to a filtered vector bundle $(\bm{N}_{\dR,\Q},\Fil^\bullet\bm{N}_{\dR,\Q})$ over $Y_K$.
\end{proposition}
\begin{proof}
This is essentially a consequence of Deligne's theorem showing that all Hodge cycles on abelian varieties are absolutely Hodge~\cite[Ch. I]{dmos}; see also~\cite[\S 3.15]{Harris1985-tv}. We sketch a proof here.

Take $H$ to be a faithful $\Q$-representation of $T$ as in \S~\ref{ss:abelian schemes}, so that the associated variation of Hodge structures $\bm{H}_{\mathrm{Hdg}}$ (associated with some choice of $K$-invariant lattice in $H_{|A_f}$) corresponds to a canonical abelian scheme $\mathcal{A}_H$ over $\mathcal{Y}_K$. We can always find such a representation; see Proposition~\ref{prop:h sharp repn}. The relative first de Rham homology of $\mathcal{A}_{H}$ over $Y_K$ gives a canonical descent of $\bm{H}_{\dR,C}$ to a vector bundle $\bm{H}_{\dR,\Q}$ over $Y_K$.

Let $H^\otimes$ (resp. $\bm{H}^\otimes_{\dR,\Q}$) be the direct sum of tensor powers of $H$ (resp. $\bm{H}^\otimes_{\dR,\Q}$) and its dual, and let $\{t_{\beta}\}\subset H^\otimes$ be a collection of tensors whose pointwise stabilizer in $\GL(H)$ is $T$. By the functoriality of the construction $N\mapsto \bm{N}_{\mathrm{Hdg}}$, these tensors give rise to Hodge tensors
\[
\{\bm{t}_{\beta,\dR,\C}\}\subset H^0(\mathcal{Y}_{K,\C},\Fil^0\bm{H}_{\dR,\C}^\otimes).
\]
By Deligne's theorem, these tensors descend to a collection
\[
\{\bm{t}_{\beta,\dR,\Q}\}\subset H^0(\mathcal{Y}_{K,\Q},\Fil^0\bm{H}_{\dR,\Q}^\otimes).
\]
See~\cite[Corollary 2.2.2]{KisinJAMS}.

Now, the functor on $Y_K$-schemes carrying a $Y_K$-scheme $S$ to the set of isomorphisms
\[
\eta:\co_S\otimes_\Q H \xrightarrow{\simeq}\bm{H}_{\dR,\Q}\vert_S
\]
satisfying $\eta(1\otimes t_\beta) = t_{\beta,\dR,\Q}$, for all indices $\beta$ is represented by a $T$-torsor $\mathcal{P}_{T,\Q}\to Y_K$. Moreover, this $T$-torsor is canonical and does not depend on the choice of representation $H$. This can be seen by comparing the torsors obtained from $H$ and a different representation $H'$ with the one associated with the direct sum $H\oplus H'$; see the argument in~\cite[p. 177]{Harris1985-tv}.

The construction of the functor $M\mapsto \bm{M}_{\dR,\Q}$ is now simple: We will take $\bm{M}_{\dR,\Q}$ to be the contraction product
\[
\bm{M}_{\dR,\Q} \define (\mathcal{P}_{T,\Q}\times M)/T,
\]
where $T$ acts diagonally on $\mathcal{P}_{T,\Q}\times M$.

\end{proof}

We now want to extend this construction over the integral model $\mathcal{Y}_K$. We will do this using integral $p$-adic Hodge theory. Let $\mathfrak{q}\subset \co_E$ be a prime lying above a rational prime $p$. Fix an algebraic closure $\F_{\mathfrak{q}}^{\alg}$ for $\F_{\mathfrak{q}}$ and also an algebraic closure $\mathrm{Frac}(W)^{\alg}$ of the fraction field $\mathrm{Frac}(W)$ of $W = W(\F_{\mathfrak{q}}^{\alg})$. Choose an embedding $\Q^{\alg} \hookrightarrow \mathrm{Frac}(W)^{\alg}$ inducing the place $\mathfrak{q}$ on $E = \iota_0(E)$. 

Let $\co_y$ be the completion of $\mathcal{Y}_K$ at an $\F^{\alg}_{\mathfrak{q}}$-valued point $y$. Write $W_{\mathfrak{q}}$ for the ring of integers in the extension of $\mathrm{Frac}(W)$ generated by the image of $E_{\mathfrak{q}}$. 
Let 
\[
I_{\mathfrak{q}} = \Gal(\mathrm{Frac}(W)^{\alg}/\mathrm{Frac}(W_{\mathfrak{q}}))
\]
 be the absolute Galois group of $\mathrm{Frac}(W_{\mathfrak{q}})$. If $\Q^{\alg}_p\subset\mathrm{Frac}(W)^{\alg}$ is the algebraic closure of $\Q_p$, then $I_{\mathfrak{q}}$ is identified with the inertia subgroup of $\Gamma_{E_{\mathfrak{q}}} = \Gal(\Q_p^{\alg}/E_{\mathfrak{q}})$.
Fix an embedding of $\mathrm{Frac}(W_{\mathfrak{q}})$-algebras $\mathrm{Frac}(\co_y)\hookrightarrow\mathrm{Frac}(W)^{\alg}$, and let 
\[
\Gamma_y = \Gal(\mathrm{Frac}(W)^{\alg}/\mathrm{Frac}(\co_y))
\] 
be the absolute Galois group of $\mathrm{Frac}(\co_y)$. Then $\Gamma_y$ is a finite index subgroup of $I_{\mathfrak{q}}$.

If $N_p$ is a continuous $p$-adic representation of $K_p$, we obtain from it a lisse $p$-adic sheaf $\bm{N}_p$ over $Y_K$, and restricting further to $\Spec(\mathrm{Frac}(\co_y))$ gives us a continuous representation of $\Gamma_y$, which we will denote by $\bm{N}_{p,y}$.

\begin{proposition}
\label{prop:p_adic_crystalline}
Suppose that $N_p$ is a $K_p$-stable $\Z_p$-lattice in an algebraic $\Q_p$-representation $N$ of $T_{\Q_p}$. Then the $\Gamma_y$-representation $\bm{\Lambda}_{p,y}[p^{-1}]$ is crystalline. 
\end{proposition}
\begin{proof}
This is essentially due to Rapoport-Zink~\cite{rapoport_zink}. We give some details of the proof. 

Consider the map
\begin{equation}\label{eqn:reflexnorm_mfq}
\Gamma_y\hookrightarrow I_{\mathfrak{q}}^{\mathrm{ab}}\xrightarrow{\simeq}\co_{E,\mathfrak{q}}^\times\hookrightarrow T_E(\Q_p),
\end{equation}
where the isomorphism in the middle is the reciprocity isomorphism of local class field theory. 

% As in the proof of Proposition~\ref{prop:Integral_Model_0_cycle}, we find from the explicit description of $\mu_0$ in~\eqref{eqn:mu_0} that $r(T_E,\mu_0)_{\mathfrak{q}}$ is simply the natural inclusion
% \[
% \co_{E,\mathfrak{q}}^\times \hookrightarrow E_{\mathfrak{q}}^\times \hookrightarrow T_E(\Q_p).
% \]

Via the map~\eqref{eqn:reflexnorm_mfq}, given any algebraic $\Q_p$-linear representation 
$M$ of $T_E$ and a $K_p$-stable $\Z_p$-lattice $M_p\subset M$, 
we obtain, in a functorial way, a continuous $\Z_p$-linear representation $\bm{M}_p$ of $\Gamma_y$. 
The associated $\Q_p$-linear representation $\bm{M}_p[p^{-1}]$ does not depend on the choice of the lattice $M_p$.

In particular, since any representation of $T$ is naturally a representation of $T_E$, we obtain a functor from $\Q_p$-linear algebraic representations of $T$ to continuous $\Q_p$-representations of $\Gamma_y$. Applying this functor to $N$, we obtain a continuous $\Q_p$-representation of $\Gamma_y$. Using the description of the functor $M_p\mapsto\bm{M}_{p}$ above, as well as of the $\Q$-structure on $Y_K(\Q^{\alg})$ in~\S\ref{ss:zero dimensional}, it is easy to verify that this representation is precisely $\bm{N}_{p,y}[p^{-1}]$. 

Therefore, to finish the proof, it is enough to show that, for any $\Q_p$-representation $N$ of $T_E$, $\bm{N}_{p}[p^{-1}]$ is a crystalline representation of $\Gamma_y$. It suffices to do this for a single faithful representation of $T_E$: Indeed, any other representation will yield a Galois representation that is a subquotient of tensor powers of the Galois representation associated with the chosen faithful representation of $T$. 

We choose our faithful representation to be the tautological representation $H_0$ of $T_E$ obtained from its multiplication action on the $\Q$-vector space $E$.    We have 
\[
H_{0,\Q_p} = \bigoplus_{\mathfrak{q}'\vert p} H_{0,\mathfrak{q}'},
\]
where $\mathfrak{q}'$ ranges over the $p$-adic primes of $E$, and $H_{0,\mathfrak{q}'}$ is simply $E_{\mathfrak{q}'}$ viewed as a representation of $T_{E,\Q_p}$. 

By the explicit description of~\eqref{eqn:reflexnorm_mfq}, we find that the associated representation of $\Gamma_y$ also admits a direct sum decomposition
\[
\bm{H}_{0,p}[p^{-1}] = \bigoplus_{\mathfrak{q}'\vert p}\bm{H}_{0,\et,\mathfrak{q}'}[p^{-1}],
\]
where $\Gamma_y$ acts on $\bm{H}_{0,\et,\mathfrak{q}}$ via the reciprocity isomorphism $I_{\mathfrak{q}}\xrightarrow{\simeq}\co_{E,\mathfrak{q}}^\times$ of local class field theory, and \emph{trivially} on $\bm{H}_{0,\et,\mathfrak{q}'}$ for $\mathfrak{q}\neq \mathfrak{q}'$. 

Therefore, it is enough to show that $\bm{H}_{0,\et,\mathfrak{q}}[p^{-1}]$ is crystalline. But, by the construction of the local reciprocity isomorphism using Lubin-Tate theory, this is simply the rational Tate module $T_{\pi_E}(\mathcal{G}_{\mathfrak{q}})[p^{-1}]$, where $\mathcal{G}_{\mathfrak{q}}$ is the Lubin-Tate group over $\co_{E,\mathfrak{q}}$ associated with some choice of uniformizer $\pi\in E_{\mathfrak{q}}$. 
\end{proof}

\begin{remark}
\label{rem:Eq times reps}
From the proof above, we see that the homomorphism
\[
\Gamma_y \to T(\Q_p) 
\]
giving rise to the functor $N_p\mapsto \bm{N}_{p,y}$ factors through the image of $E_{\mathfrak{q}}^\times$ in $T(\Q_p)$. Let $T_{\mathfrak{q}}\subset T_{\Q_p}$ be the image of $\mathrm{Res}_{E_{\mathfrak{q}}/\Q_p}\mathbb{G}_m$. Then, if we set 
\[
K_{\mathfrak{q}} = K_p\cap T_{E_{\mathfrak{q}}}(\Q_p),
\]
we actually obtain a functor from $K_{\mathfrak{q}}$-stable lattices in algebraic $\Q_p$-representations of $T_{\mathfrak{q}}$ to continuous $\Gamma_y$-representations on finite free $\Z_p$-modules. When restricted to $K_p$-stable lattices in algebraic representations of $T_{\Q_p}$, this recovers the functor $N_p\mapsto\bm{N}_{p,y}$ considered above.
\end{remark}

For the next result, we slightly expand the usual notion of an $F$-crystal over $W$: For us, it will be a finite free $W$-module $J_0$ equipped with an isomorphism $\mathrm{Fr}^*J_0[p^{-1}]\xrightarrow{\simeq}J_0[p^{-1}]$ of $W$-modules. Also a \emph{filtered finite free module} over $\co_y$ is a finite free module $J$ over $\co_y$ equipped with a filtration $\Fil^\bullet J$ by $\co_y$-linear direct summands.

\begin{corollary}
\label{cor:realizations y}
Let $M$ be an algebraic $\Q_p$-representation of $T_{\mathfrak{q}}$, and let $M_p\subset M$ be a $K_{\mathfrak{q}}$-stable $\Z_p$-lattice. Then we can associate with $M_p$ an $F$-crystal $\bm{M}_{\cris,y}$ over $W$ and a filtered finite free $\co_y$-module $\bm{M}_{\dR,\co_y}$ with the following properties:
\begin{enumerate}
	\item The assignments $M_p\mapsto \bm{M}_{\cris,y}$ and $M_p\mapsto \bm{M}_{\dR,\co_y}$ are functorial in $M_p$.
	\item If $\bm{M}_{p,y}$ is the crystalline $\Gamma_y$-representation associated with $M_p$ via Remark~\ref{rem:Eq times reps}, then we have canonical comparison isomorphisms
	\begin{align*}
	B_\cris\otimes_{\Z_p}\bm{M}_{p,y}&\xrightarrow{\simeq}B_{\cris}\otimes_{W}\bm{M}_{\cris,y};\\
	B_\dR\otimes_{\Z_p}\bm{M}_{p,y}&\xrightarrow{\simeq}B_{\dR}\otimes_{\co_y}\bm{M}_{\dR,\co_y}.
	\end{align*}
	\item If $\bm{M}_{p,y} = T_p(\mathcal{H})^\vee$ is the dual of the $p$-adic Tate module of a $p$-divisible group $\mathcal{H}$ over $\co_y$, then, in the notation of Theorem~\ref{thm:kisin_p_divisible}, we have canonical isomorphisms
	\begin{align*}
	\bm{M}_{\cris,y}&\xrightarrow{\simeq}\mathbb{D}(\mathcal{H})(W);\\
	\bm{M}_{\dR,y}&\xrightarrow{\simeq}\mathbb{D}(\mathcal{H})(\co_y)
	\end{align*}
	of $F$-crystals over $W$ and filtered finite free $\co_y$-modules, respectively. Under these isomorphisms, the comparison isomorphisms in assertion (2) are carried to the canonical $p$-adic comparison isomorphisms for abelian schemes over $\co_y$.
\end{enumerate}
\end{corollary}
\begin{proof}
Choose a uniformizer $\pi_y\in  \co_y$, and let $\mathcal{E}(u)\in W[u]$ be its associated monic Eisenstein polynomial. Then, by the theory in~\S\ref{ss:breuil_kisin}, we obtain a functor:
\[
M_p\mapsto\mathfrak{M}(M_p) \define \mathfrak{M}(\bm{M}_{p,y})
\]
from $K_{\mathfrak{q}}$-stable lattices in algebraic $\Q_p$-representations of $T_{\mathfrak{q}}$ to Breuil-Kisin modules over $\co_y$ (associated with the uniformizer $\pi_y$).

Reducing $\varphi^*\mathfrak{M}(M_p)$ mod $u$ gives us an $F$-crystal $\bm{M}_{\cris,y}$ over $W$. Reducing it mod $\mathcal{E}(u)$ gives us a finite free $\co_y$-module $\bm{M}_{\dR,y}$. The existence of the canonical comparison isomorphisms in assertion (2) follows from the properties of the functor $\mathfrak{M}$ as explained in \S~\ref{ss:breuil_kisin}. 

In particular,
\[
\bm{M}_{\dR,\co_y}[p^{-1}] = \mathrm{Frac}(\co_y)\otimes_{\mathrm{Frac}(W)}D_{\cris}(\bm{M}_{p,y})
\]
has a canonical filtration, and we will equip $\bm{M}_{\dR,\co_y}$ with the induced filtration. 

The constructions are clearly functorial in $M_p$, and their compatibility with Dieudonn\'e theory as stated in assertion (3) follows from Theorem~\ref{thm:kisin_p_divisible}.
\end{proof}

\begin{proposition}
\label{prop:realizations integral model}
Fix an algebraic $\Q$-representation $N$ of $T$ and a $K$-stable lattice $N_{\widehat{\Z}}\subset N_{\A_f}$. Then we can canonically associate with this pair a filtered vector bundle $(\bm{N}_{\dR},\Fil^\bullet\bm{N}_{\dR})$ over $\mathcal{Y}_K$. Given a prime $\mathfrak{q}\subset \co_E$, we can also canonically associate with the pair an $F$-crystal $\bm{N}_{\cris}$ over $\mathcal{Y}_{K,\F_{\mathfrak{q}}}$. These constructions have the following properties:
\begin{enumerate}
	\item They are functorial in the pair $(N,N_{\widehat{\Z}})$.
	\item The restriction of $\bm{N}_{\dR}$ to $Y_K$ is canonically isomorphic to $\bm{N}_{\dR,\Q}$ as a filtered vector bundle.
	\item If $N=H$ is as in Proposition~\ref{prop:abelian schemes realization zero} with associated abelian scheme $\mathcal{A}_H$ over $\mathcal{Y}_K$, then the filtered vector bundle $\bm{H}_{\dR}$ is canonically identified with the relative first de Rham homology of $\mathcal{A}_H$. Moreover, the $F$-crystal $\bm{H}_{\cris}$ over $\mathcal{Y}_{K,\F_{\mathfrak{q}}}$ is canonically isomorphic to the dual of the Dieudonn\'e $F$-crystal associated with the restriction of $\mathcal{A}_H$ over $\mathcal{Y}_{K,\F_{\mathfrak{q}}}$.
	\item If $y\in \mathcal{Y}_K(\F^\alg_{\mathfrak{q}})$, then the evaluation of $\bm{N}_{\dR}$ on $\Spec~\co_y$ is canonically isomorphic, as a filtered free $\co_y$-module, to the filtered module $\bm{N}_{\dR,\co_y}$ obtained from $N_{\Z_p}$ via Corollary~\ref{cor:realizations y}.
	\item With $y$ as above, the evaluation of $\bm{N}_{\cris}$ on $\mathrm{Spf}~W(\F_{\mathfrak{q}})$, viewed as a formal divided power thickening of $y$, is canonically isomorphic to the $F$-crystal $\bm{N}_{\cris,y}$, obtained from $N_{\Z_p}$ via Corollary~\ref{cor:realizations y}.
\end{enumerate}
\end{proposition}
\begin{proof}
Fix a representation $H$ of $T$ as in Proposition~\ref{prop:abelian schemes realization zero}, and a lattice $H_\Z\subset H$ such that $H_{\widehat{\Z}}\subset H_{\A_f}$ is $K$-stable, giving us an abelian scheme $\mathcal{A}_H\to \mathcal{Y}_K$. Let $\bm{H}_{\dR}$ be the first relative de Rham homology of $\mathcal{A}_H$ over $\mathcal{Y}_K$. As in the proof of Proposition~\ref{prop:zero dim derham}, if we fix tensors $\{t_{\beta}\}\subset H^\otimes$, whose pointwise stabilizer is $T$, we obtain canonical global sections $\{\bm{t}_{\beta,\dR,\Q}\}$ of $\Fil^0\bm{H}_{\dR,\Q}^\otimes$. 

We can assume that each $t_\beta$ actually lies in $H_{\widehat{\Z}}^\otimes$. Then, given a prime $p$, the $p$-adic \'etale realizations of these invariant tensors give us canonical global sections $\{\bm{t}_{\beta,p}\}$ of $\bm{H}_p^\otimes$ over $\mathcal{Y}_K[p^{-1}]$. 

Fix a prime $\mathfrak{q}\subset \co_E$ lying above $p$, and a point $y\in \mathcal{Y}_K(\F_{\mathfrak{q}}^\alg)$. Then we obtain $\Gamma_y$-invariant tensors $\{\bm{t}_{\beta,p,y}\}\subset\bm{H}_{p,y}$.

From Corollary~\ref{cor:realizations y} we obtain canonical tensors $\{\bm{t}_{\beta,\cris,y}\}\subset \bm{H}_{\cris,y}^\otimes$ and $\{\bm{t}_{\beta,\dR,\co_y}\}\subset\bm{H}_{\dR,y}^\otimes$ such that the comparison isomorphisms
\[
B_\cris\otimes_{\Z_p}\bm{H}_{p,y}\xrightarrow{\simeq}B_\cris\otimes_{W(\F_{\mathfrak{q}})}\bm{H}_{\cris,y}\;;\;B_{\dR}\otimes_{\Z_p}\bm{H}_{p,y}\xrightarrow{\simeq}B_\dR\otimes_{\co_y}\bm{H}_{\dR,\co_y}
\]
carry $1\otimes\bm{t}_{\beta,p,y}$ to $1\otimes\bm{t}_{\beta,\cris,y}$ and $1\otimes\bm{t}_{\beta,\dR,\co_y}$, respectively.

By a theorem of Blasius-Wintenberger~\cite{blasius:padic}, the restriction of $\bm{t}_{\beta,\dR,\co_y}$ to $\mathrm{Frac}(\co_y)$ is precisely the evaluation of the de Rham tensor $\bm{t}_{\beta,\dR,\Q}$ on $\Spec~\mathrm{Frac}(\co_y)$. Therefore, we find that $\bm{t}_{\beta,\dR,\Q}$ extends to a section $\bm{t}_{\beta,\dR}$ of $\bm{H}_{\dR}^\otimes$ over $\mathcal{Y}_K$. 

This also shows that, for any $K_p$-stable lattice $N_p\subset N_{\Q_p}$ in an algebraic $\Q$-representation $N$, there is a canonical isomorphism 
\[
\bm{N}_{\dR,\co_y}[p^{-1}]\xrightarrow{\simeq}\bm{N}_{\dR,\Q}\vert_{\Spec(\mathrm{Frac}(\co_y))}
\]
of filtered vector spaces. Indeed, as in the proof of Proposition~\ref{prop:zero dim derham}, both constructions arise from the $T$-torsor over $\mathrm{Frac}(\co_y)$ parameterizing trivializations $\mathrm{Frac}(\co_y)\otimes H\xrightarrow{\simeq}\bm{H}_{\dR,\co_y}[p^{-1}]$, which carry $1\otimes t_\beta$ to $\bm{t}_{\beta,\dR,\co_y}$, for each index $\beta$.

In particular, using the functoriality of the construction $N_p\mapsto \bm{N}_{\dR,\co_y}$, one deduces that there is a canonical filtered vector bundle $\bm{N}_{\dR,\mathfrak{q}}$ over $\mathcal{Y}_{K,\mathfrak{q}}$, whose restriction to $\mathcal{Y}_{K,\mathfrak{q}}[p^{-1}]$ is isomorphic to the restriction of $\bm{N}_{\dR,\Q}$, and whose evaluation at $\Spec~\co_y$, for any point $y\in \mathcal{Y}_K(\F_{\mathfrak{q}}^\alg)$, is the lattice $\bm{N}_{\dR,\co_y}$. 

The construction of the functor $N_p\mapsto \bm{N}_{\cris}$ proceeds similarly, but we only give it in the case where $\mathcal{Y}_{K,\mathfrak{q}}$ is \'etale over $\co_{E,\mathfrak{q}}$, which will suffice for our purposes. By a descent argument, we can assume that $K$ is neat, so that $\mathcal{Y}_{K,\mathfrak{q}}$ is a scheme over $\co_{E,{\mathfrak{q}}}$, and is in fact a disjoint union of schemes of the form $\mathcal{Y}'=\Spec \co_{E'}$, where $E'/E_{\mathfrak{q}}$ is a finite, unramified extension. 

Let $\F'$ be the residue field of $\co_{E'}$. Fix an embedding $\F\hookrightarrow\F^\alg_{\mathfrak{q}}$: This determines a point $y\in \mathcal{Y}'(\F^\alg_{\mathfrak{q}})$. The construction in Corollary~\ref{cor:realizations y} gives us an $F$-crystal $\bm{N}_{\cris,y}$ over $W(\F^\alg_{\mathfrak{q}})$. It is now enough to show that it has a canonical descent to an $F$-crystal $\bm{N}_{\cris,\F'}$ over $W(\F')$, which recovers the Dieudonn\'e $F$-crystal of $\mathcal{A}_H$ when $N=H$. This can be deduced from the functoriality of Kisin's functor $\mathfrak{M}$. Alternatively, it can also be deduced by observing that Kisin's functor is already defined for crystalline Galois representations of $\Gal(\Q_p^\alg/E')$, as is its compatibility with Dieudonn\'e theory of $p$-divisible groups, and we can therefore use it to produce $F$-crystals over $W(\F')$, and not just over $W(\F_{\mathfrak{q}}^\alg)$.

It remains to globalize the construction of the de Rham realization. Let $D$ be the product of the finitely many rational primes at which $E$ is ramified, or at which we have $K_p \neq K_{0,p}$. Note that $T$ extends to a torus over $\Z[D^{-1}]$. We will denote this extension again by $T$. From the construction of the compact open subgroup $K_{0,p}$ in \S~\ref{ss:zero dimensional}, we find that, for $p\nmid D$, $K_p = K_{0,p} = T(\Z_p)$. Moreover, for each such $p$, we can choose the tensors $\{t_\beta\}$ so that their stabilizer in $\GL(H_{\Z_{(p)}})$ is $T_{\Z_{(p)}}$.

We can now consider the functor on $\mathcal{Y}_K[D^{-1}]$-schemes carrying $S$ to the set of isomorphisms
\[
\xi:\co_S\otimes_{\Z}H_{\Z}\xrightarrow{\simeq}\bm{H}_{\dR}\vert_S
\]
of vector bundles over $S$ satisfying $\xi(1\otimes t_\beta) = \bm{t}_{\beta,\dR}$, for all indices $\beta$. Since $T$ is a reductive group over $\Z[D^{-1}]$, it follows from~\cite[Corollary 1.4.3]{KisinJAMS} that this functor is represented by a $T$-torsor $\mathcal{P}_T$ over $\mathcal{Y}_K[D^{-1}]$. Just as in the proof of Proposition~\ref{prop:zero dim derham}, this functor is independent of the choice of data $(H,H_{\Z})$, and we obtain from it a canonical functor 
\[
N_{\Z[D^{-1}]}\mapsto (\bm{N}_{\dR,\Z[D^{-1}]},\Fil^\bullet\bm{N}_{\dR,\Z[D^{-1}]})
\]
from algebraic representations of $T$ on finite free $\Z[D^{-1}]$-modules to filtered vector bundles over $\mathcal{Y}_K[D^{-1}]$, which has properties (2), (3) and (4).

Given an arbitrary $K$-stable lattice $N_{\widehat{\Z}}\subset N_{\A_f}$, by enlarging the set of primes appearing in the factorization of $D$ if necessary, we can assume that $N_{\widehat{\Z}[D^{-1}]}$ arises from an algebraic $\Z[D^{-1}]$-representation of $T$, and so the desired filtered vector bundle $(\bm{N}_{\dR},\Fil^\bullet\bm{N}_{\dR})$ is canonically determined outside of the primes dividing $D$. For a prime $\mathfrak{q}\subset\co_E$ dividing $D$, it is determined by the condition that its restriction to $\mathcal{Y}_{K,\mathfrak{q}}$ is isomorphic to $\bm{N}_{\dR,\mathfrak{q}}$.

\end{proof}

To summarize the results of \S~\ref{ss:sheaves_i} and of this subsection, from a pair $(N,N_{\widehat{\Z}})$ as in the Proposition above, we have obtained the following realizations:
\begin{itemize}
\item $\bm{N}_{\mathrm{Hdg}}$ in the category of variations of $\Z$-Hodge structures over $Y_K(\C)$;
\item $\bm{N}_{\dR}$ in the category of filtered vector bundles over $\mathcal{Y}_K$;
\item For each prime $\ell$, $\bm{N}_{\ell}$ in the category of lisse $\ell$-adic sheaves over $\mathcal{Y}_K[\ell^{-1}]$;
\item For each prime ${\mathfrak{q}}\subset\co_E$, $\bm{N}_{\cris}$ in the category of $F$-crystals over $\mathcal{Y}_{K,\F_{\mathfrak{q}}}$.
\end{itemize}

For $?=\mathrm{Hdg},\dR,\ell,\cris$, let $\End(\bm{N}_{?})_{\Q}$ be the endomorphism algebra of $\bm{N}_?$ in the appropriate isogeny category; this is a finite dimensional algebra over $\Q_?$, where $\Q_?=\Q$ if $?=\mathrm{Hdg}$; $\Q_? = E$ if $?=\dR$; $\Q_? = \Q_{\ell}$, if $?=\ell$; and $\Q_?=\Q_p$ if $?=\cris$. This algebra depends only on $N$ and not on the choice of $K$-stable lattice $N_{\widehat{\Z}}$.
Let $\Aut^\circ(\bm{N}_?)$ be the algebraic group over $\Q_?$ associated with the group of units in this algebra. 

Fix a representation $H$ as in Proposition~\ref{prop:abelian schemes realization zero} and a $K$-stable lattice in $H$, and let $\mathcal{A}_H$ be the associated abelian scheme over $\mathcal{Y}_K$. Let $\Aut^{\circ}(\mathcal{A}_H)$ be the algebraic group over $\Q$ obtained as the group of units in the $\Q$-algebra $\End(\mathcal{A}_H)_{\Q}$.

\begin{proposition}\
\label{prop:tQ_action}
\begin{enumerate}
\item There is a canonical map of algebraic groups,
\[
\theta_?(N):T_{\Q_?} \to \Aut^\circ(\bm{N}_?)
\]
functorial in the representation $N$.

\item There is a canonical embedding $T\hookrightarrow \Aut^{\circ}(\mathcal{A}_H)$ whose homological realizations induce the maps $\theta_?(H)$ for the representation $H$.
\end{enumerate}
\end{proposition}
\begin{proof}
The simplest way to see this is to use the torus 
\[
\widetilde{T} \define T_E.
\]
 In complete analogy with the construction of $\mathcal{Y}_K$, given a compact open $\widetilde{K}\subset \widetilde{T}(\A_f)$, we can associate with it and the cocharacter $\mu_0$ an arithmetic curve $\widetilde{\mathcal{Y}}_{\widetilde{K}}$ over $\co_E$. If the image of $\widetilde{K}$ in $T(\A_f)$ is contained in $K$, then we obtain a finite map 
\[
\widetilde{\mathcal{Y}}_{\widetilde{K}}\to \mathcal{Y}_K
\]
of algebraic $\co_E$-stacks.

We also obtain realization functors $(N,N_{\widehat{\Z}})\mapsto \bm{N}_?$ over $\widetilde{\mathcal{Y}}_{\widetilde{K}}$. Here, $N$ is an algebraic representation of $\widetilde{T}$ and $N_{\widehat{\Z}}\subset N$.

We apply this to the representation $H_0$ and the lattice $H_{0,\widehat{\Z}}$ from the proof of Proposition~\ref{prop:p_adic_crystalline} to obtain sheaves $\bm{H}_{0,?}$ over $\widetilde{\mathcal{Y}}_{\widetilde{K}}$. Since the $E$-action on $H_0$ is $\widetilde{T}$-equivariant, the sheaves just obtained are $E$-linear objects in the appropriate isogeny category. We now recover $\widetilde{T}_{\Q_?}$ as the group of $E$-equivariant automorphisms
\[
\widetilde{T}_{\Q_?} = \Aut_E^\circ(\bm{H}_{0,?}) \subset \Aut^\circ(\bm{H}_{0,?}).
\]

From this, and the fact that $H_0$ is a faithful representation of $\widetilde{T}$, it is not hard to deduce that this actually gives us a canonical map
\[
\widetilde{T}_{\Q_?} \to \Aut^\circ(\bm{N}_?),
\]
for every $\widetilde{T}$-representation $N$. We obtain the map from assertion (1) by specializing now to representations of $\widetilde{T}$ that factor through $T$.

As for assertion (2), since abelian varieties over $\C$ are a fully faithful subcategory of $\Z$-Hodge structures, the Betti realization $\theta_B(H)$ corresponds to a map
\[
T \to \Aut^{\circ} ( \mathcal{A}_{H, {Y}(\C)}).
\]

Since the \'etale realizations of this map descend over $Y$, it is easily checked that the map itself descends:
\[
T \to \Aut^{\circ} ( \mathcal{A}_{H, {Y}}).
\]

Our desired embedding is just the composition of this one with the inclusion
\[
\Aut^{\circ}(\mathcal{A}_{H,{Y}}) \hookrightarrow \Aut^{\circ}(\mathcal{A}_H).
\]
\end{proof}

%%%%%%%%%%%%%%%%%%%%%%%%%%%%%%%%%%

\subsection{The standard representation and its realizations}
\label{ss:standard}

%%%%%%%%%%%%%%%%%%%%%%%%%%%%%%%%%%

We will now consider a particular representation of $T$. As in the proof of Proposition~\ref{prop:p_adic_crystalline}, we have the tautological representation $H_0$ of $T_E$ acting on $E$ via multiplication. Let $c:E\to E$ be complex conjugation, and in the notation of~\S\ref{ss:lubin-tate_special_endomorphisms}, set:
\[
V_0 = V(H_0,c) = \{x\in\End(H_0): \; x(a\cdot h) = c(a) x(h)\text{, for all $a\in E$}\}.
\]
This is a $T_E$-subrepresentation of $\End(H_0)$ on which the action factors through $T$, and in fact through $T_{so}$. We call this the \emph{standard representation} of $T$.

The ring of integers $\co_E\subset E$ gives a natural lattice $H_{0,\Z}\subset H_0$, and hence a lattice
\[
V_{0,\Z} = V(H_{0,\Z},c) \subset V_0.
\]

% The associated $\widehat{\Z}$-lattice $V_{0,\widehat{\Z}}\subset V_{0,\A_f}$ is $K_0$-stable, and hence $K$-stable. By the theory of~\S\ref{ss:sheaves}, we obtain Betti, \'etale and de Rham realizations:
% \[
% \bm{V}_{0,B}, \bm{V}_{0,\ell}, \bm{V}_{0,\dR}
% \]
% over $\mathcal{Y}_K(\C)$, $\mathcal{Y}_K[\ell^{-1}]$ and $\mathcal{Y}_K$, respectively. Moreover, for each prime $\mathfrak{q}\subset\co_E$, we obtain a crystalline realization $\bm{V}_{0,\cris,\F_{\mathfrak{q}}}$ over $\F_{\mathfrak{q}}$. The interrelationships between these realizations are explained in Propositions~\ref{prop:betti etale zero} and~\ref{prop:de rham realization integral}.

Fix a prime $\mathfrak{q}\subset\co_E$, an algebraic closure $\mathrm{Frac}(W)^{\alg}$ of $\mathrm{Frac}(W)$ (here, $W = W(\F_{\mathfrak{q}}^{\alg})$), and an embedding $\Q_p^{\alg}\hookrightarrow\mathrm{Frac}(W)^{\alg}$ inducing the place $\mathfrak{q}$ on $E = \iota_0(E)$. Let $\Q^{\alg}_p$ be the algebraic closure of $\Q_p$ in $\mathrm{Frac}(W)^{\alg}$. We can now view $\iota_0$ as an embedding $E_{\mathfrak{q}}\hookrightarrow\Q_p^{\alg}$.

Fix a point $y\in\mathcal{Y}_K(\F_{\mathfrak{q}}^{\alg})$. We can now describe the $F$-crystal $\bm{V}_{0,\cris,y}$ quite explicitly. Fix a uniformizer $\pi_{\mathfrak{q}}\in E_{\mathfrak{q}}$, and let $\mathcal{G}_{\mathfrak{q}}$ be the Lubin-Tate group over $\co_{E,\mathfrak{q}}$ associated with this uniformizer. Let 
\[
\bm{H}_{0,\cris,\mathfrak{q}} = \mathbb{D}(\mathcal{G}_{\mathfrak{q}})(W)
\]
be the Dieudonn\'e $F$-crystal over $\mathrm{Frac}(W)$ associated with $\mathcal{G}_{\mathfrak{q}}$.
 The $\co_{E,\mathfrak{q}}$-equivariant structure on $\mathcal{G}_{\mathfrak{q}}$ induces an $\co_{E,\mathfrak{q}}$-equivariant structure on $\bm{H}_{0,\cris,\mathfrak{q}}$. 

For a prime $\mathfrak{q}'\subset\co_E$ lying over $p$ with $\mathfrak{q}'\neq \mathfrak{q}$, let
\[
\bm{H}_{0,\cris,\mathfrak{q}'} = W\otimes_{\Q_p}\co_{E,\mathfrak{q}'}
\]
be the rank $1$ $W\otimes_{\Q_p}E_{\mathfrak{q}'}$-module equipped with the constant $F$-isocrystal structure arising from the automorphism $\mathrm{Fr}\otimes 1$. 

Now, set 
\[
\bm{H}_{0,\cris,y} = \bigoplus_{\mathfrak{q}'\vert p}\bm{H}_{0,\cris,\mathfrak{q}'}.
\]
From the proof of Proposition~\ref{prop:p_adic_crystalline}, we find that this is precisely the crystalline realization obtained from the tautological representation $H_0$ of $T_E$, equipped with the standard $\co_E$-stable lattice.

The inclusion $V_0 \hookrightarrow \End(H_0)$ now gives us identifications:
\[
\bm{V}_{0,\cris,y} = V(\bm{H}_{0,\cris,y},c)\subset \End(\bm{H}_{0,\cris,y}).
\]
In particular, the decomposition of $\bm{H}_{0,\cris,y}$ gives us a decomposition:
\[
\bm{V}_{0,\cris,y} = \bigoplus_{\mathfrak{p}'\vert p}\bm{V}_{0,\cris,\mathfrak{p}'},
\]
into $F$-crystals, where $\mathfrak{p}'$ ranges over the primes of $\co_F$ lying above $p$, and where
\[
\bm{V}_{0,\cris,\mathfrak{p}'} = V(\bm{H}_{0,\cris,\mathfrak{p}'},c)\subset \End(\bm{H}_{0,\cris,\mathfrak{p}'}).
\]
Here, 
\[
\bm{H}_{0,\cris,\mathfrak{p}'} = \bigoplus_{\mathfrak{q}'\vert \mathfrak{p}'}\bm{H}_{0,\cris,\mathfrak{q}'}
\]
is an $\co_{F,\mathfrak{p}'}$-linear $F$-crystal over $W$.

\begin{proposition}
\label{prop:v0 cris realization}
Let $\mathfrak{p}\subset\co_F$ be the prime lying under $\mathfrak{q}$. Then the following statements are equivalent:
\begin{enumerate}
\item $\mathfrak{p}$ is not split in $F$. 
\item The space of $\varphi$-invariants $\bm{V}_{0,\cris,\mathfrak{p}}^{\varphi = 1}$ is non-zero. 
\item The natural map
\[
\mathrm{Frac}(W)\otimes_{\Z_p}\bm{V}_{0,\cris,\mathfrak{p}}^{\varphi = 1} \to \bm{V}_{0,\cris,\mathfrak{p}}[p^{-1}]
\]
is an isomorphism.
\item The natural map
\[
\mathrm{Frac}(W)\otimes_{\Z_p}\bm{V}_{0,\cris,y}^{\varphi = 1} \to \bm{V}_{0,\cris,y}[p^{-1}]
\]
is an isomorphism.
\end{enumerate}
\end{proposition}
\begin{proof}
If $\mathfrak{p}$ is split in $E$, we have:
\[
\bm{H}_{0,\cris,\mathfrak{p}} = \bm{H}_{0,\cris,\mathfrak{q}} \oplus \bm{H}_{0,\cris,\mathfrak{q}^c}.
\]

Moreover, $\bm{V}_{0,\cris,\mathfrak{p}} = V(\bm{H}_{0,\cris,\mathfrak{p}},c)$ consists of pairs $(x_1,x_2)$ of $\co_{F,\mathfrak{p}}$-linear maps
\[
x_1:\bm{H}_{0,\cris,\mathfrak{q}^c}\to\bm{H}_{0,\cris,\mathfrak{q}},\quad x_2:\bm{H}_{0,\cris,\mathfrak{q}}\to \bm{H}_{0,\cris,\mathfrak{q}^c}.
\]
Therefore, the space of $\varphi$-invariants consists of $\varphi$-equivariant such pairs. However, by definition, $\bm{H}_{0,\cris,\mathfrak{q}^c}$ is generated by its $\varphi$-invariants, while $\bm{H}_{0,\cris,\mathfrak{q}}$, being the Dieudonn\'e $F$-isocrystal associated with a Lubin-Tate group, has no non-zero $\varphi$-invariant elements. Thus we conclude that $\bm{V}_{0,\cris,\mathfrak{p}}^{\varphi=1}$ has no non-zero elements.

On the other hand, suppose that $\mathfrak{p}$ is not split in $E$. Then we can identify $\bm{V}_{0,\cris,\mathfrak{p}} = V(\bm{H}_{0,\cris,\mathfrak{p}},c)$ with the space $V_{\cris}(\mathcal{G}_{\mathfrak{q}})$ defined in~\S\ref{ss:lubin-tate_special_endomorphisms}. 

In Propositions~\ref{prp:unramified_vcris} and~\ref{prp:ramified_vcris}, we described the structure of $\varphi$-invariants in this space explicitly, and in particular showed that they generate the whole space over $\mathrm{Frac}(W)$.

From these considerations, the equivalence of statements (1), (2) and (3) of the proposition are immediate. The equivalence of these statements with (4) now follows from the fact that, for $\mathfrak{q}'\neq \mathfrak{q}$, $\bm{H}_{0,\cris,\mathfrak{q}'}$ is generated by its $\varphi$-invariants.
\end{proof}

Fix a representation $H$ as in Proposition~\ref{prop:abelian schemes realization zero} and a $K$-stable lattice in $H$, and let $\mathcal{A}_H$ be the associated abelian scheme over $\mathcal{Y}_K$.

% From~\S\ref{ss:zero dimensional}, we have an identification
% \begin{equation}\label{eqn:Y_K_C_points}
% \mathcal{Y}_K(\Q^{\alg}) = T(\Q)\backslash\{\mu_0\}\times T(\A_f)/K
% \end{equation}
% of groupoids equipped with a $\Gamma_E$-action. For each $y\in \mathcal{Y}(\Q^{\alg})$, this identifies the automorphism group $\Aut(y)$ with the finite group $T(\Q)\cap K$. Since the $\Gamma_E$-action is via the map~\eqref{eqn:rkTmu0} to $T(\A_f)/T(\Q)K$, it commutes with the right multiplication action of $T(\A_f)/T(\Q)K$ on $\mathcal{Y}_K(\Q^{\alg})$. Therefore, the action of $T(\A_f)/T(\Q)K$ descends to $Y_K$, and extends uniquely to an action on $\mathcal{Y}_K$. 

\begin{proposition}
\label{prop:abelian scheme reduction}
Fix a prime $\mathfrak{q}\subset\co_E$ above a rational prime $p$ and let $\mathfrak{p}\subset\co_F$ be the prime lying under it. 
The following equivalences hold:
\begin{align*}
\text{$\mathfrak{p}$ is not split in $E$}&\Leftrightarrow\text{$\mathcal{A}_{H,y}$ is supersingular for all $y\in \mathcal{Y}_K(\F_\mathfrak{q}^{\alg})$};
% \text{$\mathfrak{p}$ is split in $E$}&\Leftrightarrow\text{$\mathcal{A}_{H,y}$ is ordinary for all  $y\in \mathcal{Y}_K(\F_\mathfrak{q}^{\alg})$.}
\end{align*}
\end{proposition}
\begin{proof}
Fix a point $y\in \mathcal{Y}_K(\F_{\mathfrak{q}}^{\alg})$. By Proposition~\ref{prop:abelian schemes realization zero}, the Dieudonn\'e $F$-isocrystal associated with $\mathcal{A}_{H,y}$ is isomorphic to the $F$-isocrystal $\bm{H}_{\cris,y}$. 

Now, the slopes of the $F$-isocrystal $\bm{H}_{\cris,y}[p^{-1}]$ are determined by its Newton cocharacter
\[
\nu(H): \mathbb{D} \to \Aut^{\circ}_{\varphi}(\bm{H}_{\cris,y}),
\]
where $\mathbb{D}$ is the pro-torus over $\Q_p$ with character group $\Q$, and $\Aut^{\circ}_{\varphi}(\bm{H}_{\cris,y})$ is the algebraic group of $\Q_p$ obtained as the group of units in the $\Q_p$-algebra $\End_{\varphi}(\bm{H}_{\cris,y})_{\Q}$ of $\varphi$-equivariant endomorphisms of $\bm{H}_{\cris,y}[p^{-1}]$. 

\begin{lemma}
\label{lem:newton cocharacter}
For any $\Q$-representation $N$ of $T$, the Newton cocharacter $\nu(N)$ for $\bm{N}_{\cris,y}[p^{-1}]$ factors through the map $\theta_{\cris}(N):T_{\Q_p}\to \Aut^\circ_{\varphi}(\bm{N}_{\cris,y})$ from Proposition~\ref{prop:tQ_action}.
\end{lemma}
\begin{proof}
The Newton cocharacter is functorial in $N$. If $H_0$ is the tautological representation of $T_E$, then it is clear from the construction of $\bm{H}_{0,\cris,y}$ in~\S\ref{ss:standard} that its slope decomposition is stable under the $E\otimes_\Q\Q_p$-action, and hence that the slope cocharacter for $\bm{H}_{0,\cris,y}[p^{-1}]$ factors through the commutant in $\Aut^\circ_{\varphi}(\bm{H}_{\cris,y})$ of $E\otimes_\Q\Q_p$. This is precisely the torus $T_{E,\Q_p}$. Combining this with the fact that any $\Q$-representation of $T$, when viewed as a representation of $T_E$, appears as a subquotient of a tensor power of $H_0$, one easily deduces the lemma.
\end{proof}

Now, we find from~\eqref{prop:v0 cris realization} that $\bm{V}_{0,\cris,y}[p^{-1}]$ is generated by $\varphi$-invariants, and hence that $\nu(V_0)$ is trivial, if and only if $\mathfrak{p}$ is not split in $E$. Since the quotient $T_{so}$ of $T$ acts faithfully on $V_0$, this implies in turn that $\mathfrak{p}$ is not split in $E$ if and only if $\nu(H)$ factors through 
\[
\mathbb{G}_m = \ker(T\to T_{so}).
\]
This is the case if and only if $\nu(H)$ is constant, and hence if and only if $\mathcal{A}_{H,y}$ is supersingular.

% It only remains to show that, if $\mathfrak{p}$ is split in $E$, then $\mathcal{A}_{H,y}$ is ordinary. In this case, the computation in~\eqref{prop:v0 cris realization} shows that the Newton cocharacter $\nu(V_0)$, which factors through $T_{so,\Q_p}$, corresponds to the element $\nu_{so}\in X_*(T_{so})_\Q$ satisfying:
% \[
% \langle \nu_{so},\sum_{\iota}a_{\iota}[\iota]\rangle = - a_{\iota_0}.
% \] 
% Now, the only cocharacter $\nu\in X_*(T)_\Q$ lifting $\nu_{so}$ and satisfying $\langle \nu,\mathrm{Nm}\rangle = -1$ is $-\mu_0$. Therefore, we have $\nu(H) = -\mu_0$. 

% The abelian variety $\mathcal{A}_{H,y}$ is ordinary exactly when its Newton and Hodge polygons coincide. However, the Hodge filtration on $\bm{H}_{\dR,\co_y}[p^{-1}]$ is split by the cocharacter $-\mu_0$. This can be seen from the explicit construction of the Hodge filtration on $\bm{H}_{\dR,\C}$ over $\mathcal{Y}_K(\C)$. By what we have seen above, this coincides with the Newton cocharacter $\nu(H)$. This finishes the proof.
\end{proof}

%%%%%%%%%%%%%%%%%%%%%%%%%%%%%%%%%%%%%%%%%%%%%%%%%%

\section{Orthogonal Shimura varieties}
\label{s:orthogonal shimura}

%%%%%%%%%%%%%%%%%%%%%%%%%%%%%%%%%%%%%%%%%%%%%%%

Let $(V,Q)$ be a quadratic space over $\Q$ of signature $(n,2)$, with $n\ge 1$.    
Fix a  \emph{maximal} lattice $L \subset V$, and let $L^\vee$ be the dual lattice.
As in the introduction, the \emph{discriminant} of $L$ is $D_L = [L^\vee : L]$. 

In this section, we lay out the theory of GSpin Shimura varieties associated with $(V,Q)$ and $L$. The main references are~\cite{MadapusiSpin} and~\cite{AGHMP}. The models constructed in these references have to be modified slightly for our purposes here, and we explain this in \S~4.4. 

The main notion studied is that of a \emph{special endomorphism}, which allows us to give a moduli interpretation for the \emph{special divisors} considered by Kudla in~\cite{kudla:special_cycles}. This interpretation is crucial for the degree computations underlying the proof of Theorem~\ref{thm:arithmetic BKY}. 

%%%%%%%%%%%%%%%%%%%%%%%%%%%%%%%%%%%

\subsection{The GSpin Shimura variety}\label{ss:gspin_char_0}

%%%%%%%%%%%%%%%%%%%%%%%%%%%%%%%%%%%

Let $C(V)$ be the Clifford algebra of $(V,Q)$, with its $\Z/2\Z$-grading 
\[
C(V)=C^+(V)\oplus C^-(V) .
\]  
Recall  from \cite{MadapusiSpin} that the spinor similitude group $G=\GSpin(V)$ is the algebraic group over $\Q$  with 
\[
G(R) = \big\{ g\in C^+(V_R)^\times :  gV_Rg^{-1} = V_R \big\}
\]
for any $\Q$-algebra $R$.  It sits in an exact sequence
\[
 1 \to \mathbb{G}_m \to G \map{ g\mapsto g \action} \SO(V) \to 1,
\]  
where $g\action v=gvg^{-1}$.
Let $\nu: G \to \mathbb{G}_m$ be the spinor similitude.

The group of real points $G(\R)$ acts on the hermitian symmetric domain
\begin{equation}\label{hermitian domain}
\mathcal{D} = \{ z\in V_\C : [z,z] =0 ,\, [z,\overline{z} ] <0 \} / \C^\times \subset \mathbb{P}(V_\C)
\end{equation}
through the morphism $G \to \SO(V)$.
There are two connected components $\mathcal{D}=\mathcal{D}^+ \sqcup \mathcal{D}^-$, interchanged by the action of 
any $\gamma\in G(\R)$ with $\nu(\gamma)<0$.

The pair $(G,\mathcal{D})$ is a Shimura datum. More precisely, given a class $z\in\mathcal{D}$, we can choose a representative $z $ of the form $ x + i y$, where $x,y\in V_\R$ are mutually orthogonal vectors satisfying $Q(x) = Q(y) = -1$. Then we obtain a homomorphism
\[
\bm{h}_z : \mathbb{S} = \mathrm{Res}_{\C/\R}\mathbb{G}_m \to G_\R
\]
satisfying $\bm{h}_z(i) = xy\in G(\R)\subset C^+(V)_\R^\times$. In this way, we can identify $\mathcal{D}$ with the $G(\R)$-conjugacy class of $\bm{h}_z$, for any $z\in \mathcal{D}$. The reflex field of $(G,\mathcal{D})$ is $\Q$.

Recall that we have fixed a  maximal lattice $L \subset V$. Define a compact open subgroup
\begin{equation}\label{compact open}
K = G( \A_f ) \cap C(\widehat{L})^\times \subset G(\A_f).
\end{equation}
Here, we have set $\widehat{L} = L_{\widehat{\Z}}$. The image of $K$ in $\SO(V)(\A_f)$ is the \emph{discriminant kernel} of $\widehat{L}$; 
this is the largest subgroup of the stabilizer of $\widehat{L}$ that acts trivially on $\widehat{L}^\vee/\widehat{L}$.

By the theory of canonical models of Shimura varieties, we obtain an $n$-dimensional algebraic stack $M$ over $\Q$, the \emph{GSpin Shimura variety} associated with $L$. Its space of complex points is the $n$-dimensional complex orbifold
\begin{equation}\label{gspin shimura}
M (\C) = G(\Q) \backslash \mathcal{D} \times G(\A_f) / K.
\end{equation}

\begin{proposition}\label{prop:generic connected}
Suppose that one of the following conditions holds:
\begin{itemize}
	\item $n\geq 2$;
	\item $D_L$ is square-free. 
\end{itemize}
Then the complex orbifold $M(\C)$ is connected.  
\end{proposition}
\begin{proof}
The kernel of $\nu: G\to \mathbb{G}_m$ is the usual spin double cover of 
$\SO(V)$, and hence is simply connected.  Using strong approximation, it follows that the connected components of $M(\C)$ are indexed by 
$\Q^\times_{>0} \backslash \A_f^\times / \nu(K)$, and so the claim follows once we prove that $\nu(K_\ell) = \Z_\ell^\times$ for every prime $\ell$. When $L_{\Z_\ell}$ contains a hyperbolic plane, the assertion is clear, so we only need to consider the case where $V_{\Q_\ell}$ is anisotropic of dimension at least $3$, and is such that $\ell^2$ does not divide $D_L$. In this case, the result can be deduced from the classification of maximal anisotropic lattices over $\Z_\ell$; see~\cite[\S 29.10]{ShimuraQuadratic}.
\end{proof}

Given an algebraic representation $G \to \mathrm{Aut}(N)$ on a $\Q$-vector space $N$, and  a $K$-stable lattice $N_{\widehat{\Z}}\subset N_{\A_f}$,
we obtain a  $\Z$-local system $\bm{N}_{B}$ on $M(\C)$ whose fiber at a point $[(z,g)]\in M(\C)$ is identified with
$N\cap g N_{\widehat{\Z}}$. The corresponding vector bundle $\bm{N}_{\dR,M(\C)}=\co_{M(\C)}\otimes\bm{N}_{B}$
is equipped with a filtration $\Fil^\bullet\bm{N}_{\dR,M(\C)}$, which at any point $[(z,g)]$ equips the fiber of $\bm{N}_{B}$
with the Hodge structure determined by the cocharacter $\bm{h}_z$. This gives us a functorial assignment from pairs $(N,N_{\widehat{\Z}})$ as above to variations of $\Z$-Hodge structures over $M(\C)$.

Applying this to $V$ and the lattice $\widehat{L}\subset V_{\A_f}$, we obtain a canonical variation of polarized $\Z$-Hodge structures $(\bm{V}_{B},\Fil^\bullet\bm{V}_{\dR,M(\C)})$.   For each point $z$ of (\ref{hermitian domain}) the induced Hodge decomposition of $V_\C$
has
\[
V_\C^{(1,-1)} = \C z ,
\qquad  V_\C^{(-1,1)}=\C\overline{z},
\qquad V_\C^{(0,0)} = (\C z + \C \overline{z} )^\perp.
\]
It follows that $\Fil^1\bm{V}_{\dR,M(\C)}$ is an isotropic line and $\Fil^0\bm{V}_{\dR, M(\C)}$ is its annihilator with respect
to the pairing on $\bm{V}_{\dR, M(\C)}$ induced from that on $L$.

Let $H$ be the representation of $G$ on $C(V)$ via left multiplication. It is equipped with a $K$-stable lattice $H_{\widehat{\Z}} = C(\widehat{L})\subset H_{\A_f}$. From this, we obtain a variation of $\Z$-Hodge structures $(\bm{H}_{B},\Fil^\bullet\bm{H}_{\dR,M(\C)})$. This variation has type $(-1,0),(0,-1)$ and is therefore the homology of a family of complex tori over $M(\C)$. This variation of $\Z$-Hodge structures is \emph{polarizable}, and so the family of complex tori in fact arises from an \emph{abelian scheme} $A_\C \to M_\C$. For all this, see~\cite[(2.2)]{AGHMP}. 

By~\cite[\S 3]{MadapusiSpin}, this abelian scheme descends to an abelian scheme $A\to M$. We call this the \emph{Kuga-Satake abelian scheme}. It is equipped with a right $C(L)$-action and a compatible $\Z/2\Z$-grading
\[
A = A^+ \times A^-.
\]

The first relative de Rham homology sheaf of $A$ gives a canonical descent of $\bm{H}_{\dR,\C}$ over $M$ as a filtered vector bundle with an integrable connection. We denote this descent by $\bm{H}_{\dR}$. 
Using it, and Deligne's results on absolute Hodge cycles on abelian varieties, we obtain a canonical tensor functor from algebraic $\Q$-representations $N$ of $G$ to filtered vector bundles $(\bm{N}_{\dR},\Fil^\bullet\bm{N}_{\dR})$ over $M$, which descends the already constructed functor to objects over $M_\C$. 

Similarly, if we fix a lattice $N_{\widehat{\Z}}\subset N_{\A_f}$, then, for any prime $\ell$, the $\ell$-adic sheaf $\bm{N}_{\ell,\C} = \Z_\ell\otimes\bm{N}_B$ over $M(\C)$ descends canonically to an $\ell$-adic sheaf $\bm{N}_{\ell}$ over $M$. When $N=H$, $\bm{H}_{\ell}$ is canonically isomorphic to the $\ell$-adic Tate module of $A$. 

For all this, see~\cite[(4.15)]{MadapusiSpin}.

In particular, for $? = B,\ell,\dR$, the $G$-equivariant embedding $V \hookrightarrow \End_{C(V)}(H)$ determined by left multiplication gives rise to embeddings of homological realizations
\begin{equation}\label{Eqn:V_embeddings_char_0}
\bm{V}_{?} \hookrightarrow \underline{\End}_{C(L)}(\bm{H}_{?}).
\end{equation}

For $x\in V$ with $Q(x) > 0$, define a divisor on  $\mathcal{D}$ by
\[
\mathcal{D}(x) = \{ z \in \mathcal{D} : z \perp x \}.
\]
As in the work of  Borcherds~\cite{Bor98}, Bruinier~\cite{Bru}, and Kudla~\cite{kudla:special_cycles}, for every   $m\in \Q_{>0}$ and $\mu\in L^\vee/L$ we define a complex orbifold
\begin{equation*}
%\label{eqn:Zmmu}
Z(m, \mu ) (\C)= \bigsqcup_{ g\in G(\Q)\backslash G(\A_f) /K }  \Gamma_g \backslash \Big(
\bigsqcup_{  \substack{  x\in  \mu_g+ L_g \\ Q(x)=m  }    }   \mathcal{D}(x)
\Big).
\end{equation*}
Here  $\Gamma_g = G(\Q)\cap g K g^{-1}$, $L_g \subset V$ is the $\Z$-lattice determined by $\widehat{L}_g = g\action \widehat{L}$,  and
\[
\mu_g=g\action \mu\in L_g^\vee/L_g.
\]

By construction $Z(m,\mu)(\C)$ is the space of complex points of a disjoint union of GSpin Shimura varieties associated with quadratic spaces of signature $(n-1,2)$. As such, it has a canonical model $Z(m,\mu)$ over $\Q$, and the obvious map $Z(m,\mu)(\C) \to M(\C)$ descends to a 
finite and unramified morphism
\begin{equation}\label{eqn:Zmmu_canonical}
Z(m, \mu ) \to M.
\end{equation}
Using the complex uniformization, one can check that,  \'etale locally on the source, (\ref{eqn:Zmmu_canonical}) is a closed immersion defined by a single equation.  Thus (\ref{eqn:Zmmu_canonical}) determines  an effective Cartier divisor on $M$, which we call a \emph{special divisor}. Via abuse of notation, we will usually refer to $Z(m,\mu)$ itself as a special divisor on $M$.

%%%%%%%%%%%%%%%%%%%%%%%%%%%%%%%%%%

\subsection{Integral models in the self-dual case}\label{ss:gspin_2}

%%%%%%%%%%%%%%%%%%%%%%%%%%%%%%%%%%%

In this subsection, we will fix a prime $p$ such that the lattice $L$ is self-dual over $\Z_{(p)}$, and abbreviate 
\[
L_{(p)} = L_{\Z_{(p)}}.
\] 
The group  $G_{(p)} = \GSpin(L_{(p)})$ is a reductive model for $G$ over $\Z_{(p)}$. 

The goal is to show that a large part of the results of~\cite[\S4]{MadapusiSpin} also work without the assumption $p>2$.

Consider the \emph{Kuga-Satake abelian scheme}  $A\to M$. Its homological realizations are the sheaves associated with the representation $H$ of $G$ on $C(V)$ via left multiplication, and the lattice $H_{\widehat{\Z}} = C(\widehat{L})\subset H_{\A_f}$. 

We can choose a $G$-invariant symplectic pairing $\psi:H\times H\to \Q(\nu)$ such that induced pairing on the Betti realization $\bm{H}_{B}$ is a polarization of variations of Hodge structures; see~\cite[(2.2)]{AGHMP} for details. This gives rise to a polarization $\lambda$ on $A_{M(\C)}$, which descends to a polarization of $A$ over $M$ of degree $m^2$, where $m^2$ is the discriminant of the lattice $H_{\widehat{\Z}}$ in the symplectic space $H_{\A_f}$. In this way, we obtain a map
\begin{equation*}
%\label{eqn:siegel_embedding}
M \to \mathcal{X}_{2^{n+2},m,\Q},
\end{equation*}
which is finite and unramified. Here, $\mathcal{X}_{2^{n+2},m}$ is the moduli stack over $\Z$ of polarized abelian schemes of dimension $2^{n+2}$ and degree $m^2$.

\begin{definition}\label{defn:normalization}
Given an algebraic stack $\mathcal{X}$ over $\Z_{(p)}$, and a normal algebraic stack $Y$ over $\Q$ equipped with a finite map $j_{\Q}:Y\to\mathcal{X}_{\Q}$, the \emph{normalization} of $\mathcal{X}$ in $Y$ is the finite $\mathcal{X}$-stack $j:\mathcal{Y}\to\mathcal{X}$, characterized by the property that $j_*\co_{\mathcal{Y}}$ is the integral closure of $\co_{\mathcal{X}}$ in $(j_{\Q})_*\co_Y$. It is also characterized by the following universal property: given a finite morphism $\mathcal{Z}\to\mathcal{X}$ with $\mathcal{Z}$ a normal algebraic stack, flat over $\Z_{(p)}$, any map of $\mathcal{X}_{\Q}$-stacks $\mathcal{Z}_{\Q}\to Y$ extends uniquely to a map of $\mathcal{X}$-stacks $\mathcal{Z}\to\mathcal{Y}$.
\end{definition}

We now obtain an integral model $\mathcal{M}_{(p)}$ for $M$ over $\Z_{(p)}$ by taking the normalization of $\mathcal{X}_{2^{n+2},m}$ in $M$. By construction, the Kuga-Satake abelian schemes   extends to a polarized abelian scheme 
\[
A \to \mathcal{M}_{(p)}.
\]

\begin{theorem}\label{thm:integral_models_good_reduction}
The stack $\mathcal{M}_{(p)}$ is smooth over $\Z_{(p)}$.
\end{theorem}
\begin{proof}
When $p>2$, this follows from the main result of~\cite{KisinJAMS}. The general case is shown in~\cite[Theorem 3.10]{mp:2adic}.
\end{proof}

\begin{remark}\label{rem:ell adic extn}
Fix a prime $\ell\neq p$. Recall from~\S\ref{ss:gspin_char_0} the functor which assigns a lisse $\ell$-adic sheaf $\bm{N}_\ell$ over $M$ to each $K_\ell$-stable $\Z_\ell$-lattice $N_\ell\subset N_{\Q_\ell}$ in an algebraic representation $N$ of $G$. This functor extends (necessarily uniquely, by the normality of $\mathcal{M}_{(p)}$) to lisse $\ell$-adic sheaves over $\mathcal{M}_{(p)}$, and carries $H_{\Z_\ell}$ to the $\ell$-adic Tate module $\bm{H}_{\ell}$ of $A$. Indeed, it is enough to show that the induced functor to lisse $\Q_\ell$-sheaves over $M$ extends over $\mathcal{M}_{(p)}$. As shown in~\cite[(4.11),(7.9)]{MadapusiSpin}, this functor is associated with a canonical \'etale $G(\Q_\ell)$-torsor over $M$, which admits an extension over $\mathcal{M}_{(p)}$.
\end{remark}

We also have a canonical functor
\begin{equation}\label{eqn:de_rham_functor}
N\mapsto\bm{N}_{\dR}
\end{equation}
from algebraic $\Q$-representations of $G$ to filtered vector bundles over $M$ equipped with an integrable connection. 
The following result is  \cite[Proposition 3.7]{mp:2adic}.

\begin{proposition}\label{prop:de_rham_realization}
The functor~\eqref{eqn:de_rham_functor} on algebraic $\Q$-representations of $G$ extends canonically to an exact tensor functor 
\[N\mapsto\bm{N}_{\dR}\] from algebraic $\Z_{(p)}$-representations $N$ of $G_{(p)}$ to filtered vector bundles on $\mathcal{M}_{(p)}$ equipped with an integrable connection. When $N = H_{(p)}$, the associated filtered vector bundle with integrable connection is simply $\bm{H}_{\dR}$, the relative first de Rham homology of $A \to \mathcal{M}_{(p)}$.
\end{proposition}

In particular, from the representation $L_{(p)}$, we obtain an embedding
\[
\bm{V}_{\dR}\hookrightarrow\underline{\End}_{C(L)}(\bm{H}_{\dR})
\]
of filtered vector bundles over $\mathcal{M}_{(p)}$ with integrable connections, mapping onto a local direct summand of its target, and extending its counterpart~\eqref{Eqn:V_embeddings_char_0} over $M$.

We now expand our definition of an $F$-crystal over $\mathcal{M}_{(p),\F_p}$ to mean a crystal of vector bundles $\bm{N}$ over $\mathcal{M}_{(p),\F_p}$ equipped with an isomorphism 
\[
\mathrm{Fr}^*\bm{N} \xrightarrow{\simeq}\bm{N}
\] 
in the $\Q_p$-linear \emph{isogeny} category associated with the category of crystals over $\mathcal{M}_{(p),\F_p}$.

Write $\widehat{\mathcal{M}}_{p}$ for the formal completion of $\mathcal{M}_{(p)}$ along $\mathcal{M}_{(p),\F_p}$. The relative first crystalline homology of $A$ over $\mathcal{M}_{(p),\F_p}$ gives an $F$-crystal $\bm{H}_{\cris}$ over $\mathcal{M}_{(p),\F_p}$ whose evaluation on $\widehat{\mathcal{M}}_p$ is canonically isomorphic to the $p$-adic completion of $\bm{H}_{\dR}$ as a vector bundle with integrable connection.

\begin{proposition}\label{prop:cris_realization}
There is a canonical functor $N\mapsto \bm{N}_{\cris}$ from algebraic $\Z_{(p)}$-representations of $G_{(p)}$ to $F$-crystals over $\mathcal{M}_{(p),\F_p}$, which recovers $\bm{H}_{\cris}$ when applied to $H_{(p)}$, and whose evaluation on the formal thickening $\widehat{\mathcal{M}}_p$ is canonically isomorphic to the $p$-adic completion of $\bm{N}_{\dR}$ as a vector bundle with integrable connection.

In particular, there is a canonical $F$-crystal $\bm{V}_{\cris}$ over $\mathcal{M}_{(p),\F_p}$, whose evaluation on $\widehat{\mathcal{M}}_p$ is canonically isomorphic to the $p$-adic completion of $\bm{V}_{\dR}$ as a vector bundle with integrable connection. It admits a canonical embedding
\[
\bm{V}_{\cris} \hookrightarrow \underline{\End}_{C(L)}(\bm{H}_{\cris})
\]
mapping onto a local direct summand of its target, and compatible with the embeddings of de Rham realizations.
\end{proposition}
\begin{proof}
See Proposition 3.9 of~\cite{mp:2adic}.
\end{proof}

%%%%%%%%%%%%%%%%%%%%%%%%%%%%%%%%%%

\subsection{Special endomorphisms in the self-dual case}\label{ss:special endomorphisms 2}

%%%%%%%%%%%%%%%%%%%%%%%%%%%%%%%%%%%

By Proposition~\ref{prop:de_rham_realization} and Proposition~\ref{prop:cris_realization}, 
the embedding of $G_{(p)}$-representations $L_{ (p) }\hookrightarrow H_{(p)}$ gives rise to embeddings
\begin{equation}\label{eqn:sheaves_embedding}
\bm{V}_{?} \hookrightarrow \underline{\End}_{C(L)}(\bm{H}_?)
\end{equation}
for $? = B,\ell,\dR,\cris$ that map onto local direct summands of their targets. 

If $?=B$, let $\mathbf{1}_B$ be the locally constant sheaf $\underline{\Z}$ over $\mathcal{M}(\C)$; if $?=\ell$, let $\mathbf{1}_{\ell}$ be the lisse $\ell$-adic sheaf $\underline{\Z}_{\ell}$ over $\mathcal{M}_{(p)}[\ell^{-1}]$; if $?=\dR$, let $\mathbf{1}_{\dR}$ be the structure sheaf $\co_{\mathcal{M}_{(p)}}$, equipped with the connection $a\mapsto da$ and the one-step filtration concentrated in degree $0$; and, if $?=\cris$, let $\mathbf{1}_{\cris}$ be the structure sheaf over $(\mathcal{M}_{(p),\F_p}/\Z_p)_{\cris}$, equipped with its natural structure of an $F$-crystal.

The quadratic form on $L_{(p)}$ induces a form on the associated realizations. More precisely, for any section $\bm{f}$ of $\bm{V}_?$, we have
\[
\bm{f}\circ\bm{f} = \bm{Q}(\bm{f})  \cdot  \mathrm{id}
\]
under composition in $\underline{\End}_{C(L)}(\bm{H}_?)$. 
Here  $\bm{Q}(\bm{f})$ is a section of $\bm{1}_?$. The assignment $\bm{f}\mapsto\bm{Q}(\bm{f})$ is a quadratic form on $\bm{V}_?$ with values in $\mathbf{1}_?$. The associated bilinear form is non-degenerate when $?=\dR$ or $\cris$. 

\begin{definition}
For any $\mathcal{M}_{(p)}$-scheme $S$, we define an endomorphism $x\in \End_{C(L)}(A_S)$ to be \emph{special} if all its homological realizations land in the images of the embeddings~\eqref{eqn:sheaves_embedding}. More precisely, we require the $\ell$-adic realizations over $S[\ell^{-1}]$ to lie in $\bm{V}_{\ell}$, the crystalline realizations over $S_{\F_p}$ to lie in $\bm{V}_{\cris}$, and the de Rham realizations to lie in $\bm{V}_{\dR}$. 

We will write $V(A_S)$ for the space of special endomorphisms.
\end{definition}

We  now study the deformation theory of a special endomorphism $x$. In what follows, we will frequently cite results from~\cite[\S 5]{MadapusiSpin}, where there is a standing assumption that $p$ is odd. However, the proofs there do not use this assumption, as the reader can easily verify.

Suppose that $S = \Spec(\co)$, with $\co$ a $p$-adically complete $\Z_{(p)}$-algebra. It will be useful to have a notion of special endomorphisms for the $p$-divisible group $A_S[p^\infty]$. We will call an endomorphism $x\in \End_{C(L)}(A_S[p^\infty])$ \emph{special} if its crystalline realization lands in the image of the embedding~\eqref{eqn:sheaves_embedding} for $?=\cris$. We will write $V(A_S[p^\infty])$ for the space of special endomorphisms.

Suppose that we have a surjection $\co\rightarrow \overline{\co}$ of $p$-adically complete $\Z_{(p)}$-algebras, whose kernel $I$ admits nilpotent divided powers. 
Suppose that we have a map $y:\Spec(\co) \to\mathcal{M}_{(p)}$ and let $\overline{y}:\Spec(\overline{\co}) \to \mathcal{M}_{(p)}$ be the restriction to $\Spec(\overline{\co})$. 

Let $\bm{H}_{\co}$ be the $\co$-module obtained by restricting $\bm{H}_{\dR}$ to $\Spec(\co)$, 
and let $\bm{V}_{\co}\subset\End(\bm{H}_{\co})$ be the corresponding realization of $\bm{V}_{\dR}$, 
so that $\bm{V}_{\co}$ is equipped with its Hodge filtration $\Fil^1\bm{V}_{\co}$, which is a rank $1$ projective module over $\co$. 
Denote by $\bm{H}_{\overline{\co}}$ and $\bm{V}_{\overline{\co}}$ the induced modules over $\overline{\co}$.

Let $x\in V(A_{\overline{y}}[p^\infty])$ be a special endomorphism. The crystalline realization of $x$ gives us an element $\bm{x}_{\cris} \in \bm{V}_{\co}$. Pairing against $\Fil^1\bm{V}_{\co}$ induces a linear functional:
\begin{equation}\label{eqn:deformation_pairing}
 [\bm{x}_{\cris},\cdot]:\Fil^1\bm{V}_{\co} \to \co.
\end{equation}

The following two results are shown just as in \cite[Proposition 5.16 and Corollary 5.17]{MadapusiSpin}.

\begin{proposition}\label{prop:special_endomorphism_deform}
The endomorphism $x$ lifts to an element of $V(A_y[p^\infty])$ if and only if the functional~\eqref{eqn:deformation_pairing} is identically $0$.
\end{proposition}

\begin{corollary}\label{cor:special_endomorphism}
Suppose that $k$ is an algebraically closed field of  characteristic $p$, that $t\in \mathcal{M}_{(p)}(k)$ and that $x\in V(A_t[p^\infty])$ is a special endomorphism. 
Let $\co_t$ be the completed \'etale local ring of $\mathcal{M}_{(p),W(k)}$ at $t$. 
There is a principal ideal $(f_x)\subset \co_t$ such that, for any map $f:\co_t\to R$ to a local $\Z_{(p)}$-algebra $R$, $x$ lifts to an element in $\End_{C(L)}(A_{f(t)}[p^\infty])$ if and only if $f$ factors through $\co_t/(f_x)$. 
\end{corollary}

In other words, the deformation space of the endomorphism $x$ within the formal scheme $\mathrm{Spf}(\co_t)$ is pro-representable, 
and cut out by a single equation.

We have to explain how our notion of a special endomorphism relates to the one defined in~\cite[\S 5]{MadapusiSpin}. 
The main difference is that in [\emph{loc.~cit.}] an endomorphism $x\in\End_{C(L)}(A_S)$ was defined 
to be special if it is special in our present sense at every geometric point of $S$, which appears to be a less restrictive definition.
It is not, as the following result demonstrates.

\begin{proposition}\label{prop:special_endomorphism_points}
Let $S$ be a connected $\mathcal{M}_{(p)}$-scheme, and suppose that 
\[
x\in \End_{C(L)}(A_S) 
\]
is a $C(L)$-equivariant endomorphism. Then the following statements are equivalent:
\begin{enumerate}
\item $x$ is special;
\item For any geometric point $s\to S$, the fiber of $x$ at $s$ is special;
\item For \underline{some} geometric point $s\to S$, the fiber of $x$ at $s$ is special.
\end{enumerate}
\end{proposition}
\begin{proof}
If $S$ is a scheme of finite type over $\Q$, then this is clear, since the conditions can be checked over $S_\C$, where everything follows from the fact that the Betti realization is locally constant, and determines the \'etale and de Rham realizations.   The case of an arbitrary scheme over $\Q$ follows from this, as $\mathcal{M}_{(p)}$ is itself of finite type.

If $S$ is an arbitrary $\Z_{(p)}$-scheme, then combining this with  \cite[Lemmas 5.9 and 5.13]{MadapusiSpin} shows that (2) and (3) are equivalent. To complete the proof of the proposition, we now need to know that, if $x$ is special at a point $s\to S$ in characteristic $p$, then the crystalline realization of $x$ lands in the image of~\eqref{eqn:sheaves_embedding} globally over $S_{\F_p}$. This follows from Lemma~\ref{lem:special_endomorphism_lift} below.
\end{proof}

\begin{lemma}\label{lem:special_endomorphism_lift}
Let $R$ be a complete local algebra over $W$ with perfect residue field $k$. Suppose that we have a point $t\in \mathcal{M}_{(p)}(R)$ and an endomorphism $x\in\End_{C(L)}(A_t)$. Let $t_0\in\mathcal{M}_{(p)}(k)$ be the induced point, and let $x_0\in\End_{C(L)}(A_{t_0})$ be the fiber of $x$ at $t_0$. Then $x$ is special if and only if $x_{0}$ is special.
\end{lemma}
\begin{proof}
Let $\co_{t_0}$ be the complete local ring of $\mathcal{M}_{(p),W(k)}$ at $t_0$. By Theorem~\ref{thm:integral_models_good_reduction}, $\co_{t_0}$ is isomorphic to a power series ring over $W(k)$ in $n$ variables. By Corollary~\ref{cor:special_endomorphism}, the deformation ring for the endomorphism $x_0$ is a quotient $\co_{t_0,x_0} = \co_{t_0}/(f_{x_0})$ of $\co_{t_0}$ by a principal ideal. Now, by our hypothesis, $x_0$ lifts over $R$, and so the map $\co_{t_0}\to R$ factors through $\co_{t_0,x_0}$. In particular, it suffices to verify the lemma for $R=\co_{t_0,x_0}$, and so we can assume that we have:
\[
R = \frac{W(k)\pow{u_1,\ldots,u_n}}{(f)},
\]
for some element $f\in W(k)\pow{u_1,\ldots,u_n}$.

The crystalline realization of $x$ is a section of $\underline{\End}_{C(L)}(\bm{H}_{\cris})$. We want to show that it is in fact a section of $\bm{V}_{\cris}$. This is equivalent to showing that its image in $\underline{\End}_{C(L)}(\bm{H}_{\cris})/\bm{V}_{\cris}$ is $0$. 

Let $D_R\to R$ be the $p$-adic completion of the divided power envelope of $R$ in $W(k)\pow{u_1,\ldots,u_n}$. In other words, this is the $p$-adic completion of the subalgebra:
\[
W(k)\pow{u_1,\ldots,u_n}\left[\frac{f^n}{n!}: n\in\Z_{\geq 0}\right]\subset \mathrm{Frac}\bigl(W(k)\pow{u_1,\ldots,u_n}\bigr).
\]
Note that the Frobenius lift 
\[
\varphi:W(k)\pow{u_1,\ldots,u_n}  \to W(k)\pow{u_1,\ldots,u_n} 
\]
defined by $u_i \mapsto u_i^p$ extends continuously to an endomorphism $\varphi:D_R\to D_R$. 

Evaluation along the formal divided power thickening \[\Spec(R_{\F_p}) \hookrightarrow\Spec(D_R)\] establishes an equivalence from the category of crystals over 
$(\Spec (R_{\F_p})/\Z_p)_{\cris}$ to the category of finite free $D_R$-modules equipped with a topologically nilpotent integrable connection. 
Furthermore, this establishes an equivalence between $F$-crystals and finite free $D_R$-modules $M$ equipped with a topologically nilpotent integrable connection as well as a map $\varphi^*M[p^{-1}]\to M[p^{-1}]$ that is parallel for this connection.

Therefore the lemma is now immediate from Lemma~\ref{lem:lci_fibers_zero_cris} below.
\end{proof}

\begin{lemma}\label{lem:lci_fibers_zero_cris}
Let $M$ be a finite free $D_R$-module with a topologically nilpotent integrable connection:
\[
\nabla: M\to M\otimes\widehat{\Omega}^1_{R/W(k)}
\]
and an isomorphism $\varphi^*M[p^{-1}]\to M[p^{-1}]$ that is parallel for $\nabla$. Suppose that $m\in M^{\nabla=0}$ is a parallel element that goes to $0$ under the reduction map $M\to M\otimes_{D_R}W$. Then $m=0$.
\end{lemma}
\begin{proof}
 Let $\widehat{U}^{\mathrm{rig}}$ be the rigid analytic space over $\mathrm{Frac}(W(k))$ associated with the power series ring $W(k)\pow{u_1,\ldots,u_n}$ via Berthelot's analytification functor; see~\cite[\S 7]{dejong:formal_rigid}. This is simply the rigid analytic unit disc. The endomorphism $\varphi$ induces a contraction map $\varphi^*:\widehat{U}^{\mathrm{rig}}\to \widehat{U}^{\mathrm{rig}}$. 

 Now, there is a rational number $r_f\in (0,1)$ such that all elements in $D_R$ converge in the open disc $\widehat{U}^{\mathrm{rig}}(r_f)\subset\widehat{U}^{\mathrm{rig}}$ of radius $r_f$. Let $R^{\mathrm{rig}}$ be the ring of global sections of the structure sheaf on $\widehat{U}^{\mathrm{rig}}(r_f)$. Then we have an inclusion $D_R\subset R^{\mathrm{rig}}$, and extending scalars along this inclusion gives us an $R^{\mathrm{rig}}$-module
 \[
M^{\mathrm{rig}} = R^{\mathrm{rig}}\otimes_{D_R}M
 \]
 equipped with an integrable connection and an isomorphism
 \[
\varphi^*M^{\mathrm{rig}}\xrightarrow{\simeq} M^{\mathrm{rig}}.
 \]

 In this situation, the image of the natural map
 \[
  (M^{\mathrm{rig}})^{\nabla = 0}\to M^{\mathrm{rig}}
 \]
 \emph{generates} $M^{\mathrm{rig}}$ as an $R^{\mathrm{rig}}$-module. This is just Dwork's trick; see for instance~\cite[\S 3.4, Prop. 4]{vologodsky:hodge}.

 Therefore, if a parallel section of $M^{\mathrm{rig}}$ vanishes at a point, then it vanishes everywhere on $\widehat{U}^{\mathrm{rig}}(r_f)$, and is hence the zero section. This proves the lemma.
\end{proof}

\begin{proposition}\label{prop:special_quadratic_form 2}
Let $S$ be an $\mathcal{M}_{(p)}$-scheme. For each $x\in V(A_S)$, we have 
\[
x\circ x = Q(x) \cdot \mathrm{id}_{A_S}\in \End(A_S) 
\]
for some integer $Q(x)$. The assignment $x\mapsto Q(x)$ is a positive definite quadratic form on $V(A_S)$. 
\end{proposition}
\begin{proof}
This is shown as in~\cite[Lemma 5.12]{MadapusiSpin}.
\end{proof}

%%%%%%%%%%%%%%%%%%%%%%%%%%%%%%%%%%

\subsection{Integral model over $\Z$}

%%%%%%%%%%%%%%%%%%%%%%%%%%%%%%%%%%

We will now explain how to construct an integral model for $M$ over $\Z$. In~\cite{AGHMP}, using the results of~\cite{MadapusiSpin}, we gave a construction that worked over $\Z[1/2]$. This is inadequate for our current purposes for two reasons: First, of course, it omits the prime $2$; second, at primes $p$ such that $p^2\mid D_L$, the integral model from [\emph{loc.~cit.}]  excluded points in the special fiber that will be relevant to this article; see~\cite[Remark 2.4.4]{AGHMP}.

Fix a prime $p$. Choose an auxiliary quadratic space $(V^\beef,Q^\beef)$ over $\Q$ of signature $(n^\diamond,2)$, admitting a maximal lattice $L^\beef\subset V^\beef$ that is self-dual over $\Z_{(p)}$, and admitting an isometric embedding
\[
(V,Q) \hookrightarrow (V^\beef,Q^\beef)
\]
carrying $L$ into $L^\beef$.  Set 
\[
\Lambda = L^\perp = \{x\in L^\beef: [x,L] = 0\}\subset L^\beef.
\]

Set $G^\beef = \GSpin(V^\beef)$, and let $\mathcal{D}^\beef \subset \mathbb{P}(V^\beef_\C)$ be the associated hermitian domain; 
then there is a natural embedding of Shimura data
\[
(G,\mathcal{D})\hookrightarrow (G^\diamond,\mathcal{D}^\diamond),
\]
giving rise to a finite, unramified map of Shimura varieties $M\to M^\beef$. Here, $M^\beef$ is the Shimura variety associated with the maximal lattice $L^\beef$.

Since $L^\beef$ is self-dual over $\Z_{(p)}$, $M^\beef$ admits a smooth integral model $\mathcal{M}^\beef_{(p)}$ over $\Z_{(p)}$. We have the Kuga-Satake abelian scheme $A^\diamond\to \mathcal{M}^\beef_{(p)}$ with associated de Rham sheaf $\bm{H}_{\dR}^\beef$, as well as the embeddings
\[
\bm{V}^\beef_{?} \hookrightarrow \underline{\End}_{C(L^\beef)}(\bm{H}^\beef_{?})
\]
where $?=B,\ell,\dR,\cris$. For any $\mathcal{M}^\beef_{(p)}$-scheme $S$, we have the subspace
\[
V(A^\beef_S) \subset \End_{C(L^\beef)}(A^\beef_S)
\]
of special endomorphisms, whose homological realizations are sections of $\bm{V}^\beef_?$.

Define $\mathcal{M}_{(p)}$ to be the normalization of $\mathcal{M}^\beef_{(p)}$ in $M$ (see Definition~\ref{defn:normalization}). 
The restriction of $\Fil^1\bm{V}^\beef_{\dR}$ to $\mathcal{M}_{(p)}$ gives us a line bundle $\bm{\omega}$ over $\mathcal{M}_{(p)}$, 
which extends the line bundle $\Fil^1\bm{V}_{\dR}$ over the generic fiber $M$.

\begin{proposition}\label{prop:integral_model_bad_p}
\mbox{}
\begin{enumerate}
	\item The integral model $\mathcal{M}_{(p)}$ and the line bundle $\bm{\omega}$ are independent of the choice of the auxiliary data $(V^\beef,Q^\beef)$ and $L^\beef\subset V^\beef$. 
	\item The Kuga-Satake abelian scheme $A\to M$ extends to an abelian scheme $A\to\mathcal{M}_{(p)}$ and there is a canonical $C(L^\beef)$-equivariant graded isomorphism
	\begin{equation}\label{eqn:kuga-satake_beef_no beef}
     A\otimes_{C(L)}C(L^\beef) \xrightarrow{\simeq} A^\beef
	\end{equation}
	of abelian schemes over $\mathcal{M}_{(p)}$.
	\item There is a canonical isometric embedding:
	\begin{equation}\label{eqn:Lambda_emb}
    \Lambda \hookrightarrow V(A^\beef_{\mathcal{M}_{(p)}}).
	\end{equation}
    \item $\mathcal{M}_{(p)}$ has the following extension property: If $E/\Q_p$ is a finite extension, and $t\in M(E)$ is a point such that $A_t$ has potentially good reduction over $\co_E$, then the map $t:\Spec(E) \to M$ extends to a map $\Spec(\co_E) \to \mathcal{M}_{(p)}$.
\end{enumerate}
\end{proposition}
\begin{proof}
	Assertion (1) is shown just as in the proof of~\cite[Prop. 2.4.5]{AGHMP}. 

	As for assertion (2), first note that, given the existence of the extension $A\to\mathcal{M}_{(p)}$, the fact that $C(L^\beef)$ is \emph{free} over $C(L)$ gives meaning to $A\otimes_{C(L)}C(L^\beef)$ as an abelian scheme over $\mathcal{M}_{(p)}$;  this is the Serre tensor construction. We always have the canonical $C(L^\beef)$-equivariant graded isomorphism~\eqref{eqn:kuga-satake_beef_no beef} over the generic fiber $M$; see~\cite[(2.12)]{AGHMP}. In particular, as abelian schemes over $M$, there is a canonical closed immersion $A\hookrightarrow A^\beef$. Note that $A^\beef\to\mathcal{M}^\beef_{(p)}$ admits a polarization of degree prime to $p$; indeed, in the notation of \cite[\S~2.4]{AGHMP}, arranging for this amounts to choosing an element $\delta\in C^+(L^\beef)_{\Z_{(p)}}^\times$ satisfying $\delta^* = -\delta$. That $A$ extends to an abelian scheme over $\mathcal{M}_{(p)}$ now follows from the argument in~\cite[Prop. 4.2.2]{Madapusi}. The argument actually shows the following: As in \S~\ref{ss:gspin_2}, let $\psi:H\times H\to \Q(\nu)$ be a $G$-invariant symplectic pairing giving rise to a polarization on $A\vert_M$, and thus to a finite map $M\to \mathcal{X}_{2^{n+2},m,\Q}$ to the generic fiber of a Siegel moduli space. Then this map extends to a finite map $\mathcal{M}_{(p)}\to \mathcal{X}_{2^{n+2},m,\Z_{(p)}}$ parameterizing the abelian scheme $A\to \mathcal{M}_{(p)}$. 

	The existence of the isomorphism~\eqref{eqn:kuga-satake_beef_no beef} of abelian schemes over $\mathcal{M}_{(p)}$, as well as the embedding~\eqref{eqn:Lambda_emb} are now shown exactly as in the proof of~\cite[Prop. 2.5.1]{AGHMP}.

	Assertion (4) is immediate from the finiteness (hence properness) of the map $\mathcal{M}_{(p)}\to \mathcal{X}_{2^{n+2},m,\Z_{(p)}}$.
\end{proof}

Given the proposition, we can choose our auxiliary lattice $L^\beef$ to our convenience. We will choose it so that $\Lambda = L^\perp\subset L^\beef$ has rank at most $2$. This is not strictly necessary, but will make some proofs shorter. Moreover, it is always possible to make such a choice, as can be easily verified using the classification of quadratic forms over $\Q$.

 Let $\mathcal{Z}(\Lambda)\to \mathcal{M}^\beef_{(p)}$ be the stack such that, for any $\mathcal{M}^\beef_{(p)}$-scheme $S$ we have
  \[
  \mathcal{Z}(\Lambda)(S) = \{\text{Isometric embeddings }\Lambda\hookrightarrow V(A^\beef_S)\}.
  \]
  The argument from~\cite[Proposition 2.7.4]{AGHMP} shows that $\mathcal{Z}(\Lambda)$ is an algebraic stack that is finite and unramified over $\mathcal{M}^\beef_{(p)}$.

  The embedding~\eqref{eqn:Lambda_emb} corresponds to a map $\mathcal{M}_{(p)}\to\mathcal{Z}(\Lambda)$. It is shown in~\cite[Lemma 7.1]{MadapusiSpin} that this map identifies $M$ with an open and closed substack of $\mathcal{Z}(\Lambda)_\Q$. 

\begin{proposition}\label{prop:model_special_endomorphisms}
Let $p$ be an odd prime. Suppose either that $p^2\nmid D_L$ or that $n\geq 3$. Then $\mathcal{Z}(\Lambda)$ is normal and flat over $\Z_{(p)}$. In particular, the map $\mathcal{M}_{(p)}\to\mathcal{Z}(\Lambda)$ identifies $\mathcal{M}_{(p)}$ with an open and closed substack of $\mathcal{Z}(\Lambda)$.
\end{proposition}
\begin{proof}
  Let $\Lambda\hookrightarrow V(A^\beef_{\mathcal{Z}(\Lambda)})$ be the tautological isometric embedding and let 
\[
  \bm{\Lambda}_{\dR}\subset\bm{V}^\beef_{\dR}\vert_{\mathcal{Z}(\Lambda)} 
\]
  be the coherent subsheaf generated by the de Rham realization of this embedding. As in~\cite[Lemma 6.16]{MadapusiSpin}, there is a canonical open substack
  \[
\mathcal{Z}^{\mathrm{pr}}(\Lambda)\subset\mathcal{Z}(\Lambda)
  \]
  containing $\mathcal{Z}(\Lambda)_\Q$, and over which $\bm{\Lambda}_{\dR}$ is a local direct summand of $\bm{V}^\beef_{\dR}$. 
  It is shown in~\cite[Corollary 6.22]{MadapusiSpin} that, under our hypotheses, $\mathcal{Z}^{\mathrm{pr}}(\Lambda)$ is a flat, normal $\Z_{(p)}$-stack.

  When $p^2\nmid D_L$, it is shown in~\cite[Lemma 6.16]{MadapusiSpin} that $\mathcal{Z}^{\mathrm{pr}}(\Lambda) = \mathcal{Z}(\Lambda)$, and so the proposition follows in this case.
  For the remaining cases, we will need two lemmas.

  \begin{lemma}\label{lem:Z_Lambda_CM}
  Suppose that $n\geq 2$. The stack $\mathcal{Z}(\Lambda)$ (resp. $\mathcal{Z}(\Lambda)_{\F_p}$) is a local complete intersection over 
  $\Z_{(p)}$ (resp. $\F_p$) of relative dimension $n$.
  \end{lemma}

  \begin{proof}
  Since $p>2$, we can find $\Lambda'\subset \Lambda$ and $v\in (\Lambda')^{\perp}\subset \Lambda$ such that $p^2$ does not divide the discriminant of $\Lambda'$ and such that, over $\Z_{(p)}$, we have an orthogonal decomposition
\[
\Lambda_{\Z_{(p)}} = \Lambda'_{\Z_{(p)}}\perp\langle v\rangle.
\]

  Then we have a factorization
  \[
  \mathcal{Z}(\Lambda) \to \mathcal{Z}(\Lambda') \to \mathcal{M}^\beef_{(p)}
  \]
  into finite and unramified morphisms of $\Z_{(p)}$-stacks. 

  As above, it follows from \cite[Corollary~6.22 and Lemma~6.16]{MadapusiSpin} that $\mathcal{Z}(\Lambda')$ is a faithfully flat regular algebraic stack over $\Z_{(p)}$, whose special fiber is a geometrically normal, local complete intersection algebraic stack of dimension $n + 1\geq 3$. 

  Fix a point $t\in \mathcal{Z}(\Lambda)(\F_p^{\alg})$. We can also view this as a point $t\in \mathcal{Z}(\Lambda')(\F_p^{\alg})$. Let $\co_{\mathcal{Z}',t}$ (resp. $\co_{\mathcal{Z},t}$) be the complete local ring of $\mathcal{Z}(\Lambda')$ (resp. $\mathcal{Z}(\Lambda)$) at $t$. Then it is shown in~\cite[Corollary 5.17]{MadapusiSpin} that $\co_{\mathcal{Z},t}$ is a quotient of $\co_{\mathcal{Z}(\Lambda')}$ cut out by a single equation. 

  In particular, this implies that $\mathcal{Z}(\Lambda)$ is \'etale locally an effective Cartier divisor on $\mathcal{Z}(\Lambda')$, and is in particular a local complete intersection over $\Z_{(p)}$. To show that $\mathcal{Z}(\Lambda)_{\F_p}$ is a local complete intersection stack over $\F_p$, it now suffices to show that it does not contain any irreducible components of the normal algebraic stack $\mathcal{Z}(\Lambda')_{\F_p}$. But this follows from~\cite[Prop. 6.17]{MadapusiSpin}, which shows that, if $\eta\to \mathcal{Z}(\Lambda')_{\F_p}$ is a generic point, then the tautological map $\Lambda' \to  V(A^\beef_\eta)$ is an isomorphism. 
  \end{proof}

  \begin{lemma}\label{lem:Zpr_codimension}
   The codimension of the complement of $\mathcal{Z}^{\mathrm{pr}}(\Lambda)_{\F_p}$ in $\mathcal{Z}(\Lambda)_{\F_p}$ is at least
   $
    n-\left\lfloor {n^\beef}/{2}\right\rfloor.
   $
  \end{lemma}
  
 \begin{proof}
 By~\cite[(6.27)]{MadapusiSpin}, we find that this complement is supported entirely on the supersingular locus 
\[
  \mathcal{M}^{\beef,\mathrm{ss}}_{(p),\F_p}\subset \mathcal{M}^{\beef}_{(p),\F_p}.
 \] 
  But, by~\cite{Howard-Pappas}, this locus has dimension at most $\left\lfloor {n^\beef}/{2}\right\rfloor$. This dimension count can also be deduced using the methods of Ogus from~\cite{Ogus2001-wy}.
  From this the lemma is clear.
 \end{proof}

  By our assumption on $\Lambda$, $n^\beef \leq n+2$. Therefore, by Lemma~\ref{lem:Zpr_codimension}, we see that $\mathcal{Z}^{\mathrm{pr}}(\Lambda)$ is fiberwise dense in $\mathcal{Z}(\Lambda)$ as soon as $n\geq 3$. On the other hand, Lemma~\ref{lem:Z_Lambda_CM} shows that $\mathcal{Z}(\Lambda)$ is a Cohen-Macaulay stack over $\Z_{(p)}$. Therefore, by the normality of $\mathcal{Z}^{\mathrm{pr}}(\Lambda)$ and Serre's criterion for normality, we find that $\mathcal{Z}(\Lambda)$ is itself normal and flat over $\Z_{(p)}$, as soon as $n\geq 3$.
\end{proof}

\begin{theorem}\label{thm:irreducible_fibers}
Assume one of the following conditions:
\begin{itemize}
    \item  $L_{(p)}$ is self-dual;
    \item  $p$ is odd, $p^2\nmid D_L$ and $n\geq 2$;
    \item $p$ is odd and $n\geq 5$.
\end{itemize}
Then $\mathcal{M}_{(p),\F_p}$ is a geometrically connected and geometrically normal algebraic stack over $\F_p$.
\end{theorem}
\begin{proof}
By Proposition \ref{prop:generic connected}, under our hypotheses, $M$ is a geometrically connected smooth algebraic stack over $\Q$. 
Therefore, we only have to show that $\mathcal{M}_{(p)}$ has normal geometric fibers. 
Indeed, as soon as this is known, it will follow from~\cite[Corollary 4.1.11]{Madapusi} that $\mathcal{M}_{(p),\F_p}$ is geometrically connected.

  If $L_{(p)}$ is self-dual, then $\mathcal{M}_{(p)}$ is smooth over $\Z_{(p)}$ by Theorem~\ref{thm:integral_models_good_reduction}, so the theorem is clear under this hypothesis.

  If $p$ is odd, to prove the theorem, by Proposition~\ref{prop:model_special_endomorphisms} it is enough to show that, under the given hypotheses, $\mathcal{Z}(\Lambda)$ is a normal algebraic stack, flat over $\Z_{(p)}$, with normal geometric special fiber. By~\cite[Corollary 6.22]{MadapusiSpin}, we find that, under our hypotheses, $\mathcal{Z}^{\mathrm{pr}}(\Lambda)$ has geometrically normal fibers.
  
  Therefore, by Lemma~\ref{lem:Z_Lambda_CM} and Serre's criterion for normality, to show that $\mathcal{Z}(\Lambda)_{\F_p}$ is normal, it is enough to show that the complement of $\mathcal{Z}^{\mathrm{pr}}(\Lambda)_{\F_p}$ in $\mathcal{Z}(\Lambda)_{\F_p}$ has codimension at least $2$. When $p^2\nmid D_L$, this is clear, since $\mathcal{Z}^{\mathrm{pr}}(\Lambda) = \mathcal{Z}(\Lambda)$.

  For the general case, by Lemma~\ref{lem:Zpr_codimension}, it suffices to show
  \[
   \left\lfloor\frac{n+2}{2}\right\rfloor \leq n-2 
  \]
  whenever $n\geq 5$. This is an easy verification.
\end{proof}

The construction of $\mathcal{M}\to\Spec(\Z)$ now proceeds as in~\cite[\S~2.4]{AGHMP}. Choose a finite collection of maximal quadratic spaces $L^\beef_1,L^\beef_2,\ldots,L^\beef_r$ with the following properties:
\begin{itemize}
\item For each $i=1,2,\ldots,r$, $L^\beef_i$ has signature $(n^\beef_i,2)$, for $n^\beef_i\in\Z_{>0}$;
\item For each $i$, there is an isometric embedding $L\hookrightarrow L^\beef_i$;
\item If, for each $i$,  we denote by $D_i = D_{L^\beef_i}$  the discriminant of $L^\beef_i$, then $\mathrm{gcd}(D_1,\ldots,D_r) = 1$.
\end{itemize}
It is always possible to find such a collection.

For $i=1,2,\ldots,r$, let $M^\beef_i$ be the GSpin Shimura variety over $\Q$ attached to $L^\beef_i$. Then $M^\beef_i$ admits a smooth integral model $\mathcal{M}^\beef_{i,\Z[D_i^{-1}]}$ over $\Z[D_i^{-1}]$. Let $\mathcal{M}_{\Z[D_i^{-1}]}$ be the normalization of $\mathcal{M}^\beef_{i,\Z[D_i^{-1}]}$ in $M$.

\begin{theorem}\label{thm:M_model}
There is a unique flat, normal algebraic $\Z$-stack $\mathcal{M}$ such that, for each $i$, the restriction of $\mathcal{M}$ over $\Z[D_i^{-1}]$ is isomorphic to $\mathcal{M}_{\Z[D_i^{-1}]}$. Moreover:
\begin{enumerate}
\item The Kuga-Satake abelian scheme $A\to M$ extends to an abelian scheme $A\to \mathcal{M}$. 
\item The line bundle $\Fil^1\bm{V}_{\dR}$ over $M$ extends canonically to a line bundle $\bm{\omega}$ over $\mathcal{M}$.
\item If $L_{(p)}$ is self-dual; or if $p$ is odd and $p^2\nmid D_L$; or if $p$ is odd and $n\geq 5$, then $\mathcal{M}_{\F_p}$ is a geometrically connected and geometrically normal algebraic stack over $\F_p$.
\end{enumerate}
\end{theorem}
\begin{proof}
This is immediate from Proposition~\ref{prop:integral_model_bad_p} and Theorem~\ref{thm:irreducible_fibers}.
\end{proof}

Suppose now that we have an isometric embedding
\[
(V,Q) \hookrightarrow (V^\beef,Q^\beef)
\]
into a quadratic space of signature $(n^\beef,2)$, and a maximal lattice $L^\beef\subset V^\beef$ containing $L$. Then we have a finite and unramified map of Shimura varieties $M\to M^\beef$ over $\Q$. 

The next result is easily deduced from the construction of our integral models; see~\cite[Prop. 2.5.1]{AGHMP} for details.

\begin{proposition}\label{prop:functoriality}
The map $M\to M^\beef$ extends to a finite map of integral models $\mathcal{M} \to \mathcal{M}^\beef$. Moreover:
\begin{enumerate}
\item If $A^\beef\to \mathcal{M}^\beef$ is the Kuga-Satake abelian scheme, then there is a canonical isomorphism
\[
A\otimes_{C(L)}C(L^\beef)\xrightarrow{\simeq} A^\beef\vert_{\mathcal{M}}
\]
of abelian schemes over $\mathcal{M}$.
\item Let $\bm{\omega}^\beef$ be the canonical line bundle over $\mathcal{M}^\beef$ from assertion (2) of Theorem~\ref{thm:M_model}. Then there is a canonical isomorphism
\[
\bm{\omega}^\beef\vert_{\mathcal{M}}\xrightarrow{\simeq}\bm{\omega}
\]
of line bundles over $\mathcal{M}$ extending the natural identification 
\[
\Fil^1\bm{V}^\beef_{\dR}\vert_M = \Fil^1\bm{V}_{\dR}
\]
over the generic fiber $M$.
\end{enumerate}
\end{proposition}

\subsection{Special endomorphisms and special divisors over $\Z$}\label{ss:special divisors}

%%%%%%%%%%%%%%%%%%%%%%%%%%%%%%%%%%%

Let $S$ be a scheme over $\mathcal{M}$. We have already encountered the notion of a special endomorphism of $A_S$ in~\S\ref{ss:special endomorphisms 2}, at least in the situation where $S$ is a $\Z_{(p)}$-scheme, with $p\nmid D_L$. In~\cite[\S~2.6]{AGHMP}, we gave a definition that worked without this condition, but since we have somewhat modified our integral models here, the theory there does not apply directly. We now explain how to fix this.

Fix a prime $p$, and set $\mathcal{M}_{(p)} = \mathcal{M}_{\Z_{(p)}}$. Choose an auxiliary maximal lattice $L^\beef$ of signature $(n^\beef,2)$ that is self-dual at $p$, and which admits an isometric embedding $L\hookrightarrow L^\beef$. This gives us a finite map of $\Z_{(p)}$-stacks
\[
\mathcal{M}_{(p)} \to \mathcal{M}^\beef_{(p)}.
\]

It will be useful to have a notion of special endomorphisms for the $\ell$-divisible group of $A^\beef$ as $\ell$ varies over the rational primes. If $\ell\neq p$, we will define $V(A^\beef[\ell^\infty]_S)$ to be the space of endomorphisms of $A^\beef[\ell^\infty]_S$, whose $\ell$-adic realizations land in the space $\bm{V}^\beef_\ell$. If $\ell=p$, we will define $V(A^\beef[p^\infty]_S)$ to be the space of endomorphisms of $A^\beef[p^\infty]_S$, whose $p$-adic realizations land in the space $\bm{V}^\beef_p$ over $S[p^{-1}]$ and whose crystalline realizations land in $\bm{V}^\beef_{\cris}$ over $S_{\F_p}$.

The isomorphism of Kuga-Satake abelian schemes 
\[
A\otimes_{C(L)}C(L^\beef)\xrightarrow{\simeq}A^\beef\vert_{\mathcal{M}_{(p)}}
\]
 induces, for any $\mathcal{M}_{(p)}$-scheme $S$, canonical embeddings
\begin{align}\label{eqn:End_A_embedding}
\End_{C(L)}(A_S)  & \hookrightarrow \End_{C(L^\beef)}(A^\beef_S), \\ 
\End_{C(L)}(A_S[\ell^\infty])  & \hookrightarrow \End_{C(L^\beef)}(A^\beef_S[\ell^\infty]),\nonumber
\end{align}
for any prime $\ell$.
We now declare an endomorphism 
\begin{equation}\label{eq:an endo}
x\in\End_{C(L)}(A_S)\quad \mbox{or} \quad  x\in\End_{C(L)}(A_S[\ell^\infty])
\end{equation}
to be \emph{special} if its image under~\eqref{eqn:End_A_embedding} is a special endomorphism of $A^\beef_S$ or  $A^\beef_S[\ell^\infty]$, respectively. 

Let
\[
\Lambda = L^\perp = \{x\in L^\beef:\;[x,L]=0\}
\]
be the orthogonal complement of $L$ in $L^\beef$. Then there is a canonical embedding 
\begin{equation}\label{eqn:lambda emb end}
\Lambda\hookrightarrow \End_{C(L^\beef)}(A^\beef_M), 
\end{equation}
described in the proof of~\cite[Prop. 2.5.1]{AGHMP}.

\begin{proposition}\label{prop:special_independent_beef}
The notion of $x\in\End_{C(L)}(A_S)$ or $x\in\End_{C(L)}(A_S[\ell^\infty])$ being special does not depend on the choice of the auxiliary lattice $L^\beef$.
\end{proposition}

\begin{proof}
As in the proof of~\cite[Lemma 2.6.1]{AGHMP}, we can reduce to the case where $L$ is itself self-dual over $\Z_{(p)}$. 

In this case, we have homological realizations $\bm{V}_?$ over $\mathcal{M}_{(p)}$, and a commuting square of embeddings
\[
\xymatrix{
{\bm{V}_? } \ar[r]  \ar[d] &  { \underline{\End}_{C(L)}(\bm{H}_?)  }   \ar[d] 
\\
{  \bm{V}_?^\beef\vert_{\mathcal{M}_{(p)}} } \ar[r]  &  { \underline{\End}_{C(L^\beef)}(\bm{H}^\beef_?)\vert_{\mathcal{M}_{(p)}}     }     .
}
\]
of sheaves over $\mathcal{M}_{(p)}$ mapping onto local direct summands of their targets; see~\cite[Prop. 2.5.1(ii)]{AGHMP}. 

In fact this square is cartesian: Both vertical arrows identify sections of their domain with those of the target that anti-commute with the homological realizations of the embedding~\eqref{eqn:lambda emb end}; see the argument in \cite[\S~7.3]{MadapusiSpin}.

From this, we find that $x$ is special with respect to the lattice $L^\beef$ if and only if its homological realizations land in $\bm{V}_{?}$, and so the notion of being special is in this case \emph{intrinsic} to the stack $\mathcal{M}_{(p)}$, and independent of choices.
\end{proof}

If $S$ is now an arbitrary $\mathcal{M}$-scheme, we  declare an endomorphism (\ref{eq:an endo})
 to be \emph{special}  if its restriction to $S_{\Z_{(p)}}$ is special for every prime $p$. 
 Write $V(A_S)$ and $V(A_S[\ell^\infty])$ for the respective spaces of special endomorphisms.

We will also need certain distinguished subsets $V_\mu(A_S)\subset V(A_S)_\Q$ parameterized by $\mu\in L^\vee/L$. To define these, we will first define subsets
\[
V_{\mu_{\ell}}(A_S[\ell^\infty])\subset V(A_S[\ell^\infty])_Q
\]
parameterized by $\mu_\ell\in \Z_\ell\otimes(L^\vee/L)$.

For this, note that over $M(\C)$, as $K$ acts trivially on the quotient $L^\vee/L$, we have a canonical isometry 
\begin{equation*}
%\label{eqn:disc_Betti_isometry}
\underline{\Z}\otimes(L^\vee/L) \xrightarrow{\simeq} \bm{V}^\vee_B/\bm{V}_B
\end{equation*}
of locally constant sheaves.
For each prime $\ell$, this gives an isometry
\begin{equation*}
%\label{eqn:disc_l-adic_isometry}
\bm{\alpha}_{\ell}:\underline{\Z}_{\ell}\otimes(L^\vee/L) \xrightarrow{\simeq} \bm{V}^\vee_{\ell}/\bm{V}_{\ell}.
\end{equation*}
of \'etale sheaves of abelian groups over $M$. In fact, this can be extended to an isometry of sheaves over $\mathcal{M}[\ell^{-1}]$.  By the normality of $\mathcal{M}$, it is enough to show that the sheaves $\bm{V}_{\ell}$ and $\bm{V}_{\ell}^\vee$ extend to lisse sheaves over $\mathcal{M}[\ell^{-1}]$. This can be deduced using the argument from Remark~\ref{rem:ell adic extn}.

Fix a prime $p$, and let $S$ be an $\mathcal{M}_{(p)}$-scheme. 
Then, for any $\ell\neq p$, the $\ell$-adic realization of a special endomorphism $x\in V(A_S)$ is a section of the submodule 
$\bm{V}_{\ell}\subset\underline{\End}_{C(L)}(\bm{H}_{\ell})$. 

Now identify $\bm{V}^\vee_{\ell}$ with an $\ell$-adic subsheaf of $\bm{V}_{\ell}[\ell^{-1}]$. Any element of the dual subspace
\[
V(A_S[\ell^\infty])^\vee = \big\{y\in V(A_S[\ell^\infty])_\Q: [V(A_S[\ell^\infty]),y]\subset\Z_\ell\big\}\subset V(A_S[\ell^\infty])_\Q
\] 
has a  realization $\bm{x}_{\ell}$ in $\bm{V}_{\ell}[\ell^{-1}]$, where  the pairing $[\cdot,\cdot]$ on $V(A_S[\ell^\infty])$ is the one induced from composition in $\End(A_S[\ell^\infty])$.  This allows us to define
\[
V_{\mu_\ell}(A_S[\ell^\infty])\subset V(A_S[\ell^\infty])^\vee
\] 
to be the subset of elements $x$ such that  $\bm{x}_{\ell}$ lies in $\bm{V}^\vee_{\ell}$, and maps to $\bm{\alpha}_{\ell}(1\otimes\mu_{\ell})$ in $\bm{V}_{\ell}^\vee/\bm{V}_{\ell}$. 

Next, we will define the subset 
\[
V_{\mu_p}(A_S[p^\infty]) \subset V(A_S[p^\infty])^\vee
\] 
for  $\mu_p\in \Z_p\otimes(L^\vee/L)$. If $S$ is a $\Q$-scheme, then this can be defined just as for $\ell\neq p$. 
For the general case, choose an auxiliary maximal lattice $L^\beef$ that is self-dual over $\Z_{(p)}$ of signature $(n^\diamond,2)$, which admits an isometric embedding $L\hookrightarrow L^\beef$. By Proposition~\ref{prop:functoriality} and~\cite[Prop. 2.6.4]{AGHMP}, this gives a map of $\Z$-stacks $\mathcal{M}\to\mathcal{M}^\beef$ along with an isometric embedding
\begin{equation*}
%\label{eqn:lambda_beef_mu_embedding}
\Lambda \hookrightarrow V(A^\beef_{\mathcal{M}}).
\end{equation*}
Here, $\Lambda = L^\perp\subset L^\beef$.

\begin{lemma}
\label{lem:lambda_perp}
For any $\mathcal{M}_{(p)}$-scheme $S$, there are canonical isometries
\begin{align*}
V(A_S) &\xrightarrow{\simeq} \Lambda^\perp\subset V(A^\beef_S) \\
V(A_S[p^\infty])&\xrightarrow{\simeq} \Lambda_{\Z_p}^\perp\subset V(A^\beef_S[p^\infty]).
\end{align*}
\end{lemma}
\begin{proof}
The first isometry follows from the definitions and the fact that the subspace
\[
\End_{C(L)}(A_S) \subset \End_{C(L^\beef)}(A^\beef_S)
\]
consists precisely of those endomorphisms that anti-commute with $\Lambda$; see also~\cite[Prop. 2.5.1]{AGHMP}.
The second is proven in similar fashion.
\end{proof}

Now, there are canonical isomorphisms
\[
\Z_p\otimes(L^\vee/L) \xleftarrow{\simeq} \Z_p\otimes(L^\beef/(L\oplus\Lambda)) \xrightarrow{\simeq}\Z_p\otimes(\Lambda^{\vee}/\Lambda).
\]
Therefore, we can canonically view $\mu_p$ as an element of $\Z_p\otimes(\Lambda^\vee/\Lambda)$. Moreover, the inclusions
\[
V(A_S[p^\infty])\oplus\Lambda \subset V(A^\beef_S[p^\infty]) \subset (V(A_S[p^\infty])\oplus \Lambda_{\Z_p})^\vee  = V(A_S[p^\infty])^\vee \oplus \Lambda_{\Z_p}^\vee
\]
induce an embedding
\begin{equation}
\label{eqn:Vbeef_V_Lambda}
\frac{V(A^\beef_S[p^\infty])}{V(A_S[p^\infty])\oplus \Lambda_{\Z_p}} 
\hookrightarrow \frac{V(A_S[p^\infty])^\vee}{V(A_S[p^\infty])} \oplus\frac{\Lambda^\vee_{\Z_p}}{\Lambda_{\Z_p}}.
\end{equation}
We now set
\[
V_{\mu_p}(A_S[p^\infty]) = \big\{x\in V(A_S[p^\infty])^\vee: ([x],\mu_p)\text{ is in the image of~\eqref{eqn:Vbeef_V_Lambda}}\big\},
\]
where  
\[
[x]\in \frac{V(A_S[p^\infty])^\vee}{V(A_S[p^\infty])}
\] 
is the class of $x$.

\begin{proposition}
\label{prop:Vmu_ind_beef}
The subset $V_{\mu_p}(A_S[p^\infty])\subset V(A_S[p^\infty])^\vee$ just defined does not depend on the choice of the auxiliary lattice $L^\beef$.
Moreover, if $S$ is a $\Q$-scheme, then this definition agrees with the one already given above.
\end{proposition}
\begin{proof}
As usual, for the independence statement, we can reduce to the case where $L$ is itself self-dual over $\Z_{(p)}$. In this case, $\mu_p=0$, and we have to show that, if $x\in V(A_S[p^\infty])^\vee$ is such that $([x],0)$ is in the image of~\eqref{eqn:Vbeef_V_Lambda}, then $x$ must belong to $V(A_S[p^\infty])$. However, $([x],0)$ being in the image of~\eqref{eqn:Vbeef_V_Lambda} means exactly that $x$ belongs to $V(A^\beef_S[p^\infty])$ and is orthogonal to $\Lambda_{\Z_p}$. So we are now done by Lemma~\ref{lem:lambda_perp}.

We leave the verification of the second assertion to the reader. 
\end{proof}

Now suppose that  $S$ is an arbitrary $\mathcal{M}$-scheme,  and $p$ is any prime.  We decree that an element of 
$V(A_S[p^\infty])^\vee$ belongs to $V_{\mu_p}(A_S[p^\infty])$ if and only if it does so over $S_{\Z_{(p)}}$.

Consider the dual space
\[
V(A_S)^\vee = \{y\in V(A_S)_\Q: [V(A_S),y]\subset\Z\}\subset V(A_S)_\Q
\] 
of $V(A_S)$ with respect to the bilinear form induced from composition in $\End(A_S)$. For each prime $p$ and each $\mu_p\in\Z_p\otimes(L^\vee/L)$, let 
\[
V_{\mu_p}(A_S)\subset V(A_S)^\vee 
\]
be the subspace of elements mapping into $V_{\mu_p}(A_S[p^\infty])$. In general, if $\mu\in L^\vee/L$ has $p$-primary part $\mu_p$ for each prime $p$, set 
\[
V_\mu(A_S) = \bigcap_{p} V_{\mu_p}(A_S) \subset V(A_S)^\vee.
\]

The next result is immediate from Proposition~\ref{prop:special_quadratic_form 2} and the definitions; see also~\cite[Prop. 2.6.3]{AGHMP}.
\begin{proposition}\label{prop:special_quadratic_form}
For each $x\in V(A_S)$, we have 
\[
x\circ x = Q(x) \cdot \mathrm{id}_{A_S}\in \End(A_S) 
\]
for some integer $Q(x)$. The assignment $x\mapsto Q(x)$ is a positive definite quadratic form on $V(A_S)$. If $x\in V_\mu(A_S)$, then we have the congruence
\begin{equation}\label{Q cong}
Q(x) \equiv Q(\mu) \pmod{\Z}.
\end{equation}
\end{proposition}

Fix a maximal lattice $L^\beef$ of signature $(n^\beef,2)$, equipped with an isometric embedding $L\hookrightarrow L^\beef$, so that we have the corresponding finite map of algebraic stacks $\mathcal{M}\to \mathcal{M}^\beef$. Set $\Lambda = L^\perp\subset L^\beef$.

\begin{proposition}\label{prop:decomposition_Vmu}
Fix an $\mathcal{M}$-scheme $S\to \mathcal{M}$.
\begin{enumerate}

\item There is a canonical isometric embedding $\Lambda\hookrightarrow V(A^\beef_S)$ and an isometry
\begin{equation}\label{eqn:tensor_inject}
 V(A_{S})\xrightarrow{\simeq} \Lambda^{\perp}\subset V( A^\beef_S ).
\end{equation}

\item
For every $\mu\in L^{\beef,\vee}/L^\beef$ and every
$(\mu_1,\mu_2)\in \bigl(\mu+L^\beef\bigr)/\bigl(L\oplus \Lambda\bigr)$
the map (\ref{eqn:tensor_inject}), tensored with $\Q$, restricts to an injection
\[
V_{\mu_1}(A_{S}) \times ({\mu_2}+\Lambda) \hookrightarrow V_{\mu}(A^\beef_S).
\]

\item
The above injections determine  a decomposition
\[
V_{\mu}(A^\beef_S)
=\bigsqcup_{(\mu_1,\mu_2)\in (\mu+ L^\beef)/(L\oplus \Lambda)}  V_{\mu_1}(A_{S}) \times
\bigl({\mu_2}+\Lambda\bigr).
\]
\end{enumerate}
\end{proposition}
\begin{proof}
Assertion (1) is shown just as in Lemma~\ref{lem:lambda_perp}. Everything else is immediate from this and the definitions.
\end{proof}

\begin{definition}\label{defn:special divisor}
For $m\in \Q_{>0}$ and $\mu\in L^\vee/L$, define the \emph{special cycle} $\mathcal{Z}(m, \mu) \to \mathcal{M}$
as the  stack over $\mathcal{M}$ with functor of points
\begin{equation}\label{special divisor}
\mathcal{Z}(m, \mu) (S) = \left\{ x \in V_\mu( A_S) : Q(x) = m \right\}
\end{equation}
for any scheme $S\to \mathcal{M}$.

Note that, by (\ref{Q cong}), the stack (\ref{special divisor}) is  empty unless the image of $m$ in $\Q/\Z$ agrees with $Q(\mu)$.

For later purposes we also define the stacks $\mathcal{Z}(0, \mu)$ in exactly the same way. As the only special endomorphism $x$ with  $Q(x)=0$  is the zero map, we have
\[
\mathcal{Z}(0, \mu) =
\begin{cases}
 \emptyset & \hbox{{\rm if }} \mu\neq 0 \cr \mathcal{M} & \hbox{{\rm if }} \mu=0.\cr
\end{cases}
\]
\end{definition}

Once again, fix a maximal lattice $L^\beef$ of signature $(n^\beef,2)$, equipped with an isometric embedding $L\hookrightarrow L^\beef$, so that we have the corresponding finite map of algebraic stacks $\mathcal{M}\to \mathcal{M}^\beef$. Set $\Lambda = L^\perp\subset L^\beef$. For $m\in\Q_{\geq 0}$ and $\mu\in L^{\beef,\vee}/L^{\beef}$, write $\mathcal{Z}^\beef(m,\mu)\to\mathcal{M}^\beef$ for the stack associated with the pair $(m,\mu)$. The following result is immediate from Proposition~\ref{prop:decomposition_Vmu}.

\begin{proposition}\label{prop:Zmu_functoriality}
Fix $\mu\in L^{\beef,\vee}/L^\beef$. Then there is an isomorphism of $\mathcal{M}$-stacks
\[
\mathcal{Z}^\beef(m, \mu )\times_{\mathcal{M}^\beef}\mathcal{M}\\
\simeq \bigsqcup_{ \substack{ m_1+m_2=m \\ (\mu_1,\mu_2)\in ( \mu+L^\beef)/(L\oplus \Lambda)}}
 \mathcal{Z}(m_1,\mu_1)  \times \Lambda_{m_2,\mu_2},
\]
where
\[
\Lambda_{m_2,\mu_2} = \{x\in {\mu_2}+\Lambda : Q(x)=m_2\},
\]
and $ \mathcal{Z}(m_1,\mu_1)  \times \Lambda_{m_2,\mu_2}$ denotes the disjoint union of
$\# \Lambda_{m_2,\mu_2}$ copies of $\mathcal{Z}(m_1,\mu_1)$.
\end{proposition}

\begin{proposition}\label{prop:Zmu_properties}
There is a natural isomorphism 
\[
\mathcal{Z}(m,\mu)_\Q \xrightarrow{\simeq} Z(m,\mu)
\]
 of stacks over $M$. Moreover:
\begin{enumerate}
\item Suppose that $m>0$.  \'Etale locally on the source, $\mathcal{Z}(m,\mu)$ is an effective Cartier divisor on $\mathcal{M}$.
\item Suppose also that $n\geq 3$. Then $\mathcal{Z}(m,\mu)$ is flat over $\Z[1/2]$. If, in addition, $L_{(2)}$ is self-dual, then $\mathcal{Z}(m,\mu)$ is flat over $\Z$.
\end{enumerate}
\end{proposition}
\begin{proof}
Assertion (1) is deduced from Proposition~\ref{prop:Zmu_functoriality},  exactly as in the proof of \cite[Proposition 2.7.4]{AGHMP}, by reducing to the case where $L$ is self-dual over $\Z_{(p)}$  and using Corollary~\ref{cor:special_endomorphism}. 

As for assertion (2), since $\mathcal{Z}(m,\mu)$ is \'etale locally a divisor on $\mathcal{M}$, it fails to be flat exactly when its image in $\mathcal{M}$ contains an irreducible component of $\mathcal{M}_{\F_p}$ for some prime $p$. 

If $L_{(p)}$ is self-dual at $p$, then the argument used in~\cite[Prop. 5.21]{MadapusiSpin} applies to show that $\mathcal{Z}(m,\mu)$ is flat over $\Z_{(p)}$.

For the other cases, we can now suppose that $p>2$. Choose an auxiliary maximal lattice $L^\beef$ that is self-dual over $\Z_{(p)}$ and an embedding $L\hookrightarrow L^\beef$ as usual. If $\Lambda = L^\perp\subset L^\beef$, then by Proposition~\ref{prop:model_special_endomorphisms}, we can identify $\mathcal{M}_{(p)}$ with a closed and open substack of the stack $\mathcal{Z}(\Lambda)\to\mathcal{M}^\beef_{(p)}$ parameterizing isometric embeddings $\Lambda\hookrightarrow V(A^\beef_S)$. 

By Proposition~\ref{prop:Zmu_functoriality}, it suffices to show that, for every $m\in\Q$ and every $\mu\in L^{\beef,\vee}/L^\beef$, the restriction of $\mathcal{Z}^\beef(m,\mu)$ to $\mathcal{Z}(\Lambda)$ is flat over $\Z_{(p)}$. Equivalently, it is enough to show that the image of the map
\[
\mathcal{Z}^\beef(m,\mu)\times_{\mathcal{M}^\beef_{(p)}}\mathcal{Z}(\Lambda)_{\F_p} \to \mathcal{Z}(\Lambda)_{\F_p}
\]
does not contain an irreducible component of its target. 

For this, let $\mathcal{Z}^{\mathrm{pr}}(\Lambda)\subset\mathcal{Z}(\Lambda)$ be as in the proof of Proposition~\ref{prop:model_special_endomorphisms}. We saw there that, under the hypothesis $n\geq 3$, $\mathcal{Z}^{\mathrm{pr}}(\Lambda)$ is a fiberwise dense open substack of $\mathcal{Z}(\Lambda)$. Therefore, it is enough to show that the image of the map
\[
\mathcal{Z}^\beef(m,\mu)\times_{\mathcal{M}^\beef_{(p)}}\mathcal{Z}^{\mathrm{pr}}(\Lambda)_{\F_p} \to \mathcal{Z}^{\mathrm{pr}}(\Lambda)_{\F_p}
\]
does not contain an irreducible component of its target. 

Note that the $p$-adic component of $\mu$ is necessarily trivial, and note also that $\mathcal{Z}^{\mathrm{pr}}(\Lambda)_{\F_p}$ is normal and hence generically smooth. Therefore, the desired assertion follows from~\cite[Corollary 6.18]{MadapusiSpin}.
\end{proof}

%%%%%%%%%%%%%%%%%%%%%%%%%%%%%%%%%%%%%%%

\subsection{Metrized line bundles}
\label{ss:line bundles}

%%%%%%%%%%%%%%%%%%%%%%%%%%%%%%%%%%%%%%%

Let $F_\infty : \mathcal{M}(\C) \to \mathcal{M}(\C)$ be complex conjugation.
An \emph{arithmetic divisor} on $\mathcal{M}$ is a pair 
\[
\widehat{\mathcal{Z}} = (\mathcal{Z} , \Phi)
\]
 consisting of a Cartier divisor $\mathcal{Z}$ on $\mathcal{M}$  and a Green function $\Phi$ for $\mathcal{Z}$.  
 This means that $\Phi$ is an $F_\infty$-invariant  smooth $\R$-valued function defined on the complement of  
 $\mathcal{Z}(\C)$ in $\mathcal{M}(\C)$, such that  if $\Psi=0$ is any local equation for $\mathcal{Z}(\C)$,
the function $\Phi+ \log |\Psi|^2$  extends smoothly across the singularity $\mathcal{Z}(\C)$.  A \emph{principal arithmetic divisor} is an arithmetic divisor of the form
\[
\widehat{\mathrm{div}}(\Psi) = (\mathrm{div}(\Psi) , - \log|\Psi|^2 )
\]
for a rational function $\Psi$ on $\mathcal{M}$.  The group of all arithmetic divisors is denoted $\widehat{\mathrm{Div}}(\mathcal{M})$,
and its quotient by the subgroup of principal arithmetic divisors is the \emph{arithmetic Chow group} 
$\widehat{\mathrm{CH}}^1(\mathcal{M})$ of Gillet-Soul\'e \cite{GS}.

A \emph{metrized line bundle} on $\mathcal{M}$ is  a line bundle
endowed with  a smoothly varying  $F_\infty$-invariant Hermitian metric on its complex points.
The isomorphism classes of metrized line bundles form a group $\widehat{\mathrm{Pic}}( \mathcal{M} )$ under tensor product.
As in \cite[III.4]{SouleBook}, there is an isomorphism
\begin{equation}\label{pic chow}
\widehat{\mathrm{Pic}}( \mathcal{M} ) \iso \widehat{\mathrm{CH}}^1(\mathcal{M})  
\end{equation}
defined by sending a metrized line bundle $\widehat{\mathcal{L}}$ on $\mathcal{M}$ to the arithmetic divisor
\[
\widehat{\mathrm{div}}(\Psi) = (\mathrm{div}(\Psi) , - \log|| \Psi ||^2).
\]
for any  nonzero rational section $\Psi$ of $\mathcal{L}$.

 By assertion (3) of Theorem~\ref{thm:M_model}, we obtain a canonical line bundle $\bm{\omega}$ over $\mathcal{M}$. We call this the \emph{tautological bundle}, or the \emph{line bundle of weight one modular forms}. 
 Its fiber at a complex point $[(z,g)] \in M(\C)$ is identified with  the isotropic line $\C z \subset V_\C$.  Using this identification, we define  the \emph{Petersson metric} on  the fiber $\bm{\omega}_{ [(  z,g) ]}$  by $|| z ||^2 =  - [z,\overline{z}]$.  In this way we obtain the \emph{metrized tautological bundle}
 \[
 \widehat{ \bm{\omega}} \in \widehat{\mathrm{Pic}}( \mathcal{M} ).
 \]

%%%%%%%%%%%%%%%%%%%%%%%%%%%%%%%%%%%%%%%

\subsection{Harmonic weak Maass forms}
\label{ss:harmonic forms}

%%%%%%%%%%%%%%%%%%%%%%%%%%%%%%%%%%%%%%%

We recall some generalities about the Weil representation and vector-valued harmonic forms from
 \cite{BF,BKY,BY,KuBorcherds}.

Let $S( \widehat{V})$ be the space of Schwarz functions on $\widehat{V} = V\otimes  \A_f$, and  denote by 
\[
S_L \subset S( \widehat{V} )
\]
the (finite dimensional) subspace of  functions that are invariant under translation by $\widehat{L} = L\otimes \widehat{\Z}$, and supported on 
$\widehat{L}^\vee = L^\vee \otimes \widehat{\Z}$.  We often identify $S_L$ with the space of complex-valued functions on 
\[
\widehat{L}^\vee /   \widehat{ L }   \iso L^\vee / L.
\] 
 In particular, for each $\mu \in L^\vee /L$ there is a corresponding Schwartz function 
\begin{equation}\label{mu schwartz}
\varphi_\mu  \in S_L ,
\end{equation}
defined as the characteristic function of $\mu + \widehat{L} \subset  \widehat{ V }$.

Write $\widetilde{\SL}_2(\A)$ for the metaplectic double cover of $\SL_2(\A)$.  This cover splits over $\SL_2(\Q)$, yielding a canonical injection
\begin{equation}\label{meta splitting}
\SL_2(\Q) \hookrightarrow \widetilde{\SL}_2(\A).
\end{equation}    
Pulling back the cover by the inclusions
\[
\SL_2(\R) \to \SL_2(\A) , \quad \SL_2(\A_f ) \to \SL_2(\A)
\]
yields double covers
\[
\widetilde{\SL} _2(\R) \to \SL_2(\R), \quad  \widetilde{\SL}_2(\A_f) \to \SL_2(\A_f),
\]
and we define $\widetilde{\SL}_2(\Z)$ and $\widetilde{\SL}_2(\widehat{\Z})$ by the cartesian diagrams
\[
\xymatrix{
{ \widetilde{\SL}_2(\Z) } \ar[r]  \ar[d] &  { \widetilde{\SL}_2(\R)  }   \ar[d] & { \widetilde{\SL}_2(\widehat{\Z} )  }  \ar[r] \ar[d] &   { \widetilde{\SL}_2(\A_f )  }   \ar[d] 
\\
{  \SL_2(\Z) } \ar[r]  &  { \SL_2(\R)      }     & {  \SL_2(\widehat{\Z} )  }  \ar[r] & {  \SL_2(\A_f )  } .
}
\]

The inclusion (\ref{meta splitting}) induces an injection
$
\widetilde{\SL}_2(\Z) \to \widetilde{\SL}_2(\widehat{\Z} ),
$
denoted $\widetilde{\gamma}  \mapsto \widehat{\gamma}$,  defined by demanding  that the product 
\[
\widetilde{\gamma} \cdot \widehat{\gamma} \in  \widetilde{\SL}_2(\R) \cdot  \widetilde{\SL}_2(\A_f) \subset \widetilde{\SL}_2(\A)
\]
be equal to the image of $\widetilde{\gamma}$ under  the composition 
\[
\widetilde{\SL}_2(\Z) \to \SL_2(\Z) \hookrightarrow \widetilde{\SL}_2(\A).
\]

Denote by $\psi_\Q : \Q \backslash \A \to \C^\times$ the unramified  character with archimedean component
$\psi_{\Q,\infty}(x) = e^{2\pi i x}$.  The group $\widetilde{\SL}_2(\A_f)$ acts on $S(\widehat{V})$ via the Weil representation $\omega$
determined by  $\psi_\Q$, and the restriction of this representation to $\widetilde{\SL}_2(\Z) \subset  \widetilde{\SL}_2(\widehat{\Z} )$
leaves invariant the finite dimensional subspace $S_L$.  Denote this representation by 
\[
\omega_L : \widetilde{\SL}_2(\Z) \to \Aut(S_L),
\]
and define the complex conjugate representation 
$
\overline{\omega}_L : \widetilde{\SL}_2(\Z) \to \Aut(S_L)
$
 by
\[
\overline{\omega}_L( \widetilde{\gamma} ) \cdot  \varphi = \overline{ \omega_L (\widetilde{\gamma}) \cdot  \overline{\varphi} }.
\]
If $\mathrm{dim}(V)$ is even then $\omega_L$ and $\overline{\omega}_L$   factor through $\SL_2(\Z)$.  Note that our $\overline{\omega}_L$
is the representation denoted $\rho_L$ in \cite{Bor98,Bru,BF,BKY,BY}.

Denote by $H_{1-  n/2  }(\omega_L)$ the space of harmonic weak Maass forms of weight $1-n/2$ for $\widetilde{\SL}_2(\Z)$
of representation $\omega_L$, in the sense of  \cite[\S 3]{BY}, and denote by 
\[
 S_{1-  n/2  }(\omega_L) \subset  M^!_{1-  n/2  }(\omega_L)\subset  H_{1-  n/2  }(\omega_L)
\]
the subspaces of cusp forms and weakly modular forms, respectively.  By a result of Bruinier-Funke \cite{BF}, these spaces are
related by an exact sequence
\begin{equation}\label{BF exact sequence}
0 \to  M^!_{1-  n/2  }(\omega_L) \to H_{1-  n/2  }(\omega_L) \map{\xi} S_{ 1+  n/2  }( \overline{\omega}_L)  \to 0,
\end{equation}
where $\xi$ is a certain explicit differential operator.

As in \cite[(3.4a)]{BY}, any  $f\in H_{1-  n/2  }(\omega_L)$ has a \emph{holomorphic part}
\[
f^+(\tau) = \sum_{  \substack{ m\in D_L^{-1}  \Z \\ m \gg -\infty} } c_f^+(m)  \cdot q^m ,
\]
which is a formal $q$-expansion with coefficients 
\[
c_f^+(m) = \sum_{\mu \in L^\vee /L} c_f^+(m,\mu) \cdot \varphi_\mu \in S_L.
\]
When the  \emph{principal part} 
\[
P_f(\tau) =  \sum_{   m \ge  0}  c_f^+(-m)  \cdot q^{-m}.
\]
is  integral, in the sense that $c^+_f(-m,\mu) \in \Z$ for all $m \ge 0$ and $\mu \in L^\vee/L$,  we define the corresponding \emph{special divisor}
\[
\mathcal{Z}(f) = \sum_{  \substack{  m>0 \\ \mu \in L^\vee / L} }  c_f^+(-m,\mu) \mathcal{Z}(m,\mu) 
\]
on $\mathcal{M}$.  
 There is a natural Green function $\Phi(f)$  for  $\mathcal{Z}(f)$, defined as a regularized theta lift as in \cite[(4.7)]{BY}.  See also \cite{Bru,BF,BKY}.  In particular, we obtain an arithmetic divisor
\begin{equation}\label{arithmetic divisor}
\widehat{ \mathcal{Z} } (f)  = \big( \mathcal{Z}(f)  , \Phi(f) \big) \in \widehat{\mathrm{CH}}^1(\mathcal{M}    ).
\end{equation}

%%%%%%%%%%%%%%%%%%%%%%%%%%%%%%%%%%%

\subsection{Borcherds products}

%%%%%%%%%%%%%%%%%%%%%%%%%%%%%%%%%%%

Suppose 
\begin{equation}\label{borcherds input}
f (\tau) = \sum_{   \substack{   m\in D_L^{-1}  \Z \\ m\gg 0   } }  c_f(m)\cdot q^m \in M_{ 1 -  n /2 }^!(\omega_L)
\end{equation}
is a weakly holomorphic form, so that $f=f^+$ and $c_f(m)= c_f^+(m)$.

The following result will be shown in the companion paper~\cite{HMP}, generalizing a result of F. H\"ormann~\cite{Hormann}. Here, we only sketch its proof. For the applications to Colmez's conjecture, we will only require the assertion over primes of good reduction, which is already contained in~\cite{Hormann}.

\begin{theorem}\label{thm:borcherds}
Suppose that $n\geq 3$ and that the principal part $P_f(\tau)$ is integral. 
Then, after replacing $f$ by a multiple $kf$, for any sufficiently divisible $k\in\Z_{>0}$, there exists a rational section $\Psi(f)$ of $\omega^{  \otimes c_f(0,0) }$, defined over $\Q$,  such that
\[
\Phi(f)    =  - \log|| \Psi (f)||^2  + c_f(0,0) \log(4\pi e^\gamma) .
\]
Here  $\gamma = -\Gamma'(1)$ is the Euler-Mascheroni constant.

In particular, the canonical isomorphism (\ref{pic chow}) produces identifications
\begin{align*}
 \widehat{\bm{\omega}}^{ \otimes c_f(0,0) }  & = \widehat{\mathrm{div}}(\Psi(f))  \\
 & =   \widehat{ \mathcal{Z} } (f)   -  c_f(0,0) \cdot  \big(0,  \log(4\pi e^\gamma)  \big) + \widehat{ \mathcal{E} }(f),
\end{align*}
where  $ (0,  \log(4\pi e^\gamma)  )$  denotes  the trivial divisor endowed with the constant Green function $\log(4\pi e^\gamma)$, 
and $\widehat{ \mathcal{E} }(f) = (\mathcal{E}(f),0)$ is the divisor
\[
\mathcal{E}(f) = \mathrm{div}( \Psi (f) ) -  \mathcal{Z}(f)
\]  
endowed with the trivial Green function.
Moreover, there is a decomposition 
\[
\mathcal{E}(f) = \sum_{p\mid D_L} \mathcal{E}_p(f)
\]
 in which the divisor $\mathcal{E}_p(f)$ is supported on the special fiber $\mathcal{M}_{\F_p}$, and:
\begin{itemize}
\item
If $p$ is odd and $p^2\nmid D_L$ then $\mathcal{E}_p(f)=0$;
\item
 If  $n\geq 5$  then $\mathcal{E}(f) = \mathcal{E}_2(f)$ is supported on $\mathcal{M}_{\F_2}$. 
 \item
  If $n\ge 5$  and $L_{(2)}$ is self-dual, then $\mathcal{E}(f)=0$.
  \end{itemize}
\end{theorem}

\begin{proof}[Sketch of proof]
For any sufficiently divisible $k$, all the Fourier coefficients of $k\cdot f$ are integral, and the Borcherds lift of $k\cdot f$, after a normalization, descends to a section of $\bm{\omega}^{\otimes k c_f(0,0)}$. Replacing $f$ by this multiple, we take our desired section $\Psi(f)$ to be this descent of the Borcherds lift. 
It is known that the divisor of $\Psi(f)$ in $M$ is exactly $\mathcal{Z}(f)\vert_M$; see~\cite{Bor98} or \cite{Bru}. 
Thus $\mathcal{E}(f) = \sum_p \mathcal{E}_p(f)$ is supported in finitely many nonzero characteristics.

 Assume that $L_{(p)}$ is self-dual, or that $p$ is odd and $p^2\nmid D_L$, or that $p$ is odd and $n\ge 5$.  
 To check that $\mathcal{E}_p(f)=0$, it suffices to show that both $\mathcal{Z}(f)$ and $\mathrm{div}(\Psi(f))$ are flat over 
 $\Z_{(p)}$.  The flatness of $\mathcal{Z}(f)$ follows from Proposition~\ref{prop:Zmu_properties}.
 For the flatness of $\mathrm{div}(\Psi(f))$, note that the special fiber $\mathcal{M}_{\F_p}$ is irreducible, by Theorem~\ref{thm:M_model}.  
 Thus $\mathrm{div}(\Psi(f))$, if not flat, contains a multiple of the entire special fiber $\mathcal{M}_{\F_p}$.  
Since a theory of integral $q$-expansions is now available through~\cite{Madapusi}, we can use the explicit product $q$-expansion of $\Psi(f)$ to check that the support of $\mathrm{div}(\Psi(f))$ cannot contain $\mathcal{M}_{\F_p}$, and hence that $\mathrm{div}(\Psi(f))$ is also flat. To be more precise, the Fourier coefficients in the $q$-expansion of $\Psi(f)$ are integral and without a non-trivial common divisor. Hence, the mod $p$ reduction of such an expansion cannot vanish identically. The $q$-expansion principle now implies that the form $\Psi(f)$ also cannot vanish identically along the special fiber.
\end{proof}

\section{Big CM cycles on orthogonal Shimura varieties}\label{ss:big CM}

%%%%%%%%%%%%%%%%%%%%%%%%%%%%%%%%%%%

As in Section~\ref{s:cm shimura}, we will fix a CM field $E$ with totally real subfield $F$. We will also take $\Q^{\alg}$ to be the algebraic closure in $\C$ of $\Q$ and write $\Gamma_{\Q}$ for the absolute Galois group $\mathrm{Gal}(\Q^{\alg}/\Q)$. We will also fix a distinguished embedding $\iota_0:E \to \Q^\alg$.

The goal here is to embed the zero dimensional Shimura variety from Section~\ref{s:cm shimura} into the GSpin Shimura varieties from Section~\ref{s:orthogonal shimura}, and to study the interaction between the various `motives' that live over the two spaces. The main result is Corollary~\ref{cor:special end structure}, which explains the structure of the space of special endomorphisms associated with points of the zero dimensional Shimura variety.

%%%%%%%%%%%%%%%%%%%%%%%%%%%%%%%%%%

\subsection{Hermitian spaces}
\label{ss:hermitian}

%%%%%%%%%%%%%%%%%%%%%%%%%%%%%%%%%%

Let    $(\mathscr{V},\langle\cdot,\cdot\rangle)$ be a rank one Hermitian space over $E$ that is negative definite at $\iota_0$, and positive definite at the remaining archimedean places. The assignment
\[
x\mapsto \langle x,x\rangle = \mathscr{Q}(x)
\]
induces a quadratic form $\mathscr{Q}:\mathscr{V} \to F$ on the underlying $F$-vector space of signature
\begin{equation}\label{sig}
\mathrm{sig}( \mathscr{V} ) = \big( (0,2), (2,0), \ldots, (2,0)  \big).
\end{equation}

The Clifford algebra of $(\mathscr{V},\mathscr{Q})$ is a quaternion algebra over $F$, with a $\Z/2\Z$-grading  
\[
C(\mathscr{V}  ) = C^+(\mathscr{V})\oplus C^-(\mathscr{V}).
\]
The even part $C^+(\mathscr{V})$ is isomorphic to $E$ as an $F$-algebra. We will fix an isomorphism $E\xrightarrow{\simeq}C^+(\mathscr{V})$ of $F$-algebras. Now, the odd part $C^-(\mathscr{V})$ is identified with the $F$-vector space $\mathscr{V}$. The action of $E$ on $\mathscr{V}$ given by left multiplication in the Clifford algebra is none other than the given $E$-module structure on $\mathscr{V}$.

\begin{remark}\label{rem:hermitian construction}
If we fix any $E$-module isomorphism $\mathscr{V}\simeq E$,  there  is a unique  $\xi \in F^\times$ such that the hermitian form on $\mathscr{V}$
is identified with the hermitian form 
$
\langle x,y\rangle = \xi x \overline{y}
$ 
on $E$.  The element $\xi$ is negative at $\iota_0$ and positive at $\iota_1,\ldots, \iota_{d-1}$, and the isomorphism class of $\mathscr{V}$
is uniquely determined by 
\[
\xi \in F^\times/ \mathrm{Nm}_{E/F} ( E^\times).
\]
Conversely, if we start with any CM field $E$ with  totally real subfield $F$, and any $\xi \in F^\times$ negative at $\iota_0$ 
and positive at $\iota_1,\ldots, \iota_{d-1}$, we obtain an $F$-quadratic space 
$
(\mathscr{V},\mathscr{Q})  =   (E ,  \xi\cdot \mathrm{Nm}_{E/F} )
$  
of signature (\ref{sig}) as above. 
\end{remark}

Let $\chi:\A_F^\times \to \{\pm 1\}$ be the quadratic character determined by $E/F$.
Keeping the notation of Remark \ref{rem:hermitian construction}, for every place $v$ of $F$ define the \emph{local invariant}
\[
\mathrm{inv}_v(\mathscr{V}) = \chi_v(\xi)  \in  \{ \pm 1\} .
\]
Thus $\mathrm{inv}_v(\mathscr{V})=1$ if and only if $\xi$ is a norm from $E_v^\times$,
and  $\alpha\in F_v^\times$ is represented by $\mathscr{V}_v$ if and only if $\chi_v(\alpha)=\mathrm{inv}_v( \mathscr{V} )$.
The hermitian space $\mathscr{V}$ is uniquely determined by its collection of local invariants, and the product of the local invariants is $1$.

\begin{definition}\label{def:nearby}
Suppose that $\mathfrak{p} \subset \co_F$ is a prime ideal  nonsplit in $E$.  The \emph{nearby hermitian space} $\near \mathscr{V}$ is
obtained from $\mathscr{V}$ by interchanging invariants at $\iota_0$ and $\mathfrak{p}$. 
In other words,  $\near \mathscr{V}$  is  the unique  rank one hermitian space over $E$ with 
\[
\mathrm{inv}_v( \near \mathscr{V} ) \iso
\begin{cases}
- \mathrm{inv}_v(  \mathscr{V} ) &  \mbox{if } v \in \{ \mathfrak{p}  , \iota_0 \} \\
\mathrm{inv}_v(  \mathscr{V} ) &  \mbox{otherwise.}
\end{cases}
\]
The (positive definite) hermitian form on $\near \mathscr{V}$ is denoted $\near\langle x_1,x_2\rangle$, and the 
associated  $F$-quadratic form is $\near \mathscr{Q}(x) = \near \langle x,x\rangle$.
\end{definition}

%%%%%%%%%%%%%%%%%%%%%%%%%%%%%%%%%%

\subsection{Reflex algebras and Clifford algebras}
\label{ss:reflex algebra}

%%%%%%%%%%%%%%%%%%%%%%%%%%%%%%%%%%

Associated with $(\mathscr{V},\mathscr{Q})$ is the  $\Q$-quadratic space   
\begin{equation}\label{isometric equality}
 (V,Q)   =  (\mathscr{V} , \mathrm{Tr}_{F/\Q}  \circ \mathscr{Q}) 
\end{equation}
of signature $(n,2)=(2d-2,2)$.

Let $E^\sharp$ be the total reflex algebra associated with $E$. It is an \'etale $\Q$-algebra whose associated $\Gamma_\Q$-set is canonically identified with the set $\mathrm{CM}(E)$ of CM types for $E$; see \S~\ref{ss:abelian schemes}.

\begin{proposition}\label{prop:reflex inclusion}
The relation (\ref{isometric equality}) determines a distinguished embedding of $\Q$-algebras 
$
E^\sharp \hookrightarrow C^+(V).
$
\end{proposition}
\begin{proof}
The $E$-action on $V=\mathscr{V}$ gives us a decomposition
\begin{equation}\label{eqn:V_Q_alg_decomp}
V_{\Q^{\alg}}  =  \bigoplus_{\iota\in\mathrm{Emb}(E)}\mathscr{V}(\iota),
\end{equation}
into one-dimensional $\Q^{\alg}$-vector spaces, where $\mathscr{V}(\iota) = \mathscr{V}\otimes_{E,\iota}\Q^{\alg}$. By construction, the quadratic form $Q$ induces a perfect pairing
\[
\mathscr{V}(\iota) \times \mathscr{V}(\overline{\iota}) \to  \Q^{\alg}.
\]
Therefore, for each embedding $\iota_i:F\to\Q^{\alg}$, $i=0,1,\ldots,d-1$, $Q$ restricts to a non-degenerate form on
\[
\mathscr{V}_i = \mathscr{V}\otimes_{F,\iota_i}\Q^{\alg}.
\]

If $i\neq j$  then $\mathscr{V}_i$ and $\mathscr{V}_j$ are orthogonal, and so we obtain a $\Q^{\alg}$-linear orthogonal decomposition 
\[
V_{\Q^{\alg}} = \bigoplus_{i=0}^{d-1}\mathscr{V}_i
\]
 into two-dimensional non-degenerate quadratic subspaces.
In turn, this gives us a natural $\Gamma_{\Q}$-stable commutative subalgebra
\begin{equation}\label{eqn:fake_reflex_algebra}
\bigotimes_{i=0}^{d-1}C^+\bigl(\mathscr{V}_i\bigr)\subset C^+(V_{\Q^{\alg}}),
\end{equation}
which descends to a $\Q$-subalgebra $B\subset C^+(V)$. 

We claim that there is a canonical isomorphism of $\Q$-algebras $E^\sharp\xrightarrow{\simeq}B$. For this, it is enough to show that there is a canonical isomorphism of $\Gamma_{\Q}$-sets:
\[
\Hom_{\Q-\mathrm{alg}}(B,\Q^{\alg})\xrightarrow{\simeq}\mathrm{CM}(E).
\]

But this is clear from the description in~\eqref{eqn:fake_reflex_algebra}, since, for each $i=0,1,\ldots,d-1$, we have canonical isomorphisms of $\Q^{\alg}$-algebras with an involution:
\[
E\otimes_{F,\iota_i}\Q^{\alg}\xrightarrow{\simeq}C^+(\mathscr{V})\otimes_{F,\iota_i}\Q^{\alg}\xrightarrow{\simeq}C^+(\mathscr{V}_i).
\]
\end{proof}

%%%%%%%%%%%%%%%%%%%%%%%%%%%%%%%%%%%

\subsection{Morphisms of Shimura varieties}\label{ss:shimura data}

%%%%%%%%%%%%%%%%%%%%%%%%%%%%%%%%%%%

Assume now that $d>1$, so that $n=2d-2 > 0$. Write $H$ for $C(V)$, viewed as a faithful representation of 
\[
G=\GSpin(V)
\] 
via the left multiplication action of $C(V)$ on itself.   
Using the inclusion  $E^\sharp \subset C(V)$ of Proposition \ref{prop:reflex inclusion}, the group $T_{E^\sharp}$ also
acts faithfully on $H$ via left multiplication. The torus $T=T_E/T_F^1$ can be identified with the intersection of $G$ and $T_{E^\sharp}$ inside of $\GL(H)$.
In other words, there is a   cartesian diagram
\[
\xymatrix{
 { T }  \ar[rr]^{ \mathrm{Nm}^\sharp }  \ar[d] & & {T_{ E^\sharp } }  \ar[d]  \\
 G \ar[rr]   &&  {  \GL(H)  }
}
\]
in which all arrows are injective.  Here, $\mathrm{Nm}^\sharp$ is the total reflex norm defined in \S~\ref{ss:abelian schemes}.

Now, we have canonical identifications
\[
\mathrm{Res}_{E/\Q}\SO(\mathscr{V}) = T_E^1 = T_{so}
\]
of tori over $\Q$. This exhibits $T_{so}$ as a maximal torus in $\SO(V)$, and it also identifies $V$ with the standard representation $V_0$ of $T$. 
Moreover, we have a commutative diagram 
\[
\xymatrix{
{ 1 } \ar[r]  &  {  \mathbb{G}_m }   \ar[r]  \ar@{=}[d]  &   {  T }  \ar[r]^{\theta}  \ar[d]  &   { T_{so}  }  \ar[d]   \ar[r]   &  1 \\
{ 1 } \ar[r]  &  {  \mathbb{G}_m }   \ar[r]  &   {  G }  \ar[r] &   { \SO(V) }   \ar[r]   &  1
}
\]
with exact rows, and all vertical arrows are injective.

% \begin{proposition}\label{prop:commutant}
% $E$ is the commutant of $T_{so}(\Q)$ in $\End(V)$.
% \end{proposition}

Via the decomposition~\eqref{eqn:V_Q_alg_decomp}, we obtain a   $T(\C)$-stable line
\[
z^{cm} = \mathscr{V}(\iota_0)_{\C} \subset \mathscr{V}_\C.
\]
This line is isotropic with respect to the quadratic form $\mathrm{Tr}_{F/\Q} \circ \mathscr{Q}$, and we use 
(\ref{isometric equality}) to view $z^{cm}$ as a point of the hermitian domain (\ref{hermitian domain}).

The  morphism  $T \to G$ induces a morphism of Shimura data
\begin{equation}\label{eqn:morphism_data}
( T , \{\mu_0\} ) \to (G,\mathcal{D})
\end{equation}
mapping $\mu_0$ to $z^{cm}\in\mathcal{D}$.

As in \S \ref{s:orthogonal shimura},  let  $L\subset V$ be a maximal lattice of discriminant  $D_L$.  
Recall that the choice of maximal lattice determines a compact  open subgroup $K  \subset G(\A_f)$
and a Shimura variety (\ref{gspin shimura}), with a canonical model $M \to \Spec(\Q)$.  

Consider the compact open subgroup $K_{L,0} = K_0 \cap K \subset T(\A_f)$. In \S~\ref{s:cm shimura}, we associated with it a zero dimensional Shimura variety $Y_{K_{L,0}}$, as well as a normal integral model $\mathcal{Y}_{K_{L_0}}$ over $\co_E$. From now on we abbreviate  
\[
\mathcal{Y}= \mathcal{Y}_{K_{L,0}}.
\] 
This is an arithmetic curve over $\co_E$, whose generic fiber we denote by  $Y\to \Spec(E)$.
By the theory of canonical models, we now obtain a morphism 
\begin{equation}\label{cm morphism}
 Y\to M
\end{equation}
of $\Q$-stacks, induced by the morphism of Shimura data~\eqref{eqn:morphism_data}.

%Recall the Kuga-Satake abelian scheme $A\to  M$, and let $A_Y \to Y$ be its pullback via (\ref{cm morphism}).

% Define an order in $E$ by  $\co_L = \{ \alpha \in E :\alpha L \subset L \},$  and abbreviate
%\begin{equation}\label{D total}
%D_{bad}=2  D_E  D_L  [\co_E : \co_L] ,
%\end{equation}
%where $D_E=\mathrm{disc}(E)$, and $D_L=[L^\vee : L]$ as in \S \ref{s:orthogonal shimura}.

%\begin{proposition}
% There are abelian schemes $B_1,\ldots, B_{2^d}$ over $Y$,  each of dimension $2^{d-1}$ and  admitting  complex 
% multiplication by the maximal order in $E^\sharp$ with  CM type $\Phi^\sharp$, and an isogeny 
%\[
% \psi :  A_Y \to B_1 \times \cdots \times B_{2^d }
%\]
%of abelian schemes over $Y$.  The isogeny $\psi$ may be chosen so that  $\deg(\psi)$ divides $(D_{bad})^{2^{2d+1}}$. 
%\end{proposition}

\begin{proposition}\label{prop:morphism_integral_Y}
The map~\eqref{cm morphism} extends to a map of $\Z$-stacks
\begin{equation*}
%\label{cm morphism integral}
\mathcal{Y} \to \mathcal{M}
\end{equation*}
\end{proposition}

\begin{proof}
This follows from Proposition~\ref{prop:abelian scheme reduction} and assertion (4) of Proposition~\ref{prop:integral_model_bad_p}.
\end{proof}

We will need some information about the compatibility of this map with constructions of automorphic sheaves. For this, fix a prime $\mathfrak{q}\subset\co_E$ lying above a rational prime $p$, and an auxiliary quadratic lattice $L^\beef$ of signature $(n^\beef,2)$, self-dual at $p$ and admitting $L$ as an isometric direct summand. Associated with it is the Shimura variety $M^\beef$ with a smooth integral canonical model $\mathcal{M}^\beef_{(p)}$ over $\Z_{(p)}$ and a finite map $\mathcal{M}_{(p)}\to \mathcal{M}^\beef_{(p)}$.

From Propositions~\ref{prop:de_rham_realization} and~\ref{prop:cris_realization}, we obtain functors $N_{(p)}\mapsto\bm{N}^\beef_{\dR}$ and $N_{(p)}\mapsto\bm{N}^\beef_{\cris}$ from $G^\beef_{(p)}\define \GSpin(L^\beef_{\Z_{(p)}})$-representations to filtered vector bundles over $\mathcal{M}^\beef_{(p)}$ and $F$-crystals over $\mathcal{M}^\beef_{\F_p}$, respectively. On the other hand, any $G^\beef_{(p)}$-representation $N_{(p)}$ gives a $\Q$-representation $N = N_{(p)}[p^{-1}]$ of $T$, and a $K_{0,L}$-stable lattice $N_p = N_{(p)}\otimes\Z_p\subset N_{\Q_p}$. Therefore, by Proposition~\ref{prop:realizations integral model} (or more precisely, its proof), it gives us a filtered vector bundle $\bm{N}_{\dR}$ over $\mathcal{Y}_{(\mathfrak{q})} = \mathcal{Y}\otimes_{\co_E}\co_{E,(\mathfrak{q})}$, and an $F$-crystal $\bm{N}_{\cris}$ over $\mathcal{Y}_{\F_{\mathfrak{q}}}$.

\begin{proposition}
\label{prop:realizations compatibility}
There are canonical isomorphisms
\[
\bm{N}_{\dR}\xrightarrow{\simeq} \bm{N}_{\dR}^\beef\vert_{\mathcal{Y}_{(\mathfrak{q})}}\;;\;\bm{N}_{\cris}\xrightarrow{\simeq} \bm{N}_{\cris}^\beef\vert_{\mathcal{Y}_{\F_{\mathfrak{q}}}}
\]
of filtered vector bundles and $F$-crystals, respectively.
\end{proposition}

We omit the proof of the proposition, which follows immediately from unwinding the constructions. The main point is that both constructions, when restricted to the completed \'etale local ring at a point $y\in \mathcal{Y}(\F_{\mathfrak{q}})$, recover the functors $N_p\mapsto\bm{N}_{\dR,\co_y}$ and $N_p\mapsto\bm{N}_{\cris,y}$ of Corollary~\ref{cor:realizations y}, obtained from Kisin's functor $\mathfrak{M}$

For any prime $p$, note that the $F$-action on $V$ gives us an orthogonal decomposition:
\[
V_{\Q_p} = \bigoplus_{\mathfrak{p}\mid p}V_{\mathfrak{p}},
\]
where $\mathfrak{p}$ ranges over the $p$-adic places of $F$, and where we have set $V_{\mathfrak{p}} = F_{\mathfrak{q}}\otimes_FV$. For each $\mathfrak{p}\mid p$, set 
\[
L_{\mathfrak{p}} = L_p \cap V_{\mathfrak{p}}\subset V_{\mathfrak{p}}.
\]

\begin{definition}\label{defn:D bad}
Call a prime $p$ \emph{good for $L$}, or simply \emph{good},  if the following conditions hold:
\begin{itemize}
	\item For every  $\mathfrak{p}\mid p$  unramified in $E$, the $\Z_p$-lattice $L_{\mathfrak{p}}$ is $\co_{E,{\mathfrak{p}}}$-stable and self-dual for the induced $\Z_p$-valued quadratic form.
	\item For every $\mathfrak{p}\mid p$  ramified in $E$,  the $\Z_p$-lattice $L_{\mathfrak{p}}$ is   maximal for the induced 
	$\Z_p$-valued quadratic form, 
	and there exists an $\co_{E,{\mathfrak{q}}}$-stable lattice $\Lambda_{\mathfrak{p}}\subset V_{\mathfrak{p}}$ such that
	\[
     \Lambda_{\mathfrak{p}}\subset L_{\mathfrak{p}}\subsetneq \mathfrak{d}_{E_{\mathfrak{q}}/F_{\mathfrak{p}}}^{-1}\Lambda_{\mathfrak{p}}.
	\]
	Here  $\mathfrak{q}\subset \co_E$ is the unique prime above $\mathfrak{p}$.
\end{itemize} 

All but finitely many primes are good: Choose any $\co_E$-stable lattice $\Lambda\subset L$. Then, for all but finitely many primes $p$, $\Lambda_{\Z_p}=L_{\Z_p}$ will be self-dual and hence good.

We will call a prime \emph{bad} if it is not good, and we let $D_{bad}$ be the product of the bad primes. If we wish to make its dependence on the lattice $L$ explicit, we will write $D_{bad,L}$ for this quantity.
\end{definition}

\begin{lemma}\label{lem:level_subgroup}
For every $p\nmid D_{bad}$, we have
\[
K_{L,0,p} = K_{0,p} \subset T(\Q_p).
\]
In particular, $\mathcal{Y}$ is finite \'etale over $\co_E[D_{bad}^{-1}]$.
\end{lemma}
\begin{proof}
Note that $K_{L,0,p}$ contains the subgroup $\Z_p^\times$ of scalars. Therefore, it is enough to show that the image $K_{0,p,so}$ of $K_{0,p}$ in $T_{so}(\Q_p)$ is contained in the discriminant kernel of $L_{\Z_p}$. This is easy to see from the explicit description of $L_{\Z_p}$ in Definition~\ref{defn:D bad}, as well as of $K_{0,p,so}$ in~\eqref{eqn:K0p so}. 

There are two main points: First, $K_{0,p,so}$ preserves all $\co_{E,p}$-stable lattices in $L_{\Z_p}$. Second, for any prime $\mathfrak{p}\subset\co_F$ ramified in $E$ with $\mathfrak{q}\subset\co_E$ the prime above it, if $\alpha\in\co_{E,\mathfrak{q}}$, then $\overline{\alpha}$ and $\alpha$ are congruent mod $\mathfrak{d}_{E_{\mathfrak{q}}/F_{\mathfrak{p}}}$. 

Combining these two facts, if $\Lambda_{\mathfrak{p}}\subset L_{\mathfrak{p}}$ is a maximal $\co_{E,\mathfrak{q}}$-stable lattice, then $K_{0,p,so}$ stabilizes $\Lambda_{\mathfrak{p}}$ and acts trivially on ${\mathfrak{d}_{E_{\mathfrak{q}}/F_{\mathfrak{p}}}^{-1}\Lambda_{\mathfrak{p}} }/ {\Lambda_{\mathfrak{p}}}$. This implies that it stabilizes $L_{\mathfrak{p}}$ and acts trivially on $L^\vee_{\mathfrak{p}}/L_{\mathfrak{p}}$. Since $L^\vee/L$ is a subquotient of $\bigoplus_{\mathfrak{p}}L^\vee_{\mathfrak{p}}/L_{\mathfrak{p}}$, we find that $K_{0,p,so}$ preserves $\widehat{L}$ and acts trivially on $L^\vee/L$.
\end{proof}

%%%%%%%%%%%%%%%%%%%%%%%%%%%%%%%%%%%%%%%

\subsection{The space of special endomorphisms}\label{ss:special end zero cycle}

%%%%%%%%%%%%%%%%%%%%%%%%%%%%%%%%%%%%%%%

\begin{proposition}
\label{prop:no special char 0}
Suppose that $y\in \mathcal{Y}(\C)$. Then $V(A_y) = 0$.
\end{proposition}
\begin{proof}
Let $z^{cm} = \mathscr{V}(\iota_0)_\C \subset V_\C$ be as in~\S\ref{ss:shimura data}. The proposition amounts to the statement that there are no positive elements $x\in V$ that are orthogonal to $z^{cm}$. But if such an $x$ existed, then it would generate the one-dimensional $E$-vector space $\mathscr{V} = V$, and, since $z^{cm}\subset \mathscr{V}_\C$ is $E$-stable, this would imply that \emph{every} element of $V$ is orthogonal to $z^{cm}$, which is clearly impossible. 
\end{proof}

Recall that $V$ is isomorphic as a $T$-representation to the standard representation $V_0 = V(H_0,c)$ from~\S\ref{ss:standard}. If we are viewing $V$ or $V_0$ as an $E$-module, we will emphasize this by writing $\mathscr{V}$ and $\mathscr{V}_0$, instead. There is a canonical Hermitian form on $\mathscr{V}_0$: For $x,y\in \mathscr{V}_0$, we define $\langle x,y\rangle_0 \in E$ by the relation $x\circ y =\langle x,y\rangle_0$ as elements of
$\End(H_0)$. Under the isomorphism $\mathscr{V} \xrightarrow{\simeq} \mathscr{V}_0$, the Hermitian form on $\mathscr{V}$ induced from $\mathscr{Q}$ is carried to the form $\xi\langle x,y\rangle_0$, for some element $\xi\in F$ such that $\iota_0(\xi)<0$ and $\iota_j(\xi)>0$, for $j>0$.

The lattice $L_{\widehat{\Z}}\subset V_{\A_f}$ is $K_{0,L}$-stable, and we have a $K_{0,L}$-equivariant embedding $L_{\widehat{\Z}}\hookrightarrow\End_{C(L)}(H_{\widehat{\Z}})$. From this data, and the constructions in \S~\ref{ss:sheaves_i} and \S~\ref{ss:sheaves_ii}, we obtain embeddings
\begin{equation}\label{eqn:V embedding zero}
\bm{V}_{?} \hookrightarrow \underline{\End}_{C(L)}(\bm{H}_?)\vert_{\mathcal{Y}}
\end{equation}
of sheaves over $\mathcal{Y}$ for $?=B,\ell,\dR,\cris$. The images of these embeddings are local direct summands of their targets when $? = B,\ell$, but not necessarily when $?=\dR,\cris$. However, we have an embedding
\begin{equation}\label{eqn:V embedding zero isogeny}
\bm{V}_{\cris,\Q} \hookrightarrow \underline{\End}_{C(L)}(\bm{H}_\cris)_{\Q}\vert_{\mathcal{Y}}
\end{equation}
in the isogeny category associated with the category of $F$-crystals over $\mathcal{Y}$.

The next result is clear from the definitions and Proposition~\ref{prop:realizations compatibility}.
\begin{proposition}
\label{prop:special zero dim realizations}
For any $\mathcal{Y}$-scheme $S$, and any prime $p$, 
\[
V(A_S[p^\infty])\subset\End_{C(L)}(A_S[p^\infty])
\]
consists precisely of those endomorphisms whose homological realizations land in the images of the embedding~\eqref{eqn:V embedding zero} for $?=p$ over $S[p^{-1}]$, and in the embedding~\eqref{eqn:V embedding zero isogeny} for $?=\cris$ over $S_{\F_p}$. In particular, 
\[
V(A_S)\subset\End_{C(L)}(A_S)
\] 
consists of those endomorphisms whose $\ell$-adic realizations over $S[\ell^{-1}]$ land in the image of the embedding~\eqref{eqn:V embedding zero}, and whose crystalline realizations over $S_{\F_p}$ land in the embedding~\eqref{eqn:V embedding zero isogeny}. 
\end{proposition}

Fix a rational prime $p$. Let $\mathfrak{q} \subset \co_E$ be a prime above $p$, let $\mathfrak{p} \subset \co_F$ be the prime below $\mathfrak{q}$.

\begin{proposition}
\label{prop:special ordinary}
If $\mathfrak{p}$ is split in $F$, then 
\[
V(A_y) = V(A_y[p^\infty]) = 0
\] 
for all $y\in \mathcal{Y}(\F_{\mathfrak{q}}^{\alg})$.
\end{proposition}
\begin{proof}
 Indeed, $V(A_y[p^\infty]) \subset \bm{V}_{\cris,y}[p^{-1}]^{\varphi = 1} = 0$, by Proposition~\ref{prop:v0 cris realization}.
\end{proof}

By Propositions~\ref{prop:no special char 0} and~\ref{prop:special ordinary}, for a geometric point $y$ of $\mathcal{Y}$, if $V(A_y)\neq 0$, then $y$ must be an $\F_{\mathfrak{q}}^{\alg}$-valued point with $\mathfrak{q}\subset \co_E$ the unique prime lying above a prime $\mathfrak{p}\subset \co_F$ that is not split in $E$. 

Until otherwise specified, we will assume from now on that we have fixed the data of such $\mathfrak{p}, \mathfrak{q}$ and $y$. In this case, by Proposition~\ref{prop:abelian scheme reduction}, the abelian variety $A_y$ is supersingular. Therefore, for any prime $\ell$, the natural map
\[
\Z_\ell\otimes\End(A_y)\to \End(A_y[\ell^\infty]).
\]
is an isomorphism. This implies that, if $\ell\neq p$, then the natural map
\begin{equation}\label{eqn:V ell isom}
V(A_y[\ell^\infty]) \to \bm{V}_{\ell,y}
\end{equation}
is an isomorphism.

Also, if $\ell=p$, then the natural map
\begin{equation}\label{eqn:V cris isom}
V(A_y[p^\infty])_\Q \to \bm{V}_{\cris,y}[p^{-1}]^{\varphi=1}
\end{equation}
is also an isomorphism. Moreover, by Proposition~\ref{prop:v0 cris realization}, $\bm{V}_{\cris,y}[p^{-1}]$ is generated by its $\varphi$-invariants.

For any $?$, since $E$ acts $T$-equivariantly on $V$, we have a natural map $E\to \End(\bm{V}_?)_\Q$ giving an action of $E$ on $\bm{V}_?$ in the appropriate isogeny category. In particular, if $y$ is as above, then, via the isomorphisms~\eqref{eqn:V ell isom} and~\eqref{eqn:V cris isom}, the space $V(A_y[\ell^\infty])_\Q$ has an $E$-action, making it a rank $1$ module over $\Q_\ell\otimes_{\Q}E$. If we want to emphasize this structure, we will write $\mathscr{V}(A_y[\ell^\infty])_\Q$ for this space, and $\mathscr{V}(A_y[\ell^\infty])$ for the lattice $V(A_y[\ell^\infty])$ within it.

Recall that there is a natural quadratic form $Q$ on $V(A_y[\ell^\infty])$ induced from composition in $\End(A_y[\ell^\infty])$. There is now a unique Hermitian form $\langle\cdot,\cdot\rangle_{\ell}$ on $\mathscr{V}(A_y[\ell^\infty])_\Q$ with associated $\Q_\ell\otimes_{\Q}F$-quadratic form $\mathscr{Q}_\ell(x) = \langle x,x\rangle_\ell$ such that, for any $x$, we have:
\[
Q(x) = \mathrm{Tr}_{(\Q_\ell\otimes_{\Q}F)/\Q_{\ell}}(\mathscr{Q}_\ell(x)).
\]

Set 
\[
\mathscr{V}(A_y[\infty]) = \prod_\ell \mathscr{V}(A_y[\ell^\infty]).
\]
Then $\mathscr{V}(A_y[\infty])_\Q$ has the structure of a Hermitian space over $\A_{f,E}$.

\begin{proposition}
\label{prop:nearby hermitian space ell}
The Hermitian space $\mathscr{V}(A_y[\infty])_\Q$ is isometric to $\near \mathscr{V}_{\A_f}$, where $\near \mathscr{V}$ is the nearby Hermitian space from Definition~\ref{def:nearby}.
\end{proposition}
\begin{proof}
For each prime $\ell\neq p$,~\eqref{eqn:V ell isom} shows that $\mathscr{V}(A_y[\ell^\infty])$ is isometric to $L_{\Z_\ell}$. This shows that $\mathscr{V}(A_y[\infty])_\Q$ is isometric to $\mathscr{V}_{\A_f}$, and hence to $\near \mathscr{V}_{\A_f}$, away from the prime  $p$.

Now  consider what happens at the prime  $p$. By~\eqref{eqn:V cris isom} there is an isometry
\[
\mathscr{V}(A_y[p^\infty])_\Q \xrightarrow{\simeq} \bm{V}_{\cris,y}[p^{-1}]^{\varphi = 1},
\]
and there is an orthogonal decomposition
\[
\bm{V}_{\cris,y}[p^{-1}] = \bigoplus_{\mathfrak{p}'\mid p}\bm{V}[p^{-1}]_{\cris,\mathfrak{p}'},
\]
where $\mathfrak{p}'$ ranges over the primes in $\co_F$ lying above $p$. 
By the proof of Proposition~\ref{prop:v0 cris realization}, for each $\mathfrak{p}'$  we have
\[
\bm{V}[p^{-1}]_{\cris,\mathfrak{p}'} = V(\bm{H}_{0,\cris,\mathfrak{p}'},c)[p^{-1}].
\]
Under this isomorphism, the Hermitian form on $\bm{V}[p^{-1}]_{\cris,\mathfrak{p}'}$ is carried to the form $\xi\langle \cdot,\cdot\rangle$, where $\langle\cdot,\cdot\rangle$ is the Hermitian form induced from composition in $\End(\bm{H}_{0,\cris,\mathfrak{p}'})$.

If $\mathfrak{p}'\neq \mathfrak{p}$, there is an isomorphism of $F$-crystals
\[
W_{\mathfrak{p}'}\otimes_{\co_F}H_0 \xrightarrow{\simeq}\bm{H}_{0,\cris,\mathfrak{p}'},
\]
where the left hand side is equipped with the semi-linear map $\mathrm{Fr}^{d_0}\otimes 1$. Therefore, we obtain an isomorphism
\begin{equation*}
%\label{eqn:v cris not p}
\bm{V}[p^{-1}]_{\cris,\mathfrak{p}'}^{\varphi=1} = F_{\mathfrak{p}'}\otimes_{F}\mathscr{V}_0
\end{equation*}
carrying the Hermitian form on the left hand side to $\xi\langle\cdot,\cdot\rangle_0$. This shows that $\mathscr{V}(A_y[\infty])_\Q$ is isometric to $\mathscr{V}_{\A_f}$ and hence to $\near \mathscr{V}_{\A_f}$ away from the place $\mathfrak{p}$.

Finally, if $\mathfrak{p}' = \mathfrak{p}$, the $F$-crystal $\bm{H}_{0,\cris,\mathfrak{p}}$ is the Dieudonn\'e $F$-crystal of a Lubin-Tate group over $\co_{E,\mathfrak{q}}$ associated with some uniformizer $\pi\in E_{\mathfrak{q}}$. If $\mathfrak{q}$ is unramified over $F$ and $\pi$ is chosen to lie in $F_{\mathfrak{p}}$, then Proposition~\ref{prp:unramified_vcris} shows that we have an isomorphism
\[
\bm{V}[p^{-1}]_{\cris,\mathfrak{p}}^{\varphi = 1} \xrightarrow{\simeq} E_{\mathfrak{q}}\otimes_E \mathscr{V}_0
\]
carrying the Hermitian form on the left hand side to $\pi\xi\langle \cdot,\cdot\rangle_0$.

If $\mathfrak{q}$ is ramified over $F$ and $\pi$ is chosen to lie in $F_{\mathfrak{p}}$, then Proposition~\ref{prp:unramified_vcris} shows that we have an isomorphism:
\[
\bm{V}[p^{-1}]_{\cris,\mathfrak{p}}^{\varphi = 1} \xrightarrow{\simeq} E_{\mathfrak{q}}\otimes_E \mathscr{V}_0
\]
carrying the Hermitian form on the left hand side to $\gamma\xi\langle \cdot,\cdot\rangle_0$, where $\gamma = \beta\overline{\beta}$, for some $\beta\in W_{\mathfrak{q}}^\times$ satisfying $\pi\mathrm{Fr}^{d_0}(\beta) = \overline{\pi}\beta$. 

In either case, it is easily checked that this establishes an isometry of $\bm{V}[p^{-1}]_{\cris,\mathfrak{p}}^{\varphi=1}$ with $\near \mathscr{V}_{\mathfrak{p}}$. This finishes the proof of the proposition.
\end{proof}

\begin{proposition}\label{prop:cm special supersingular}
Suppose that $\mathfrak{p}$ is not split in $E$ and that $\mathfrak{q}\subset \co_E$ is the unique prime above it. Fix a point $y\in\mathcal{Y}(\F_{\mathfrak{q}}^{\alg})$. Then $A_y$ is a supersingular abelian variety. Moreover:
\begin{enumerate}
\item $V(A_y)\neq 0$;
\item The natural map 
\begin{equation}\label{eqn:special end lots}
\widehat{\Z}\otimes_{\Z} V(A_y)\to V(A_y[\infty])
\end{equation}
is an isometry of quadratic spaces over $\widehat{\Z}$.
\end{enumerate}
\end{proposition}
\begin{proof}
It was already observed above that $A_y$ being supersingular follows from Proposition~\ref{prop:abelian scheme reduction}.  

From Proposition~\ref{prop:nearby hermitian space ell}, we find that, for any prime $\ell$, the rank of the $\Z_\ell$-module $V(A_y[\ell^\infty])$ is equal to $2d = \dim V$. Moreover, we can find a finite extension of $\F_{\mathfrak{q}}$ over which $y$, and all the elements of $V(A_y[\ell^\infty])$ are defined. 

This shows that  \cite[Assumption 6.2]{mp:tatek3} is satisfied, and so our proposition now follows from  [\emph{loc.~cit.}, Theorem 6.4].   
The statement of the cited result assumed $p>2$, but its  proof goes through without this assumption.
\end{proof}

\begin{corollary}\
\label{cor:special end structure}
\begin{enumerate}
	\item  
	For any connected $\mathcal{Y}$-scheme $S$, $V(A_S)_\Q$ has a canonical structure of an $E$-vector space equipped with a positive definite Hermitian form $\langle\cdot,\cdot\rangle$. 
	\item 
	We have $V(A_S)_\Q = 0$ unless the image of $S\to \mathcal{Y}$ is supported on a single special fiber $\mathcal{Y}_{\F_{\mathfrak{q}}}$ with $\mathfrak{q}\subset\co_E$ a prime lying over a non-split prime $\mathfrak{p}\subset \co_F$. 
	\item 
	If $S\to\mathcal{Y}$ is supported on a single special fiber $\mathcal{Y}_{\F_{\mathfrak{q}}}$ as in (2), then there is an isometry
	\[
\mathscr{V}(A_S)_\Q \xrightarrow{\simeq} \near \mathscr{V}
\]
of Hermitian spaces over $E$. Here, we have written $\mathscr{V}(A_S)_\Q$ for $V(A_S)_\Q$ equipped with its additional Hermitian $E$-vector space structure.
\end{enumerate}
\end{corollary}
\begin{proof}
From Proposition~\ref{prop:tQ_action}, we obtain an embedding $T\hookrightarrow\Aut^\circ(A_S)$, whose homological realizations are the maps $\theta_?(H)$ of [\emph{loc.~cit.}]. This implies that $V(A_S)_\Q\subset \End_{C(L)}(A_S)_\Q$ is a $T$-stable subspace.

First assume that $S$ is a geometric point $y\in \mathcal{Y}(\F_{\mathfrak{q}}^{\alg})$, where $\mathfrak{q}\subset\co_E$ lies over a non-split prime $\mathfrak{p}\subset\co_F$. For any $\ell\neq p$, by~\eqref{eqn:V ell isom}, $V(A_y)_{\Q_\ell}$ is isomorphic as a $T_{\Q_\ell}$-representation to $V_{\Q_\ell}$. It is easy to see that, for each $\ell$, the $E$ action on $V_{\Q_\ell}$ identifies $E_{\Q_\ell}$ with the commutant of $T_{\Q_\ell}$ in $\End(V_{\Q_\ell})$. Therefore, the commutant of $T$ in $\End(V(A_y)_\Q)$ is a commutative $\Q$-algebra that, for every $\ell\neq p$, is isomorphic to $E_{\Q_\ell}$ over $\Q_\ell$. As such, this commutant must be the field $E$. In this way, we obtain an $E$-action on $V(A_y)_\Q$, making it a $1$-dimensional $E$-vector space, which is irreducible as a representation of $T$. In particular, there is a unique Hermitian form $\langle\cdot,\cdot\rangle$ on it, which when composed with $\mathrm{Tr}_{F/\Q}$ gives the canonical quadratic form on $V(A_y)_\Q$, It now follows from the Hasse principle for $E$-Hermitian spaces, and Propositions~\ref{prop:special_quadratic_form},~\ref{prop:nearby hermitian space ell} and~\ref{prop:cm special supersingular} that we have an isometry
\[
\mathscr{V}(A_y)_\Q \xrightarrow{\simeq} \near \mathscr{V}
\]
of Hermitian spaces over $E$.

If $S$ is any $\mathcal{Y}$-scheme with $V(A_S)\neq 0$, it follows from Proposition~\ref{prop:no special char 0} that the image of $S$ in $\mathcal{Y}$ does not intersect the generic fiber ${Y}$, and thus is supported in finite characteristics. Suppose that $s\in S(\F_{\mathfrak{q}}^{\alg})$ is a geometric point lying above a point $y\in \mathcal{Y}(\F_{\mathfrak{q}}^{\alg})$. This implies that $\mathfrak{q}$ lies over a non-split prime $\mathfrak{p}\subset\co_F$. Moreover, since $\mathscr{V}(A_y)_\Q$ is an irreducible representation of $T$, the map
\[
V(A_S)_\Q \to \mathscr{V}(A_y)_\Q
\] 
must be an isomorphism. In particular, $V(A_S)_\Q$ has a canonical structure of a Hermitian space over $E$, equipped with which it is isomorphic to $\near \mathscr{V}$. It follows from this that the image of $S$ in $\mathcal{Y}$ has to be supported over $\mathcal{Y}_{\F_{\mathfrak{q}}}$.
\end{proof}

  %%%%%%%%%%%%%%%%%%%%%%%%%%%%%%%%%%%%%%%

\section{Arithmetic intersections and derivatives of $L$-functions}\label{s:intersection}

%%%%%%%%%%%%%%%%%%%%%%%%%%%%%%%%%%%%%%%

In this section, we set up the terminology required to state the main technical result of this paper, Theorem~\ref{thm:arithmetic BKY}.
In particular, following~\cite{BKY}, we discuss the theory of incoherent Eisenstein series and their q-expansion
and recall the main theorem of \emph{loc. cit.}

Keep $E/F$ and   $(\mathscr{V},\mathscr{Q})$ as in \S \ref{ss:hermitian}.  Once again define a $\Q$-quadratic space
\[
 (V,Q)   =  (\mathscr{V} , \mathrm{Tr}_{F/\Q}  \circ \mathscr{Q}) 
 \]
of  signature $(2d-2,2) = (n,2)$, where   $d=[F:\Q]$. We will assume that $d>1$, so that $n>0$.

Let  $\chi : \A_F^\times \to \{ \pm 1\}$  be the quadratic character defined by the CM extension $E/F$, 
and  let  $D_E$ and $D_F$ be the discriminants of $E$ and $F$. 
If we set $\Gamma_\R(s) =      \pi^{-s/2} \Gamma(s/2)$,  the completed $L$-function
\begin{equation}\label{completed L}
\Lambda(s, \chi )  =   \left| \frac{ D_E } {D_F} \right|^{s/2}    \Gamma_\R(s+1)^d L(s,\chi)
\end{equation}
satisfies the function equation  $\Lambda(1-s, \chi) = \Lambda(s,\chi)$.  Furthermore, 
\begin{equation}\label{eqn:completed log der}
\frac{ \Lambda'(0,\chi )  }{\Lambda(0,\chi )}    
=    \frac{ L'(0,\chi)  }{ L(0,\chi) }  + \frac{1}{2} \log \left| \frac{  D_{E}  }{  D_{F} }\right| 
- \frac{ d   \log(4\pi e^\gamma)}{ 2 }   
\end{equation}
where $\gamma = -\Gamma'(1)$ is the Euler-Mascheroni constant.

%%%%%%%%%%%%%%%%%%%%%%%%%%%%%%%%

\subsection{Incoherent Eisenstein series}
\label{ss:incoherent}

%%%%%%%%%%%%%%%%%%%%%%%%%%%%%%%%%

Recalling the standard additive character $\psi_\Q: \Q\backslash \A \to \C^\times$ of \S \ref{ss:harmonic forms}, define
\[
\psi_F: F\backslash \A_F \to \C^\times
\]
by $\psi_F = \psi_\Q \circ \mathrm{Tr}_{F/\Q}$.

If $v$ is an arichmedean place of $F$, denote by  $\mathscr{C}_v$ the unique positive definite rank $2$ quadratic space over $F_v$.
Set $\mathscr{C}_\infty = \prod_{v\mid \infty} \mathscr{C}_v$.  The  rank $2$ quadratic space 
\[
\mathscr{C}  = \mathscr{C}_\infty \times \widehat{\mathscr{V}}
\]
over $\A_F$  is \emph{incoherent}, in the sense that it is not the adelization of any $F$-quadratic space.  In fact, $\mathscr{C}$ is isomorphic to 
$\mathscr{V} \otimes_F \A_F$ everywhere locally, except at the unique archimedean place $\iota_0$ at which $\mathscr{V}$ is negative definite.

To any Schwartz function  
\[
\varphi_\infty \otimes \varphi  \in S(\mathscr{C}_\infty) \otimes S(\widehat{\mathscr{V}}) \iso S( \mathscr{C} )
\] 
we may associate  an incoherent  
 Hilbert modular Eisenstein series via the process described in \cite{KuAnnals, KYEis, Yang}.
Briefly,  the construction is as follows.  Denote by $I( s, \chi)$
the degenerate principal series representation of $\SL_2(\A_F)$ induced from the character $\chi | \cdot |^s$ 
on the subgroup $B\subset \SL_2$ of upper triangular matrices.  Thus $I(s,\chi)$ consists of all smooth functions $\Phi(g,s)$ on $\SL_2(\A_F)$
satisfying the transformation law
\[
\Phi \left(   \left( \begin{matrix}     a& b \\ & a^{-1}      \end{matrix}  \right)   g    , s   \right) = \chi(a) |a|^{s+1}  \Phi(g,s).
\]
As in \S \ref{ss:harmonic forms}, the Weil representation $\omega_{\mathscr{C}}$ (determined by the character $\psi_F$) defines  an action of 
$\SL_2(\A_F)$ on $S( \mathscr{C} )$, and the function 
\[
\Phi (g,0) =  \omega_{\mathscr{C}}(g) ( \varphi_\infty \otimes \varphi ) (0)
\] 
lies in the induced representation $I( 0, \chi)$.
It extends uniquely to a standard section $\Phi ( g , s )$ of $I(s,\chi)$, which determines an  Eisenstein series
\begin{equation}\label{eisenstein formation}
E(g ,s ,  \Phi ) = \sum_{  \gamma \in  B(F) \backslash \SL_2(F)   }\Phi ( \gamma g ,s) 
\end{equation}
on $\SL_2(\A_F)$.   As in  \cite[Theorem 2.2]{KuAnnals}, the incoherence of $\mathscr{C}$ implies   that $E(g ,s , \Phi)$ vanishes identically at $s=0$.

Endow $\A_F$ with the Haar measure self-dual with respect to $\psi_F$, and give $F\backslash \A_F$ the quotient measure.
For every $\alpha\in F$ define  the \emph{Whittaker function}
\begin{equation}\label{whittaker def}
W_\alpha ( g,s , \Phi )  = \int_{\A_F}  \Phi ( w n(b) g, s)\cdot \psi_F( -    \alpha b )\, db,
\end{equation}
where $w=\left(\begin{smallmatrix}    & -1 \\ 1     \end{smallmatrix}\right)$ and $n(b) = \left( \begin{smallmatrix} 
1 & b \\ & 1 \end{smallmatrix}\right).$ 
The  Eisenstein series  (\ref{eisenstein formation}) has a Fourier expansion
\[
E(g,s, \Phi ) = \sum_{\alpha \in F}E_\alpha (g,s, \Phi )
\]
in which the coefficient 
\[
E_\alpha (g,s,  \Phi ) = \int_{ F\backslash \A_F } E\big( n(b) g,s, \Phi  \big) \cdot \psi_F( -  \alpha b )\, db 
\]
is related to the Whittaker function by
\begin{equation}\label{eisenstein whittaker}
E_\alpha(g,s,  \Phi ) = 
\begin{cases}
W_\alpha ( g,s , \Phi )  & \mbox{if }\alpha \not=0 \\
\Phi(g,s) + W_0 ( g,s , \Phi )  & \mbox{if }\alpha =0.
\end{cases}
\end{equation}

The degenerate principal series decomposes  $I(s,\chi) = \otimes_v I_v(s,\chi_v)$, where the tensor product
is over all places of $F$. There is an obvious factorization $\Phi = \Phi_\infty \otimes \Phi_f $ into archimedean and nonarchimedean parts, which
induces a corresponding factorization
\[
W_\alpha ( g,s , \Phi ) = W_{\alpha,\infty}(g_\infty , s , \Phi_\infty) \cdot W_{\alpha,f}( g_f, s, \Phi_f)
\]
of the integral (\ref{whittaker def}).
In practice there will be a further factorization 
$
\varphi = \otimes_p \varphi_p \in S(\widehat{\mathscr{V}})
$ 
over the rational primes, and hence   a factorization
\[
W_\alpha ( g,s , \varphi ) =  W_{\alpha,\infty} (g_\infty,s,\phi_\infty ) \cdot \prod_{ p } W_{\alpha,p} ( g_p,s , \varphi_p )
\]
of Whittaker functions.
When the component $\varphi_p$ admits a further factorization $\varphi_p = \otimes_{\mathfrak{p} \mid p} \varphi_\mathfrak{p}$ 
so does
\[
W_{\alpha,p} ( g_p,s , \varphi_p ) = \prod_{\mathfrak{p} \mid p} W_{\alpha,\mathfrak{p}} ( g_ \mathfrak{p} ,s , \varphi_\mathfrak{p} ).
\]

%\begin{remark}
%Suppose  $\alpha\in F^\times$, and $v$ is a place of $F$ for  which the local quadratic  space 
%$\mathscr{C}_v$ does not represent $\alpha$. Exactly as  in   \cite[Proposition 1.4]{KuAnnals}, the linear functional
%$\varphi_v \mapsto W_{\alpha,v} ( I  ,0 , \varphi_v )$ on $S(\mathscr{C}_v)$ is identically $0$.  Here $I\in \SL_2(F_v)$ is the identity matrix.  
%\end{remark}

From now on we will always take the archimedean component  $\varphi_\infty$ of our Schwartz function to be the 
Gaussian distribution
\[
\varphi_\infty^{\bm{1}} = \otimes \varphi^{\bm{1}}_v \in  \bigotimes_{v\mid \infty} S( \mathscr{C}_v)
\] 
defined by
$
\varphi^{\bm{1}}_v(x) = e^{-2\pi \mathscr{Q}_v(x)}
$
($\mathscr{Q}_v$ is the quadratic form on $\mathscr{C}_v$.)   By \cite[Lemma 4.1]{KYEis}  the resulting Eisenstein series 
(\ref{eisenstein formation})  has parallel weight $1$.
As the archimedean component  will remain fixed, the section $\Phi$ is determined by 
 $\varphi \in S(\widehat{\mathscr{V}})$, and we   will often write
\[
E(g,s,\varphi)   =E(g,s,  \Phi ) .
\]

%%%%%%%%%%%%%%%%%%%%%%%%%%%%%%%%%%%

\subsection{A formal $q$-expansion}
\label{ss:q-expansion}

%%%%%%%%%%%%%%%%%%%%%%%%%%%%%%%%%%%

As in  the previous section, fix a Schwartz function $\varphi  \in S(\widehat{\mathscr{V}})$, and let 
$E(g,s,\varphi)$ be the corresponding incoherent weight $1$ Eisenstein series on $\SL_2(\A_F)$.

For any $\vec{\tau} \in \mathcal{H}^d$ let  $g_{\vec{\tau}} \in \SL_2(\A_F )$   be the matrix with archimedean components
\[
g_{\tau_i} = 
 \Bigg(\begin{matrix}
1 &  u_i   \\
& 1
\end{matrix}\Bigg)
\Bigg(\begin{matrix}
v_i^{ 1/2 } &     \\
& v_i^{- 1/2} 
\end{matrix}\Bigg),
\]
and take all finite components to be the identity matrix.  Here 
\[
\vec{u} = (u_0,\ldots, u_{d-1}) , \quad \vec{v} = (v_0 , \ldots, v_{d-1})
\]
are the real and imaginary parts of  $\vec{\tau}$.   Exactly as in  \cite[(4.4)]{BKY},  define a classical weight $1$ Hilbert modular Eisenstein series 
\[
E ( \vec{\tau} , s , \varphi ) =  \frac{ 1}{  \sqrt{ N( \vec{v} ) }  } \cdot E(g_{\vec{\tau} } ,s ,\varphi),
\] 
where $N(\vec{v}) = v_0 \cdots v_{d-1}$.  Its  derivative at $s=0$ has the Fourier expansion
\[
E^{   \prime } ( \vec{\tau} , 0 , \varphi ) = \frac{ 1 }{  \sqrt{ N( \vec{v} )}   } \cdot \sum_{\alpha \in F}  E_\alpha'( g_{\vec{\tau}} , 0 ,\varphi).
\]

As in \cite{KuAnnals,KYtheta}, for any $\alpha \in F^\times$ define the  \emph{difference set} 
\[
\mathrm{Diff}(\alpha) = \{  \mbox{places $v$ of $F$} : \mathscr{C}_v \mbox{ does not represent } \alpha \}.
\]
Usually  $\alpha$ will be totally positive, in which case  
\begin{align*}
\mathrm{Diff}(\alpha) & = \{  \mbox{primes }\mathfrak{p} \subset \co_F : \mathscr{V}_\mathfrak{p} \mbox{ does not represent } \alpha \} \\
& = \{  \mbox{primes }\mathfrak{p} \subset \co_F :  \chi_\mathfrak{p}( \alpha ) \neq \mathrm{inv}_\mathfrak{p} (\mathscr{V}) \}.
\end{align*}

\begin{remark}\label{rem:diff}
Note that $\mathrm{Diff}(\alpha)$ is a finite set of odd cardinality, and any place $v \in \mathrm{Diff}(\alpha)$ is nonsplit in $E$.  
If $\mathfrak{p}\subset \co_F$ is a finite place, then  $\mathrm{Diff}(\alpha) = \{ \mathfrak{p} \}$
if and only if $\alpha$ is represented by the nearby hermitian  space  $\near \mathscr{V}$ of Definition \ref{def:nearby}.
\end{remark}

All parts of the following proposition follow from the statement and proof of \cite[Proposition 4.6]{BKY}.

\begin{proposition}\label{prop:coefficient support}
For any  totally positive   $\alpha \in F$ we have  
\[
\frac{1 }{   \sqrt{ N( \vec{v} ) }   }\cdot  E_\alpha'( g_{\vec{\tau}} , 0 ,\varphi) = 
 \frac{ a_F( \alpha  , \varphi) }{ \Lambda( 0 , \chi )   }  \cdot q^\alpha
\]
for some  constant $a_F( \alpha  , \varphi)$ independent of $\vec{\tau}$.   Furthermore:
\begin{enumerate}
\item
If  $| \mathrm{Diff}(\alpha)| >1$, then $a_F(\alpha,\varphi)=0$.
\item
If  $\mathrm{Diff}(\alpha) =\{ \mathfrak{p}\}$, then 
\[
\frac{ a_F(\alpha ,\varphi)  } {  \Lambda(0,\chi ) } \in \Q (\varphi) \cdot \log N(\mathfrak{p}),
\] 
where  $\Q(\varphi)/\Q$ is the extension obtained by adjoining all values of $\varphi$.
\end{enumerate}
\end{proposition}

Now we study  the constant term.  Much of the following proposition is  implicit in the statement and proof of  \cite[Proposition 4.6]{BKY},
but the relevant part of [\emph{loc.~cit.}] is misstated, and we need more information than is found there.

\begin{proposition}\label{prop:coarse constant}
There is a meromorphic function $M(s,\varphi)$ such that 
\begin{equation}\label{full constant}
\frac{ E_0( g_{\vec{\tau}} , s  ,\varphi) }{  \sqrt{ N(\vec{v} )  }  } =   \varphi(0) \cdot  N(\vec{v})^{s/2}  
-
 N(\vec{v})^{-s/2}   \frac{\Lambda(s , \chi ) }{\Lambda(s+1,\chi)}  \cdot   M (s,\varphi)  .
\end{equation}
  If $\varphi=\otimes_\mathfrak{p}\varphi_\mathfrak{p}$  factors over the primes of $\co_F$, then so does
\[
M(s,\varphi)=\prod_\mathfrak{p} M_\mathfrak{p}(s,\varphi_\mathfrak{p}).
\]  
Each factor  $M_\mathfrak{p}(s,\varphi_\mathfrak{p})$ is a rational function, with coefficients in $\Q(\varphi_\mathfrak{p})$,  
in the variable  $N(\mathfrak{p})^s$, and all but finitely many factors are equal to $1$.  Finally, 
 \begin{equation}\label{lazy constant}
\frac{ 1 }{   \sqrt{ N( \vec{v} ) }   }\cdot  E_0'( g_{\vec{\tau}} , 0 ,\varphi) 
 =   \varphi(0)  \log N( \vec{v} ) +    \frac{ a_F(0,\varphi) }{ \Lambda( 0 , \chi )    }  ,
 \end{equation}
 where the constant $a_F(0,\varphi)$ is defined by the relation
\[
\frac{ a_F(0,\varphi) }{\Lambda(0,\chi)   }  =   -2 \varphi(0) \cdot  \frac{  \Lambda'(0,\chi) }{ \Lambda(0,\chi)  }-   M'(0,\varphi) .
\]
\end{proposition}

\begin{proof}
Assume  that $\varphi=\otimes_\mathfrak{p} \varphi_\mathfrak{p}$  admits a factorization over the finite places of $F$, so that 
there are similar factorizations 
\[
\Phi(g,s) = \prod_v \Phi_v(g,s), 
\quad W_0(g,s,\Phi) = \prod_v W_{0,v} (g,s, \Phi_v) 
\]
over all places of $F$.     We define 
\[ 
M (s,\varphi ) =     \prod_\mathfrak{p} M_\mathfrak{p}(s,\varphi_\mathfrak{p}),
\]
where 
\begin{align}\label{local M}
M_\mathfrak{p}(s,\varphi_\mathfrak{p}) 
&=  \frac{      \mathrm{N}(\mathfrak{p})^{ f(\mathfrak{p})  /2 }    } {  \gamma_\mathfrak{p} (\mathscr{V} )    }  \cdot 
\frac{L_\mathfrak{p}(s+1,\chi)} {L_\mathfrak{p}(s,\chi)}  \cdot   W_{0,\mathfrak{p}} (I,s, \Phi_\mathfrak{p})  \\
&=  
\frac{      \mathrm{N}(\mathfrak{p})^{ f(\mathfrak{p})  /2 }    } {  \gamma_\mathfrak{p} (\mathscr{V} )    }  \cdot 
\frac{L_\mathfrak{p}(s+1,\chi)} {L_\mathfrak{p}(s,\chi)}  
\cdot  \int_{F_\mathfrak{p} }   \Phi_\mathfrak{p} \left(    \begin{matrix}  w n(b)      \end{matrix}      , s \right)\, db . \nonumber 
\end{align}
Here $I\in \SL_2(F_\mathfrak{p})$ is the identity matrix,
$
f(\mathfrak{p})  =  \ord_\mathfrak{p}(\mathfrak{D}_F D_{E/F} ),
$
where  $\mathfrak{D}_F$ and $D_{E/F}$ are the different and relative discriminants of $F/\Q$ and $E/F$, respectively, and 

\begin{equation}\label{eqn:gamma p defn}
\gamma_\mathfrak{p}(\mathscr{V}  ) = \chi_\mathfrak{p}(-1) \cdot  \mathrm{inv}_\mathfrak{p}(\mathscr{V})  \cdot \epsilon_\mathfrak{p} (\chi, \psi_F)
\in \{ \pm 1, \pm i\}
\end{equation}
is the local Weil index (relative to  $\psi_F$) as in  \cite{Yang}.  These satisfy
\[
(-i)^d \prod_\mathfrak{p} \gamma_\mathfrak{p}(\mathscr{V}) = -1.
\]
Note that for a given $\varphi$, all but finitely many $\mathfrak{p}$ satisfy  $M_\mathfrak{p}(s,\varphi_\mathfrak{p})=1$.  
This is an easy exercise.  Alternatively, as two factorizable Schwartz functions are equal in all but finitely many components,
it suffices to prove the claim for any one  factorizable Schwartz function.  This is done below.

%For all but finitely many  $\mathfrak{p}$  the local section $\Phi_\mathfrak{p}(g,s)$ satisfies $\Phi_\mathfrak{p}(k,s)=1$ 
%for all   $k\in \SL_2(\Z_p)$, and when this holds   
%\[
% \int_{F_\mathfrak{p} }   \Phi_\mathfrak{p} \left(    \begin{matrix}  w n(b)      \end{matrix}      , s \right)\, db 
% = \frac{L_\mathfrak{p}(s,\chi)} {L_\mathfrak{p}(s+1,\chi)}  .
%\]
%This can be seen by substituting 
%\[
% \Phi_\mathfrak{p} \left(    \begin{matrix}  w n(b)      \end{matrix}      , s \right) =
% \begin{cases}
% 1 & \mbox{if } | b| \le 1 \\
%  \chi_\mathfrak{p}(b) \cdot |b|^{-s-1} & \mbox{if } | b| \ge 1
% \end{cases}
%\]
%into the integral  and computing directly.   

 Extend $\varphi \mapsto M (s,\varphi )$ linearly to all Schwartz functions.   
Combining the definition (\ref{local M}) with  the calculation
\[
W_{0,\infty}( g_{\vec{\tau}} , s ,\Phi_\infty)  =  
 (-i)^d    \frac{  \Gamma_\R(s+1)^d }{ \Gamma_\R(s+2)^d }  \cdot  N(\vec{v})^{(1-s)/2}  
\]
of  \cite[Proposition 2.4]{Yang},  we find 
\[
W_0(g_{\vec{\tau}} ,s,\Phi) =   - N(\vec{v})^{(1-s)/2}   \frac{\Lambda(s , \chi ) }{\Lambda(s+1,\chi)}  \cdot   M(s,\varphi).
\]
Plugging this equality and 
\[
\Phi( g_{ \vec{\tau}}  , s ) = N(\vec{v})^{(s+1)/2}  \cdot \Phi ( I ,s ) =  N(\vec{v})^{(s+1)/2}  \cdot \varphi(0)
\]
  into  the equality
\[
 E_0( g_{\vec{\tau}} , s  ,\varphi)  =  \Phi(g_{\vec{\tau}}  ,s) + W_0(g_{\vec{\tau}} ,s,\Phi) 
\]
of (\ref{eisenstein whittaker})  proves (\ref{full constant}).  
As the left hand side of (\ref{full constant}) vanishes at $s=0$, the functional equation $\Lambda(1-s,\chi) = \Lambda(s,\chi)$ implies
 $M  (0,\varphi)= \varphi(0)$, and (\ref{lazy constant})  then follows directly from (\ref{full constant}) by taking the derivative.

It only remains to prove the claims concerning the rationality of the local factors $M_\mathfrak{p}(s,\varphi_\mathfrak{p})$.
First we describe $M_\mathfrak{p}(s,\varphi_\mathfrak{p})$ for a specific choice of $\varphi_\mathfrak{p}$.  Fix an isomorphism
\begin{equation}\label{local herm coords}
( \mathscr{V}_\mathfrak{p} , \mathscr{Q}_\mathfrak{p}) \iso ( E_\mathfrak{p} , \xi_\mathfrak{p}  \cdot  \mathrm{Nm}_{E_\mathfrak{p} / F_\mathfrak{p} } )
\end{equation}
with $\xi_\mathfrak{p}\in F_\mathfrak{p}^\times$.   If $\mathfrak{p}$ is either split or ramified in $E$, we choose this isomorphism so that 
$\xi_\mathfrak{p} \in \co^\times_{F,\mathfrak{p}}$.  If $\mathfrak{p}$ is inert in $E$, we choose the isomorphism so that $\ord_\mathfrak{p}(\xi_\mathfrak{p})\in \{0,1\}$.
Now let $\widetilde{\varphi}_\mathfrak{p}$ be the characteristic function of $\co_{E,\mathfrak{p}} \subset E_\mathfrak{p} =\mathscr{V}_\mathfrak{p}$.
For this choice of Schwartz function,  the calculations of \cite{Yang} (see also Corollary~\ref{cor:good prime M constant} below) show that
\[
M_\mathfrak{p}( s,\widetilde{\varphi}_\mathfrak{p}  ) =
\begin{cases}
\mathrm{N} (\mathfrak{p} )^{-1}  \cdot
 \frac{  L_\mathfrak{p}( s+1,\chi)   }{   L_\mathfrak{p}(s-1,\chi)   } 
  & \mbox{if  $\mathfrak{p}$ is inert in $E$ and  $\mathrm{inv}_\mathfrak{p}(\mathscr{V})=-1$}  \\
 1 & \mbox{otherwise.}
 \end{cases}
\]

 By the linearity of $\varphi_\mathfrak{p} \mapsto M_\mathfrak{p}(s,\varphi_\mathfrak{p})$,
it now  suffices to show that  when $\varphi_\mathfrak{p}(0)=0$, the function
\begin{equation}\label{truncated factor}
    \frac{ L_\mathfrak{p} (s,\chi)   }{   L_\mathfrak{p} (s+1,\chi)  }  \cdot M_\mathfrak{p}(s , \varphi_\mathfrak{p})
 =  \frac{  \mathrm{N}(\mathfrak{p})^{ f(\mathfrak{p})  /2 } }{ \gamma_\mathfrak{p} (\mathscr{V}) } 
 \cdot \int_{  F_\mathfrak{p}  } \Phi_\mathfrak{p}( w n(b) ,s)\, db
\end{equation}
is a polynomial in $N(\mathfrak{p})^s$ with coefficients in $\Q(\varphi_\mathfrak{p})$.
We assume $\varphi_\mathfrak{p}(0)=0$  in all that follows.

If $| b | \le 1$ then $\Phi_\mathfrak{p}( w n(b) ,s)$ is independent of $s$, by the definition of a standard section.
If $| b | \ge 1$ then the factorization
\[
w n(b) = 
\left( \begin{matrix}      & -1 \\  1 & b        \end{matrix} \right)
=
\left( \begin{matrix}     b^{-1} & -1 \\  & b        \end{matrix} \right)
\left( \begin{matrix}     1 &   \\  b^{-1}  & 1        \end{matrix} \right)
\]
shows that
\[
\Phi_\mathfrak{p}( w n(b) ,s) =  \chi_\mathfrak{p}(b)   | b |^{ -s-1 }  
\Phi_\mathfrak{p} \left(  \left( \begin{matrix}     1 &   \\  b^{-1}  & 1        \end{matrix} \right) ,0  \right).
\]
As $\Phi_\mathfrak{p}(g,0)$ is locally constant, this last equality also implies that for all $b$ 
outside of  some sufficiently large ball $\mathfrak{p}^{-c}$, we have
\[
\Phi_\mathfrak{p}( w n(b) ,s) = \chi_\mathfrak{p}(b)  | b |^{ -s-1 } \Phi_\mathfrak{p}( I  ,0)
= \chi_\mathfrak{p}(b)  | b |^{ -s-1 } \varphi_\mathfrak{p}(0) =0.
\]
 Using these observations, one can check that  (\ref{truncated factor}) is a polynomial in 
$\mathrm{N}(\mathfrak{p})^s$ by decomposing the integral as a sum of integrals over annuli $\mathfrak{p}^k \smallsetminus \mathfrak{p}^{k+1}$
in the usual way.

For all sufficiently large $c$ we  have
\begin{align*}
  \frac{ 1} { \gamma_\mathfrak{p} (\mathscr{V}) }  \int_{  F_\mathfrak{p}  } \Phi_\mathfrak{p}( w n(b) , 0 )\, db
 & = 
 \frac{1}{ \gamma_\mathfrak{p} (\mathscr{V}) }    \int_{  \mathfrak{p}^{-c}  } \Phi_\mathfrak{p}( w n(b) , 0 )\, db   \\
 & =
 \int_{  \mathfrak{p}^{ -c }  } \int_{ \mathscr{V}_\mathfrak{p} } \varphi_\mathfrak{p}(x)  \psi_{F,\mathfrak{p}} \big( b \mathscr{Q} (x) \big)\, dx\, db \\
  & =
\int_{ \mathscr{V}_\mathfrak{p} } \varphi_\mathfrak{p}(x)  
 \left( \int_{  \mathfrak{p}^{-c}  }    \psi_{F,\mathfrak{p}} \big( b \mathscr{Q}_\mathfrak{p} (x) \big)\, db \right) \, dx.
\end{align*}
The second equality is easily obtained from the explicit formulas \cite[(4.2.1)]{HY} defining the Weil representation.
In the above equalities, Haar measure on $\mathscr{V}_\mathfrak{p}$ is normalized as in \cite[Lemma 4.6.1]{HY}, 
so that, for any isomorphism (\ref{local herm coords}),
\[
\mathrm{Vol}( \co_{E,\mathfrak{p}} ) 
= \mathrm{N}( \mathfrak{p} )^{  -\ord_\mathfrak{p}(D_{E/F} ) / 2 }
 \mathrm{N}( \mathfrak{p} )^{  -\ord_\mathfrak{p}( \xi_\mathfrak{p} ) } .
\]
The Haar measure on $F_\mathfrak{p}$ is chosen to be self-dual with respect to $\psi_{F,\mathfrak{p}}$, so that
\[
\mathrm{Vol}(\mathfrak{p}^{-c}) =   \mathrm{N}(\mathfrak{p})^c  \cdot  \mathrm{Vol}(\co_{F,\mathfrak{p}} )=
\mathrm{N}(\mathfrak{p})^{ c } \cdot \mathrm{N}(\mathfrak{p})^{ - \ord_\mathfrak{p}(\mathfrak{D}_F) /2 }.
\]
The inner integral above is 
\[
 \int_{  \mathfrak{p}^{-c}  }    \psi_{F,\mathfrak{p}} \big( b \mathscr{Q}_\mathfrak{p} (x) \big)\, db
= \begin{cases}
\mathrm{Vol}( \mathfrak{p}^{-c} ) & \mbox{if }\mathscr{Q}_\mathfrak{p}(x) \in  \mathfrak{p}^c\mathfrak{D}_{F,\mathfrak{p}}^{-1} \\
0 & \mbox{otherwise,}
\end{cases}
\]
and from this it is clear that the value at $s=0$ of (\ref{truncated factor}) lies in $\Q(\varphi_\mathfrak{p})$.

By the interpolation trick of Rallis, as in \cite[Lemma 4.2]{KYEis}, the calculation above can be extended to show that the
value of  (\ref{truncated factor}) lies in $\Q(\varphi_\mathfrak{p})$ for \emph{any} $s\in \Z_{\ge 0}$. 
This shows  that (\ref{truncated factor}) has the form $R(\mathrm{N}(\mathfrak{p})^s)$ where 
$R(T)\in \C[T]$ is $\Q(\varphi_\mathfrak{p})$-valued at infinitely many  $T\in\Z$, and from this it follows that 
$R(T)$  has coefficients in $\Q(\varphi_\mathfrak{p})$.   

This completes the proof of Proposition \ref{prop:coarse constant}.
\end{proof}

As in   \cite[Proposition 4.6]{BKY},  define a formal $q$-expansion
\[
\mathcal{E} ( \vec{\tau}  , \varphi ) =    a_F(0,\varphi) +  \sum_{  \alpha \in F_+ } a_F(\alpha,\varphi) \cdot q^\alpha,
\]
where $F_+ \subset F$ is the subset of totally positive elements.  Its formal diagonal restriction is the formal $q$-expansion
\[
\mathcal{E}(\tau ,\varphi)  = \sum_{m\in \Q } a(m,\varphi )  \cdot q^m
\]
defined by  $a(0,\varphi) =  a_F(0,\varphi)$,  and 
\begin{equation}\label{eisenstein decomp}
a (m,\varphi )  = \sum_{ \substack{  \alpha \in F_+ \\ \mathrm{Tr}_{F/\Q}(\alpha)=m    } } a_F(\alpha,\varphi ) 
\end{equation}
for all $m\not=0$.  In particular $a(m,\varphi) =0$ if $m<0$.

%%%%%%%%%%%%%%%%%%%%%%%%%%%%%%%%%%%%%%%

\subsection{The Bruinier-Kudla-Yang theorem}

%%%%%%%%%%%%%%%%%%%%%%%%%%%%%%%%%%%%%%%

Fix a maximal lattice $L$ in the $\Q$-quadratic space $(V,Q)$.
Recalling the Schwartz function 
$
\varphi_\mu \in S(\widehat{\mathscr{V}}) = S(\widehat{V})
$ 
of  (\ref{mu schwartz}),  abbreviate 
\[
a(m,\mu) = a(m,\varphi_\mu) , \quad  a_F(\alpha, \mu) = a_F(\alpha,\varphi_\mu) 
\]
for any  $\mu\in L^\vee /L $.

 Fix also a  harmonic weak Maass form $f\in H_{ 2-d }(\omega_L)$ with integral principal part.
Let us temporarily denote by 
\[
\bm{f}= \xi( f) \in S_d(\overline{\omega}_L)
\]
 the image of $f$ under the Bruinier-Funke differential operator of (\ref{BF exact sequence}).
Decompose $\bm{f} (\tau)= \sum_\mu \bm{f}_\mu(\tau) \varphi_\mu$, where the sum is over $\mu \in L^\vee /L$, and define 
a generalized  $L$-function
\[
\mathcal{L} \big(s, \xi(f)  \big) = 
\Lambda(s+1 , \chi ) 
\int_{ \SL_2(\Z) \backslash \mathcal{H} } \sum_{\mu \in L^\vee/L}    \overline{ \bm{f} _\mu(\tau) } E( \tau, s, \varphi_\mu) \, \frac{du\, dv}{  v^{2-d} }
\]
exactly as in \cite[(5.3)]{BKY}.  Here  $\tau = u+iv\in\mathcal{H}$, and $E( \tau, s, \varphi)$ 
is the restriction of the   Hilbert modular Eisenstein series   $E ( \vec{\tau} , s , \varphi )$
 to the diagonally embedded $\mathcal{H} \hookrightarrow \mathcal{H}^d$.   
This $L$-function is an entire function of the variable $s$, and vanishes  at  $s=0$.

Abbreviate
\[
\deg_\C({Y})   \define  \sum_{y \in {Y}(\C) }  \frac{ 1 } { |  \Aut(y) | } = 
\frac{ |  T(\Q) \backslash T(\A_f)   / K_{L,0}  |  }{  |   T(\Q) \cap  K_{L,0}    |   },
\]
where $Y(\C)$ is the set of complex points of $Y$, viewed as an $E$-stack.  If  we set
\[
 \mathcal{Y}^\infty  = \mathcal{Y} \times_{ \Spec(\Z) } \Spec( \C),
 \]
 then
 \[
 \sum_{y \in \mathcal{Y}^\infty(\C) }  \frac{ 1 } { |  \Aut(y) | } = 2d\cdot  \deg_\C({Y}) .
 \]

The following theorem is the main result of \cite{BKY}.

\begin{theorem}[Bruinier-Kudla-Yang]\label{thm:BKY}
In the notation above,
\[
   \frac{ \Phi (f ,\mathcal{Y}^\infty) }{  2  \deg_\C ({Y})  } 
= 
-  \frac{  \mathcal{L}'(0 ,  \xi(f)  )  }  {  \Lambda( 0 , \chi )  } 
  +  \sum_{  \substack{  \mu \in L^\vee / L  \\ m \ge 0 }  }    \frac{ a(m,\mu) \cdot c_f^+(-m,\mu) } { \Lambda( 0 , \chi )  },
\]
where  $\Phi(f)$ is the Green function for $\mathcal{Z}(f)$ appearing in (\ref{arithmetic divisor}), 
 and, using the morphism  $  \mathcal{Y}^\infty(\C) \to \mathcal{M}(\C)$
 induced by (\ref{cm morphism}), we abbreviate
 \[
 \Phi (f ,\mathcal{Y}^\infty) =   \sum_{ y \in  \mathcal{Y}^\infty (\C) }  \frac{ \Phi( f ,y) } { | \Aut(y) | }.
 \]
\end{theorem}

%%%%%%%%%%%%%%%%%%%%%%%%%%%%%%%%%%%%%%%

\subsection{The arithmetic intersection formula}

%%%%%%%%%%%%%%%%%%%%%%%%%%%%%%%%%%%%%%%

Exactly as in \S \ref{ss:line bundles}, we may form the group of metrized line bundles $\widehat{\mathrm{Pic}}(\mathcal{Y})$  on $\mathcal{Y}$. 

Let $F_\infty : \mathcal{Y}^\infty(\C) \to \mathcal{Y}^\infty(\C)$ be complex conjugation. 
 As $\mathcal{Y}$ is flat of relative dimension $0$ over $\co_E$, all Cartier divisors on $\mathcal{Y}$ are supported in nonzero characteristics.
 If $\mathcal{Z}$ is such a divisor, by a  Green function for $\mathcal{Z}$ we mean \emph{any} $F_\infty$-invariant $\R$-valued function $\Phi$ on 
 $\mathcal{Y}^\infty(\C)$.    Exactly as in  \S \ref{ss:line bundles}, we define an \emph{arithmetic divisor} on $\mathcal{Y}$ to be a pair
 \[
 \widehat{\mathcal{Z}} = ( \mathcal{Z} ,\Phi)
 \]
 consisting of a Cartier divisor on $\mathcal{Y}$ together with a Green function.   The codimension one arithmetic Chow group 
 $\widehat{\mathrm{CH}}^1(\mathcal{Y})$ is the quotient of the group of all arithmetic divisors by the subgroup of  principal arithmetic 
 divisors
 \[
 \widehat{\mathrm{div}}( \Psi )  = ( \mathrm{div}( \Psi ) , - \log | \Psi |^2 ),
 \]
 for $\Psi$ a nonzero rational function on $\mathcal{Y}$.  Once again we have an isomorphism
\[
\widehat{\mathrm{Pic}}(\mathcal{Y}) \iso \widehat{\mathrm{CH}}^1(\mathcal{Y}).
\]

\begin{remark}
Any arithmetic divisor $(\mathcal{Z},\Phi)$ decomposes as  $(\mathcal{Z},  0 ) + ( 0,\Phi)$, and $\mathcal{Z}$ can be further
decomposed as the difference of two effective Cartier divisors.  
\end{remark}

To define the  \emph{arithmetic degree}, as in  \cite{GS,KRY2,KRY3}, of  an arithmetic divisor 
$\widehat{\mathcal{Z}}$ as above,  we first assume that $\widehat{\mathcal{Z}} = (\mathcal{Z},0)$
with $\mathcal{Z}$ an effective Cartier divisor.  Then
\[
\widehat{\deg}(\widehat{\mathcal{Z}}) = \sum_{\mathfrak{q} \subset \co_E} \log \mathrm{N}(\mathfrak{q})
\sum_{ z\in \mathcal{Z}(  \F^\alg_\mathfrak{q}  ) } \frac{  \mathrm{length}( \co_{\mathcal{Z},z} )    }{| \Aut(z) | }
\]
where  $\co_{\mathcal{Z},z}$ is the \'etale local ring of $\mathcal{Z}$ at $z$.
If $\widehat{\mathcal{Z}} = ( 0 , \Phi)$ is purely archimedean, then
\[
\widehat{\deg}(\widehat{\mathcal{Z}}) =  \frac{1}{2} \sum_{ y \in \mathcal{Y}^\infty(\C)}  \frac{  \Phi(y)  }{ | \Aut(y) | }.
\]
The arithmetic degree extends linearly to all arithmetic divisors, and defines a homomorphism
\[
\widehat{\deg}  :\widehat{\mathrm{Pic}}( \mathcal{Y} )  \to \R .
\]

We now define a homomorphism
\[
[ \cdot: \mathcal{Y} ] :  \widehat{\mathrm{Pic}}( \mathcal{M} ) \to \R ,
\]
the \emph{arithmetic degree along $\mathcal{Y}$},   as the composition
\[
\widehat{\mathrm{Pic}}( \mathcal{M} ) \to \widehat{\mathrm{Pic}}( \mathcal{Y} )  \map{\widehat{\deg} } \R .
\]

\begin{theorem}\label{thm:arithmetic BKY}
Recall the integer $D_{bad}=D_{bad,L}$ defined following Definition \ref{defn:D bad}.
For any $f\in H_{ 2-d }(\omega_L)$ with integral principal part, the equality 
\[
 \frac{  [ \widehat{\mathcal{Z}}(f) : \mathcal{Y} ]   }{   \deg_\C ({Y})  }   =
-  \frac{ \mathcal{L}'(0 , \xi(f) )  }  {  \Lambda( 0 , \chi )  }  
+     \frac{  a(0,0)   \cdot  c_f^+( 0,0) } {  \Lambda( 0 , \chi )  } 
\]
holds up to a $\Q$-linear combination of $\{ \log(p) : p\mid D_{bad}\}$. 
\end{theorem}

Theorem  \ref{thm:arithmetic BKY} is the technical core of this paper;  its proof will occupy all of \S \ref{s:BKY proof}, 
with the completion of the proof appearing in  \S \ref{ss:arithmetic intersection proof}.

\begin{remark}\label{rem:proper intersection}
By Proposition \ref{prop:no special char 0}, the $\Z$-quadratic space of special endomorphisms $V(A_y)$ is $0$ for any 
complex point $y\in \mathcal{Y}(\C)$.  By the very definition of the special divisors $\mathcal{Z}(m,\mu)$, it follows that the image
of $\mathcal{Y} \to \mathcal{M}$ is disjoint from the support of all $\mathcal{Z}(m,\mu)$, and hence from the support of  $\mathcal{Z}(f)$,
 in the complex fiber.  As $\mathcal{Y}$ is flat over $\Z$ 
of relative dimension $0$, this implies that the image of $\mathcal{Y}$ meets the support of $\mathcal{Z}(f)$ properly; \emph{i.e.~}the intersection 
has dimension $0$, and is supported in finitely many nonzero characteristics.
\end{remark}

%%%%%%%%%%%%%%%%%%%%%%%%%%%%%%%%%%%%%%%

\section{Proof of the arithmetic intersection formula}
\label{s:BKY proof} 

%%%%%%%%%%%%%%%%%%%%%%%%%%%%%%%%%%%%%%%

In this section we prove Theorem  \ref{thm:arithmetic BKY}. There are two main computations that are independent of each other: Proposition~\ref{prop:explicit siegel-weil} and Theorem~\ref{thm:local ring}. The first computes the Fourier coefficients of an incoherent Eisenstein series, and the second computes the lengths of the local rings of the intersection between the special divisors on the ambient GSpin Shimura variety with the zero dimensional Shimura variety from Section~\ref{ss:zero dimensional}. These combine to give Theorem~\ref{thm:degree}, which is at the heart of the proof of the main theorem.

 %%%%%%%%%%%%%%%%%%%%%%%%%%%%%%%%%%%%%%%

 % \subsection{Nearby lattices at good primes}
 % \label{ss:lattice calc}

% %%%%%%%%%%%%%%%%%%%%%%%%%%%%%%%%%%%%%%%

 %%%%%%%%%%%%%%%%%%%%%%%%%%%%%%%%%%%%%%%

 \subsection{Local Whittaker functions}
 \label{ss:whittaker functions}

% %%%%%%%%%%%%%%%%%%%%%%%%%%%%%%%%%%%%%%%

Let $p$ be a good prime, in the sense of Definition \ref{defn:D bad}, 
and let $\mathfrak{p}\subset\co_F$ be a prime above it. 
We will assume that $\mathfrak{p}$ is not split in $\co_E$. 
Let $\mathfrak{q}\subset \co_E$ be the unique prime above $\mathfrak{p}$.

Let $m(\mathfrak{p})$ and  $n(\mathfrak{p})$ be the $\mathfrak{p}$-adic valuations of the different $\mathfrak{d}_{F_{\mathfrak{p}}/\Q_p}$ 
and relative discriminant $D_{E_{\mathfrak{q}}/F_{\mathfrak{p}}} = \mathrm{Nm}_{E_{\mathfrak{q}}/F_{\mathfrak{p}}}(\mathfrak{d}_{E_{\mathfrak{q}}/F_{\mathfrak{p}}})$, respectively.  The integer $n(\mathfrak{p})$ is non-zero if and only if $\mathfrak{q}$ is ramified over $F$. 
Set $f(\mathfrak{p}) = m(\mathfrak{p}) + n(\mathfrak{p})$; this is the $\mathfrak{p}$-adic valuation of $\mathfrak{d}_{F_{\mathfrak{p}}/\Q_p}D_{E_{\mathfrak{q}}/F_{\mathfrak{p}}}$. 
% We will write $r(\mathfrak{p})\in\{0,1\}$ for the parity of $n(\mathfrak{p})$.

Let $e(\mathfrak{p})$ be the absolute ramification index of $\mathfrak{p}$. If $p\neq 2$, then the only possible non-zero value for $n(\mathfrak{p})$ is $1$. If $p=2$, then $n(\mathfrak{p})$ belongs to the set $\{2e(\mathfrak{p})+1\}\cup\{2i:\,0\leq i\leq e(\mathfrak{p})\}$.  

% Set:
% \begin{equation}
% \label{eqn:jq}
% j(\mathfrak{p}) =
% \left\lfloor\frac{n(\mathfrak{p})}{2} \right\rfloor.
% \end{equation}

Since $p$ is good, the quadratic space $L_{\mathfrak{p}} = L_p \cap V_{\mathfrak{p}}$ contains a maximal $\co_{E,\mathfrak{{q}}}$-stable lattice $\Lambda_{\mathfrak{p}}$. Moreover, if $\mathfrak{p}$ is unramified in $E$, then this lattice is itself self-dual and in particular is equal to $L_{\mathfrak{p}}$.

Fix a uniformizer $\pi_{\mathfrak{p}}\in\co_{F,\mathfrak{p}}$. If $\mathfrak{p}$ is unramified in $E$, we will also write $\pi_{\mathfrak{q}}$ for this element, when we view it as a uniformizer for $E_{\mathfrak{q}}$. If $\mathfrak{p}$ is ramified in $E$, we assume that $\pi_\mathfrak{p}$ has the form
$\mathrm{Nm}(\pi_{\mathfrak{q}}) = \pi_{\mathfrak{p}}$ for a uniformizer   $\pi_{\mathfrak{q}}\in E_{\mathfrak{q}}$.  
Here  $\mathrm{Nm}$ is  the norm from $E_{\mathfrak{q}}$ to $F_{\mathfrak{p}}$.

We will now explicitly describe the possibilities for $\Lambda_{\mathfrak{p}}$.

\begin{itemize}
	\item 
	If $\mathfrak{p}$ is inert in $E$, then the self-dual quadratic form on $L_{\mathfrak{p}}$ is the trace of an $E_{\mathfrak{q}}$-valued Hermitian form. In this case, $L_{\mathfrak{p}} = \Lambda_{\mathfrak{p}}$, and we have an isometry of Hermitian lattices:
	\begin{equation*}%\label{unr local hermitian +}
     (L_{\mathfrak{p}},\langle x_1,x_2\rangle) \simeq (\co_{E,\mathfrak{q}},\pi_{\mathfrak{p}}^{-m(\mathfrak{p})}x_1\overline{x}_2).
	\end{equation*}
	The \emph{nearby} Hermitian module $\near L_{\mathfrak{p}} = \near \Lambda_{\mathfrak{p}}$ is defined by
	\begin{equation*}%\label{unr local hermitian -}
 ( \near L_{\mathfrak{p}} , \near  \langle x_1 , x_2 \rangle ) = ( \co_{E,\mathfrak{q}} ,   \pi_{\mathfrak{p}}^{-m(\mathfrak{p})+1} x_1\overline{x_2} ).
  \end{equation*}
(In other words, the underlying $\co_{E,\mathfrak{q}}$-module is the same, but the hermitian form is rescaled by $\pi_{\mathfrak{p}}$.)

	\item
	 If $\mathfrak{p}$ is ramified in $E$, with $\mathfrak{q}\subset\co_E$ the prime above it, then, for an appropriate choice of unit $\beta_+\in\co_{F,\mathfrak{p}}^\times$, we have an isometry of Hermitian lattices:
	\begin{equation*}%\label{ram local r0 hermitian alpha}
     (\Lambda_{\mathfrak{p}},\langle x_1,x_2\rangle) \simeq (\co_{E,\mathfrak{q}},\beta_+\pi_{\mathfrak{p}}^{-m(\mathfrak{p})}x_1\overline{x}_2).
	\end{equation*}
	The \emph{nearby} Hermitian module $\near \Lambda_{\mathfrak{p}}$ is defined by
	\begin{equation*}%\label{ram local r0 hermitian beta}
 ( \near \Lambda_{\mathfrak{p}} , \near  \langle x_1 , x_2 \rangle ) = ( \co_{E,\mathfrak{q}} ,  \beta_{-}\pi_{\mathfrak{p}}^{-m(\mathfrak{p})} x_1\overline{x}_2 ),
  \end{equation*}
  where  $\beta_{-} = \delta\beta_{+}$, and $\delta\in 1+\pi_{\mathfrak{p}}^{n(\mathfrak{p})-1}\co_{F,\mathfrak{p}}$\footnote{If $n(\mathfrak{p}) = 1$, then we set $1+\pi_{\mathfrak{p}}^{n(\mathfrak{p})-1}\co_{F,\mathfrak{p}} = \co_{F,\mathfrak{p}}^\times$.}    is such that $\chi(\delta) = -1$.
(In other words, the underlying $\co_{E,\mathfrak{q}}$-module is the same, but the hermitian form is rescaled by $\delta$.)

	% Write $\near L_{\mathfrak{p}}\subset \near \Lambda_{\mathfrak{p}}^\vee$ for the image of $L_{\mathfrak{p}}$ under the obvious identifications
	% \[
 %     \Lambda_{\mathfrak{p}}^\vee = \pi_{\mathfrak{q}}^{-1}\Lambda_{\mathfrak{p}} = \pi_{\mathfrak{q}}^{-1} (\near \Lambda_{\mathfrak{p}}) = \near \Lambda_{\mathfrak{p}}^\vee.
	% \]
\end{itemize}

Let $\near \mathscr{V}$ be the nearby Hermitian space as in Definition~\ref{def:nearby}. Then, by construction, the nearby lattice $\near \Lambda_{\mathfrak{p}}$ is a lattice in $\near \mathscr{V}_{\mathfrak{p}}$. Moreover, again by construction, we have an identification of $\co_{E,\mathfrak{q}}$-modules (though not an isometry) 
\begin{equation}\label{eqn:near ident}
\Lambda_{\mathfrak{p}} = \near \Lambda_{\mathfrak{p}}.
\end{equation}

Fix a coset
\[
\lambda + \Lambda_{\mathfrak{p}} \subset \pi_{\mathfrak{q}}^{-n(\mathfrak{p})}\Lambda_{\mathfrak{p}}
\]
of $\Lambda_{\mathfrak{p}}$, and let $\near \lambda+ \near\Lambda_{\mathfrak{p}}$ be the associated coset of $\near\Lambda_{\mathfrak{p}}$ obtained from the identification~\eqref{eqn:near ident}. 

Let $\nearp \varphi_\lambda \in S(\nearp\mathscr{V}_{\mathfrak{p}})$ be the characteristic function of $\nearp \lambda + \nearp \Lambda_{\mathfrak{p}}$. Here, and in the sequel, we will use the superscript $\nearp$ to indifferently denote objects related to both $\mathscr{V}$ and $\near \mathscr{V}$;
\emph{e.g.}, $S(\nearp\mathscr{V}_{\mathfrak{p}})$  means either   
$S( \mathscr{V}_{\mathfrak{p}})$ or   $S(\near \mathscr{V}_{\mathfrak{p}})$.

Write $\nearp \Phi^\lambda_{\mathfrak{p}} \in I_{\mathfrak{{p}}}(s,\chi)$ for the standard section associated with $\nearp \varphi_\lambda$ as in \S\ref{ss:incoherent}, with corresponding Whittaker function
\[
W_{\alpha,\mathfrak{p}}(I,s,\nearp \Phi^\lambda_{\mathfrak{p}}) = \int_{F_{\mathfrak{p}}}  \nearp\Phi^\lambda_{\mathfrak{p}} ( w n(b), s)\cdot \psi_{F_{\mathfrak{p}}}( -    \alpha b )\, db.
\]

Let $I\in \mathrm{SL}_2(F_{\mathfrak{p}})$ be the identity. For convenience, set
 \[
W^*_{\alpha,\mathfrak{p}}(I,s,\nearp \Phi^\lambda_{\mathfrak{p}}) = \frac{\gamma_{\mathfrak{p}}(\nearp \mathscr{V})}{N(\mathfrak{p})^{f(\mathfrak{p})/2}}\cdot W_{\alpha,\mathfrak{p}}(I,s,\nearp \Phi^\lambda_{\mathfrak{p}}).
 \]
Here, $\gamma_{\mathfrak{p}}(\nearp \mathscr{V})$ is defined by~\eqref{eqn:gamma p defn}.

The next result follows from~\cite[Proposition 1.4]{KuAnnals}.
\begin{proposition}
\label{prop:whittaker vanishing}
Suppose that $\alpha\in F^\times_{\mathfrak{p}}$ is not represented by $\nearp \mathscr{V}_{\mathfrak{p}}$. Then
\[
W_{ \alpha, \mathfrak{p}} (g_{\mathfrak{p}} , 0, \nearp \Phi^\lambda_{\mathfrak{p}}) = 0.
\]
\end{proposition}

Set
 \[
\xi_{\mathfrak{p}} = \begin{cases}
\pi_{\mathfrak{p}}^{-m(\mathfrak{p})} &\text{if $\mathfrak{p}$ is unramified in $E$} \\
\beta_+\pi_{\mathfrak{p}}^{-m(\mathfrak{p})} &\text{if $\mathfrak{p}$ is ramified in $E$},
\end{cases}
 \]
 and
\begin{align}\label{eqn:near xi}
 \near \xi_{\mathfrak{p}} = \begin{cases}
 	\pi_{\mathfrak{p}}^{-m(\mathfrak{p})+1} &\text{if $\mathfrak{p}$ is unramified in $E$} \\
\beta_-\pi_{\mathfrak{p}}^{-m(\mathfrak{p})} &\text{if $\mathfrak{p}$ is ramified in $E$}.
 \end{cases}
\end{align}

The proofs of the two propositions below are essentially contained in~\cite[\S 4.6]{HY} and~\cite{Yang}. 
In particular, see \cite[Propositions 2.1, 2.2, and 2.3]{Yang}. 

% There will however be a detailed calculation of Whittaker functions associated with $\lambda\neq 0$ in Proposition~\ref{prop:whittaker Phi lambda} later in this subsection. This might be helpful to the reader attempting to reconcile the statements given here with those in the cited references.

 \begin{proposition}
\label{prop:whittaker char function unram}
Suppose that $\mathfrak{p}$ is unramified in $E$. 
\begin{enumerate}

\item If $\ord_{\mathfrak{p}}(\alpha) < -m(\mathfrak{p}) $, then
\[
W_{\alpha,\mathfrak{p}}(I,s,\nearp \Phi^0_{\mathfrak{p}}) = 0.
\]
\item If  $\ord_{\mathfrak{p}}(\alpha) \geq -m(\mathfrak{p})$, then
	\[
    W^*_{\alpha,\mathfrak{p}}(I,s,\Phi^0_{\mathfrak{p}}) = \frac{1}{L_{\mathfrak{p}}(s+1,\chi)}\sum_{0\leq k\leq \ord_{\mathfrak{p}}(\alpha)+m(\mathfrak{p})}(-1)^kN(\mathfrak{p})^{-ks},
	\]
and
\begin{align*}
W_{\alpha,\mathfrak{p}}(I,s,\near \Phi^0_{\mathfrak{p}}) &=  W_{\alpha,\mathfrak{p}}(I,s,\Phi^0_{\mathfrak{p}}) - (1 + N(\mathfrak{p})^{-1}).
\end{align*}
% Therefore, 
% \begin{align*}
% W_{\alpha,\mathfrak{p}}(I,0,\Phi_0) &= \frac{1+(-1)^{\ord_{\mathfrak{p}}(\alpha)}}{2\cdot L_{\mathfrak{p}}(1,\chi)};\\
% W'_{\alpha,\mathfrak{p}}(I,0,\Phi_0) &= \log(N(\mathfrak{p}))\cdot\frac{1}{2\cdot L_{\mathfrak{p}}(1,\chi)}\cdot \left(1-(-1)^{\ord_{\mathfrak{p}}(\alpha)}\ord_{\mathfrak{p}}(\alpha)\right).
% \end{align*}
\end{enumerate}
\end{proposition}

\begin{proposition}
\label{prop:whittaker char function ram}
Suppose that $\mathfrak{p}$ is ramified in $E$. 
\begin{enumerate}
\item If $\ord_{\mathfrak{p}}(\alpha) < -m(\mathfrak{p})$, then
\[
W_{\alpha,\mathfrak{p}}(I,s,\nearp \Phi^0_{\mathfrak{p}}) = 0.
\]
\item If  $\ord_{\mathfrak{p}}(\alpha) \geq -m(\mathfrak{p})$, then
	\[
    W^*_{\alpha,\mathfrak{p}}(I,s,\nearp \Phi^0_{\mathfrak{p}}) = 1+\chi_{\mathfrak{p}}(\nearp\xi_{\mathfrak{p}}\alpha)N(\mathfrak{p})^{-(\ord_{\mathfrak{p}}(\alpha)+m(\mathfrak{p}) + n(\mathfrak{p}))s}.
	\]
% Therefore, 
% \begin{align*}
% W_{\alpha,\mathfrak{p}}(I,0,\Phi_0) &= \frac{\gamma_{\mathfrak{p}}(\mathscr{V}_{0,\mathfrak{p}})}{N(\mathfrak{p})^{n(\mathfrak{p})/2}}\cdot\left(1+\chi(\alpha)\right);\\
% W'_{\alpha,\mathfrak{p}}(I,0,\Phi_0) &= - \log(N(\mathfrak{p}))\cdot\chi(\alpha)\cdot\frac{\gamma_{\mathfrak{p}}(\mathscr{V}_{0,\mathfrak{p}})}{N(\mathfrak{p})^{n(\mathfrak{p})/2}}\cdot\left(n(\mathfrak{p}) + \ord_{\mathfrak{p}}(\alpha)\right).
% \end{align*}
\end{enumerate}
\end{proposition}

Now, suppose that $\mathfrak{p}$ is ramified in $E$. As above, let $\pi_{\mathfrak{q}}\in E_{\mathfrak{q}}$ be a uniformizer, chosen so that $\mathrm{Nm}(\pi_{\mathfrak{q}}) = \pi_{\mathfrak{p}}$. 

For any $a_1,a_2,\zeta\in F^\times$, write $a_1\equiv a_2\pmod{\zeta}$ to mean $a_1\equiv a_2\pmod{\zeta\co_{F,\mathfrak{p}}}$.
\begin{proposition}
\label{prop:whittaker Phi lambda}
Suppose that $\lambda\notin \Lambda_{\mathfrak{p}}$. 
\begin{enumerate}
\item If $p\neq 2$, then
\begin{align*}
W^*_{\alpha,\mathfrak{p}}(I,s,\nearp\Phi^\lambda_{\mathfrak{p}}) &= 
\begin{cases}
1 &\text{if $\alpha\equiv \nearp\mathscr{Q}(\nearp \lambda)\pmod{\nearp \xi_{\mathfrak{p}}}$} \\
0 &\text{otherwise}.
\end{cases}
\end{align*} 
\item Suppose that $p=2$. Then 
\[
\alpha\equiv\mathscr{Q}(\lambda)\pmod{\xi_{\mathfrak{p}}}\;\Leftrightarrow\;\alpha\equiv \near\mathscr{Q}(\near \lambda)\pmod{\near \xi_{\mathfrak{p}}}.
\]
Moreover, $W_{\alpha,\mathfrak{p}}(I,s,\nearp \Phi^\lambda_{\mathfrak{p}})$ is identically $0$ unless these equivalent congruences hold, and when they hold we have
	\begin{align*}
    W^*_{\alpha,\mathfrak{p}}(I,s,\nearp \Phi^\lambda_{\mathfrak{p}}) & =
      1 + \chi_{\mathfrak{p}}(\nearp\xi_{\mathfrak{p}} \alpha)N(\mathfrak{p})^{-(n(\mathfrak{p})- r(\lambda))s},
   \end{align*}
  where $r(\lambda)\in\Z_{>0}$ is the smallest positive integer such that $\lambda\in \pi_{\mathfrak{q}}^{-r(\lambda)}\Lambda_{\mathfrak{p}}$.
\end{enumerate}
% Therefore, 
% \begin{align*}
% W_{\alpha,\mathfrak{p}}(I,0,\Phi^\lambda_0) &= \frac{\gamma_{\mathfrak{p}}(\mathscr{V}_{0,\mathfrak{p}})}{N(\mathfrak{p})^{n(\mathfrak{p})/2}};\\
% W'_{\alpha,\mathfrak{p}}(I,0,\Phi^\lambda_0) &= \log(N(\mathfrak{p}))\cdot \frac{\gamma_{\mathfrak{p}}(\mathscr{V}_{0,\mathfrak{p}})}{N(\mathfrak{p})^{n(\mathfrak{p})/2}}.
% \end{align*} 

\end{proposition}

\begin{proof}
When $p\neq 2$, this computation is contained in \cite[Proposition 4.6.4]{HY}. When $p=2$, the result appears to be new. We present a mostly self-contained proof here that covers both possibilities. 

For simplicity, write $\Phi$, $\chi$, $\psi$ and $\xi$ for $\nearp\Phi^\lambda_{\mathfrak{p}}$, $\chi_{\mathfrak{p}}$, $\psi_{F_{\mathfrak{p}}}$ and $\xi_{\mathfrak{p}}$, respectively. By a standard argument, we have a decomposition:
\[
W_{\alpha,\mathfrak{p}}(I,s,\Phi) = W_{\alpha,\mathfrak{p}}(I,s,\Phi)^{\leq 1} + W_{\alpha,\mathfrak{p}}(I,s,\Phi)^{>1},
\]
where
\begin{align*}
W_{\alpha,\mathfrak{p}}(I,s,\Phi)^{\leq 1} & = \int_{\vert b\vert\leq 1}\Phi(wn(b))\psi(-\alpha b)db;\\
W_{\alpha,\mathfrak{p}}(I,s,\Phi)^{> 1} &=  \int_{\vert b\vert > 1}\chi(b)\vert b\vert^{-(s+1)}\Phi(n_{-}(b^{-1}))\psi(-\alpha b)db.
\end{align*}
Here, $n_{-}(b^{-1}) = \left( \begin{matrix}     1 &  0  \\  b^{-1}  & 1        \end{matrix} \right)$, and we have abbreviated $\Phi(g,0)$ to $\Phi(g)$.

By the definition of $\Phi$, and basic properties of the Weil representation, for any $b\in\co_{F,\mathfrak{p}}$, we have
\begin{align*}
\Phi(wn(b)) &= \gamma_{\mathfrak{p}}(\nearp \mathscr{V})\int_{\nearp \lambda + \nearp \Lambda_{\mathfrak{p}}}\psi(b\cdot\nearp\mathscr{Q}(x))dx\\
& = \gamma_{\mathfrak{p}}(\nearp \mathscr{V})\cdot\psi(b\cdot\nearp \mathscr{Q}(\nearp \lambda))\cdot\int_{\nearp\Lambda_{\mathfrak{p}}}\psi(b\cdot\nearp	\mathscr{Q}(x) + b\cdot\nearp \langle \nearp \lambda , x\rangle)\; dx\\
& = \gamma_{\mathfrak{p}}(\nearp \mathscr{V})\cdot\psi(b\cdot\nearp \mathscr{Q}(\nearp \lambda))\cdot\int_{\nearp\Lambda_{\mathfrak{p}}}\psi(b\cdot\nearp	\mathscr{Q}(x))\; dx.
\end{align*}
Here, $dx$ is the Haar measure on $\nearp\mathscr{V}_{\mathfrak{p}}$ that is self-dual with respect to the pairing 
\[
(x_1,x_2)\mapsto \psi\bigl(\mathrm{Tr}_{E_{\mathfrak{q}}/F_{\mathfrak{p}}}(\nearp \langle x_1,x_2\rangle)\bigr).
\]
We have also used the fact that, for any $x\in \Lambda_{\mathfrak{p}}$, $\nearp \langle \nearp\lambda, x\rangle$ belongs to $\mathfrak{d}^{-1}_{F_{\mathfrak{p}}/\Q_p}$, and hence
\[
\psi(b\cdot\nearp \langle \nearp \lambda , x\rangle) = 1.
\]

Set $s_\lambda = \nearp \mathscr{Q}(\nearp \lambda)$.
% , and let $u_{\lambda}\in \co_{F,\mathfrak{p}}^\times$ be such that
% \[
% s_\lambda = \pi_{\mathfrak{p}}^{-1}u_{\lambda}\cdot\xi.
% \]
Using \cite[Lemma 4.6.1]{HY}, we then obtain
\begin{align*}
\Phi(wn(b)) 
&= N(\mathfrak{p})^{-f(\mathfrak{p})/2}\gamma_{\mathfrak{p}}(\nearp \mathscr{V})\psi(b s_\lambda)
\int_{\pi_{\mathfrak{p}}^{-m(\mathfrak{p})}\co_{F,\mathfrak{p}}}(1+\chi(\xi y))\psi(b y)dy     \\
& = N(\mathfrak{p})^{-f(\mathfrak{p})/2}\gamma_{\mathfrak{p}}(\nearp \mathscr{V})\psi(bs_\lambda)  \\
&\quad \times  \biggl[N(\mathfrak{p})^{m(\mathfrak{p})/2}  +  
\sum_{k=-m(\mathfrak{p})}^\infty N(\mathfrak{p})^{-k}\int_{\co_{F,\mathfrak{p}}^\times}\chi(\xi y)\psi(\pi_{\mathfrak{p}}^kby)dy \biggr]  \\
& = N(\mathfrak{p})^{-n(\mathfrak{p})/2}\gamma_{\mathfrak{p}}(\nearp \mathscr{V})\psi(bs_\lambda).
\end{align*}
The second equality here is deduced by noting
\begin{equation}\label{eqn:psi integral}
\int_{\pi_{\mathfrak{p}}^{-r}\co_{F,\mathfrak{p}}}\psi(\zeta y) dy = \begin{cases}
N(\mathfrak{p})^{-r}\mathrm{Vol}(\co_{F,\mathfrak{p}}) = N(\mathfrak{p})^{r-m(\mathfrak{p})/2}
&\text{if $\ord_{\mathfrak{p}}(\zeta)\geq r - m(\mathfrak{p})$} \\
0 &\text{otherwise},
\end{cases}
\end{equation}
and the last by using the following lemma, which is a standard Gauss sum computation, using the fact that $\chi$ has conductor $n(\mathfrak{p})$. 
\begin{lemma}\label{lem:gauss sum}
For $\alpha\in F_{\mathfrak{p}}$,
\[
\int_{\co_{F,\mathfrak{p}}^\times}\chi(y)\psi(\alpha y)dy = \begin{cases}
N(\mathfrak{p})^{-f(\mathfrak{p})/2}\cdot\chi(\alpha)\cdot\epsilon_{\mathfrak{p}}(\chi,\psi)
&\text{ if $\ord_{\mathfrak{p}}(\alpha) = -f(\mathfrak{p})$} \\
0 &\text{ otherwise.}
\end{cases}
\]
\end{lemma}

Therefore, we have
\begin{align}\label{eqn:whittaker ord b +}
W_{\alpha,\mathfrak{p}}(I,s,\Phi)^{\leq 1} & = N(\mathfrak{p})^{-n(\mathfrak{p})/2}\gamma_{\mathfrak{p}}(\nearp \mathscr{V})\cdot\int_{\co_{F,\mathfrak{p}}}\psi((s_\lambda-\alpha) b)db\nonumber\\
& = \begin{cases}
N(\mathfrak{p})^{-f(\mathfrak{p})/2}\cdot \gamma_{\mathfrak{p}}(\nearp \mathscr{V}) &\text{if $\alpha\equiv s_\lambda\pmod{\xi}$} \\
0 &\text{otherwise.}
\end{cases}
\end{align}

To compute $W_{\alpha,\mathfrak{p}}(I,s,\Phi)^{>1}$, we will need
\begin{lemma}
\label{lem:phi n- c}
Suppose that $c\in\co_{F,\mathfrak{p}}$ and that $k=\ord_{\mathfrak{p}}(c) > 1$. For any integer $t\in \Z_{\geq 1}$, set
\[
U^{t}_{F,\mathfrak{p}} = 1 + \pi_{\mathfrak{p}}^{t} \co_{F,\mathfrak{p}}.
\]
Set
\[
d(k,\lambda) = 2k - (n(\mathfrak{p}) - r(\lambda)).
\]

Then $\Phi(n_{-}(c)) \neq 0$ only if
\[
\frac{n(\mathfrak{p}) - r(\lambda)}{2} < k < n(\mathfrak{p}) - \frac{r(\lambda)}{2}.
\]
In this case, we have
\[
\Phi(n_{-}(c)) = 
\frac{\psi( c^{-1} s_\lambda) }{\mathrm{Vol}(U^{d(k,\lambda)}_{F,\mathfrak{p}})} \cdot \int_{U^{d(k,\lambda)}_{F,\mathfrak{p}}}  \chi( y)  \psi(- c^{-1} s_\lambda y)dy.
\]
In particular, if $p\neq 2$, then $W_{\alpha,\mathfrak{p}}(I,s,\Phi)^{>1} = 0$. 
\end{lemma}
\begin{proof}
As in~\cite{HY} and~\cite{Yang}, this uses the identity $n_{-}(c) = -wn(-c)w$, so that
\begin{align*}\label{eqn:phi n- c}
\Phi(n_{-}(c)) &= \chi(-1)\Phi(wn(-c)w)\\
&= \chi(-1)  \cdot \gamma_{\mathfrak{p}} ( \nearp\mathscr{V}) \int_{\nearp\mathscr{V}_{\mathfrak{p}}}\psi(-c\cdot \nearp\mathscr{Q}(x)) \omega(w) (\nearp\varphi_\lambda) (x) dx\\
& =  \int_{\nearp\mathscr{V}_{\mathfrak{p}}}\psi(-c\cdot \nearp\mathscr{Q}(x)) \int_{\nearp \lambda + \nearp \Lambda_{\mathfrak{p}}}\psi(-\mathrm{Tr}(\nearp \langle x,y\rangle)) dy\; dx.
\end{align*}
Here, we have used the identity $ \gamma_{\mathfrak{p}} ( \nearp\mathscr{V}) ^ 2 = \epsilon_{\mathfrak{p}}(\chi,\psi)^2 = \chi(-1)$.

For $x\in \nearp \mathscr{V}_{\mathfrak{p}}$, set $t_\lambda(x) = \mathrm{Tr}( \nearp \langle \nearp\lambda, x\rangle)$. We compute
\begin{align*}
\int_{\nearp \lambda + \nearp \Lambda_{\mathfrak{p}}}\psi(-\mathrm{Tr}(\nearp \langle x,y\rangle)) dy & =  \psi( - t_\lambda(x) ) \int_{\Lambda_{\mathfrak{p}}}  \psi ( - \mathrm{Tr}(  \nearp \langle x, y\rangle) ) dy\\
& = N(\mathfrak{p})^{-n(\mathfrak{p})/2} \cdot  \psi( - t_\lambda(x) ) \cdot \mathrm{char}( \pi_{\mathfrak{q}}^{-n(\mathfrak{p})} \Lambda_{\mathfrak{p}} ) ( x ),
\end{align*}
and hence
\begin{align*}
\Phi(n_{-}(c)) & = N(\mathfrak{p})^{-n(\mathfrak{p})/2} \int_{\pi_{\mathfrak{q}}^{-n(\mathfrak{p})} \Lambda_{\mathfrak{p}}} \psi(- c\cdot \nearp \mathscr{Q}(x) - t_\lambda(x)) \; dx.
\end{align*}

If $ k \geq n(\mathfrak{p})$  then
\[
\ord_{\mathfrak{p}}( c \cdot \nearp \mathscr{Q}(x)) \geq k - n(\mathfrak{p}) - m(\mathfrak{p}) \geq -m(\mathfrak{p})
\]
for all $x\in \pi_{\mathfrak{q}}^{-n(\mathfrak{p})} \Lambda_{\mathfrak{p}}$.
Therefore, under this assumption, we have
\begin{align*}
\Phi(n_{-}(c))  = N(\mathfrak{p})^{-n(\mathfrak{p})/2} \int_{\pi_{\mathfrak{q}}^{-n(\mathfrak{p})} \Lambda_{\mathfrak{p}}} \psi(- t_\lambda(x)) \; dx = 0,
\end{align*}
where we have used~\eqref{eqn:psi integral}.

Now, suppose that $k<n(\mathfrak{p})$. Note that
\[
-c\cdot \nearp \mathscr{Q}(x) - t_\lambda(x) = -c\cdot \nearp \mathscr{Q}( x  + c^{-1}\cdot\nearp \lambda ) + c^{-1} s_\lambda.
\]
Therefore,
\begin{align*}
\Phi(n_{-}(c)) & = N(\mathfrak{p})^{-n(\mathfrak{p})/2} \psi( c^{-1} s_\lambda)\cdot \int_{c^{-1}\cdot\nearp \lambda + \pi_{\mathfrak{q}}^{-n(\mathfrak{p})} \Lambda_{\mathfrak{p}}} \psi(- c\cdot \nearp \mathscr{Q}( x )  )\; dx.
\end{align*}

Using Lemma 4.6.1 of~\cite{HY}, we find
\begin{align*}
\Phi (n_{-}(c)) & = N(\mathfrak{p})^{-(f(\mathfrak{p})+n(\mathfrak{p}))/2} \psi( c^{-1} s_\lambda) \cdot \int_{c^{-2}s_\lambda + \pi_{\mathfrak{p}}^{-f(\mathfrak{p})}\co_{F,\mathfrak{p}}} (1 + \chi(\xi y)) \psi( - c y)\; dy\\
& =  N(\mathfrak{p})^{-(f(\mathfrak{p})+n(\mathfrak{p}))/2} \psi( c^{-1} s_\lambda) \cdot \int_{c^{-2}s_\lambda + \pi_{\mathfrak{p}}^{-f(\mathfrak{p})}\co_{F,\mathfrak{p}}} \chi(\xi y) \psi( - c y)\; dy.
\end{align*}
Here, for the last identity, we have also used~\eqref{eqn:psi integral} combined with the inequality
\[
\ord_{\mathfrak{p}} (c) = k < n(\mathfrak{p}) =  f(\mathfrak{p}) - m(\mathfrak{p}).
\]

If $2k > n(\mathfrak{p}) - r(\lambda)$, then
\[
\ord_{\mathfrak{p}}(c^{-2}s_\lambda) = -2k  - m(\mathfrak{p}) - r(\lambda) < -f(\mathfrak{p}).
\]
Therefore, the substitution $y \mapsto (c^{-2}s_\lambda)^{-1}y$, combined with the observation that $\chi(s_\lambda) = \chi(\xi)$ gives us
\[
\Phi(n_{-}(c)) = \frac{\psi( c^{-1} s_\lambda) }{\mathrm{Vol}(U^{d(k,\lambda)}_{F,\mathfrak{p}})} \cdot \int_{U^{d(k,\lambda)}_{F,\mathfrak{p}}}  \chi( y)  \psi(- c^{-1} s_\lambda y)dy.
\]
If $k \geq n(\mathfrak{p}) - \frac{r(\lambda)}{2}$, then $d(k,\lambda) \geq n(\mathfrak{p})$. Since $\chi$ has conductor $n(\mathfrak{p})$, in this case we get
\[
\Phi(n_{-}(c)) = \frac{\psi( c^{-1} s_\lambda) }{\mathrm{Vol}(U^{d(k,\lambda)}_{F,\mathfrak{p}})} \cdot \int_{U^{d(k,\lambda)}_{F,\mathfrak{p}}} \psi(- c^{-1} s_\lambda y)dy,
\]
which vanishes by~\eqref{eqn:psi integral}.

If $2k \leq n(\mathfrak{p}) - r(\lambda)$, then we have
\[
 \int_{c^{-2}s_\lambda + \pi_{\mathfrak{p}}^{-f(\mathfrak{p})}\co_{F,\mathfrak{p}}} \chi(y) \psi( - c y)\; dy =  \int_{\pi_{\mathfrak{p}}^{-f(\mathfrak{p})}\co_{F,\mathfrak{p}}} \chi(y) \psi( - c y)\; dy.
\]
In this case, it is not hard to see, using Lemma~\ref{lem:gauss sum}, that this integral vanishes, and hence that $\Phi ( n_{-}(c) ) = 0$.

\end{proof}

When $p\neq 2$, this, combined with~\eqref{eqn:whittaker ord b +}, finishes the proof of Proposition~\ref{prop:whittaker Phi lambda}. Therefore, we now specialize to the case where $p=2$. In this case, we have
\[
\near\mathscr{Q}(\near\lambda) = \delta\mathscr{Q}(\lambda),
\]
where $\delta\in U^{n(\mathfrak{p})-1}_{F,\mathfrak{p}}$. From this, and the condition $r(\lambda) < n(\mathfrak{p})$, it follows easily that the conditions
\[
\alpha\equiv\mathscr{Q}(\lambda)\pmod{\xi},\quad\alpha\equiv \near\mathscr{Q}(\near \lambda)\pmod{\xi}
\]
are equivalent. This shows the first part of assertion (2) of the proposition.

For the second part, observe that Lemma~\ref{lem:phi n- c} gives us:
\begin{align*}
W_{\alpha,\mathfrak{p}} ( I , s , \Phi )^{ > 1} & = \sum_k N(\mathfrak{p})^{-k(s+1)} \int_{\ord_{\mathfrak{p}}(b) = -k} \chi(b) \Phi ( n_{-}(b^{-1})) \psi(-\alpha b)\; db.,
\end{align*}
where $\frac{n - r(\lambda)}{2} < k < n(\mathfrak{p}) + \frac{\ord_{\mathfrak{q}}(\lambda}{2}$, and the summand indexed by $k$ is equal to
\begin{align}\label{eqn:kth term}
&\frac{N(\mathfrak{p})^{-ks}}{\mathrm{Vol}(U^{d(k,\lambda)}_{F,\mathfrak{p}})}\int_{U^{d(k,\lambda)}_{F,\mathfrak{p}}}\chi(y)\int_{\co_{F,\mathfrak{p}}^\times}\chi(b)\psi(b\pi_{\mathfrak{p}}^{-k}(s_\lambda - \alpha -s_\lambda y)) db\; dy.
\end{align}.

Now, we have
\[
\ord_{\mathfrak{p}}(\pi_{\mathfrak{p}}^{-k}s_\lambda(1-y)) = k - f(\mathfrak{p}) > - f(\mathfrak{p}),
\]
and  therefore
\begin{align*}
\ord_{\mathfrak{p}}(\pi_{\mathfrak{p}}^{-k}(s_\lambda - \alpha - s_\lambda y)) &= -k + \ord_{\mathfrak{p}}(s_\lambda(1-y) - \alpha)
\end{align*}
can equal $-f(\mathfrak{p})$ if and only if $\ord_{\mathfrak{p}}(\alpha) = k - f(\mathfrak{p})$. 
So, using Lemma~\ref{lem:gauss sum}, we see that
\begin{align*}
\int_{\co_{F,\mathfrak{p}}^\times}\chi(b)\psi(b\pi_{\mathfrak{p}}^{-k}(s_\lambda - \alpha -s_\lambda y)) db 
=
N(\mathfrak{p})^{-f(\mathfrak{p})/2}\cdot \chi(-s_\lambda )\cdot \chi(y - (1 - s_\lambda^{-1}\alpha) )\cdot \epsilon_{\mathfrak{p}}(\chi,\psi)
\end{align*}
if $\ord_{\mathfrak{p}}(\alpha) = k - f(\mathfrak{p})$, and that it vanishes otherwise.
Therefore~\eqref{eqn:kth term} is non-zero only if $\mathrm{ord}_{\mathfrak{p}}(\alpha) = k - f(\mathfrak{p})$, in which case it is equal to
\begin{align*}
\frac{N(\mathfrak{p})^{-f(\mathfrak{p})/2}\cdot \gamma_{\mathfrak{p}}(\nearp \mathscr{V})\cdot N(\mathfrak{p})^{-ks}}{\mathrm{Vol}(U^{d(k,\lambda)}_{F,\mathfrak{p}})}\int_{U^{d(k,\lambda)}_{F,\mathfrak{p}}}\chi(y(y - (1-s_\lambda^{-1}\alpha)))\; dy.
\end{align*}
Here, we have also used the formula for $\gamma_{\mathfrak{p}}(\nearp\mathscr{V})$ from~\eqref{eqn:gamma p defn}, combined with the identity $\chi(s_\lambda) = \chi(\xi) = \mathrm{inv}_{\mathfrak{p}}(\nearp \mathscr{V})$.

Combining this with~\eqref{eqn:whittaker ord b +}, we obtain
\begin{align}
\label{eqn:whittaker prelim formula}
W^*_{\alpha,\mathfrak{p}}(I, s, \Phi) & = \begin{cases}
N(\mathfrak{p})^{-ks}\cdot M(\alpha,\lambda) &\text{if $\alpha\nequiv s_\lambda\pmod{\xi}$} \\
1 + N(\mathfrak{p})^{-ks}\cdot M(\alpha,\lambda) &\text{if $\alpha\equiv s_\lambda\pmod{\xi}$},
\end{cases}
\end{align}
where $k = \ord_{\mathfrak{p}}(\alpha) + f(\mathfrak{p})$, and where
\[
M(\alpha,\lambda) = \frac{1}{\mathrm{Vol}(U^{d(k,\lambda)}_{F,\mathfrak{p}})}\int_{U^{d(k,\lambda)}_{F,\mathfrak{p}}}\chi(y(y - (1-s_\lambda^{-1}\alpha)))\; dy.
\]

Now, if $\alpha\equiv s_\lambda\pmod{\xi}$, then $\ord_{\mathfrak{p}}(s_\lambda^{-1}\alpha) = 0$, and so $k =  n(\mathfrak{p}) - r(\lambda)$. Therefore, the proof of the proposition will be completed by the following lemma:

\begin{lemma}
\label{lem:M alpha lambda}
We have
\[
M(\alpha,\lambda) = \begin{cases}
1 &\text{if $\alpha\equiv s_\lambda\pmod{\xi}$ and $\nearp\mathscr{V}_{\mathfrak{p}}$ represents $\alpha$}\\
-1 &\text{if $\alpha\equiv s_\lambda\pmod{\xi}$ and $\nearp\mathscr{V}_{\mathfrak{p}}$ does not represent $\alpha$}\\
0 &\text{if $\alpha\nequiv s_\lambda\pmod{\xi}$}.
\end{cases}
\]
\end{lemma}

\begin{proof}
If $\alpha\equiv s_\lambda\pmod{\xi}$ and $\nearp \mathscr{V}_{\mathfrak{p}}$ represents $\alpha$, then we can choose our coset representative $\nearp \lambda$ so that
\[
s_\lambda = \nearp \mathscr{Q}(\nearp \lambda) = \alpha.
\]
Therefore, $s_\lambda^{-1}\alpha = 1$, and the formula for $M(\alpha,\lambda)$ reduces to
\[
M(\alpha,\lambda) = \frac{1}{\mathrm{Vol}(U^{d(k,\lambda)}_{F,\mathfrak{p}})}\int_{U^{d(k,\lambda)}_{F,\mathfrak{p}}}\chi(y^2)\; dy = 1.
\]

If $\alpha\equiv s_\lambda\pmod{\xi}$ is not represented by $\nearp \mathscr{V}_{\mathfrak{p}}$, then Proposition~\ref{prop:whittaker vanishing} shows that
\[
1 + M(\alpha,\lambda) = W^*_{\alpha,\mathfrak{p}}(I, 0 ,\Phi) = 0,
\]
and so $M(\alpha,\lambda) = -1$.

Now, suppose that $\alpha \nequiv s_\lambda \pmod{\xi}$. Set $\zeta = 1 - s_\lambda^{-1}\alpha$. We have
\begin{align}\label{eqn:M lambda simplified}
\int_{U^{d(k,\lambda)}_{F,\mathfrak{p}}}\chi(y ( y - \zeta )) \; dy& = \int_{U^{d(k,\lambda)}_{F,\mathfrak{p}}}\chi(1- y^{-1}\zeta )\; dy\nonumber\\
& = \int_{U^{d(k,\lambda)}_{F,\mathfrak{p}}}\chi(1- y\zeta )\; dy\nonumber\\
& = \chi(-\zeta) \int_{U^{d(k,\lambda)}_{F,\mathfrak{p}}}\chi(y - \zeta^{-1})\; dy.
\end{align}

Note that
\[
\ord_{\mathfrak{p}}(\zeta) =\begin{cases}
0 &\text{if $k> n(\mathfrak{p}) - r(\lambda)$} \\
k - (n(\mathfrak{p}) - r(\lambda)) &\text{if $k< n(\mathfrak{p}) - r(\lambda)$}.
\end{cases}
\]
Moreover, when $k = n(\mathfrak{p}) - r(\lambda)$, $\ord_{\mathfrak{p}}(\zeta)$ is an integer between $0$ and $r(\lambda) - 1$.
In particular, we find that we always have
\[
n(\mathfrak{p}) - 1 - \ord_{\mathfrak{p}}(\zeta) \geq d(k,\lambda).
\]

Choose $\eta \in U^{n(\mathfrak{p})-1}_{F,\mathfrak{p}}$ such that $\chi(\eta) = -1$. This choice determines  a measure preserving bijection
\[
 \alpha:U^{d(k,\lambda)}_{F,\mathfrak{p}} \xrightarrow{\simeq} U^{d(k,\lambda)}_{F,\mathfrak{p}}
\]
by $ y \mapsto y + (1 - \eta)\zeta^{-1}.$
We now compute
\begin{align*}
\int_{U^{d(k,\lambda)}_{F,\mathfrak{p}}} \chi ( y - \zeta^{-1} ) \; dy & = \int_{U^{d(k,\lambda)}_{F,\mathfrak{p}}} \chi( \alpha(y) - \zeta^{-1} ) \; dy\\
& = \int_{U^{d(k,\lambda)}_{F,\mathfrak{p}}} \chi( y - \eta \zeta^{-1}) \; dy\\
& = - \int_{U^{d(k,\lambda)}_{F,\mathfrak{p}}} \chi( \eta^{-1}y - \zeta^{-1}) \; dy\\
& = - \int_{U^{d(k,\lambda)}_{F,\mathfrak{p}}} \chi( y -  \zeta^{-1}) \; dy.
\end{align*}
Combining this with~\eqref{eqn:M lambda simplified} shows that $M(\alpha,\lambda) = 0$.
\end{proof}
This completes the proof of Proposition \ref{prop:whittaker Phi lambda}.
\end{proof}

We now record a few more results that are easy consequences of Propositions~\ref{prop:whittaker char function unram},~\ref{prop:whittaker char function ram} and~\ref{prop:whittaker Phi lambda}. We omit their proofs.

\begin{proposition}
\label{prop:whittaker alpha 0}
We have 
\[
W_{ 0, \mathfrak{p}} (I , s, \Phi^\lambda_{\mathfrak{p}}) = \begin{cases}
\frac{\gamma_{\mathfrak{p}}(\mathscr{V})}{N(\mathfrak{p})^{f(\mathfrak{p})/2}}\cdot\frac{L_{\mathfrak{p}}(s,\chi)}{L_{\mathfrak{p}}(s+1,\chi)}&\text{if $\lambda \in \Lambda_{\mathfrak{p}}$} \\
0&\text{if $\lambda\notin \Lambda_{\mathfrak{p}}$.}
\end{cases}
\]
\end{proposition}

\begin{corollary}\label{cor:good prime M constant}
Let $M_{\mathfrak{p}}(s,\varphi_{\lambda})$ be as in~\eqref{local M}. Then $M_\mathfrak{p}(s,\varphi_\lambda)$ is constant. In fact, it is either $1$ or $0$ depending on whether $\lambda$ is zero or non-zero.
\end{corollary}

% Fix a representative $\near \mu_{\mathfrak{p}}\in \near \mu_{\mathfrak{p}} + \near L_{\mathfrak{p}}$. Write $c_{\mathfrak{p}}(\alpha,\mu)$ for the class of $\alpha - \near \langle \near \mu_{\mathfrak{p}}, \near \mu_{\mathfrak{p}}\rangle\in F_{\mathfrak{p}}$ modulo $\pi_{\mathfrak{p}}^{-(m({\mathfrak{p}}) + j(\mathfrak{p}))}\co_{F,\mathfrak{p}}$. This does not depend on the choice of the representative $\near \mu_{\mathfrak{p}}$.

% \begin{proposition}
% \label{prop:whittaker nearby value}
% Suppose that $\alpha\in F\smallsetminus\{0\}$ is such that $\mathfrak{p}\in \mathrm{Diff}(\alpha)$. Then $W_{ \alpha, \mathfrak{p}} (I , 0, \near \Phi^\mu_{\mathfrak{p}}) = 0$ unless $c_{\mathfrak{p}}(\alpha,\mu)=0$. Suppose that $c_{\mathfrak{p}}(\alpha,\mu) = 0$. Then:
% \[
%  W_{ \alpha, \mathfrak{p}} (I , 0, \near \Phi^\mu_{\mathfrak{p}}) = \frac{\gamma_{\mathfrak{p}}(\mathscr{V})}{N(\mathfrak{p})^{f(\mathfrak{p})/2}}\cdot\frac{e_{\mathfrak{p}}(E/F)}{L_{\mathfrak{p}}(1,\chi)}
% \]
% In particular, 
% \[
% W_{ \alpha, \mathfrak{p}} (I , 0, \near \Phi^\mu_{\mathfrak{p}})\neq 0.
% \]
% Here, $e_{\mathfrak{p}}(E/F)$ is the ramification index of $\mathfrak{q}$ over $F$.
% \end{proposition}

\begin{proposition}
\label{prop:whittaker derivative}
Suppose that $\alpha\in F^\times$ is such that $\mathrm{Diff}(\alpha) = \{\mathfrak{p}\}$. Set
\[
X(\alpha,\lambda) = \{x\in \near \lambda+ \near L_{\mathfrak{p}}: \near \mathscr{Q}(x) = \alpha\in F^\times_{\mathfrak{p}}\}.
\]
\begin{enumerate}
\item If $X(\alpha,\lambda) = \emptyset$, then 
\[
W'_{\alpha,\mathfrak{p}}(I, 0, \Phi^\lambda_{\mathfrak{p}}) = 0.
\]

\item If $X(\alpha,\lambda) \neq \emptyset$, then $ W_{ \alpha, \mathfrak{p}} (I , 0, \near \Phi^\lambda_{\mathfrak{p}})\neq 0$. Moreover, in this case, we have
\[
\frac{W'_{ \alpha, \mathfrak{p}} (I , 0, \Phi^\lambda_{\mathfrak{p}})}{W_{\alpha,\mathfrak{p}}( I, 0 , \near \Phi^\lambda_{\mathfrak{p}})} = \frac{\ell_{\mathfrak{p}}(\alpha)}{2}\cdot \log N(\mathfrak{q}),
\]
where
\[
\ell_{\mathfrak{p}}(\alpha) = \begin{cases}
\frac{\ord_{\mathfrak{p}}(\xi_{\mathfrak{p}}^{-1}\alpha)+1}{2}  &\text{if $\mathfrak{p}$ is unramified in $E$} \\
\ord_{\mathfrak{p}}(\xi_{\mathfrak{p}}^{-1} \alpha) + n(\mathfrak{p})  &\text{if $\mathfrak{p}$ is ramified in $E$}.
\end{cases}
\] 
\end{enumerate}
\end{proposition}

%%%%%%%%%%%%%%%%%%%%%%%%%%%%%%%%%%%

\subsection{Nearby Schwarz functions}\label{ss:nearby lattices}

%%%%%%%%%%%%%%%%%%%%%%%%%%%%%%%%%%%
We will keep our notation from the previous subsection.

If $\mathfrak{p}'\mid p$ is a prime of $\co_F$ not equal to $\mathfrak{p}$, set $\near \Lambda_{\mathfrak{p}'} = \Lambda_{\mathfrak{p}}$ as Hermitian spaces over $\co_{E,\mathfrak{p}}$. Note that $\near \Lambda_{\mathfrak{p}'}[p^{-1}]$ is isometric to $\near\mathscr{V}_{\mathfrak{p}'}$, and set
\[
\Lambda_p = \bigoplus_{\mathfrak{p}'\mid p}\Lambda_{\mathfrak{p}'}, \quad  \near \Lambda_p = \bigoplus_{\mathfrak{p}'\mid p}\near \Lambda_{\mathfrak{p}'}.
\]

As in~\eqref{eqn:near ident}, we have a canonical $\co_{E,p}$-linear isomorphism (but not an isometry):
\begin{equation}\label{eqn:near isometry}
\Lambda_p \xrightarrow{\simeq}\near\Lambda_p\subset \near \mathscr{V}_p.
\end{equation}
We set 
\[
\nearp \Lambda^\vee_p = \bigoplus_{\mathfrak{p}'\mid p}(\mathfrak{p}')^{-n(\mathfrak{p}')}(\nearp \Lambda_{\mathfrak{p}'}).
\]
Note that $\near \Lambda^\vee_p$ is not necessarily the dual lattice associated with $\near \Lambda_p$, but the notation will be convenient.

Suppose that we are given a class
\[
\lambda = (\lambda_{\mathfrak{p}'}) \in \Lambda_p^\vee/\Lambda_p = \bigoplus_{\mathfrak{p}'\mid p}\left((\mathfrak{p}')^{-n(\mathfrak{p}')}\Lambda_{\mathfrak{p}'}/\Lambda_{\mathfrak{p}'}\right). 
\]
Observe that the isomorphism~\eqref{eqn:near isometry} carries the coset $\lambda + \Lambda_p$ to a coset $\near \lambda + \near\Lambda_p$ of $\near\Lambda_p$ in $\near\Lambda^\vee_p$. We have a further factorization
\[
\near \lambda + \near\Lambda_p = \prod_{\mathfrak{p}'\mid p}\near\lambda_{\mathfrak{p}'} + \near\Lambda_{\mathfrak{p}'}.
\]
Let $\near\varphi_{\lambda_{\mathfrak{p}'}}\in S(\near\mathscr{V}_{\mathfrak{p}'})$ be the characteristic function of $\near \lambda_{\mathfrak{p}'} + \near\Lambda_{\mathfrak{p}'}$. We now set
\[
\near\tilde{\varphi}_\lambda = \bigotimes_{\mathfrak{p}'\mid p}\near\tilde{\varphi}_{\lambda_{\mathfrak{p}'}},
\]
where $\near\tilde{\varphi}_{\lambda_{\mathfrak{p}'}} = \near\varphi_{\lambda_{\mathfrak{p}'}}$, for $\mathfrak{p}'\neq \mathfrak{p}$, and where
\begin{align}\label{eqn:defn tilde varphi}
\near\tilde{\varphi}_{\lambda_{\mathfrak{p}}} & = \begin{cases}
\near\varphi_{\lambda_{\mathfrak{p}}} &\text{ if $\ord_{\mathfrak{q}}(\lambda_{\mathfrak{p}}) > - n(\mathfrak{p})$} \\
0 &\text{ otherwise.}
\end{cases}
\end{align}

Fix $\mu\in L^\vee/L$. Associated with this is the characteristic function $\varphi_\mu\in S(\widehat{\mathscr{V}})$ of the coset $\mu + \widehat{L}$. We will now associate with this class a \emph{nearby} Schwarz function $\near\varphi_\mu\in S(\near \widehat{\mathscr{V}})$ as follows. First, we will have a factorization
\[
\near\varphi_\mu = \bigotimes_\ell\near\varphi_{\mu_\ell}\in \bigotimes_\ell S(\near \mathscr{V}_\ell) \subset S(\near \widehat{\mathscr{V}}).
\]

If $\ell\neq p$, then $\near\varphi_{\mu_\ell} = \varphi_{\mu_\ell}$ will be the characteristic function of $\mu_\ell + L_\ell$ under the obvious identification
\[
S(\near \mathscr{V}_\ell) = S(\mathscr{V}_\ell).
\]

If $\ell = p$, we can view $\mu_p + L_p$ as a subset of $\Lambda^\vee_p$, and, as such, it is a disjoint union
\[
\mu_p + L_p = \bigsqcup_{\lambda\in \mu_p + L_p}\lambda + \Lambda_p
\]
of cosets of $\Lambda_p$ in $\Lambda^\vee_p$. We now set
\begin{equation}\label{eqn:near varphi mu decomp}
\near\varphi_{\mu_p} = \sum_{\lambda\in \mu_p + L_p}\near\tilde{\varphi}_\lambda = \sum_{\lambda\in \mu_p + L_p}\bigotimes_{\mathfrak{p}'\mid p}\near\tilde{\varphi}_{\lambda_{\mathfrak{p}'}}.
\end{equation}

As in the discussion of \S \ref{ss:incoherent}, let $I(s,\chi)$ be the space of the degenerate principal series representation of $\SL_2(\A_F)$
induced from $\chi |\cdot |^s$, and let $\nearp\Phi^\mu(g,s)$ be the standard section of $I(s,\chi)$ determined by the Schwartz function
\[
\varphi_\infty^{\bm{1}} \otimes \nearp\varphi_\mu \in S(\mathscr{C}_\infty) \otimes S(\nearp\mathscr{\widehat{V}}),
\]
Associated to this and each $\alpha\in F$ is the Whittaker function
\[
W_\alpha(g,s,\nearp\Phi^\mu) = W_{\alpha,\infty}(g_\infty,s,\nearp\Phi^\mu_\infty)\cdot W_{\alpha,f}(g_f,s,\nearp\Phi^\mu_f)
\]
admitting a factorization into infinite and finite parts.

We have a decomposition of the finite part
\[
\nearp\Phi^\mu_f = \otimes_\ell\nearp\Phi^\mu_\ell\in \otimes_\ell I_\ell(s,\chi),
\]
where $\nearp\Phi^\mu_\ell\in I_\ell(s,\chi)$ is the standard section associated with $\nearp\varphi_{\mu_\ell}$. This gives us a decomposition of Whittaker functions
\[
W_{\alpha,f}(g_f,s,\nearp\Phi^\mu_f) = \prod_\ell W_{\alpha,\ell}(g_\ell, s, \nearp\Phi^\mu_\ell).
\]

We have
\[
\varphi_{\mu_p} = \sum_{\lambda_p\in \mu_p + L_p}\varphi_{\lambda}\in S(\mathscr{V}_p),
\]
where $\varphi_{\lambda}$ is the characteristic function of $\lambda + \Lambda_p$, and the analogous decomposition for $\near\varphi_\mu$ from~\eqref{eqn:near varphi mu decomp}.  Each coset $\lambda + \Lambda_p$ admits a further decomposition
\[
\lambda + \Lambda_p = \bigsqcup_{\mathfrak{p}'\mid p}\lambda_{\mathfrak{p}'} + \Lambda_{\mathfrak{p}'},
\]
and so we obtain a finer decomposition
\begin{equation}\label{eqn:mu to lambda decomp}
\varphi_{\mu_p} = \sum_{\lambda_p\in \mu_p + L_p}\bigotimes_{\mathfrak{p}'\mid p}\varphi_{\lambda_{\mathfrak{p}'}},
\end{equation}
where $\varphi_{\lambda_{\mathfrak{p}'}}$ is the characteristic function of $\lambda_{\mathfrak{p}'} + \Lambda_{\mathfrak{p}'}$.

% On the level of Whittaker functions, we now have
% \[
% W_{\alpha,p}(g_p, s, \nearp \Phi^\mu_p) = \sum_{\lambda_p\in\mu_p + L_p}\prod_{\mathfrak{p}'\mid p}W_{\alpha,\mathfrak{p}'}(g_{\mathfrak{p}'}, s, \nearp\Phi^\lambda_{\mathfrak{p}'}).
% \]

% \begin{proposition}
% \label{prop:lambda to mu}
% \begin{enumerate}
% \item The function 
% \end{enumerate}
% \end{proposition}

% Fix a prime $p\nmid D_{bad}$. Then we can describe $L_p$ quite explicitly, using Lemma~\ref{lem:good_prime}: First, we have an orthogonal decomposition:
% \[
% L_p = \bigoplus_{\mathfrak{p}\mid p}L_{\mathfrak{p}}
% \]
% as $\mathfrak{p}$ ranges over the $p$-adic places of $F$. 

% For every prime $\ell$ there is a subset
% \[
% \mathscr{V}_{\mu_\ell} ( A_y [\ell^\infty] ) \subset \End(A_y[\ell^\infty])
% \]
% of special endomorphisms.  Define
% \[
% \mathscr{V}_\mu(A_y[\infty]) = \prod_\ell \mathscr{V}_{\mu_\ell}( A_y [\ell^\infty] ).
% \]
% When $\mu=0$ we omit from notation.

% \begin{proposition}\label{prop:unramified special module}
% The natural map
% \[
% \mathscr{V} (A_y)  \otimes_\Z \widehat{\Z} \to \mathscr{V}(A_y[\infty]) 
% \]
% is an isomorphism.  Moreover,  there is an isomorphism of  $\widehat{E}$-hermitian modules 
% \[
% \mathscr{V}(A_y [\infty] ) \otimes_{\widehat{\Z}} \A_f   \iso \near \widehat{\mathscr{V}},
% \]
% identifying  
% $
%  \mathscr{V} _\mu (A_y [\infty]  )    \iso  \mu+ \near \widehat{L}.
% $
% \end{proposition}

%%%%%%%%%%%%%%%%%%%%%%%%%%%%%%%%%%%

\subsection{Orbital integrals and Fourier coefficients}

%%%%%%%%%%%%%%%%%%%%%%%%%%%%%%%%%%%

Normalize the Haar measure on 
\[
T_{so}(\R) = \{ s\in (E\otimes_\Q \R)^\times : s\overline{s} =1\}
\]
to have total volume $1$, and fix any Haar measure on
\[
 T_{so}(\A_f ) = \{ s\in \widehat{E}^\times : s\overline{s} =1\}.
\] 
There is an induced quotient measure on $T_{so}(\Q) \backslash T_{so}(\A)$, and for any  compact open subgroup 
$U \subset T_{so}(\A_f )$ we have 
\[
\mathrm{Vol}(U) = \frac{  | T_{so}(\Q) \cap U |  }{  | T_{so}(\Q) \backslash T_{so}(\A_f) / U  |  } 
\cdot \mathrm{Vol}\big(  T_{so}(\Q) \backslash T_{so}(\A) \big).
\]

\begin{definition} 
Fix a prime $\mathfrak{p} \subset \co_F$ nonsplit in $E$, and let $\near\varphi \in S(\near \widehat{\mathscr{V}} )$  be any  Schwartz function
on  the nearby hermitian space $\near \widehat{\mathscr{V}}=\near \mathscr{V} \otimes \A_f$ of \S \ref{ss:hermitian}.
For each $\alpha \in \widehat{F}^\times$ define the \emph{orbital integral}
\[
O(\alpha, \near\varphi ) =  \frac{1}{ \mathrm{Vol} \big(    T_{so}(\Q) \backslash T_{so}(\A)    \big) }
 \int_{    T_{so}(\A_f )     }  \near\varphi ( s x) \, ds
\]
for any $x \in \near \widehat{\mathscr{V}}$ with  $\near\langle x,x\rangle=\alpha$.   If no such $x$ exists we set $O(\alpha, \near\varphi ) =0$.
\end{definition}

% As in the discussion of \S \ref{ss:incoherent}, let $I(s,\chi)$ be the space of the degenerate principal series representation of $\SL_2(\A_F)$
% induced from $\chi |\cdot |^s$, and let $\Phi(g,s)$ be the standard section determined by the Schwartz function
% \[
% \varphi_\infty^{\bm{1}} \otimes \varphi_\mu \in S(\mathscr{C}_\infty) \otimes S(\mathscr{\widehat{V}}),
% \]
% so that $E(g,s,\varphi_\mu) = E(g,s, \Phi)$.

% By a slight abuse of notation, we will write $\mu_{\mathfrak{p}}\neq 0$ to mean that there is a representative $\lambda_p \in \mu_p + L_p$ of a coset of $\Lambda_p$ in $\Lambda_p^\vee$ with $\mathfrak{p}$-isotypic component $\lambda_{\mathfrak{p}}\notin \Lambda_{\mathfrak{p}}$. Otherwise, we will set $\mu_{\mathfrak{p}} = 0$.

\begin{proposition}\label{prop:explicit siegel-weil}
Fix an  $\alpha\in F_+$  such that  $\mathrm{Diff}(\alpha)=\{\mathfrak{p}\}$ for a single prime $\mathfrak{p}\subset \co_F$.  
Let $\mathfrak{q} \subset \co_E$ be the prime above 
$\mathfrak{p}$. Suppose that $\mathfrak{p}$ lies above a good prime $p$. Then, for any $\mu\in L^\vee /L$, we have
\[
 \frac{ a_F(\alpha, \mu) } {  \Lambda(0,\chi )  } = -  \ell_{\mathfrak{p}}(\alpha) \cdot O(\alpha, \near  \varphi_\mu ) \cdot \log N(\mathfrak{q}),
\]
where $\ell_{\mathfrak{p}}(\alpha) = 0$ unless $( \near \mu_p + \near L_p )\cap \near \mathscr{V}_{\mathfrak{p}}$ represents $\alpha\in F_{\mathfrak{p}}^\times$, in which case, we have
\[
\ell_{\mathfrak{p}}(\alpha) = \begin{cases}
\frac{\ord_{\mathfrak{p}}(\alpha)+ m(\mathfrak{p}) + 1}{2} &\text{if $\mathfrak{p}$ is unramified in $E$} \\
\ord_{\mathfrak{p}}(\alpha) + m(\mathfrak{p}) + n(\mathfrak{p}) &\text{if $\mathfrak{p}$ is ramified in $E$}.
\end{cases}
\]
\end{proposition}

\begin{proof}
The proof proceeds as in~\cite[Theorem 6.1]{KuAnnals}. The strategy is to relate  the incoherent Eisenstein series $E(g,s,\varphi_\mu)$ to a nearby 
\emph{coherent} Eisenstein series, whose Fourier coefficients can be computed using the Siegel-Weil formula. This information is then combined with the computations of local Whittaker functions in \S\ref{ss:whittaker functions} to complete the proof.

We begin by repeating the construction of the incoherent Eisenstein series from \S\ref{ss:incoherent},  but we replace the incoherent $\A_F$-quadratic space
$\mathscr{C}=\mathscr{C}_\infty \times \widehat{\mathscr{V}}$ by the  coherent space   
\[
 \mathscr{C}_\infty \times  \near \widehat{\mathscr{V}} \iso  \near \mathscr{V}\otimes_F \A_F ,
\]
which differs from $\mathscr{C}$ only at the place $\mathfrak{p}$.  

Let $\varphi^p_\mu\in S(\widehat{\mathscr{V}}^p)$ be the prime-to-$p$ part of $\varphi_\mu$ so that we have
\[
\varphi_\mu = \varphi_{\mu_p}\otimes \varphi_\mu^p \in S(\mathscr{V}_p) \otimes S(\widehat{\mathscr{V}}^p).
\]

By~\eqref{eqn:mu to lambda decomp}, we have $\varphi_{\mu_p} = \sum_{\lambda_p\in \mu_p+L_p}\varphi_{\lambda_p}$, where $\varphi_{\lambda_p}$ admits a further product decomposition
\[
\varphi_{\lambda_p} = \prod_{\mathfrak{p}'\mid p}\varphi_{\lambda_{\mathfrak{p}'}}\in \otimes_{\mathfrak{p}'}S(\mathscr{V}_{\mathfrak{p}'}).
\]
Here, $\lambda_p$ ranges over representatives for cosets of $\Lambda_p$ in $\Lambda_p^\vee$ contained in $\mu_p + L_p$.

Set $\varphi_\lambda = \varphi_{\lambda_p}\otimes \varphi^p_\mu\in S(\widehat{\mathscr{V}})$. We now have
\begin{equation}\label{eqn:aF mu lambda decomp}
a_F(\alpha,\mu) = \sum_{\lambda_p\in \mu_p + L_p}a_F(\alpha,\varphi_\lambda).
\end{equation}

Fix $\lambda_p \in \mu_p + L_p$. Choose any Schwarz function $\near\varphi_{\mathfrak{p}}\in S(\near \mathscr{V}_{\mathfrak{p}})$. This gives us a global Schwarz function $\near\varphi\in S(\near \mathscr{V})$ admitting a factorization over primes $\mathfrak{p}'\subset\co_F$:
\[
\near\varphi = \otimes_{\mathfrak{p}'}\near\varphi_{\mathfrak{p}'},
\]
where $\near\varphi_{\mathfrak{p}}$ is our chosen function and, for $\mathfrak{p}'\neq \mathfrak{p}$, we have
\[
\near\varphi_{\mathfrak{p}'} = \varphi_{\lambda_{\mathfrak{p}'}} \in S(\near \mathscr{V}_{\mathfrak{p}'}) = S(\mathscr{V}_{\mathfrak{p}'}).
\]

If $\near\varphi_{\mathfrak{p}}$ is the characteristic function of $\near\lambda_{\mathfrak{p}} + \near \Lambda_{\mathfrak{p}}$, then we will write $\near\varphi_\lambda$ for the corresponding element of $S(\near \mathscr{V})$. 

Let $\Phi^\lambda\in I(s,\chi)$ (resp. $\near\Phi\in I(s,\chi)$) be the standard section associated with $\varphi^{\mathbf{1}}_\infty\otimes\varphi_\lambda$ (resp. $\varphi^{\mathbf{1}}_\infty\otimes\near\varphi$). If $\near\varphi = \near \varphi_\lambda$, we will write $\near\Phi^\lambda$ for the section $\near\Phi$, in agreement with the notation used in the local setting of \S\ref{ss:whittaker functions}.

There is a factorization 
\[
I(s,\chi) = I_\mathfrak{p}(s,\chi) \otimes I^{\mathfrak{p}}(s,\chi),
\]
into the $\mathfrak{p}$-part and prime-to-$\mathfrak{p}$-part. Since $\varphi_\lambda$ and $\near\varphi$ differ only at their $\mathfrak{p}$-components, our two sections $\Phi^\lambda$ and $\near\Phi$ have the form 
\[
\Phi^\mu = \Phi^\mu_\mathfrak{p} \otimes \Phi^{(\mathfrak{p})},\quad  \near\Phi = \near\Phi_\mathfrak{p} \otimes \Phi^{(\mathfrak{p})},
\]
for a \emph{common} section $\Phi^{(\mathfrak{p})}$ of $I^{\mathfrak{p}}(s,\chi)$.

We now have a \emph{coherent} Eisenstein series
$
E(g,s , \near\varphi )  = E(g,s , \near\Phi) 
$
defined exactly as in (\ref{eisenstein formation}), and associated with the Schwartz function $\near\varphi$. 

Given $g,g'\in \mathrm{SL}_2(\A_F)$ which have the \emph{same} prime-to-$\mathfrak{p}$ components, we deduce, using (\ref{eisenstein whittaker}), the relation
\begin{equation}\label{eqn:nearby eisenstein}
 E_\alpha(g,s,\Phi^\lambda) =  \frac{ W_{ \alpha, \mathfrak{p}} (g_\mathfrak{p} , s,  \Phi^\lambda_\mathfrak{p}) }{W_{ \alpha, \mathfrak{p}} (g'_\mathfrak{p} ,  s, \near\Phi_{\mathfrak{p}})} 
 \cdot E_\alpha(g',s, \near \Phi),
\end{equation}
which is valid for all values of $s$ at which $W_{\alpha,\mathfrak{p}}(g'_{\mathfrak{p}} , s, \near\Phi_{\mathfrak{p}})$ is non-zero.

Suppose that $g$ is such that $g_{\mathfrak{p}} = I\in \SL_2(F_\mathfrak{p})$ is the identity, and choose $g'_{\mathfrak{p}}$ and $\near\varphi_{\mathfrak{p}}$ such that $W_{\alpha,\mathfrak{p}}(g'_{\mathfrak{p}},0,\near\varphi_{\mathfrak{p}})\neq 0$. Then, using~\eqref{eqn:nearby eisenstein} and Proposition~\ref{prop:whittaker vanishing}, we get:
\begin{align*}
E'_\alpha(g,0,\Phi^\lambda)  
&=  \frac{  W'_{ \alpha, \mathfrak{p}} (I  , 0,  \Phi^\lambda_\mathfrak{p})   }{ W_{ \alpha, \mathfrak{p}} ( g'_{\mathfrak{p}} ,  0, \near\Phi_\mathfrak{p})   } 
 \cdot E_\alpha( g',0, \near\Phi).
%&= -  \left(   \frac{1 + \ord_\mathfrak{p}(\alpha) }{2} \right)
% \cdot E_\alpha( g,0, \near \Phi) \cdot \log N(\mathfrak{p}),
\end{align*}
%Compare  with \cite[(5.2.5)]{HY}, \cite[(6.10)]{KuAnnals}, or \cite[\S 5]{KYtheta}.

In the notation of Proposition \ref{prop:coefficient support}, this equality  implies 
\begin{equation}\label{coefficient swindle}
\frac{a_F(\alpha,\lambda)}{ \Lambda(0,\chi) }  \cdot q^\alpha= 
-  \frac{  W'_{ \alpha, \mathfrak{p}} (I  , 0,  \Phi^\lambda_\mathfrak{p})   }{ W_{ \alpha, \mathfrak{p}} ( g'_{\mathfrak{p}} ,  0, \near\Phi_\mathfrak{p})   }
 \cdot  \frac{ E_\alpha(g'_{\vec{\tau}} ,0, \near\Phi)  }{ \sqrt{ N(\vec{v} ) } }.
\end{equation}
for all $\vec{\tau} \in \mathcal{H}^d$.

If $\alpha$ is not represented by $\near \lambda_{\mathfrak{p}} + \near \Lambda_{\mathfrak{p}}$, then assertion (1) of Proposition~\ref{prop:whittaker derivative} now implies
\[
\frac{a_F(\alpha,\varphi_\lambda)}{\Lambda(0,\chi)} = 0.
\]
Combining this with~\eqref{eqn:aF mu lambda decomp} shows that $a_F(\alpha,\mu) = 0$, whenever $(\near\mu_p + \near L_p)\cap \near\mathscr{V}_{\mathfrak{p}}$ does not represent $\alpha$.

Now, suppose that $\alpha$ is represented by $\near \lambda_{\mathfrak{p}} + \near \Lambda_{\mathfrak{p}}$. Then Proposition~\ref{prop:whittaker derivative} implies that we can take $g' = g$ and $\near\varphi_{\mathfrak{p}} = \near\varphi_{\lambda_{\mathfrak{p}}}$. 

 As in the proof of \cite[Proposition 4.4.1]{HY}, the Siegel-Weil formula \cite{KudlaRallis} implies 
\[
\frac{ E_\alpha(g_{\vec{\tau}} ,0, \near \Phi^\lambda)  }{ \sqrt{ N(\vec{v} ) } }  
  =  \frac{ 2  q^\alpha}{   \mathrm{Vol} \big(  T_{so}(\Q) \backslash T_{so}(\A)   \big)   }  
\int_{  T_{so}(\Q) \backslash T_{so}(\A)   }
 \sum_{  \substack  {   x\in  \near\mathscr{V} \\  \near \mathscr{Q}(x)=\alpha   }  } \near \varphi_\lambda( s_f^{-1} x ) \, ds .
\]
The group $T_{so}(\Q)$ acts simply transitively on the set of all $x\in  \near\mathscr{V}$ with   $\near \mathscr{Q}(x)=\alpha$, allowing 
us to rewrite this equality as
\begin{equation}\label{eqn:siegel weil result}
\frac{ E_\alpha(g_{\vec{\tau}} ,0, \near \Phi^\lambda)  }{ \sqrt{ N(\vec{v} ) } }  
=  2\cdot O(\alpha, \near\varphi_\lambda) \cdot q^\alpha.
\end{equation}

Combining~\eqref{eqn:siegel weil result} with~\eqref{coefficient swindle}, and using the formulas for
\[
\frac{  W'_{ \alpha, \mathfrak{p}} (I  , 0,  \Phi^\lambda_\mathfrak{p})   }{ W_{ \alpha, \mathfrak{p}} ( I ,  0, \near\Phi^\lambda_\mathfrak{p})   }
\]
from Proposition~\ref{prop:whittaker derivative} shows
\begin{equation}\label{eqn:aF lambda}
\frac{a_F(\alpha,\varphi_\lambda)}{\Lambda(0,\chi)} = -\ell_{\mathfrak{p}}(\alpha)\cdot O(\alpha,\near\varphi_\lambda)\cdot\log N(\mathfrak{q}).
\end{equation}

Now, observe that $\ord_{\mathfrak{p}}(\alpha) + m(\mathfrak{p}) + n(\mathfrak{p}) = 0$ whenever $n(\mathfrak{p})\neq 0$ and $\ord_{\mathfrak{q}}(\lambda_{\mathfrak{p}}) = -n(\mathfrak{p})$. Therefore, from the definition of $\near\tilde{\varphi}_{\lambda_{\mathfrak{p}}}$ in ~\eqref{eqn:defn tilde varphi}, we see that~\eqref{eqn:aF lambda} is equivalent to
\begin{equation}\label{eqn:aF lambda 1}
\frac{a_F(\alpha,\varphi_\lambda)}{\Lambda(0,\chi)} = -\ell_{\mathfrak{p}}(\alpha)\cdot O(\alpha,\near\tilde{\varphi}_\lambda)\cdot\log N(\mathfrak{q}).
\end{equation}
Here, $\near\tilde{\varphi}_{\lambda}$ differs from $\near\varphi_\lambda$ only at $\mathfrak{p}$, and we take its factor at $\mathfrak{p}$ to be $\near\tilde{\varphi}_{\lambda_{\mathfrak{p}}}$.

Now, note that, by~\eqref{eqn:near varphi mu decomp}, 
\[
O(\alpha,\near\varphi_\mu) = \sum_{\lambda_p\in \mu_p+L_p} O(\alpha,\near\tilde{\varphi}_\lambda)
\]
and that
\[
O(\alpha,\near\tilde{\varphi}_\lambda) = 0,
\]
whenever $\near\lambda_{\mathfrak{p}} + \near\Lambda_{\mathfrak{p}}$ does not represent $\alpha$.

Combining these observations with~\eqref{eqn:aF mu lambda decomp} and~\eqref{eqn:aF lambda 1} completes the proof of the Proposition.
\end{proof}

%%%%%%%%%%%%%%%%%%%%%%%%%%%%%%%%%%%%%%%

\subsection{A decomposition of the space of special endomorphisms}
\label{ss:special decomposition}

%%%%%%%%%%%%%%%%%%%%%%%%%%%%%%%%%%%%%%%

Fix a prime $\mathfrak{p}\subset\co_F$ not split in $E$, and let $\mathfrak{q}\subset\co_E$ be the unique prime above it. Fix an algebraic closure $\F_{\mathfrak{p}}^{\alg}$ for $\F_{\mathfrak{p}}$ and also an algebraic closure $\mathrm{Frac}(W)^{\alg}$ of the fraction field $\mathrm{Frac}(W)$ of $W = W(\F_{\mathfrak{p}}^{\alg})$. Choose an embedding $\Q^{\alg} \hookrightarrow \mathrm{Frac}(W)^{\alg}$ inducing the place $\mathfrak{q}$ on $E = \iota_0(E)$. 

Let $L_{\mathfrak{p}} = L_p \cap  \mathscr{V}_{\mathfrak{p}}\subset L_p$,  and let $H_{\mathfrak{p}} = C(L_{\mathfrak{p}}) \subset C(L_p) = H_p$. Let $K_{L,0}\subset T(\A_f)$ be the compact open subgroup defined in \S~\ref{ss:shimura data}, and let $K_{0,L,\mathfrak{q}} \subset K_{0,L,p}$ be the intersection of $K_{0,L,p}$ with the image of $E_{\mathfrak{q}}^\times$ under the natural map
\[
E_{\mathfrak{q}}^\times \hookrightarrow (E\otimes_{\Q}\Q_p)^\times \to T(\Q_p).
\]

Then $H_p\subset H_{\Q_p}$ is a $K_{0,L,p}$-stable lattice, and $H_{\mathfrak{p}}\subset H_{\mathfrak{p}}[p^{-}]$ is a $K_{0,L,\mathfrak{q}}$-stable lattice. Moreover, the natural $C(L_p)$-linear map
\begin{align}\label{eqn:H mf p to Hp}
H_{\mathfrak{p}}\otimes_{C(L_{\mathfrak{p}})}C(L_p) &\to H_p\\
h\otimes z &\mapsto h\cdot z\nonumber
\end{align}
is a $K_{0,L,\mathfrak{q}}$-equivariant isomorphism, once we equip $C(L_p)$ with the trivial $K_{0,L,\mathfrak{q}}$-action.

Fix a point $y\in \mathcal{Y}(\F_{\mathfrak{q}}^{\alg})$. Then, by Remark~\ref{rem:Eq times reps},~\eqref{eqn:H mf p to Hp} gives us a crystalline $\Z_p$-representation $\bm{H}_{\mathfrak{p},y}$ of $\Gamma_y$ and a $C(L_p)$-linear isomorphism 
\begin{align}
\label{eqn:H mf p to Hp etale}
\bm{H}_{\mathfrak{p},y}\otimes_{C(L_{\mathfrak{p}})}C(L_p) &\xrightarrow{\simeq}
\bm{H}_{p,y}.
\end{align}

The following result is easily deduced from Theorem~\ref{thm:kisin_p_divisible}.
\begin{proposition}\label{prop:H mfp Tate module}
The $\Gamma_y$-module   $\bm{H}_{\mathfrak{p},y}$ is canonically isomorphic to the $p$-adic Tate module of a 
$C(L_{\mathfrak{p}})$-linear $p$-divisible subgroup
\[
A[p^\infty]_{\mathfrak{p}}\subset A[p^\infty]\vert_{\Spec (\co_y)}.
\]
Moreover, the natural $C(L_p)$-linear map of $p$-divisible groups
\[
A[p^\infty]_{\mathfrak{p}}\otimes_{C(L_{\mathfrak{p}})}C(L_p)\to A[p^\infty]\vert_{\Spec(\co_y) }
\]
is an isomorphism.
\end{proposition}

In particular, for any $\co_y$-scheme $S$, we obtain a natural map

\begin{equation}\label{eqn:end mf p to p}
\End_{C(L_{\mathfrak{p}})}(A[p^\infty]_{\mathfrak{p},S})\to \End_{C(L_p)}(A_S[p^\infty]),
\end{equation}
and so, in complete analogy with the definitions from \S\ref{ss:special divisors}, we define the space of special endomorphisms
\[
V\bigl(A[p^\infty]_{\mathfrak{p},S}\bigr) \subset \End_{C(L_{\mathfrak{p}})}(A[p^\infty]_{\mathfrak{p},S})
\]
to consist of those elements that induce special endomorphisms of $A[p^\infty]$ via~\eqref{eqn:end mf p to p}. By definition, this is a subspace of $V(A_S[p^\infty])$.

The next result is entirely analogous to Lemma~\ref{lem:lambda_perp}.
\begin{proposition}
\label{prop:V mf p to p}
Let $L_p^{\mathfrak{p}} = L_{\mathfrak{p}}^\perp \subset L_p$ be the orthogonal complement to $L_{\mathfrak{p}}$. Then there is a canonical isometric embedding
\[
L_p^{\mathfrak{p}}\hookrightarrow V(A[p^\infty]_{\Spec (\co_y) })
\]
as a direct summand, such that, for any $\co_y$-scheme $S$, we have
\[
V\bigl(A[p^\infty]_{\mathfrak{p},S}\bigr)  = \bigl(L_p^{\mathfrak{p}}\bigr)^{\perp}\subset V(A_S[p^\infty]).
\]
\end{proposition}

Given a class $\eta\in L_{\mathfrak{p}}^\vee/L_{\mathfrak{p}}$, we will also need a corresponding subset 
\begin{equation}\label{eqn:V eta}
V_{\eta}(A[p^\infty]_{\mathfrak{p},S})\subset V(A[p^\infty]_{\mathfrak{p},S})^\vee.
\end{equation}
This is once again defined as in \S\ref{ss:special divisors}: We fix an embedding $L\hookrightarrow L^\beef$ into a maximal lattice $L^\beef$ that is of signature $(n^\beef,2)$ and is self-dual at $p$. Let $\Lambda^{\mathfrak{p}}\subset L^\beef_p$ be the orthogonal complement of $L_{\mathfrak{p}}$. Then we have a canonical isometric embedding
\[
\Lambda^{\mathfrak{p}}\hookrightarrow V(A^\beef[p^\infty]_S),
\]
whose orthogonal complement is $V(A[p^\infty]_{\mathfrak{p},S})$. Hence we get a map
\[
V(A^\beef_S[p^\infty]) \to \frac{V(A[p^\infty]_{\mathfrak{p},S})^\vee}{V(A[p^\infty]_{\mathfrak{p},S})} \oplus \frac{\Lambda^{\mathfrak{p},\vee}}{\Lambda^{\mathfrak{p}}}.
\]

The subset~\eqref{eqn:V eta} now consists of elements $x$ such that the pair 
\[
([x],\eta) \in \frac{V(A[p^\infty]_{\mathfrak{p},S})^\vee}{V(A[p^\infty]_{\mathfrak{p},S})} \oplus \frac{\Lambda^{\mathfrak{p},\vee}}{\Lambda^{\mathfrak{p}}}
\] 
is in the image of $V(A_S^\beef[p^\infty])$. Here, we have used the natural isomorphisms
\[
\frac{\Lambda^{\mathfrak{p},\vee}}{\Lambda^{\mathfrak{p}}} \xleftarrow{\simeq} \frac{L^\beef_p}{L_{\mathfrak{p}}+\Lambda^{\mathfrak{p}}} \xrightarrow{\simeq} \frac{L_{\mathfrak{p}}^\vee}{L_{\mathfrak{p}}}
\]
to view $\eta$ as an element of $ \frac{\Lambda^{\mathfrak{p},\vee}}{\Lambda^{\mathfrak{p}}}$.

The following proposition is now immediate from the definitons and is analogous to assertion (3) of Proposition~\ref{prop:decomposition_Vmu}.
\begin{proposition}
\label{prop:V p to mf p}
For any $\mu_p\in L_p^\vee/L_p$, we have a canonical decompositions
\[
V_{\mu_p}(A_S[p^\infty])
=\bigsqcup_{(\mu_1,\mu_2)\in (\mu_p+ L_p)/(L_{\mathfrak{p}}\oplus L^{\mathfrak{p}_p})}  V_{\mu_1}(A[p^\infty]_{\mathfrak{p},S}) \times
\bigl({\mu_2}+ L^{\mathfrak{p}}_p\bigr),
\]
where we are viewing 
\[
\frac{\mu_p+ L_p}{L_{\mathfrak{p}}\oplus L^{\mathfrak{p}}_p} \subset 
\frac{L^\vee_{\mathfrak{p}}}{L_{\mathfrak{p}}} \oplus \frac{L^{\mathfrak{p,\vee}}_p}{L^{\mathfrak{p}}_p} .
\]
\end{proposition}

%%%%%%%%%%%%%%%%%%%%%%%%%%%%%%%%%%%%%%%

\subsection{Lubin-Tate and Kuga-Satake}
\label{ss:lt to ks}

%%%%%%%%%%%%%%%%%%%%%%%%%%%%%%%%%%%%%%%

Let $\mathfrak{p}\subset \co_F$ and $\mathfrak{q}\subset \co_E$ be as above. For the rest of this section, we will assume that $\mathfrak{p}$ lies above a \emph{good} prime $p$. Therefore, we have
\[
\Lambda_{\mathfrak{p}} \subset L_{\mathfrak{p}}\subsetneq \mathfrak{d}_{E_{\mathfrak{q}}/F_{\mathfrak{p}}}^{-1}\Lambda_{\mathfrak{p}} = \Lambda^\vee_{\mathfrak{p}},
\]
where $\Lambda_{\mathfrak{p}}\subset \mathscr{V}_{\mathfrak{p}}$ is an $\co_{E,\mathfrak{q}}$-stable lattice.

Fix a point $y\in \mathcal{Y}(\F_{\mathfrak{q}}^{\alg})$. 
Fix also a uniformizer $\pi_{\mathfrak{q}}\in E_{\mathfrak{q}}$, and let $\mathcal{G}_{\mathfrak{q}}$ be the Lubin-Tate formal $\co_{E,\mathfrak{q}}$-module over $\co_y$ associated with this uniformizer. 
If $\mathfrak{p}$ is unramified in $E$, we will assume that we have chosen $\pi_{\mathfrak{q}} = \pi_{\mathfrak{p}}$ to be a uniformizer for $F_{\mathfrak{p}}$. Otherwise, we will set $\pi_{\mathfrak{p}} = \mathrm{Nm}(\pi_{\mathfrak{q}})\in F_{\mathfrak{p}}$. As in \S\ref{ss:whittaker functions}, we will set
\[
m(\mathfrak{p}) = \ord_{\mathfrak{q}}(\mathfrak{d}_{F/\Q}).
\]

As in \S\ref{ss:denominators}, for any $\co_y$-scheme $S$, and for each $\lambda\in \mathfrak{d}_{E_{\mathfrak{q}}/F_{\mathfrak{p}}}^{-1}/\co_{E,\mathfrak{q}}$, we have a canonical subset
\[
V_\lambda(\mathcal{G}_{\mathfrak{q},S})\subset \End(\mathcal{G}_{\mathfrak{q},S})_\Q
\]
of special endomorphisms (with denominators) of $\mathcal{G}_{\mathfrak{q},S}$.
Fix an $\co_{E,\mathfrak{q}}$-linear identification $\Lambda_{\mathfrak{p}} = \co_{E,\mathfrak{q}}$, so that we can identify
\[
\frac{\Lambda^\vee_{\mathfrak{p}}}{\Lambda_{\mathfrak{p}}} = \frac{\mathfrak{d}^{-1}_{E_{\mathfrak{q}}/F_{\mathfrak{p}}}}{\co_{E,\mathfrak{q}}}.
\]
In particular, for any $\lambda\in \Lambda^\vee_{\mathfrak{p}}/\Lambda_{\mathfrak{p}}$, 
we have a corresponding set $V_\lambda(\mathcal{G}_{\mathfrak{q},S})$ of special endomorphisms of $\mathcal{G}_{\mathfrak{q}}$.
Under this identification, the Hermitian form on $\Lambda_{\mathfrak{p}}$ is carried to the form 
\[
\langle x_1,x_2\rangle = \xi_{\mathfrak{p}}x_1\overline{x}_2
\]
on $\co_{E,\mathfrak{q}}$, for some $\xi_{\mathfrak{p}}\in F_{\mathfrak{p}}^\times$ satisfying $\ord_{\mathfrak{p}}(\xi_{\mathfrak{p}}) = -m(\mathfrak{p})$.

Since we have identified $\Lambda^\vee_{\mathfrak{p}}/\Lambda_{\mathfrak{p}}$ with $\mathfrak{d}_{E_{\mathfrak{q}}/F_{\mathfrak{p}}}^{-1}/\co_{E,\mathfrak{q}}$, for $\lambda\in \Lambda^\vee_{\mathfrak{p}}/\Lambda_{\mathfrak{p}}$, we can speak of the space $V_\lambda(\mathcal{G}_{\mathfrak{q}})$ of special endomorphisms of the Lubin-Tate group $\mathcal{G}_{\mathfrak{q}}$. 

For $\mu\in L^\vee_{\mathfrak{p}}/L_{\mathfrak{p}}$, set
\[
V_\mu(\mathcal{G}_{\mathfrak{q},y}) = \bigsqcup_{\lambda \in \mu + L_{\mathfrak{p}}} V_\lambda(\mathcal{G}_{\mathfrak{q},\co_y/\pi_{\mathfrak{q}}\co_y}),
\]
where $\lambda$ varies over the classes in $\Lambda^\vee_{\mathfrak{p}}/\Lambda_{\mathfrak{p}}$ such that $\lambda + \Lambda^\vee_{\mathfrak{p}}$ lies in $\mu + L_{\mathfrak{p}}$.

\begin{proposition}
\label{prop:lubin-tate to kuga-satake}
There exists an $E_{\mathfrak{q}}$-linear isomorphism
\[
 V(A_y[p^\infty]_{\mathfrak{p}})_\Q \xrightarrow{\simeq} V(\mathcal{G}_{\mathfrak{q},y})_\Q
\]
carrying the Hermitian form on the left hand side to $\xi_{\mathfrak{p}}$ times that on the right, and such that, for each $\mu\in L_{\mathfrak{p}}^\vee/L_{\mathfrak{p}}$, it induces a bijection
\[
V_\mu(A_y[p^\infty]_{\mathfrak{p}}) \xrightarrow{\simeq} V_\mu(\mathcal{G}_{\mathfrak{q},y}).
\]
\end{proposition}
\begin{proof}
Using Remark~\ref{rem:Eq times reps}, we can associate with $\Lambda_{\mathfrak{p}}$ an $\co_{E,\mathfrak{q}}$-linear continuous representation $\bm{\Lambda}_{\mathfrak{p},\et,y}$ of the absolute Galois group $\Gamma_y$ of $\mathrm{Frac}(\co_y)$. This representation can be identified with the space $\bm{V}_{0,\mathfrak{p},\et} = V(\bm{H}_{0,\mathfrak{p},\et},c)$ of $\co_{E,\mathfrak{q}}$-semilinear endomorphisms of the Tate module $\bm{H}_{0,\mathfrak{p},\et}$ of the Lubin-Tate group $\mathcal{G}_{\mathfrak{q}}$. Moreover, its crystalline realization $\bm{\Lambda}_{\mathfrak{p},\cris,y}$ can be identified with the space $\bm{V}_{0,\mathfrak{p},\cris,y} = V(\bm{H}_{0,\mathfrak{p},\cris,y},c)$ of $\co_{E,\mathfrak{q}}$-semilinear endomorphisms of the $F$-crystal $\bm{H}_{0,\mathfrak{p},\cris,y}$ obtained from the Dieudonn\'e $F$-crystal associated with $\mathcal{G}_{\mathfrak{q}}$. 

These identifications carries the Hermitian form on $\bm{\Lambda}_{\mathfrak{p},\et}$ (resp. $\bm{\Lambda}_{\mathfrak{p},\cris,y})$ to $\xi_{\mathfrak{p}}$ times the natural Hermitian form on $\bm{V}_{0,\mathfrak{p}}$ (resp. $\bm{V}_{0,\mathfrak{p},\cris,y}$). Therefore, we now obtain an $E_{\mathfrak{q}}$-linear isomorphism
\[
V(A_y[p^\infty]_{\mathfrak{p}})_\Q = \bm{V}_{\mathfrak{p},\cris,y}^{\varphi = 1}[p^{-1}] = \bm{\Lambda}_{\mathfrak{p},\cris,y}^{\varphi = 1}[p^{-1}] \xrightarrow{\simeq}\bm{V}_{0,\mathfrak{p},\cris,y}^{\varphi = 1}[p^{-1}] = V(\mathcal{G}_{\mathfrak{q},y})_\Q
\] 
carrying the Hermitiian form on the left to $\xi_{\mathfrak{p}}$-times that on the very right.

It remains to show that it carries $V_\mu(A_y[p^\infty]_{\mathfrak{p}})$ onto $V_\mu(\mathcal{G}_{\mathfrak{q},y})$. For this, we will need a little preparation. Consider the Breuil-Kisin module $\mathfrak{M}(\Lambda_{\mathfrak{p}})$ associated with $\bm{\Lambda}_{\mathfrak{p},\et,y}$ and the uniformizer $\pi_{\mathfrak{q}}$. We have an $\co_{E,\mathfrak{q}}$-linear identification
\[
\mathfrak{M}(\Lambda_{\mathfrak{p}}) = V(\mathfrak{M}(H_{0,\mathfrak{p}}),c)
\]
of Breuil-Kisin modules. 
\begin{lemma}
\label{lem:discriminant trivial}
There is canonical, $\varphi$-equivariant isomorphism
\[
\mathfrak{S}\otimes_{\Z_p}\frac{\Lambda^\vee_{\mathfrak{p}}}{\Lambda_{\mathfrak{p}}}\xrightarrow{\simeq}\frac{\mathfrak{M}(\Lambda^\vee_{\mathfrak{p}})}{\mathfrak{M}(\Lambda_{\mathfrak{p}})}
\]
of $\mathfrak{S}$-modules, where the left hand side is equipped with the constant $\varphi$-semi-linear endomorphism $\varphi\otimes 1$. It induces a $\varphi$-equivariant isomorphism
\[
\mathfrak{S}\otimes_{\Z_p}\frac{L^\vee_{\mathfrak{p}}}{L_{\mathfrak{p}}}\xrightarrow{\simeq}\frac{\mathfrak{M}(L^\vee_{\mathfrak{p}})}{\mathfrak{M}(L_{\mathfrak{p}})}.
\]
\end{lemma}
\begin{proof}
If $\mathfrak{p}$ is unramified in $\co_E$, then $\Lambda^\vee_{\mathfrak{p}} = \Lambda_{\mathfrak{p}}$, and there is nothing to show.

Suppose therefore that $\mathfrak{p}$ is ramified in $\co_E$. We have $\mathfrak{S}\otimes_{\Z_p}\co_{E,\mathfrak{q}}$-linear isomorphisms
\begin{equation}\label{eqn:coEq linear ident 1}
\mathfrak{S}\otimes_{\Z_p}\frac{\Lambda^\vee_{\mathfrak{p}}}{\Lambda_{\mathfrak{p}}}\xrightarrow{\simeq}\mathfrak{S}\otimes_{\Z_p}\frac{\mathfrak{d}_{E_{\mathfrak{q}}/F_{\mathfrak{p}}}^{-1}}{\co_{E,\mathfrak{q}}},
\end{equation}
and
\begin{equation}\label{eqn:coEq linear ident 2}
\frac{\mathfrak{M}(\Lambda^\vee_{\mathfrak{p}})}{\mathfrak{M}(\Lambda_{\mathfrak{p}})}\xrightarrow{\simeq}\frac{\mathfrak{d}_{E_{\mathfrak{q}}/F_{\mathfrak{p}}}^{-1}}{\co_{E,\mathfrak{q}}}\otimes_{\co_{E,\mathfrak{q}}}V(\mathfrak{M}(\bm{H}_{0,\mathfrak{p}}),c).
\end{equation}

As in \S \ref{ss:lubin_tate}, we have identifications 
\begin{equation}\label{eqn:bk module ident}
\mathfrak{M}(\bm{H}_{0,\mathfrak{p}}) = \mathfrak{M}(T_{\pi_E}(\mathcal{G}_{\mathfrak{q}})) = \mathfrak{S}\otimes_{\Z_p}\co_{E,\mathfrak{q}}
\end{equation}
as $\mathfrak{S}\otimes_{\Z_p}\co_{E,\mathfrak{q}}$-modules carrying the the $\varphi$-semilinear endomorphism of $\mathfrak{M}(\bm{H}_{0,\mathfrak{p}})$ to the endomorphism $\beta(\varphi\otimes 1)$, where $\beta$ has the following description: First, let $E_{\mathfrak{q},0}\subset E_{\mathfrak{q}}$ be the maximal unramified subextension. For each embedding $\eta:E_{\mathfrak{q},0}\hookrightarrow \mathrm{Frac}(W)$, we obtain a finite $W$-algebra $W_\eta = \co_{E,\mathfrak{q}}\otimes_{\co_{E_{\mathfrak{q},0}},\eta}W$. There is a disinguished embedding $\eta_0$ induced from the distinguished embedding $\iota_0$ of $E_{\mathfrak{q}}$ into $\mathrm{Frac}(W)^{\mathrm{alg}}$.

We now have
\[
\beta = (\beta_{\eta})\in \prod_{\eta}\mathfrak{S}\otimes_WW_\eta = \mathfrak{S}\otimes_{\Z_p}\co_{E,\mathfrak{q}},
\]
where $\beta_{\eta} = 1$, if $\eta\neq \eta_0$, and $\beta_{\eta_0} = u - \eta_0(\pi_{\mathfrak{q}})$.

From~\eqref{eqn:bk module ident}, we now obtain an identification
\begin{equation}
 \label{eqn:bk module V ident}
 V(\mathfrak{M}(\bm{H}_{0,\mathfrak{p}}),c) = \mathfrak{S}\otimes_{\Z_p}\co_{E,\mathfrak{q}}
 \end{equation}
carrying the $\varphi$-semilinear endomorphism on the left hand side to the endomorphism $\alpha(\varphi\otimes 1)$, where
\[
\alpha = (\alpha_{\eta})\in \prod_{\eta}\mathfrak{S}\otimes_WW_\eta = \mathfrak{S}\otimes_{\Z_p}E_{\mathfrak{q}}
\]
with $\alpha_\eta = 1$, for $\eta \neq \eta_0$, and $\alpha_{\eta_0} = \frac{u-\eta_0(\overline{\pi}_{\mathfrak{q}})}{u - \eta_0(\pi_{\mathfrak{q}})}$. Since $\overline{\pi}_{\mathfrak{q}} - \pi_{\mathfrak{q}}\in \mathfrak{d}_{E_{\mathfrak{q}}/F_{\mathfrak{p}}}$, we have
\[
\alpha\equiv 1 \pmod{\mathfrak{d}_{E_{\mathfrak{q}}/F_{\mathfrak{p}}}}.
\]
Therefore, tensoring~\eqref{eqn:bk module V ident} with $\mathfrak{d}_{E_{\mathfrak{q}}/F_{\mathfrak{p}}}^{-1} / \co_{E,\mathfrak{q}}$, 
and using~\eqref{eqn:coEq linear ident 1} and~\eqref{eqn:coEq linear ident 2}, gives us the isomorphism whose existence is asserted in the proposition.

We leave it to the reader to check that this isomorphism is independent of all our choices.
\end{proof}

Now, base-changing along $\varphi:\mathfrak{S}\to \mathfrak{S}$ the isomorphism from Lemma~\ref{lem:discriminant trivial} and then reducing it mod $u$, we obtain a canonical isomorphism
\[
\eta:W\otimes_{\Z_p}\frac{L_{\mathfrak{p}}^\vee}{L_{\mathfrak{p}}}\xrightarrow{\simeq}\frac{\bm{V}_{\mathfrak{p},\cris,y}^\vee}{\bm{V}_{\mathfrak{p},\cris,y}}.
\]

\begin{lemma}
\label{lem:Vmu alt descp}
The subset $V_\mu(A_y[p^\infty]_{\mathfrak{p}})\subset V(A_y[p^\infty]_{\mathfrak{p}})_\Q$ consists of those elements $x$ whose crystalline realization $\bm{x}_{\cris}\in \bm{V}_{\mathfrak{p},\cris,y}[p^{-1}]$ lies in $\bm{V}^\vee_{\mathfrak{p},\cris,y}$, and such that
\[
\bm{x}_{\cris}\equiv \eta(1\otimes \mu)\pmod{\bm{V}_{\mathfrak{p},\cris,y}}
\]
\end{lemma}
\begin{proof}
Choose an auxiliary lattice $L^\beef$, self-dual at $p$, and isometric embedding $L\hookrightarrow L^\beef$, giving us the auxiliary Kuga-Satake abelian variety $A^\beef_y$ over $\F_{\mathfrak{q}}^\alg$. 

Let $L^{\beef,\mathfrak{p}}_p\subset L^\beef_p$ be the orthogonal complement of $L_{\mathfrak{p}}$. Choose a lift $\tilde{\mu}\in L_{\mathfrak{p}}^\vee$ of $\mu$, and an element $\tilde{\mu}^{\mathfrak{p}}\in L^{\beef,\mathfrak{p},\vee}_p$ such that  
\[
(\tilde{\mu},\tilde{\mu}^{\mathfrak{p}})\in L_p^{\beef}\subset L^\vee_{\mathfrak{p}}\oplus L^{\beef,\mathfrak{p},\vee}_p.
\]
Then, by definition, giving an element of $V_\mu(A_y[p^\infty]_{\mathfrak{p}})$ amounts to specifying $x\in V(A_y[p^\infty]_{\mathfrak{p}})^\vee$ such that
\[
(x,\tilde{\mu}^{\mathfrak{p}}) \in V(A_y[p^\infty]_{\mathfrak{p}})^\vee \oplus L^{\beef,\mathfrak{p},\vee}
\]
lies in the image of $V(A^\beef_y[p^\infty])$. 

Since we have a canonical isometric embedding
\[
W\otimes_{\Z_p}L^{\beef,\mathfrak{p}}\hookrightarrow \bm{V}^\beef_{\cris,y}
\]
mapping into the orthogonal complement of $\bm{V}_{\mathfrak{p},\cris,y}$, we obtain an inclusion
\begin{equation}\label{eqn:cris dual inclusion}
\bm{V}_{\cris,y}^\beef\hookrightarrow \bm{V}^\vee_{\mathfrak{p},\cris,y}\oplus(W\otimes_{\Z_p}L^{\beef,\mathfrak{p},\vee}).
\end{equation}
Let $\bm{x}_{\cris}\in \bm{V}_{\mathfrak{p},\cris,y}[p^{-1}]$ be the crystalline realization of $x$, and let $\bm{x}^\beef_{\cris}\in \bm{V}^\beef_{\cris,y}$ be the crystalline realization of $(x,\tilde{\mu})$. Then it is clear from the definitions that $\bm{x}_{\cris}$ actually lies in $\bm{V}^\vee_{\mathfrak{p},\cris,y}$, and that~\eqref{eqn:cris dual inclusion} maps $\bm{x}^\beef_\cris$ to $(\bm{x}_{\cris},\tilde{\mu}^{\mathfrak{p}})$. From this, one deduces that $\bm{x}_{\cris}$ must map into $\eta(1\otimes u)\in \bm{V}^\vee_{\mathfrak{p},\cris,y}/\bm{V}_{\mathfrak{p},\cris,y}$.
\end{proof}

Tracing through the definition of $V_\mu(\mathcal{G}_{\mathfrak{q},y})$, it is not hard to show that it has the same description as that of $V_\mu(A_y[p^\infty]_{\mathfrak{p}})$ given to us by Lemma~\ref{lem:Vmu alt descp}. This finishes the proof of the proposition.

\end{proof}

%%%%%%%%%%%%%%%%%%%%%%%%%%%%%%%%%%%%%%%

\subsection{Special zero cycles}

%%%%%%%%%%%%%%%%%%%%%%%%%%%%%%%%%%%%%%%

For any scheme $S$ over $\mathcal{Y}$, by Corollary~\ref{cor:special end structure}, the space $V(A_S)_\Q$ has a canonical structure of an $E$-vector space equipped with a positive definite Hermitian form $\langle\cdot,\cdot\rangle$ such that, for any $x\in V(A_S)_\Q$,
\[
Q(x) = x\circ x = \mathrm{Tr}_{F/\Q}(\langle x,x\rangle).
\]
Write $\mathscr{V}(A_S)_\Q$ for $V(A_S)_\Q$ equipped with this additional structure. For any $\mu\in L^\vee/L$, let $\mathscr{V}_\mu(A_S)$ denote the space $V_\mu(A_S)$ viewed as a subspace of $\mathscr{V}(A_S)_\Q$

Suppose that $\alpha\in F^\times$ and $\mu\in L^\vee / L$. Define a moduli problem  $\mathcal{Z}_F(\alpha,\mu)$ over $\mathcal{Y}$ such that, for any $\mathcal{Y}$-scheme $S$, we have
\[
\mathcal{Z}_F(\alpha,\mu)(S) = \{x\in \mathscr{V}_\mu(A_S):  \langle x,x\rangle = \alpha\}.
\]
Since $\langle\cdot,\cdot\rangle$ is positive definite, $\mathcal{Z}_F(\alpha,\mu)$ is empty unless $\alpha\in F_+$ is totally positive.

From the definitions, we now find that there is a canonical decomposition of $\mathcal{Y}$-stacks
\begin{equation}\label{intersection decomp}
\mathcal{Y} \times_\mathcal{M} \mathcal{Z}(m,\mu)
 = \bigsqcup_{  \substack{ \alpha \in F_+ \\   \mathrm{Tr}_{F/\Q}(\alpha) =m  } } \mathcal{Z}_F(\alpha,\mu) .
\end{equation}

\begin{proposition}\label{prop:support}
Suppose that $\alpha\in F_+$ and $\mu \in L^\vee /L$. Then $\mathcal{Z}_F(\alpha,\mu)$ is non-empty only if $\mathrm{Diff} (\alpha)$ consists of a \underline{single} prime $\mathfrak{p}$. In this case, $\mathcal{Z}_F(\alpha,\mu)$ is $0$-dimensional, and is supported at 
the unique prime $\mathfrak{q}\subset \co_E$ above $\mathfrak{p}$.
\end{proposition}

\begin{proof}
To begin, Proposition~\ref{prop:no special char 0} implies that the intersection of $\mathcal{Z}_F(\alpha,\mu)$ with ${Y}$ is empty. Therefore, $\mathcal{Z}_F(\alpha,\mu)$ is always either empty or $0$-dimensional.

If  $ z\in \mathcal{Z}_F(\alpha,\mu) (\F_\mathfrak{q}^\alg)$ for some prime $\mathfrak{q}\subset \co_E$, let
$y\in \mathcal{Y} (\F_\mathfrak{q}^\alg)$ be the point below it.  By the definition of $\mathcal{Z}_F(\alpha,\mu)$  the $E$-hermitian space 
$\mathscr{V}(A_y)_\Q$ represents $\alpha$.  In particular, $\mathscr{V}(A_y)\neq 0$, and so Proposition~\ref{prop:special ordinary} implies that the prime 
$\mathfrak{p}\subset \co_F$ below $\mathfrak{q}$  is nonsplit in $\co_E$. Moreover, Corollary~\ref{cor:special end structure} implies that the nearby hermitian space 
$\near \mathscr{V}$ represents $\alpha$.  This shows $\mathrm{Diff}(\alpha)=\{\mathfrak{p}\}$, by 
Remark \ref{rem:diff}, and everything follows easily.
\end{proof}

Set
\[
V_\mu(A_y[\infty]) = \prod_{\ell}V_{\mu_\ell}(A_y[\ell^\infty]) \subset V(A_y[\infty])_\Q.
\]
When viewed as a subset of the Hermitian space $\mathscr{V}(A_y[\infty])_\Q$, we will denote this set by $\mathscr{V}_\mu(A_y[\infty])$.

\begin{proposition}\label{prop:special end nearby lattice}
Suppose that $y\in \mathcal{Y}(\F_{\mathfrak{q}}^{\alg})$. Then the $\A_{f,E}$-linear isometry 
\[
\mathscr{V}(A_y[\infty])_{\Q}\xrightarrow{\simeq}\near \mathscr{V}
\] 
of Proposition~\ref{prop:nearby hermitian space ell} can be chosen so that, for each $\mu \in L^\vee/L$ the characteristic function of the image of $V_\mu(A_y[\infty])$ in $\near\widehat{\mathscr{V}}$ is the nearby Schwarz function $\near \varphi_\mu$ defined in~\eqref{eqn:near varphi mu decomp}.
\end{proposition}
\begin{proof}
The only non-trivial point is to show that the $E_p$-linear isometry
\[
\mathscr{V}(A_y[p^\infty])_{\Q}\xrightarrow{\simeq}\near \mathscr{V}_p
\]
can be chosen so that, for every $\mu_p\in L^\vee_p/L_p$, it identifies the characteristic function of $V_{\mu_p}(A_y[p^\infty])$ with the Schwarz function $\near \varphi_{\mu_p}$.

For this, first note that the sublattice $L_{\mathfrak{p}}\subset L_p$ transfers to a sublattice $\near L_{\mathfrak{p}}\subset \near L_p$, as do its cosets in $L^\vee_{\mathfrak{p}}$. Moreover, we have a canonical decomposition
\begin{equation*}
\mu_p + L_p = \bigsqcup_{(\mu_1,\mu_2)\in (\mu_p+ L_p)/(L_{\mathfrak{p}}\oplus L^{\mathfrak{p}})}  (\mu_1 + L_{\mathfrak{p}}) \times
\bigl({\mu_2}+ L^{\mathfrak{p}}_p\bigr).
\end{equation*}

Given this and Proposition~\ref{prop:V p to mf p}, it is enough to show that we can find an $E_{\mathfrak{q}}$-linear isometry
\[
V(A_y[p^\infty]_{\mathfrak{p}})_\Q \xrightarrow{\simeq} \near \mathscr{V}_{\mathfrak{p}}
\]
such that, for every $\mu_1\in L_{\mathfrak{p}}^\vee/L_{\mathfrak{p}}$, it carries the characteristic function of $V_{\mu_1}(A_y[p^\infty]_{\mathfrak{p}})$ to the Schwarz function $\near\varphi_\mu$.

By Proposition~\ref{prop:lubin-tate to kuga-satake}, we have an $E_{\mathfrak{q}}$-linear isomorphism
\[
V(A_y[p^\infty])_{\mathfrak{p}}\xrightarrow{\simeq}V(\mathcal{G}_{\mathfrak{q},y})_\Q
\]
carrying $V_{\mu_1}(A_y[p^\infty]_{\mathfrak{p}})$ to
\[
\bigsqcup_{\lambda\in \mu_1 + L_{\mathfrak{p}}}V_{\lambda}(\mathcal{G}_{\mathfrak{q},y}).
\]
Here, $\lambda$ runs over the cosets of $\Lambda_{\mathfrak{p}}$ in $\Lambda^\vee_{\mathfrak{p}}$ that are contained in $\mu_1 + L_{\mathfrak{p}}$, and we define $V_\lambda(\mathcal{G}_{\mathfrak{q},y})$ via an identification $\Lambda_{\mathfrak{p}} = \co_{E,\mathfrak{q}}$, which induces an identification $\Lambda^\vee_{\mathfrak{p}}/\Lambda_{\mathfrak{p}} = \mathfrak{d}_{E_{\mathfrak{q}}/F_{\mathfrak{p}}}^{-1}/\co_{E,\mathfrak{q}}$.

By Proposition~\ref{prop:lubin-tate special denom}, $V_\lambda(\mathcal{G}_{\mathfrak{q},y})$ is empty whenever 
$\ord_{\mathfrak{q}}(\lambda) < -n(\mathfrak{p}) + 1$. Therefore, we see that it is enough to construct an $E_{\mathfrak{q}}$-linear isometry 
\[
(V(\mathcal{G}_{\mathfrak{q},y})_\Q,\langle\cdot,\cdot\rangle) \xrightarrow{\simeq}
(E_{\mathfrak{q}},\beta x_1\overline{x}_2)
\]
such that, for every $\lambda\in \mathfrak{d}_{E_{\mathfrak{q}}/F_{\mathfrak{p}}}^{-1}/\co_{E,\mathfrak{q}}$ with 
$\ord_{\mathfrak{q}}(\lambda) > - n(\mathfrak{p})$, the isometry carries 
$V_\lambda(\mathcal{G}_{\mathfrak{q}, y })$ to $\lambda + \co_{E,\mathfrak{q}}$. 
Here, $\beta = \pi_{\mathfrak{p}} = \mathrm{Nm}(\pi_{\mathfrak{q}})$ if $\mathfrak{p}$ is inert in $E$, 
and $\beta \in\co_{F,\mathfrak{p}}^\times$ is such that $\chi_{\mathfrak{p}}(\beta) = -1$ if $\mathfrak{p}$ is ramified in $E$.

Such an isometry can be constructed using Propositions~\ref{prp:unramified_vcris} and~\ref{prp:ramified_vcris}.
\end{proof}

Recall the embedding $T\hookrightarrow\Aut^\circ(A)$ from Proposition~\ref{prop:tQ_action}, whose homological realizations induce maps
\[
\theta_?(H): T_{\Q_?} \to \Aut^\circ(\bm{H}_?)
\]
over $\mathcal{Y}$, which in turn give us maps
\[
\theta_?(V): T_{\Q_?} \to \Aut^\circ(\bm{V}_?).
\]

In particular, for each prime $\ell$, we obtain a canonical map
\[
\theta_{\ell}: T_{\Q_\ell} \to \Aut\bigl(\mathscr{V} ( A[\ell^\infty] )_Q\bigr),
\]
and thus a map
\[
\theta \define \theta_{\A_f}: T_{\A_f} \to \Aut\bigl(\mathscr{V}(A[\infty])_\Q\bigr).
\]

\begin{lemma}\label{lem:twisting isogeny}
The group $T(\A_f) /T(\Q)K_{L,0}$ acts simply transitively on 
the set of isomorphism classes in $\mathcal{Y}(\F_\mathfrak{q}^\alg)$, and every point
\[
y \in \mathcal{Y}(\F_\mathfrak{q}^\alg)
\]
has automorphism group $\Aut(y) = T(\Q) \cap K_{L,0}$. Moreover, for every  $t\in T(\A_f)$ there is a canonical isometry of $\A_{f,E}$-hermitian spaces
$
\mathscr{V}( A_{t\cdot y}[\infty])_\Q    \xrightarrow{\simeq} \mathscr{V}(A_y[\infty])_\Q 
$
identifying 
\[
\mathscr{V}_\mu(  A_{t\cdot y} [\infty] ) \xrightarrow{\simeq}  \theta(t)^{-1} \cdot \mathscr{V}_\mu(A_y [\infty] )
\]
as subsets of $\mathscr{V}(A_y[\infty])_\Q$.
\end{lemma}

\begin{proof}

By Proposition~\ref{prop:morphism_integral_Y}, $\mathcal{Y}\otimes_{\co_E}\co_{E,p}$ is finite \'etale  over $\co_{E,p}$. Therefore, the reduction map
\[
\mathcal{Y}(\mathrm{Frac}(W)^{\alg}) \to \mathcal{Y}(\F_{\mathfrak{q}}^{\alg})
\]
is an equivalence of groupoids. Furthermore, the map $\mathcal{Y}(\Q^\alg) \to \mathcal{Y}(\mathrm{Frac}(W)^{\alg})$ is also an equivalence of groupoids. Therefore, the first assertion follows from the fact that $T(\A_f)/T(\Q)K_{L,0}$ acts simply transitively on the set of isomorphism classes in $\mathcal{Y}(\Q^\alg)$ with isotropy group $T(\Q)\cap K_{L,0}$. This can be checked from the explicit description of the generic fiber ${Y}$ in \S\ref{ss:zero dimensional}.

The rest of the lemma follows easily from the definitions.
\end{proof}

\begin{proposition}\label{prop:point count}
Fix an  $\alpha\in F_+$  such that  $\mathrm{Diff}(\alpha)=\{\mathfrak{p}\}$. We have
\[
\sum _{  z\in \mathcal{Z}_F(\alpha,\mu)(\F_\mathfrak{q}^\alg)  } \frac{1}{ |\Aut(z) | }  
=
  \deg_\C( {Y})  \cdot  O(\alpha ,  \near \varphi_\mu ) ,
\]
\end{proposition}

\begin{proof}
The proof follows the same strategy as \cite[Theorem 3.5.3]{HowUnitaryCM}.
Pick any base point $y_0 \in \mathcal{Y}(\F_\mathfrak{q}^\alg)$,  and an isomorphism
\[
\mathscr{V}(A_{y_0} [\infty] )_\Q   \xrightarrow{\simeq} \near \widehat{\mathscr{V}}
\]
as in Proposition~\ref{prop:special end nearby lattice}. This identifies   the characteristic function of 
\[
 \mathscr{V} _\mu (A_{y_0} [\infty]  )  \subset \mathscr{V}(A_{y_0} [\infty] ) \otimes_{\widehat{\Z}} \A_f
\]
with the Schwartz function $\near \varphi_\mu \in S( \near \widehat{\mathscr{V}} )$ defined in \S \ref{ss:nearby lattices}.

Using  Lemma \ref{lem:twisting isogeny},  we  compute
\begin{align*}
\sum _{  z\in \mathcal{Z}_F(\alpha,\mu)(\F_\mathfrak{q}^\alg)  } \frac{1}{ |\Aut(z) | }
& =
\sum_{ y\in \mathcal{Y}(\F_\mathfrak{q}^\alg)}
  \sum_{  \substack{   x\in \mathscr{V}_\mu (A_y)    \\ \langle x,x\rangle  =\alpha   } }  \frac{ 1} {  | \Aut(y) | }
   \\
& =
\sum_{   t\in  T(\Q) \backslash T(\A_f) / K_{L,0} }   \sum_{  \substack{   x\in \mathscr{V}_\mu (A_{t\cdot y_0})    \\ \langle x,x\rangle  =\alpha  } }  \frac{1}{ |\Aut(t\cdot y_0) | } \\
& =
\frac{1  }{ | T(\Q) \cap K_{L,0} |  }
\sum_{   t\in   T(\Q) \backslash T(\A_f) / K_{L,0} }   \sum_{  \substack{   x\in  \mathscr{V} (A_{y_0}) \otimes \Q    \\ \langle x,x\rangle  =\alpha  } }   
\near \varphi_\mu \big(  \theta(t)   x \big)   .
\end{align*}
Next use the fact that  
\[
T (\Q)/ \mathrm{ker}(\theta) \iso T_{so}(\Q)=\{ s\in E^\times : s\overline{s} =1\}
\]
acts simply transitively on the set of $x \in \mathscr{V} (A_{y_0}) \otimes \Q$
with $\langle x ,x  \rangle=\alpha$.  By picking one such $x$, we compute
\begin{align*}
\sum _{  z\in \mathcal{Z}_F(\alpha,\mu)(\F_\mathfrak{q}^\alg)  } \frac{1}{ |\Aut(z) | }  
& =  \frac{ 1} {  | T(\Q) \cap K_{L,0}| }  \sum_{   \substack{   t\in   T(\Q) \backslash T(\A_f) / K_{L,0} \\  t' \in T(\Q)/\mathrm{ker}(\theta)  }    }
\near \varphi_\mu  \big(   \theta(t' t )   x \big)     \\
& =  
\frac{   \deg_\C( {Y} )   }{ \mathrm{Vol}\big(  T_{so}(\Q) \backslash  T_{so}(\A) \big)   }\int_{ T_{so}(\A_f)} \near \varphi_\mu (  s   x ) \, ds ,
\end{align*}
as desired.
\end{proof}

%%%%%%%%%%%%%%%%%%%%%%%%%%%%%%%%%%%%%%%

\subsection{Deformation theory}
\label{ss:def theory}

%%%%%%%%%%%%%%%%%%%%%%%%%%%%%%%%%%%%%%%

Fix an   $\alpha\in F_+$  such that 
$\mathrm{Diff}(\alpha)=\{\mathfrak{p}\}$ for a single prime $\mathfrak{p}\subset \co_F$.  Let 
$\mathfrak{q}\subset \co_E$ be the unique  prime above $\mathfrak{p}$. Assume that the rational prime
$p$ below $\mathfrak{p}$ is good for $L$.

Suppose that $y\in \mathcal{Y}(\F_{\mathfrak{q}}^\alg)$. For any integer $k\in\Z_{\geq 1}$  and any $\mu\in L^\vee_{\mathfrak{p}}/L_{\mathfrak{p}}$, set $A_k[p^\infty] = A_{\co_y/\pi_{\mathfrak{q}}^k\co_y}[p^\infty]$, $\mathcal{G}_{\mathfrak{q},k} = \mathcal{G}_{\mathfrak{q},\co_y/\pi_{\mathfrak{q}}^k\co_y}$, and
\begin{align*}
V_\mu(A_k[p^\infty]) & = V_\mu(A_{\co_y/\pi_{\mathfrak{q}}^k\co_y}[p^\infty])  \\
V_\mu(A_k[p^\infty]_{\mathfrak{p}})  = V_\mu(A_{\co_y/\pi_{\mathfrak{q}}^k\co_y}[p^\infty]_{\mathfrak{p}})\;&\; V_\mu(\mathcal{G}_{\mathfrak{q},k}) = V_\mu(\mathcal{G}_{\mathfrak{q},\co_y/\pi_{\mathfrak{q}}^k\co_y}).
\end{align*}

Consider the $1$-dimension $E_{\mathfrak{q}}$-vector space $V(A_y[p^\infty]_{\mathfrak{p}})_\Q$. By Proposition~\ref{prop:lubin-tate to kuga-satake}, it can be identified with the $E_{\mathfrak{q}}$-vector space $V(\mathcal{G}_{\mathfrak{q},y})_\Q$. 

\begin{proposition}
\label{prop:lt-to-ks deformation}
For every $k$, the above identification induces an equality
\[
V_\mu(A_k[p^\infty]_{\mathfrak{p}}) = V_\mu(\mathcal{G}_{\mathfrak{q},k})
\]
\end{proposition}
\begin{proof}
We will prove this by induction on $k$. When $k=1$, this follows from Proposition~\ref{prop:lubin-tate to kuga-satake}. It remains to show that the assertion holds for $k+1$ whenever it holds for $k$. 

Consider the de Rham realization $\bm{\Lambda}_{\mathfrak{p},\dR,\co_y}$ associated with the $K_{\mathfrak{q}}$-representation $\Lambda_{\mathfrak{p}}$. It is the reduction mod $\mathcal{E}(u)$ of the $\mathfrak{S}$-module $\varphi^*\mathfrak{M}(\Lambda_{\mathfrak{p}})$, and is naturally a filtered $\co_y$-submodule of $\bm{V}^\beef_{\dR,\co_y}$.

\begin{lemma}
\label{lem:filt beef lambda}
We have $\Fil^1\bm{\Lambda}_{\mathfrak{p},\dR,\co_y} = \Fil^1\bm{V}^\beef_{\dR,\co_y}$.
\end{lemma}
\begin{proof}
That the assertion holds after inverting $p$ is immediate from the construction. Therefore, it is enough to show that both $\Fil^1\bm{\Lambda}_{\mathfrak{p},\dR,\co_y}$ and $\Fil^1\bm{V}^\beef_{\dR,\co_y}$ have the same image in $\bm{\Lambda}^\vee_{\mathfrak{p},\dR,\co_y}$. Since $\Fil^1\bm{V}^\beef_{\dR,\co_y}$ is a direct summand of $\bm{V}^\beef_{\dR,\co_y}$, it actually suffices to show that its image in $\bm{\Lambda}^\vee_{\mathfrak{p},\dR,\co_y}$ is contained in $\bm{\Lambda}_{\mathfrak{p},\dR,\co_y}$.

Set
\[
\Fil^1\varphi^*\mathfrak{M}(V^\beef_p) = \{x\in \varphi^*\mathfrak{M}(V^\beef_p):\;\varphi(x)\in \mathcal{E}(u)\mathfrak{M}(V^\beef_p)\}.
\]
Then, using assertion~\eqref{kisin:derham} of Theorem~\ref{thm:kisin_p_divisible}, it can be checked that the image of $\Fil^1\varphi^*\mathfrak{M}(V^\beef_p)$ in $\bm{V}^\beef_{\dR,\co_y}$ is precisely $\Fil^1\bm{V}^\beef_{\dR,\co_y}$.

Now, given an element of $\Fil^1\bm{V}^\beef_{\dR,\co_y}$, choose a lift $x\in \Fil^1\varphi^*\mathfrak{M}(V^\beef_p)$. If $x'\in \varphi^*\mathfrak{M}(\Lambda^\vee_{\mathfrak{p}})$ is the image of $x$, then we find
\[
\varphi(x')\in \mathcal{E}(u)\mathfrak{M}(\Lambda^\vee_{\mathfrak{p}}).
\]
But then Lemma~\ref{lem:discriminant trivial} implies that
\[
x'\in \varphi^*\mathfrak{M}(\Lambda_{\mathfrak{p}}) + \mathcal{E}(u)\varphi^*\mathfrak{M}(\Lambda^\vee_{\mathfrak{p}})
\]
and hence that its image in $\bm{\Lambda}^\vee_{\mathfrak{p},\dR,\co_y}$ lies in $\bm{\Lambda}_{\mathfrak{p},\dR,\co_y}$.

% Let 
% \[
% \Fil^1\varphi^*\mathfrak{M}(\Lambda_{\mathfrak{p}})\subset \mathfrak{M}(\Lambda_{\mathfrak{p}})\text{ and }\Fil^1\varphi^*\mathfrak{M}(V^\beef_p) \subset \mathfrak{M}(V^\beef_p)
% \]
% be defined as just above~\eqref{bk:ddr}: For $\mathfrak{M} = \mathfrak{M}(\Lambda_{\mathfrak{p}}), \mathfrak{M}(V^\beef_p)$, we have
% \[
% \Fil^1\varphi^*\mathfrak{M} = \{x\in \varphi^*\mathfrak{M}:\;\varphi(x)\in \mathcal{E}(u)\mathfrak{M}\}.
% \]

% Then we claim 
% \[
% \Fil^1\varphi^* \mathfrak{M}(\Lambda_{\mathfrak{p}}) + \mathcal{E}(u)\mathfrak{M}(V^\beef_p) = \Fil^1\varphi^* \mathfrak{M}(V^\beef_p) + \mathcal{E}(u)\mathfrak{M}(V^\beef_p).
% \]

% Then, applying assertion~\ref{kisin:derham} of Theorem~\ref{thm:kisin_p_divisible} to the $p$-divisible groups $\mathcal{G}_{\mathfrak{q},\co_y}$ and $A^\beef[p^\infty]\vert_{\Spec~\co_y}$, respectively, we deduce that $\Fil^1\varphi^*\mathfrak{M}(\Lambda_{\mathfrak{p}})$ (resp. $\Fil^1\varphi^*\mathfrak{M}(V^\beef_p)$) is precisely the pre-image in $\varphi^*\mathfrak{M}(\Lambda_{\mathfrak{p}})$ (resp. $\varphi^*\mathfrak{M}(V^\beef_p)$) of $\Fil^1\bm{\Lambda}_{\mathfrak{p},\dR,\co_y}$ (resp. $\Fil^1\bm{V}^\beef_{\dR,\co_y}$).

\end{proof}

Write $\bm{V}^\beef_{\dR,k}$ (resp. $\bm{\Lambda}_{\mathfrak{p},\dR,k}$, $\bm{\Lambda}^\vee_{\mathfrak{p},\dR,k}$) for the reduction of $\bm{V}^\beef_{\dR,\co_y}$ (resp. $\bm{\Lambda}_{\mathfrak{p},\dR,\co_y}$, $\bm{\Lambda}_{\mathfrak{p},\dR,\co_y}$) mod $\pi_{\mathfrak{q}}^k$.

% Since $\bm{\Lambda}^\vee_{\mathfrak{p},\dR,k}$ is a free $\co_y/\pi_{\mathfrak{q}}^k\co_y\otimes_{\Z_p}\co_{E,\mathfrak{q}}$-module of rank $1$, we can consider its quotient
% \[
% \widetilde{\mathrm{Ob}}_k \define \bm{\Lambda}^\vee_{\mathfrak{p},\dR,k}\otimes_{\co_y\otimes_{\Z_p}\co_{E,\mathfrak{q}},1\otimes\overline{\iota}_0}\co_y/\pi_{\mathfrak{q}}^k\co_y,
% \]
% which is a free $\co_y/\pi_{\mathfrak{q}}^k\co_y$-module of rank $1$.

% Set
% \[
% \mathrm{Ob}_k = \pi_{\mathfrak{q}}^{k-1}\cdot\widetilde{\mathrm{Ob}}_k.
% \]
% This is an $\F_{\mathfrak{q}}^\alg$-vector space of dimension $1$.

Now, choose $x\in V_\mu(A_k[p^\infty]_{\mathfrak{p}})$, and let $x_{\mathrm{LT}}$ be the corresponding element of $V_\mu(\mathcal{G}_{\mathfrak{q},k})$. To finish the proof of the proposition, it remains to show that $x$ lifts to $V_\mu(A_{k+1}[p^\infty]_{\mathfrak{p}})$ if and only if $x_{\mathrm{LT}}$ lifts to an element of $V_\mu(\mathcal{G}_{\mathfrak{q},k+1})$.

Consider $x^\beef = (x,\tilde{\mu})\in V(A^\beef_{\co_y/\pi_{\mathfrak{q}}^k\co_y})$. Let $\bm{x}^\beef_{\cris}\in \bm{V}^\beef_{\dR,k+1}$ be the crystalline realization of $x^\beef$. By Proposition~\ref{prop:special_endomorphism_deform} and Lemma~\ref{lem:filt beef lambda}, $x^\beef$ lifts to $V(A^\beef_{\co_y/\pi_{\mathfrak{q}}^{k+1}\co_y})$, and hence $x$ lifts to $V_\mu(A_{k+1}[p^\infty]_{\mathfrak{p}})$, if and only if the functional
\[
[\bm{x}^\beef_{\cris},\_\_]:\bm{\Lambda}_{\mathfrak{p},\dR,k+1}\to \co_y/\pi_{\mathfrak{q}}^{k+1}\co_y
\]
lies in the annihilator of $\Fil^1\bm{\Lambda}_{\mathfrak{p},\dR,k+1}$. 

We claim that this annihilator is 
\[
\Fil^0\bm{\Lambda}^\vee_{\mathfrak{p},\dR,k+1} \define \ker\bigl(\bm{\Lambda}^\vee_{\mathfrak{p},\dR,k+1}\to\bm{\Lambda}^\vee_{\mathfrak{p},\dR,k+1}\otimes_{\co_y\otimes_{\Z_p}\co_{E,\mathfrak{q}},1\otimes\overline{\iota}_0}\co_y\bigr).
\]
Indeed, it is enough to check that the annihilator of $\Fil^1\bm{\Lambda}_{\mathfrak{p},\dR,\co_y}$ in $\bm{\Lambda}^\vee_{\mathfrak{p},\dR,\co_y}$ is 
\[
\ker\bigl(\bm{\Lambda}^\vee_{\mathfrak{p},\dR,\co_y}\to\bm{\Lambda}^\vee_{\mathfrak{p},\dR,\co_y}\otimes_{\co_y\otimes_{\Z_p}\co_{E,\mathfrak{q}},1\otimes\overline{\iota}_0}\co_y\bigr),
\]
which can be checked after inverting $p$, where it is easily verified.

Now, by Proposition~\ref{prop:lifting end}, $x_{\mathrm{LT}}$ has a crystalline realization $\bm{x}_{\mathrm{LT},\cris}\in \bm{\Lambda}^\vee_{\dR,k+1}$, and lifts to $V_\mu(\mathcal{G}_{\mathfrak{q},k+1})$ if and only if $\bm{x}_{\mathrm{LT},\cris}$ lies in $\Fil^0\bm{\Lambda}^\vee_{\mathfrak{p},\dR,k+1}$. 

To finish, it now suffices to observe that 
\begin{equation}\label{eqn:equality realizations}
\bm{x}_{\mathrm{LT},\cris}=[\bm{x}^\beef_{\cris},\_\_]\in \bm{\Lambda}^\vee_{\mathfrak{p},\dR,k+1}. 
\end{equation}
For this, let $S$ be the $p$-adic completion of the divided power envelope of the surjection $W[u]\xrightarrow{u\mapsto \iota_0(\pi_{\mathfrak{q}})}\co_y$, and set 
\[
\mathcal{M}(\Lambda^\vee_{\mathfrak{p}}) \define \varphi^*\mathfrak{M}(\Lambda^\vee_{\mathfrak{p}})\otimes_{\mathfrak{S}}S.
\]

The $\varphi$-module structure on $\mathfrak{M}(\Lambda^\vee_{\mathfrak{p}})[\mathcal{E}(u)^{-1}]$ gives us an isomorphism
\[
\varphi:\varphi^*\mathcal{M}(\Lambda^\vee_{\mathfrak{p}})[p^{-1}]\xrightarrow{\simeq}\mathcal{M}(\Lambda^\vee_{\mathfrak{p}}).
\]

Moreover, by a variation of Dwork's trick (see~\cite[6.2.1.1]{breuil:griffiths}), the reduction map
\[
\mathcal{M}(\Lambda^\vee_{\mathfrak{p}})[p^{-1}]\to \mathcal{M}(\Lambda^\vee_{\mathfrak{p}})[p^{-1}]\otimes_SW = \bm{\Lambda}^\vee_{\mathfrak{p},\cris,y}[p^{-1}]
\]
induces a bijection on $\varphi$-invariant elements. Let $\bm{x}_0\in \bm{\Lambda}^\vee_{\mathfrak{p},\cris,y}$ be the crystalline realization of $x$ viewed as an element of $V_\mu(A_y[p^\infty]_{\mathfrak{p}})$, and let $\tilde{\bm{x}}_0\in \mathcal{M}(\Lambda^\vee_{\mathfrak{p}})[p^{-1}]$ be its unique $\varphi$-invariant lift. 

If $k<e$, then the image of $\tilde{\bm{x}}_0$ in $\mathcal{M}(\Lambda^\vee_{\mathfrak{p}})[p^{-1}]\otimes_SW[u]/(u^{k+1})$ actually lies in
\[
\mathcal{M}(\Lambda^\vee_{\mathfrak{p}})\otimes_SW[u]/(u^{k+1})
\]
and, by virtue of its $\varphi$-invariance, is necessarily the crystalline realization of both $x$ and $x_{\mathrm{LT}}$ along the divided power thickening $W[u]/(u^{k+1})\xrightarrow{u\mapsto \iota_0(\pi_{\mathfrak{q}})}\co_y/\pi_{\mathfrak{q}}^k\co_y$. 

If $k\geq e$, then, once again by virtue of its $\varphi$-invariance, $\tilde{\bm{x}}_0$ is the crystalline realization of both $x$ and $x_{\mathrm{LT}}$ along the divided power thickening $S\to \co_y/\pi_{\mathfrak{q}}^k\co_y$.

From these observations, the required identity~\eqref{eqn:equality realizations} easily follows. 
\end{proof}

Define a function
\[
\mathrm{ord}_{\mathfrak{q}}:\;V(A_y[p^\infty]_{\mathfrak{p}})_\Q \to \Z,
\]
given by two defining properties:
\begin{itemize}
	\item If $a\in E_{\mathfrak{q}}$, and $x\in V(A_y[p^\infty]_{\mathfrak{p}})$, then
	\[
     \mathrm{ord}_{\mathfrak{q}}(a\cdot x) = \mathrm{ord}_{\mathfrak{q}}(a) + \mathrm{ord}_{\mathfrak{q}}(x).
	\]
	\item If $x\in V(\mathcal{G}_{\mathfrak{q},y})$ is an $\co_{E_{\mathfrak{q}}}$-module generator, then 
    \[
	\mathrm{ord}_{\mathfrak{q}}(x) = \begin{cases}
	1,&\text{ if $\mathfrak{q}$ is unramified over $F$};\\
	n(\mathfrak{p}) = \ord_{\mathfrak{q}}(\mathfrak{d}_{E/F}),&\text{ if $\mathfrak{q}$ is ramified over $F$}.
	\end{cases}
	\]
\end{itemize}

Our definition of the function $\ord_{\mathfrak{q}}$ is justified by the following result.

\begin{proposition}
\label{prop:deformation special}
 Suppose that $\mu\in L_{\mathfrak{p}}^\vee/L_{\mathfrak{p}}$, and that $x\in V_\mu(A_y[p^\infty]_{\mathfrak{p}})$. Then $x$ lifts to $V_\mu(A_{k}[p^\infty]_{\mathfrak{p}})$ if and only if $\mathrm{ord}_{\mathfrak{q}}(x)\geq k$.
\end{proposition}
\begin{proof}
This is immediate from Proposition~\ref{prop:lt-to-ks deformation} and Theorem~\ref{thm:deformation special}.
\end{proof}

\begin{theorem}\label{thm:local ring}
At any point $z\in \mathcal{Z}_F(\alpha,\mu)(\F_\mathfrak{q}^\alg)$ we have
\[
\length \left( \co_{\mathcal{Z}_F(\alpha,\mu),z}\right) = \ell_{\mathfrak{p}}(\alpha),
\]
where $\ell_{\mathfrak{p}}(\alpha)$ is defined as in Proposition~\ref{prop:explicit siegel-weil}.
\end{theorem}
\begin{proof}
The point $z$ corresponds to a point $y\in \mathcal{Y}(\F_{\mathfrak{q}}^{\alg})$ equipped with a special endomorphism $x\in V_\mu(A_y)$ satisfying $\langle x,x\rangle = \alpha$. By Serre-Tate theory its deformation theory is governed by the induced endomorphism $x_p \in V_{\mu_p}(A_y[p^\infty])$. By Proposition~\ref{prop:V p to mf p}, there is a unique pair
\[
(\mu_1,\mu_2) \in \frac{\mu_p + L_p}{L_{\mathfrak{p}}+ L^{\mathfrak{p}}_p} \subset \frac{L_{\mathfrak{p}}^\vee}{L_{\mathfrak{p}}} \oplus \frac{L^{\mathfrak{p},\vee}_p}{L^{\mathfrak{p}}_p},
\]
together  with a unique $x_{\mathfrak{p}}\in V_{\mu_1}(A_y[p^\infty]_{\mathfrak{p}})$  and $v\in \mu_2 + L^{\mathfrak{p}}_p$, such that
\[
x_p = x_{\mathfrak{p}} + v.
\]
Moreover, $\ord_{\mathfrak{p}}(\alpha) = \ord_{\mathfrak{p}}(\langle x_{\mathfrak{p}},x_{\mathfrak{p}}\rangle)$.

Also, by the same proposition, the deformation theory of $x_p$ is governed by that of $x_{\mathfrak{p}}$. More precisely, $x_p$ lifts to $V_{\mu_p}(A_k[p^\infty])$ if and only if $x_{\mathfrak{p}}$ lifts to $V_{\mu_1}(A_k[p^\infty]_{\mathfrak{p}})$. By Proposition~\ref{prop:deformation special}, this is equivalent to the condition $\ord_{\mathfrak{q}}(x_{\mathfrak{p}})\geq k$.

Therefore, to finish, we must show:
\begin{equation}\label{eqn:ord computation}
\ord_{\mathfrak{q}}(x_{\mathfrak{p}}) = \ell_{\mathfrak{p}}(\alpha).
\end{equation}
Now, note that the Hermitian form on $V(A_y[p^\infty]_{\mathfrak{p}})_\Q$ is $\xi_{\mathfrak{p}}$-times the natural Hermitian form $\langle\cdot,\cdot\rangle_{\mathrm{LT}}$ on $V(\mathcal{G}_{\mathfrak{q},y})_\Q$, and that $\ord_{\mathfrak{p}}(\xi_{\mathfrak{p}}) = -m(\mathfrak{p})$. 

Moreover, using Propositions~\ref{prp:unramified_vcris} and~\ref{prp:unramified_vcris}, we find that, if $x_0\in V(\mathcal{G}_{\mathfrak{q},y})$ is an $\co_{E,\mathfrak{q}}$-module generator, then
\[
\ord_{\mathfrak{p}}(\langle x_0,x_0\rangle_{\mathrm{LT}}) =  \begin{cases}
1 = -1 + 2\cdot\ord_{\mathfrak{q}}(x_0),&\text{ if $\mathfrak{q}$ is unramified over $F$};\\
0 = -n(\mathfrak{p}) + \ord_{\mathfrak{q}}(x_0),&\text{ if $\mathfrak{q}$ is ramified over $F$}.
\end{cases}
\]

Combining all this, we find that $\ord_{\mathfrak{p}}(\alpha)$ is equal to
\[
 \ord_{\mathfrak{p}}(\langle x_{\mathfrak{p}},x_{\mathfrak{p}}\rangle) = \begin{cases}
-m(\mathfrak{p}) - 1 + 2\cdot\ord_{\mathfrak{q}}(x_{\mathfrak{p}}),&\text{ if $\mathfrak{q}$ is unramified over $F$};\\
-m(\mathfrak{p}) - n(\mathfrak{p}) + \ord_{\mathfrak{q}}(x_{\mathfrak{p}}),&\text{ if $\mathfrak{q}$ is ramified over $F$.}
\end{cases}
\]

Comparing this with the formulas for $\ell_{\mathfrak{p}}(\alpha)$ in Proposition~\ref{prop:explicit siegel-weil} gives us~\eqref{eqn:ord computation} and hence the theorem.
\end{proof}

\subsection{Calculation of arithmetic degrees: the end of the proof of Theorem~\ref{thm:arithmetic BKY}}
\label{ss:arithmetic intersection proof}

%%%%%%%%%%%%%%%%%%%%%%%%%%%%%%%%%%%%%%%

\begin{theorem}\label{thm:degree}
Suppose  $\alpha\in F_+$  and  $\mu \in L^\vee / L$.  Assume that 
$\mathrm{Diff}(\alpha) = \{ \mathfrak{p} \}$ consists of a single prime of $\co_F$, which lies above a rational prime $p$  that is good for $L$. If we denote by 
\[
 \widehat{\mathcal{Z}}_F (\alpha,\mu) \in \widehat{\mathrm{Pic}}( \mathcal{Y}) 
\]
 the divisor $\mathcal{Z}_F(\alpha,\mu)$ on $\mathcal{Y}$ endowed with the trivial Green function, then 
\[
\frac{   \widehat{\deg} \big( \widehat{\mathcal{Z}}_F (\alpha,\mu) \big)  } {  \deg_\C( {Y})  }
= - \frac{a_F(\alpha,\mu)}{ \Lambda(0, \chi ) }.
\]
\end{theorem}

\begin{proof}
Combine Propositions ~\ref{prop:explicit siegel-weil}, and~\ref{prop:point count} with Theorem~\ref{thm:local ring}.
\end{proof}

Given $a,b\in\R$, we will write $a\approx_L b$ to mean that $a-b$ is a $\Q$-linear combination of $\{ \log(p) : p \mid D_{bad,L} \}$.

\begin{proposition}\label{prop:constant term eval}
We have
\[
\frac{ a(0,0) }{ \Lambda(0,\chi) } \approx_L - \frac{ 2 \Lambda'(0,\chi)}{ \Lambda(0,\chi) }.
\]
If $\mu\neq 0$, then
\[
 \frac{ a(0,\mu) }{ \Lambda(0,\chi) } \approx_L  
 0.
\]
%Moreover, there is a constant $B_d$, depending only on $d$, such that
%\[
%| b_p (L,\mu) | < B_d \cdot \log|D_{bad}|
%\]
%for all $p$ and  $\mu$. (To emphasize, the constant $B_d$ does not depend on the CM field $E$, the hermitian space $(\mathscr{V} , \mathscr{Q})$,
%or the lattice $L\subset V=\mathscr{V}$).
\end{proposition}

\begin{proof}
Let $\varphi=\varphi_\mu$, so that we have a factorization $\varphi=\otimes \varphi_p$ over the rational primes, and 
\begin{equation}\label{eqn:a 0 mu}
\frac{ a(0,\mu) } { \Lambda (0,\chi) } =    \frac{ a_F(0,\varphi) }{ \Lambda(0,\chi) } 
=   -2 \varphi (0) \frac{ \Lambda'(0,\chi)}{\Lambda(0,\chi)} - M'(0,\varphi).
\end{equation}
by Proposition \ref{prop:coarse constant}.

Fix a prime $p$, and suppose that we have $\varphi_p = \sum_i\varphi_i$, where each $\varphi_i$ admits a factoring $\varphi_i = \otimes_{\mathfrak{p}\mid p}\varphi_{i,\mathfrak{p}}$ over primes $\mathfrak{p}\subset \co_F$ above $p$. Then, for each $i$, by Proposition~\ref{prop:coarse constant}, we obtain a factoring
\[
M_p(s,\varphi_i)  = \prod_{\mathfrak{p}\mid p}M_{\mathfrak{p}}(s,\varphi_{i,\mathfrak{p}}),
\]
where, for any $\mathfrak{p}\mid p$, $M_{\mathfrak{p}}(s,\varphi_{i,\mathfrak{p}})$ is a rational function in $N(\mathfrak{p})^s$. Therefore, $M_{\mathfrak{p}}(0,\varphi_{i,\mathfrak{p}})$ is a rational number, and $M'_{\mathfrak{p}}(0,\varphi_{i,\mathfrak{p}})$ is a rational multiple of $\log N(\mathfrak{p})$.

Moreover, if $p$ is a good prime, then, by~\eqref{eqn:mu to lambda decomp}, we can choose our decomposition to be
\[
\varphi_{\mu_p} = \bigotimes_{\lambda_p\in\mu_p + L_p} \varphi_{\lambda_p},
\]
where $\lambda_p$ ranges over representative of cosets of $\Lambda_p$ in $\Lambda_p^\vee$ contained in $\mu_p + L_p$.

By Corollary~\ref{cor:good prime M constant}, $M_{\mathfrak{p}}(s,\varphi_{\lambda_\mathfrak{p}})$ is constant, and hence $M_{\mathfrak{p}}'(0,\varphi_{\lambda_{\mathfrak{p}}}) = 0$, for all primes $\mathfrak{p}\mid p$,   It now follows that $M'(0,\varphi)$ is a $\Q$-linear combination of $\log(p)$ with $p\mid D_{bad}$.

The identity~\eqref{eqn:a 0 mu} now gives us the proposition.
\end{proof}

\begin{proof}[Proof of Theorem \ref{thm:arithmetic BKY}]
Recalling that 
\[
\mathcal{Z}(f) = \sum_{m>0} \sum_{ \mu \in L^\vee / L} c_f^+(-m,\mu) \mathcal{Z}(m,\mu), 
\]
the stack decomposition 
\[
 \mathcal{Z}(m,\mu) \times_\mathcal{M}  \mathcal{Y}   = \bigsqcup_{   \substack{  \alpha\in F_+  \\ \mathrm{Tr}_{F/\Q}(\alpha) = m   }  }
  \mathcal{Z}_F (\alpha,\mu)
\]
of  (\ref{intersection decomp}) implies
\[
 \frac{  [ \widehat{\mathcal{Z}}(f) : \mathcal{Y} ]   }{   \deg_\C ({Y})  }  =
 \frac{ \Phi (f ,\mathcal{Y}^\infty ) }{  2  \deg_\C ({Y})  } 
+   \sum_{ \substack{     \mu\in L^\vee/L \\ m > 0 }} c_f^+(m,\mu)  \sum_{   \substack{  \alpha \in F_+ \\ \mathrm{Tr}_{F/\Q}(\alpha) =m    }     }
\frac{   \widehat{\deg} \big( \widehat{\mathcal{Z}}_F (\alpha,\mu) \big)  } {  \deg_\C( {Y})  } .
\]

For any $\alpha\in F_+$ and any $\mu\in L^\vee /L$ we have
\begin{equation}\label{approx degree}
\frac{  \widehat{\deg} \big( \widehat{ \mathcal{Z}}_F (\alpha,\mu)   \big) }  {  \deg_\C( {Y})  }
\approx_L 
- \frac{   a_F(\alpha ,\mu) } {  \Lambda( 0 , \chi )  }  .
\end{equation}
Indeed, if $|\mathrm{Diff}(\alpha)| >1$ then  Propositions \ref{prop:coefficient support} and \ref{prop:support}
imply that both sides of (\ref{approx degree}) vanish.  
If $\mathrm{Diff}(\alpha)=\{ \mathfrak{p} \}$ then let $p$ be the rational prime below $\mathfrak{p}$.
If  $p\nmid D_{bad}$  then the relation (\ref{approx degree})  follows from Theorem \ref{thm:degree}.
If $p\mid D_{bad}$ then both sides of (\ref{approx degree}) are $\approx 0$ by 
 Propositions \ref{prop:coefficient support} and \ref{prop:support}.

Combining (\ref{approx degree}) and (\ref{eisenstein decomp}) shows that
\[
\sum_{   \substack{  \alpha\in F^\times  \\ \mathrm{Tr}_{F/\Q}(\alpha) = m   }  }   
\frac{  \widehat{\deg} \big( \widehat{ \mathcal{Z}}_F (\alpha,\mu)   \big) }  {  \deg_\C( {Y})  }
\approx_L 
- \frac{   a(m ,\mu) } {  \Lambda( 0 , \chi )  }     ,
\]
and so
\[
 \frac{  [ \widehat{\mathcal{Z}}(f) : \mathcal{Y} ]   }{   \deg_\C ({Y})  }    \approx_L 
\frac{\Phi(f,   \mathcal{Y}^\infty )  }{2  \deg_\C ({Y})  } 
-    \sum_{ \substack{  \mu\in L^\vee/L \\   m > 0   }}   \frac{  a(m ,\mu)  \cdot c_f^+(m,\mu)  } {  \Lambda( 0 , \chi )  }  .
\]
Comparing with Theorem \ref{thm:BKY}, and using the approximate identity $a(0,\mu)\approx_L 0$ from Proposition~\ref{prop:constant term eval} for $\mu\not=0$,   shows that
\[
 \frac{  [ \widehat{\mathcal{Z}}(f) : \mathcal{Y} ]   }{   \deg_\C ({Y})  }    \approx_L   
-  \frac{ \mathcal{L}'(0 , \xi(f) )  }  {  \Lambda( 0 , \chi )  }  
+     \frac{ a(0, 0 ) \cdot  c_f^+(0,0) } {  \Lambda( 0 , \chi )  } ,
\]
as desired.
\end{proof}

%%%%%%%%%%%%%%%%%%%%%%%%%%%%%%%%%%

\section{The height of the tautological bundle}
\label{s:taut height}

%%%%%%%%%%%%%%%%%%%%%%%%%%%%%%%%%%%

For applications to the computations of heights of line bundles, and to Colmez's conjecture, it will be useful to have a results more general than Theorem~\ref{thm:arithmetic BKY}, involving restrictions of Borcherds products from larger GSpin Shimura varieties. In these section, we pursue such generalizations, which are mostly of a formal nature, given what has come before.

%%%%%%%%%%%%%%%%%%%%%%%%%%%%%%%%%%

\subsection{Enlarging the Shimura variety}
\label{ss:enlarged setup}

%%%%%%%%%%%%%%%%%%%%%%%%%%%%%%%%%%%

Suppose that we have a quadratic space $(V^\beef, Q^\beef)$  of signature $(n^\beef,2)$, a maximal lattice $L^\beef \subset V^\beef$ of
discriminant  $D_{L^\beef} = [ L^{\beef,\vee} : L^\beef]$, and an isometric embedding 
\[
(V,Q)  \hookrightarrow (V^\beef, Q^\beef).
\]
satisfying $L\subset L^\beef$.  Define a positive definite $\Z$-quadratic space
 \[
 \Lambda = \{ x\in L^\beef : x\perp L\}
 \]
 of rank $n^\beef-2d+2$,   so that   $L\oplus \Lambda \subset L^\beef$ with finite index, and  $V^\beef  = V \oplus \Lambda_\Q .$

From this data, we obtain maps of $\Z$-stacks
\[
\mathcal{Y}  \to \mathcal{M} \to   \mathcal{M}^\beef,
\]
where $\mathcal{M}$ and  $\mathcal{M}^\beef$ are  the integral models for the orthogonal Shimura varieties associated with the lattices $L$ and $L^\beef$, respectively, and $\mathcal{Y}=\mathcal{Y}_{K_{L,0}}$ is the stack appearing in Proposition \ref{prop:morphism_integral_Y}.

%%%%%%%%%%%%%%%%%%%%%%%%%%%%%%%%%%

\subsection{The archimedean contribution}

%%%%%%%%%%%%%%%%%%%%%%%%%%%%%%%%%%%

From the $\Z$-quadratic space $\Lambda$ we may form the representation $\omega_\Lambda$ of 
$\widetilde{\SL}_2(\Z)$ on the  finite dimensional  subspace
\[
S_\Lambda \subset S( \widehat{\Lambda}_\Q )
\]
of $\C$-valued functions on $\Lambda^\vee /\Lambda$, exactly as in \S \ref{ss:harmonic forms}, and the contragredient representation 
$\omega_\Lambda^\vee$ on the $\C$-linear dual $S_\Lambda^\vee$.

The theta series
\[
\vartheta_\Lambda(\tau) = \sum_m \rho_\Lambda(m) \cdot q^m \in M^!_{   \frac{n^\beef}{2} -d+1     }(\omega_\Lambda^\vee)
\]
has Fourier coefficients $\rho_\Lambda(m) \in S_\Lambda^\vee$ defined by 
\[
\rho_\Lambda(m,\varphi) =  \sum_{ x\in \Lambda^\vee } \varphi(x) 
\]
for any $\varphi\in S_\Lambda$.  Letting $\varphi_\mu$ denote the characteristic function of $\mu \in \Lambda^\vee /\Lambda$,
we often write $\rho_\Lambda(m,\mu)=\rho_\Lambda(m,\varphi_\mu)$.

Given a pair 
$
(\mu_1,\mu_2) \in (L^\vee / L) \oplus (\Lambda^\vee / \Lambda)
$ 
and a $\mu \in L^{\beef,\vee} / L^\beef$
we write, by abuse of notation, $\mu_1+\mu_2 = \mu$ to mean that the map
\[
(L^\vee / L) \oplus (\Lambda^\vee / \Lambda) \to (L^\vee \oplus \Lambda^\vee) / L^\beef 
\]
induced by the inclusions  
\[
L\oplus \Lambda \subset L^\beef \subset L^{\beef,\vee} \subset L^\vee \oplus \Lambda^\vee
\] 
takes   $(\mu_1,\mu_2) \mapsto \mu$.

\begin{proposition}\label{prop:CM_value_beef}
Fix any  weakly holomorphic modular form $f \in M_{ 1-n^\beef/2 } ^! (\omega_{L^\beef})$ with integral principal part, and let $\Phi^\beef(f)$ be the 
corresponding Green function on $\mathcal{M}^\beef$, as in \S \ref{ss:harmonic forms}.  If we set
\[
 \mathcal{Y}^\infty  = \mathcal{Y} \times_{ \Spec(\Z) } \Spec( \C),
\]
and define $\Phi^\beef (f ,\mathcal{Y}^\infty)$  as in Theorem \ref{thm:BKY}, then
\[
   \frac{ \Phi^\beef (f ,\mathcal{Y}^\infty) }{  2  \deg_\C ({Y})  } 
= 
   \sum_{  \substack{  \mu \in L^{\beef,\vee} / L^\beef  \\ m \in \Q }  } c_f (-m,\mu)    
   \sum_{  \substack{   m_1 + m_2 =m  \\ \mu_1+\mu_2 = \mu  }  }   \frac{  a(m_1,\mu_1) \rho_\Lambda(m_2,\mu_2) }{  \Lambda( 0 , \chi )  }.
\]
 \end{proposition}

\begin{proof}
The isomorphism
\[
S( \widehat{V}^\beef ) \iso S(\widehat{V}) \otimes S(\widehat{\Lambda}_\Q),
\]
together with the tautological pairing between $S(\widehat{\Lambda}_\Q)$ and its dual,  induce a map
\[
S( \widehat{V}^\beef )  \otimes S(\widehat{\Lambda}_\Q)^\vee \to S(\widehat{V}) .
\]
As $L\oplus \Lambda \subset L^\beef$, this restricts to a map
$
S_{L^\beef} \otimes  S_\Lambda^\vee \to S_L,
$
which we call  \emph{tensor contraction}, and denote by $\varphi_1 \otimes \varphi_2 \mapsto  \varphi_1 \odot \varphi_2$.

There is an induced map on spaces of weakly holomorphic forms 
\[
M^!_k  (\omega_{L^\beef})  \otimes M^!_\ell( \omega_\Lambda^\vee) \to M^!_{k+\ell}  (\omega_{L}) 
\]
for any half-integers $k$ and $\ell$.  In particular, $f\mapsto f \odot \vartheta_\Lambda$ defines a linear map
\[
M^!_{1- \frac{n^\beef}{2} }(\omega_L^\beef) \to M^!_{2-d}(\omega_L  ).
\]
In terms of $q$-expansions,
\[
( f \odot \vartheta_\Lambda)(\tau) = \sum_{     m\gg -\infty } \sum_{ \mu \in L^\vee / L     }
c_{ f \odot \vartheta_\Lambda }(m ,\mu) \varphi_\mu  \cdot q^m
\]
where
\begin{equation}\label{contraction coefficients}
c_{ f \odot \vartheta_\Lambda }(m ,\mu)  
= \sum_{ m_1 + m_2 =m } 
\sum_{  \substack{    \mu_2 \in \Lambda^\vee / \Lambda   \\   \mu+\mu_2 \in L^{\beef ,\vee} / L\oplus \Lambda }  }
c_f(m_1, \mu+\mu_2 ) \cdot \rho_\Lambda(m_2,\mu_2 ).
\end{equation}

The essential observation is this:
If we pull back the Green function $\Phi^\beef(f)$ for the divisor $\mathcal{Z}^\beef(f)(\C)$ on $\mathcal{M}^\beef(\C)$ via the map
 $\mathcal{M}(\C) \to \mathcal{M}^\beef(\C)$, we obtain the Green function $\Phi(f\odot \vartheta_\Lambda)$ for the 
divisor $\mathcal{Z}( f\odot \vartheta_\Lambda )(\C)$ on $\mathcal{M}(\C)$.  This is clear from the 
factorization \cite[(4.16)]{BY} of  Siegel theta functions, and the  construction of the 
Green functions as regularized theta lifts as in \cite{BKY, BY}.

As $\mathcal{Y}^\infty(\C) \to \mathcal{M}^\beef(\C)$ factors through $\mathcal{M}(\C)$, we find
\[
\frac{ \Phi^\beef (f ,\mathcal{Y}^\infty) }{  2  \deg_\C ({Y})  } 
=  \frac{ \Phi  (f \odot \vartheta_\Lambda ,\mathcal{Y}^\infty) }{  2  \deg_\C ({Y})  } .
\]
We may  apply the result of Bruinier-Kudla-Yang, as stated in Theorem \ref{thm:BKY}, directly to the right hand side.  This gives
\begin{align*}
\frac{ \Phi^\beef (f ,\mathcal{Y}^\infty) }{  2  \deg_\C ({Y})  }  
& =  - \sum_{  \substack{  \mu_1 \in L^\vee / L   \\  m_1 \in \Q } }
\frac{   a(m_1,\mu_1) \cdot c_{ f\odot \vartheta_\Lambda} (  -m_1, \mu_1)   }{ \Lambda(0,\chi)  }  \\
%& =  - \sum_{  \substack{  \mu_1 \in L^\vee / L   \\  m_1 + m_2+ m_3=0 } } \frac{   a(m_1,\mu_1)   }{ \Lambda(0,\chi)  } 
%\sum_{  \substack{    \mu_2 \in \Lambda^\vee / \Lambda   \\   \mu_1+\mu_2 \in L^{\beef ,\vee} / L\oplus \Lambda }  } 
%c_f( m_3 , \mu_1+\mu_2 ) \cdot \rho_\Lambda(m_2,\mu_2 ) \\
& =  - \sum_{  \substack{     \mu \in L^{\beef,\vee} / L^\beef    \\  m_1 + m_2+ m_3=0 } }
c_f( m_3 , \mu  ) 
\sum_{ \mu_1+\mu_2 =\mu  }  \frac{   a(m_1,\mu_1)  \rho_\Lambda(m_2,\mu_2 )   }{ \Lambda(0,\chi)  }  
\end{align*}
where the second equality follows from (\ref{contraction coefficients}). 
\end{proof}

%%%%%%%%%%%%%%%%%%%%%%%%%%%%%%%%%%

\subsection{An extended  arithmetic intersection formula}

%%%%%%%%%%%%%%%%%%%%%%%%%%%%%%%%%%%

Fix any  weakly holomorphic modular form $f \in M_{ 1-n^\beef/2 } ^! (\omega_{L^\beef})$ with integral principal part.  Let 
\[
\mathcal{Z}^\beef(f) = \sum_{ m>0} \sum_{ \mu \in L^{\beef,\vee}/L^\beef } c_f( - m,\mu) \cdot \mathcal{Z}^\beef(m,\mu),
\]
be the corresponding divisor on $\mathcal{M}^\beef$, and denote by
\[
\widehat{ \mathcal{Z}}^\beef(f)   = \big(    \mathcal{Z}^\beef(f)  ,   \Phi^\beef(f)  \big)
\in \widehat{\mathrm{CH}}^1(\mathcal{M}^\beef)
\]
the corresponding arithmetic divisor.

In what follows we will frequently demand that  $f$ satisfy the following hypothesis 
with respect to the quadratic submodule $L\subset L^\beef$ and its orthogonal complement $\Lambda = L^\perp\subset L^\beef$.

\begin{hypothesis}\label{hyp:proper}
If $m>0$, and if there  an element $x\in D_{L^\beef}^{-1}\Lambda$ with $Q(x) = m$,  then $c_f(-m,\mu) = 0$ for all $\mu \in L^{\beef,\vee} / L^\beef$.\end{hypothesis}

\begin{lemma}\label{lem:improper_intersection}
If the image of $\mathcal{Y}(\C) \to \mathcal{M}^\beef(\C)$ intersects the support of $\mathcal{Z}^\beef(m,\mu)(\C)$, then
there is an  $x\in D_{L^\beef}^{-1}  \Lambda$ with $Q(x)=m$. Therefore, if Hypothesis~\ref{hyp:proper} holds, then $\mathcal{Z}^\diamond(f)$ intersects $\mathcal{Y}$ properly.
\end{lemma}
\begin{proof}
Suppose that the image of $\mathcal{Y}(\C)$ intersects the support of $\mathcal{Z}^\beef(m,\mu)(\C)$. By Corollary~\ref{prop:Zmu_functoriality}, this implies that there exist $m_1,m_2\in\Q_{\geq 0}$ with $m_1+m_2 = m$ and 
\[
(\mu_1,\mu_2) \in \frac{L^{\beef,\vee}}{L\oplus\Lambda} \subset (L^\vee/L) \oplus (\Lambda^\vee/\Lambda)
\]
such that $\mathcal{Y}(\C)$ intersects the support of $\mathcal{Z}(m_1,\mu_1)\times\Lambda_{m_2,\mu_2}$, where
\[
\Lambda_{m_2,\mu_2} = \{x\in {\mu_2}+\Lambda : Q(x)=m_2\}.
\] 
By Proposition~\ref{prop:no special char 0}, for any point $y\in {Y}(\C)$ we must have $V(A_y) = 0$.  Thus the only way that 
$y$ can meet $\mathcal{Z}(m_1, \mu_1)$ is if  $m_1=0$ and $\mu_1=0$.   It now follows that there exists $x\in \Lambda_{m,\mu_2}$ with 
\[
(0,\mu_2)\in\frac{L^{\beef,\vee}}{L\oplus\Lambda},
\]
and from this  we deduce that there exists $x\in D_{L^\beef}^{-1}\Lambda$ with $Q(x) = m$. 

The second assertion is  clear from the first.
\end{proof}

\begin{proposition}\label{prop:height_pairing_beef}
Under Hypothesis~\ref{hyp:proper}, we have:
\[
\frac{  [ \widehat{\mathcal{Z}}^{\beef}(f) : \mathcal{Y} ]   }{   \deg_\C ({Y})  }   \approx_L
 c_f(0,0)\frac{a(0,0)}{\Lambda(0,\chi)}.
\]
\end{proposition}
\begin{proof}
Corollary~\ref{prop:Zmu_functoriality} gives us, for each pair $(m,\mu)$,   a decomposition 
\[
\mathcal{Z}^{\beef}(m,\mu)\times_{\mathcal{M}^{\beef}}\mathcal{M}  =   \bigsqcup_{\substack{   m_1 + m_2 =m  \\ \mu_1+\mu_2 = \mu }}\mathcal{Z}(m_1,\mu_1)\times\Lambda_{m_2,\mu_2}
\]
of stacks over $\mathcal{M}$.   Moreover, if  $c_f(-m,\mu)\neq 0$ then Hypothesis~\ref{hyp:proper}  implies that 
all  terms with $m_1=0$ are empty.   Therefore, we obtain a decomposition
\[
\mathcal{Z}^\beef(f)\vert_{\mathcal{M}} = \sum_{  \substack{     \mu \in L^{\beef,\vee} / L^\beef    \\  m_1 + m_2+ m_3=0 \\ m_3<0 \\ m_1 > 0} }
c_f( m_3 , \mu  ) 
\sum_{ \mu_1+\mu_2 =\mu  } \rho_{\Lambda}(m_2,\mu_2) \mathcal{Z}(m_1,\mu_1)
\]
of divisors on $\mathcal{M}$.

By Lemma~\ref{lem:improper_intersection}, the image of $\mathcal{Y}\to \mathcal{M}$ intersects $\mathcal{Z}^\diamond(f)$ properly, and so, as in the proof of Theorem~\ref{thm:arithmetic BKY} (see especially \S\ref{ss:arithmetic intersection proof}), we deduce
\begin{align*}
 \frac{  [ \widehat{\mathcal{Z}}^\beef(f) : \mathcal{Y} ]   }{   \deg_\C ({Y})  }    \approx_L 
\frac{\Phi^\beef(f,   \mathcal{Y}^\infty )  }{2  \deg_\C (Y)  } 
\;-    \sum_{  \substack{     \mu \in L^{\beef,\vee} / L^\beef    \\  m_1 + m_2+ m_3=0 \\ m_3<0, m_1 > 0 } }
c_f( m_3 , \mu  ) 
\sum_{ \mu_1+\mu_2 =\mu  } \frac{a(m_1,\mu_1)\rho_{\Lambda}(m_2,\mu_2)}{\Lambda(0,\chi)}.
\end{align*}
Combining with Proposition~\ref{prop:CM_value_beef} completes the proof.
\end{proof}

% Let $(V^\beef, Q^\beef)$ be a quadratic space of signature $(n,2)$ with $n\geq 5$, and suppose that we have an isometric embedding 
% \[
% (V,Q)  \hookrightarrow (V^\beef, Q^\beef)
% \]
% and a maximal lattice $L^\beef\subset V^\beef$ with $L\subset L^\beef$.   Suppose that $f \in M^!_{  1- n/2   } (\omega_{L^\beef})$ has integral Fourier coefficients and  nonzero constant term $c_f(0,0)$.
% Let $\mathcal{Z}^\beef(f)$ be the corresponding divisor  on $\mathcal{M}^\beef$, and  assume that Hypothesis \ref{hyp:proper} is satisfied.

Suppose now that $n^\beef\geq 3$. Denote by 
\[
\widehat{\bm{\omega}} \in  \widehat{\mathrm{Pic}} (\mathcal{M}) 
,\quad 
\widehat{\bm{\omega}}^\beef  \in  \widehat{\mathrm{Pic}} (\mathcal{M}^\beef ) 
\]
the metrized tautological bundles on $\mathcal{M}$ and $\mathcal{M}^\beef$ of \S \ref{ss:line bundles}. 
By  Theorem \ref{thm:borcherds}, after replacing $f$ by a multiple if necessary, we have the equality
\begin{align}\label{eqn:Borcherds_beef}
c_f(0,0) \cdot  \widehat{\bm{\omega}}^\beef  &=   \widehat{ \mathcal{Z} }^\beef ( f)   -  c_f(0,0) \cdot  (0,  \log(4\pi e^\gamma)  )  +  \widehat{\mathcal{E}}^\beef(f),
 \end{align}
where $\widehat{\mathcal{E}}^\beef(f) =  (\mathcal{E}^\beef(f),0)$ is empty if $L^\beef_{(2)}$ is self-dual and $n\geq 5$, and is otherwise supported on the union of special fibers $\mathcal{M}_{\F_p}$ for $p^2\nmid D_L$, as well as $\mathcal{M}_{\F_2}$ if $L_{(2)}$ is not self-dual.

\begin{theorem}\label{thm:taut degree}
We have
\[
\frac{  [  \widehat{\bm{\omega}}  :  \mathcal{Y} ] }{  \deg_\C ({Y}) }
\approx_L
   - \frac{  2 \Lambda'(0,\chi) }  {   \Lambda( 0 , \chi )    }    -  d  \cdot  \log(4\pi e^\gamma)  +  \frac{1}{c_f(0,0)}\frac{[\widehat{\mathcal{E}}^\beef(f):\mathcal{Y}]}{\mathrm{deg}_{\C}({Y})}.
\]
\end{theorem}

\begin{proof}
Combine Propositions~\ref{prop:height_pairing_beef} and \ref{prop:constant term eval} with~\eqref{eqn:Borcherds_beef}, and observe that the restriction of $\widehat{\bm{\omega}}^\beef$ to $\mathcal{M}$ is canonically isomorphic to $\widehat{\bm{\omega}}$; see Proposition~\ref{prop:functoriality}.
\end{proof}

%%%%%%%%%%%%%%%%%%%%%%%%%%%%%%%%%%%%%%%%%%%%%%%%%%

\section{Colmez's conjecture}
\label{s:colmez}

%%%%%%%%%%%%%%%%%%%%%%%%%%%%%%%%%%%%%%%%%%%%%%%%%%

In this section we prove Theorem \ref{bigthm:average colmez}, following the argument that was explained in \S~\ref{ss:intro colmez} of the introduction.

%%%%%%%%%%%%%%%%%%%%%%%%%%%%%%%%%%%%%%%%%%%%%%%%%%

\subsection{The statement of the conjecture}
\label{ss:colmez statement}

%%%%%%%%%%%%%%%%%%%%%%%%%%%%%%%%%%%%%%%%%%%%%%%%%%

%%%%%%%%%%%%%%%%%%%%%%%%%%%%%%%%%%%%%%%%%%%%%%%%%%

%\subsection{Colmez heights}

%%%%%%%%%%%%%%%%%%%%%%%%%%%%%%%%%%%%%%%%%%%%%%%%%%

In this subsection only, $E$ is an arbitrary CM algebra.
Recall that  $\Q^{\alg}$ is the algebraic closure  of $\Q$ in $\C$. 
The group $\Gamma_\Q=\mathrm{Gal}(\Q^{\alg}/\Q)$ acts on the set of all CM types of $E$ in the usual way:
$\sigma\circ \Phi = \{ \sigma \circ \varphi : \varphi \in \Phi \}$.   For each $\Phi$  let 
$\mathrm{Stab}(\Phi) \subset \Gamma_\Q$ be its stabilizer.

\begin{definition}\label{defn:CM_Colmez}
Let $c\in \Gamma_\Q$ be complex conjugation.  Write $\mathcal{CM}^0$ for the space of locally constant functions $a:\Gamma_\Q \to \Q$ that are constant on conjugacy classes and are such that the quantity
\begin{equation}\label{complex_conj_condition}
a(c\sigma) + a(\sigma)
\end{equation}
is independent of $\sigma\in\Gamma_\Q$. This notion does not depend on the choice of $c$.
\end{definition}

Every function $a\in \mathcal{CM}^0$ decomposes uniquely as a finite linear combination 
\[
a = \sum_\eta a(\eta) \cdot \eta
\]
of Artin characters.   For each Artin character  $\eta$ let 
\[
L(s,\eta) = \prod_p \frac{1}{ \det\big( 1 - p^{-s} \eta(\mathrm{Fr}_\mathfrak{p})|_{U^{I_\mathfrak{p}}} \big) }
\] 
be the usual Artin $L$-function, where $\mathfrak{p}$ is a prime of $\Q^\alg$ above $p$, and  $U$ is the 
space of the representation $\eta$.  

The independence from $\sigma$ of the quantity~\eqref{complex_conj_condition} implies that any nontrivial Artin character $\eta$ with  $a(\eta) \not=0$  must be totally odd, in the sense that $\eta(c) = -\eta(\mathrm{id})$, and therefore $L(0 ,\eta) \not=0$. We now set:
\[
\tilde{Z}(0,a) = - \sum_\eta a(\eta) \left(
\frac{L'(0,\eta)}{L(0,\eta)} + \frac{  \log( f_\eta)  }{2}
\right)
\]
where $f_\eta$ is the Artin conductor of $\eta$.

Following Colmez, we will now construct a particular function $\cofu^0_{(E,\Phi)}$ in $ \mathcal{CM}^0$ from the CM type $(E,\Phi)$. First, define a locally constant function on $\Gamma_\Q$ by the formula:
\[
\cofu_{( E,\Phi) } (\sigma)  = | \Phi \cap \sigma\circ  \Phi  |.
\]  
The average 
 \begin{equation}\label{Colmez function}
 \cofu^0_{( E,\Phi)} = \frac{1}{[  \Gamma_\Q : \mathrm{Stab}(\Phi) ] } \sum_{ \tau \in \Gamma_\Q/\mathrm{Stab}(\Phi) } \cofu_{(E,\tau \circ \Phi )}
 \end{equation}
is constant on conjugacy classes of $\Gamma_\Q$, and depends only on the $\Gamma_\Q$-orbit of $\Phi$. Moreover,
\[
\cofu^0_{(E,\Phi)}(\sigma) + \cofu^0_{(E,\Phi)}( c\sigma ) = |\Phi |
\]
is independent of $\sigma$, and so $\cofu^0_{(E,\Phi)}(\sigma)$ belongs to $\mathcal{CM}^0$, as desired.

\begin{remark}\label{rem:lifting CM}
If $E$ is a CM field, $\widetilde{E}$ is a CM field containing $E$, and 
\[
\widetilde{\Phi}= \{ \widetilde{ \varphi }  \in \Hom(  \widetilde{ E } ,  \Q^\alg) : \widetilde{ \varphi } |_E \in \Phi\}
\] 
is the lifted CM type, then 
\[
[ \widetilde{ E } : E ]   \cdot  \cofu^0_{( E,\Phi) }   =    \cofu^0_{( \widetilde{ E } , \widetilde{\Phi})}  .
\]
\end{remark}

\begin{definition}
The \emph{Colmez height} of the pair $(E,\Phi)$ is 
\[
h^\Colmez_{ (E,\Phi) } = \tilde{Z}(0,\cofu^0_{(E,\Phi)}).
\]
\end{definition}

Suppose $A$ is an abelian variety  over $\Q^\alg$ of  dimension $2\mathrm{dim}(A) = [E:\Q]$, and  
admitting complex multiplication of type $(E,\Phi)$. 
Choose a model of $A$ over a number field $\kk\subset \Q^\alg$ large enough that  the N\'eron model $\pi:\mathcal{A} \to \Spec(\co_\kk)$ has everywhere good reduction.
Pick a nonzero rational section $s$ of the line bundle 
\[
\pi_*\Omega^{\mathrm{dim}(A)}_{ \mathcal{A}/\co_\kk } \in \Pic(\co_\kk),
\] 
 and define
\[
h^\Falt _\infty ( A, s ) = \frac{-1}{ 2 [\kk:\Q] }  \sum_{ \sigma : \kk \to \C}
\log \big|   \int_{ \mathcal{A}^\sigma(\C) } s^\sigma \wedge \overline{s^\sigma}\,  \big|,
\]
and 
\[
h^\Falt_f(A,s) = \frac{1}{  [\kk:\Q] }  \sum_{ \mathfrak{p} \subset \co_\kk}  \ord_\mathfrak{p}(s) \cdot  \log  \mathrm{N}(\mathfrak{p})   .
\]

\begin{definition}
The \emph{Faltings height} of $A$ is 
\[
h^\Falt(A)= h^\Falt_f(A, s) + h^\Falt _\infty ( A,s) 
\]
It is independent of the field $\kk$, the choice of model of $A$ over $\kk$, and the choice of  section $s$.
\end{definition}

\begin{theorem}[Colmez]\label{thm:colmez}
If $A$ has complex multiplication by the maximal order $\co_E \subset E$, the Faltings height
\[
h^\Falt_{ (E,\Phi) } \define h^\Falt(A)
\]
 depends only  on the pair $(E,\Phi)$, and not on the choice of $A$. Moreover, there is a unique linear map 
 $
  \mathrm{ht}:\mathcal{CM}^0 \to \R
 $
 such that, for any pair $(E,\Phi)$, we have
 \[
 h^\Falt_{ (E,\Phi) } =  \mathrm{ht}(\cofu^0_{(E,\Phi)}).
 \]
\end{theorem}
\begin{proof}
This is~\cite[Th\'eor\`eme 0.3]{Colmez}.
\end{proof}

\begin{conjecture}[Colmez]\label{conj:colmez}
For any $a\in\mathcal{CM}^0$, we have $\mathrm{ht}(a) = \widetilde{Z}(0,a).$
In particular, taking $a=\cofu^0_{(E,\Phi)}$, for any CM pair $(E,\Phi)$ we have
\[
h^\Falt_{ (E,\Phi) } = h^\Colmez_{ (E,\Phi) } .
\]
\end{conjecture}

%%%%%%%%%%%%%%%%%%%%%%%%%%%%%%%%%%

\subsection{The reflex CM type}

%%%%%%%%%%%%%%%%%%%%%%%%%%%%%%%%%%

For the remainder of \S \ref{s:colmez} we fix a  CM field $E$ of degree $[E:\Q]=2d$, and a distinguished embedding $\iota_0 :E \to \C$.
Denote by $F$  the maximal  totally real subfield of $E$.

Recall from \S \ref{ss:reflex algebra} the total reflex algebra $E^\sharp$ associated with $E$: This is a finite \'etale $\Q$-algebra equipped with a canonical $\Gamma_\Q$-equivariant identification
\[
\Hom_{\Q-\alg}(E^\sharp,\Q^\alg) = \mathrm{CM}(E),
\]
where $\mathrm{CM}(E)$ is the $\Gamma_\Q$-set consisting of all CM types of $E$.

The embedding $\iota_0$ determines a subset
\[
\{ \mbox{CM types of $E$ containing $\iota_0$} \}   \subset   \mathrm{CM}(E).
\]
This corresponds to a subset
$
\Phi^\sharp \subset \Hom(E^\sharp , \Q^\alg) = \mathrm{CM}(E),
$
called the \emph{total reflex CM type}.  The pair $(E^\sharp, \Phi^\sharp)$ is the \emph{total reflex pair}.

The relation between  the total reflex pair and the classical notion  of reflex pairs is given by the  following proposition, which is immediate from the definitions.

\begin{proposition}
There exist representatives
$
\Phi_1,\ldots, \Phi_{m} \in \mathrm{CM}(E)
$ 
for the $\Gamma_\Q$-orbits in $\mathrm{CM}(E)$ satisfying the following condition: If for each pair $(E,\Phi_i)$, $(E_i^\prime , \Phi_i^\prime)$ is its reflex  CM pair, then there is an isomorphism of $\Q$-algebras
$
E^\sharp \iso  \prod_i E_i^\prime ,
$
such that the natural bijection
\[
\Hom(E^\sharp , \Q^\alg) \iso  \Hom( E_1^\prime ,\Q^\alg)\sqcup \cdots \sqcup\Hom( E_m^\prime ,\Q^\alg)
\]
identifies $\Phi^\sharp = \Phi_1 ^\prime\sqcup \cdots \sqcup \Phi_m^\prime$.   
In particular, $E^\sharp$ is a CM algebra and $\Phi^\sharp$ is a CM type.
\end{proposition}

%%%%%%%%%%%%%%%%%%%%%%%%%%%%%%%%%%%%%%%%%%%%%%%%%%

\subsection{The average over CM types}

%%%%%%%%%%%%%%%%%%%%%%%%%%%%%%%%%%%%%%%%%%%%%%%%%%

\begin{proposition}\label{prop:colmez height average}
Recall the completed $L$-function (\ref{completed L}).  The Colmez height satisfies
\begin{align*}
  \frac{1}{2^d} \sum_{ \Phi } h^\Colmez_{( E,\Phi) }    
  &=    - \frac{ 1  }{2} \cdot  \frac{ L'(0,\chi)  }{ L(0,\chi) }  - \frac{ 1  }{4} \cdot  \log \left| \frac{  D_{E}  }{  D_{F} }\right| 
     - \frac{ d  }{2} \cdot     \log(2\pi )  \\
& = - \frac{ 1  }{2} \cdot   \frac{ \Lambda'(0,\chi )  }{\Lambda(0,\chi )}    
-    \frac{ d }{4}   \log(16\pi^3 e^\gamma),
\end{align*}
where the sum on the left hand side is over all CM types of $E$. 
\end{proposition}

\begin{proof}
Recall that we have fixed an embedding $\iota_0:E\to \Q^\alg$.  If we let $\Gamma_F\subset \Gamma_\Q$ be the subgroup that acts as the identity on $\iota_0(F) \subset \Q^\alg$, and view the nontrivial character $\chi: \Gal(E/F) \to \{\pm 1\}$ as a character of $\Gamma_F$, then
\begin{equation}\label{eqn:cofu0 sharp}
\frac{ 1}{ [ E:\Q] }   \sum_{ \Phi } \cofu^0_{( E,\Phi) }   = 2^{d-2} \left(  \bm{1} + \frac{1}{d} \mathrm{Ind}_{\Gamma_F}^{\Gamma_\Q}( \chi ) \right)
\end{equation}
where $\bm{1}$ is the trivial character on $\Gamma_\Q$.
Indeed, if we normalize the Haar measure on $\Gamma_\Q$ to have total volume $1$, and define a function
$\psi :\Gamma_\Q \to \Z$ by 
\[
\psi(\sigma) = \begin{cases}
2^{d-1} & \mbox{if }\sigma\circ \iota_0 = \iota_0 \\
0 & \mbox{if }\sigma \circ \iota_0= \overline{\iota}_0 \\
2^{d-2} & \mbox{otherwise},
\end{cases}
\]
then an elementary calculation shows that the values of both  sides of (\ref{eqn:cofu0 sharp}) at $\sigma\in \Gamma_\Q$ are equal to
\[
\int_{\Gamma_\Q} \psi(\tau^{-1} \sigma \tau) \, d\tau.
\]

Using this, the first equality in the proposition  follows from the calculation
\begin{align*}
\frac{1}{2^d} \sum_{ \Phi } h^\Colmez_{( E,\Phi) }  &= -\frac{1}{2}\left[d\cdot\frac{\zeta'(0)}{\zeta(0)} + \frac{L'(0,\chi)}{L(0,\chi)} + \frac{1}{2}\log\left(f_{\mathrm{Ind}_{\Gamma_F}^{\Gamma_\Q}(\chi)}\right)\right]\\
 & = -\frac{d}{2}\cdot\log(2\pi) - \frac{1}{2}\frac{L'(0,\chi)}{L(0,\chi)} - \frac{1}{4}\cdot \log\bigl(|D_F|\cdot N_{E/\Q}(\mathfrak{d}_{E/F})\bigr)\\
 & = - \frac{ 1  }{2} \cdot  \frac{ L'(0,\chi)  }{ L(0,\chi) }  - \frac{ 1  }{4}\cdot  \log \left| \frac{  D_{E}  }{  D_{F} }\right| 
     - \frac{ d  }{2} \cdot     \log(2\pi ) ,
\end{align*}
and the second equality follows from  \eqref{eqn:completed log der}.
\end{proof}

\begin{proposition}\label{prop:average reflex}
The total reflex pair  $(E^\sharp,\Phi^\sharp)$ satisfies
\[
 \cofu^0_{( E^\sharp ,\Phi^\sharp) } = \frac{ 1}{ [ E:\Q] }   \sum_{ \Phi } \cofu^0_{( E,\Phi) }  ,
\]
where the sum is over all CM types of $E$. 
\end{proposition}

\begin{proof}
Let $\widetilde{E} \subset \Q^\alg$ be a finite Galois extension of $\Q$ large enough to contain all embeddings $E \to \Q^\alg$.
In particular, each $E_i^\prime \subset \widetilde{E}$. 
Use $\iota_0$  to regard $E$ as a subfield of $\widetilde{E}$.  For each  $1\le i \le k$
let  
\[
\widetilde{\Phi}_i , \widetilde{\Phi}^\prime_i  \subset \Gal(\widetilde{E}/\Q) = \Hom( \widetilde{E} ,\Q^\alg)
\]
 be the lifts of $\Phi_i$  and $\Phi^\prime_i$, respectively, so that $\sigma \in  \widetilde{\Phi}_i$ if and only if 
 $\sigma^{-1} \in \widetilde{\Phi}^\prime_i$.   An easy exercise shows that 
 $\cofu^0_{( \widetilde{E} , \widetilde{\Phi}^\prime_i)} =  \cofu^0_{( \widetilde{E} , \widetilde{\Phi}_i)}$, and hence 
 Remark \ref{rem:lifting CM} implies
 \[
 [ \widetilde{E} : E_i^\prime ] \cdot   \cofu^0_{( E_i^\prime,\Phi_i^\prime) }  
=  [ \widetilde{E} : E ]   \cdot \cofu^0_{( E,\Phi_i)}.
 \]
 It  follows that 
\[
  \cofu^0_{( E_i^\prime,\Phi_i^\prime) }  
=   \frac{  [ E_i^\prime : \Q ] }{ [ E:\Q] } \cdot \cofu^0_{( E,\Phi_i)}
= \frac{ 1}{ [ E:\Q] } \cdot  \sum_{ \tau \in \Gamma_\Q / \mathrm{Stab}(\Phi_i)   }\cofu^0_{( E, \tau\circ \Phi_i)}, 
\]
and summing over $i$ proves the  claim.  
\end{proof}

\begin{corollary}\label{Cor:total_reflex_height}
The total reflex pair  $(E^\sharp,\Phi^\sharp)$ satisfies
\begin{align*}
h^\Falt_{ ( E^\sharp,\Phi^\sharp )  }  & =   \frac{ 1}{ [ E:\Q] }   \sum_{ \Phi\in\mathrm{CM}(E) } h^\Falt_{( E,\Phi) }.   
%h^\Colmez_{ ( E^\sharp,\Phi^\sharp )  }    & =    \frac{ 1}{ [ E:\Q] }   \sum_{ \Phi } h^\Colmez_{( E,\Phi) } 
\end{align*}
\end{corollary}
\begin{proof}
Combine Theorem~\ref{thm:colmez} and Proposition~\ref{prop:average reflex}.
\end{proof}

%%%%%%%%%%%%%%%%%%%%%%%%%%%%%%%%%%%

\subsection{Faltings heights and Arakelov heights}

Recall the torus $T = T_E/T_F^1$ and the arithmetic curve 
\[
\mathcal{Y}_0\to \Spec(\co_E)
\] 
from \S \ref{ss:integral_model_Y} defined by the compact open subgroup $K_{0}\subset T(\A_f)$. In~\S\ref{ss:sheaves_i} and~\S\ref{ss:sheaves_ii}, given an algebraic representation $N$ of the torus $T$, and a $K_{0}$-stable lattice $N_{\widehat{\Z}}\subset N_{\A_f}$, we constructed various homological realizations $\bm{N}_?$ over $\mathcal{Y}_0$, functorially associated with the pair $(N,N_{\widehat{\Z}})$.

Let $H^\sharp$ be as in Proposition~\ref{prop:h sharp repn}. The subring $\co_{E^\sharp}\subset E^\sharp$ gives us a lattice $H^\sharp_{\Z}\subset H^\sharp$ stable under the multiplication action of $\co_{E^\sharp}$. The associated $\widehat{\Z}$-lattice $H^\sharp_{\widehat{\Z}}$. Therefore, from the pair $(H^\sharp,H_{\widehat{\Z}}^\sharp)$, we obtain an abelian scheme $A^\sharp\to\mathcal{Y}_0$, whose homological realizations are the sheaves associated with the pair. By construction, at any point $y\in Y_0(\C)$, $A^\sharp_y$ is an abelian variety with CM by $\co_{E^\sharp}$ and of CM type $\Phi^\sharp$.

Define
      \[
     \bm{\Omega}^\sharp = \pi_* \Omega_{A^\sharp/ \mathcal{Y}_0}^{\mathrm{dim}(A^\sharp)}.
      \]
At any complex point $y\in \mathcal{Y}^\infty_0(\C)$ we  endow the fiber
\[
\bm{\Omega}^\sharp_y = H^0( A^\sharp_y , \Omega^{\mathrm{dim}(A^\sharp)}_{A^\sharp_y/\C} )
\]
with the  \emph{Faltings metric}
\[
|| s ||^2 = \Big| \int_{ A^\sharp_y(\C) } s \wedge \overline{s} \,  \Big|,
\]
and so obtain the \emph{metrized Hodge bundle}
\[
 \widehat{ \bm{\Omega}}^\sharp \in \widehat{\mathrm{Pic}}( \mathcal{Y}_0 ).
\]
% The  canonical isomorphism $\det\big(\Lie(A^\sharp_y)\big)^\vee   \xrightarrow{\simeq}  \bm{\Omega}^\sharp_y $  induces a dual metric   on the line bundle $\det(\Lie(A^\sharp))$.

The Betti realization of $A^\sharp$ gives us a local system
\[
\bm{H}^\sharp_B \subset \bm{H}^\sharp_B \otimes \co_{Y_0(\C)} = \bm{H}^\sharp_{dR, Y_0(\C)},
\]
which determines  a local system of $\Z$-modules
$
\det(\bm{H}^\sharp_B) \subset \det( \bm{H}^\sharp_{dR, Y_0(\C)} )
$
of rank $1$.  

We define the \emph{volume metric} on $\det( \bm{H}^\sharp_{\dR} )$ by  declaring that $||e||^2 = 1$ 
for any local generator  $e$ of $\det(\bm{H}^\sharp_B)$.   At any complex point $y\in Y_0(\C)$  the dual volume metric on
\[
\det(\bm{H}^\sharp_{dR,y}) ^\vee \iso H^{2 \mathrm{dim}(A^\sharp)}_{\dR}(A^\sharp_y / \C)
\]
is just integration of top degree $C^\infty$ forms:
\[
|| \eta || = \big| \int_{A_y(\C) } \eta \, \big|.
\]
This gives a second metrized line bundle
\[
\widehat{\det}(\bm{H}^\sharp_{\dR})  \in \widehat{\mathrm{Pic}}( \mathcal{Y}_0 ).
\]

We will need a third metrized line bundle $\widehat{\bm{\omega}}_0$. This will be defined as follows. 
Consider the representation $V_0 = V(E,c)$ of $T_E$ on the space of $E$-semilinear endomorphisms of $E$. 
This representation factors through $T$ (and in fact through $T_{so} = T_E/T_F$) and has a natural lattice $L_0 = V(\co_E,c)$ such that $\widehat{L}_0 = L_{0,\widehat{\Z}}$ is stable under $K_0$. 
The natural $E$-linear structure on $V_0$ is invariant under $T$. 
Therefore, from the pair $(V_0,\widehat{L}_0)$, we obtain a de Rham realization $\bm{V}_{0,\dR}$ over $\mathcal{Y}_0$, equipped with an action of $\co_E$, making it a locally free sheaf of rank $1$ over $\co_{\mathcal{Y}_0}\otimes_{\Z}\co_E$. 
This realization is equipped with a canonical $\co_E$-stable filtration $\Fil^\bullet\bm{V}_{0,\dR}$ by local direct summands extending the one over $Y_0 = \mathcal{Y}_{0,\Q}$ obtained from Proposition~\ref{prop:zero dim derham}. 
Moreover, the degree $1$ summand 
\[
\bm{\omega}_0 \define \Fil^1\bm{V}_{0,\dR}
\]
 is a line bundle over $\mathcal{Y}_0$.

Composition in $\End(E)$ induces a canonical, $T$-invariant Hermitian form $\langle \cdot,\cdot\rangle_0$ on $V_0$ determined by the property
\[
(x\circ y)(a) = \langle x,y\rangle_0\cdot a,
\]
for any $x,y\in V_0$ and $a\in E$. From this, we obtain a $\Q$-valued quadratic form
\[
\mathcal{Q}_0 = \mathrm{Tr}_{F/\Q}(\langle x,x\rangle_0)
\]
with associated bilinear form $[x,y]_0$ on $V_0$.

Just as in \S~\ref{ss:line bundles}, for every $y\in \mathcal{Y}^\infty(\C)$, this form equips $\bm{\omega}_{0,y} = \Fil^1\bm{V}_{0,\dR,y}$ with the Hermitian form $||z||^2_0 = -[z,\overline{z}]_0$, and thus equips $\bm{\omega}_0$ with the structure of a metrized line bundle, which will denote by $\widehat{\bm{\omega}}_0$.

There is a natural $T$-equivariant embedding
\begin{equation}\label{eqn:V Esharp emb}
V_0\hookrightarrow \End(H^\sharp)
\end{equation}
defined as follows: We have
\[
H^\sharp = E^\sharp = \left(\bigotimes_{\iota\in \mathrm{Emb}(F)}\Q^{\alg}\otimes_{\iota,F}E\right)^{\Gamma_\Q}.
\]
Here, the action of $\Gamma_\Q$ on the tensor product is the obvious one compatible with permutation of the indexing set $\mathrm{Emb}(F)$.\footnote{In other words, $E^\sharp$ is the \emph{tensor induction} of the $F$-algebra $E$ to an algebra over $\Q$.}

 For each $\iota\in \mathrm{Emb}(F)$, we have an embedding
\[
\Q^\alg\otimes_{\iota,F}V_0 = V\bigl(\Q^\alg\otimes_{\iota,F}E,c\bigr)\subset \End(\Q^\alg\otimes_{\iota,F}E).
\]
The Lie algebra tensor product of these embeddings gives us a $\Gamma_\Q$-equivariant embedding
\[
\Q^\alg\otimes_{\Q}V_0 = \bigoplus_{\iota}\bigl(\Q^\alg\otimes_{\iota,F}V_0\bigr) \hookrightarrow \End\left(\bigotimes_{\iota}\Q^{\alg}\otimes_{\iota,F}E\right),
\]
so that $x\in V_0$ acts on $\Q^\alg\otimes_\Q E^\sharp$ via:
\[
x(a_0\otimes a_1\otimes\cdots \otimes a_{d-1}) = \sum_{i=0}^{d-1}a_0\otimes \cdots\otimes a_{i-1}\otimes x(a_i)\otimes\cdots a_{d-1}.
\]
Here, $\iota_0,\iota_1,\ldots,\iota_{d-1}:F\hookrightarrow \R$ are the real embeddings of $F$, and for each $i$, $a_i\in \Q^{\alg}\otimes_{\iota_i,F}E$.

The descent of this action over $\Q$ gives us~\eqref{eqn:V Esharp emb}.

Now, it is clear that this embedding induces a $K_0$-stable inclusion $\widehat{L}_0\hookrightarrow\End(H^\sharp_{\widehat{\Z}})$, and thus gives us a map of de Rham realizations
\begin{equation*}
\bm{V}_{0,\dR}\to \End(\bm{H}_{\dR}^\sharp)
\end{equation*}
allowing us to view sections of $\bm{V}_{0,\dR}$ as endomorphisms of $\bm{H}_{\dR}^\sharp$.

The action of $\bm{\omega}_0$ on $\bm{H}_{\dR}^\sharp$ induces a map
\[
\bm{\omega}_0\otimes_{\co_{\mathcal{Y}_0}}\mathrm{gr}^{-1}_{\Fil}\bm{H}^\sharp_{\dR} \to \Fil^0\bm{H}^\sharp_{\dR}
\]
of vector bundles over $\mathcal{Y}_0$, and  taking determinants yields a map
\begin{equation}\label{eqn:det map}
\bm{\omega}_0^{\otimes 2^{d-1}}\otimes_{\co_{\mathcal{Y}_0}}\det\bigl(\mathrm{gr}^{-1}_{\Fil}\bm{H}_{\dR}^\sharp\bigr)\to \det\bigl(\Fil^0\bm{H}^\sharp_{\dR}\bigr)
\end{equation}
of line bundles over $\co_{\mathcal{Y}_0}$. Set
\begin{equation}\label{eqn:diff line bundles}
\bm{\mathscr{L}} = \det\bigl(\Fil^0\bm{H}^\sharp_{\dR}\bigr)\otimes \bm{\omega}_0^{\otimes - 2^{d-1}}\otimes_{\co_{\mathcal{Y}_0}}\det\bigl(\mathrm{gr}^{-1}_{\Fil}\bm{H}_{\dR}^\sharp\bigr)^{\otimes -1}.
\end{equation}
Then~\eqref{eqn:det map} gives us a canonical section of $\bm{\mathscr{L}}$ over $\mathcal{Y}_0$, and thus an effective divisor $\bm{\Delta}$ on $\mathcal{Y}_0$. Write $\widehat{\bm{\Delta}} = (\bm{\Delta},0)$ for the associated arithmetic divisor.

\begin{proposition}
\label{prop:delta degree}
We have
\[
\frac{\widehat{\deg}(\widehat{\bm{\Delta}})}{\deg_\C (Y_0)} = 2^{d-1}\log |D_F|.
\]
\end{proposition}

This is the key technical result of this subsection, and its proof will be given further below. For now, we deduce from it the following theorem, which gives the precise relation between the degree of $\widehat{\omega}_0$, and the average of the Faltings heights of abelian varieties with CM by $\co_E$.

\begin{theorem}\label{thm:Faltings height line bundles}
We have the identity
\[
\frac{1}{2^{d}} \sum_\Phi h^\Falt_{(E,\Phi)}  = \frac{1}{4} \frac{  \widehat{\deg}(\widehat{\bm{\omega}}_0) }{    \deg_\C (Y_0) }  + \frac{1}{4}\log|D_F| - \frac{1}{2} d\cdot\log(2\pi)  .
\]
\end{theorem}

\begin{proof}
	By Corollary~\ref{Cor:total_reflex_height}, we have
	\[
      \sum_\Phi h^\Falt_{(E,\Phi)}  =  2d\cdot h^\Falt_{(E^\sharp,\Phi^\sharp)}.
	\]

	Observing that, for every $y\in Y_0(\C)$, the abelian variety  $A^\sharp_y$ has CM by $\co_{E^\sharp}$ with CM type $\Phi^\sharp$, and using Theorem~\ref{thm:colmez}, we obtain 
     \begin{equation}\label{eqn:omega sharp degree}
       \frac{ \widehat{\deg}\bigl(\widehat{\bm{\Omega}}^\sharp\bigr)  }{  \deg_\C(Y_0)  }  =  2d\cdot h^\Falt_{(E^\sharp,\Phi^\sharp)} = \sum_\Phi h^\Falt_{(E,\Phi)}.
     \end{equation}

    Consider the short exact sequence
\[
0\to \Fil^0\bm{H}^\sharp_{\dR}\to \bm{H}^\sharp_{\dR}\to \mathrm{gr}^{-1}_{\Fil}\bm{H}^\sharp_{\dR}\to 0
\]
of vector bundles over $\mathcal{Y}_0$. Taking determinants, we obtain an isomorphism
\begin{equation}\label{eqn:line bundles isomorphism}
\det(\bm{H}^\sharp_{\dR})\xrightarrow{\simeq}\bm{\mathscr{L}}\otimes\bm{\omega}_0^{\otimes 2^{d-1}}\otimes\det(\mathrm{gr}^{-1}_{\Fil}\bm{H}^\sharp_{\dR})^{\otimes 2},
\end{equation}
where $\mathscr{L}$ is as in~\eqref{eqn:diff line bundles}. If $\widehat{\bm{\Delta}}$ is as in Proposition~\ref{prop:delta degree}, then, using the canonical isomorphism
\[
\det(\mathrm{gr}^{-1}_{\Fil}\bm{H}_{\dR}^\sharp)^{\otimes -1}\xrightarrow{\simeq}\bm{\Omega}^\sharp,
\]
it is easy to check that~\eqref{eqn:line bundles isomorphism} gives us an identity
\[
\widehat{\det}(\bm{H}^\sharp_{\dR}) = \widehat{\bm{\Delta}} + 2^{d-1}\widehat{\bm{\omega}}_0 - 2\cdot\widehat{\bm{\Omega}}^{\sharp}
\]
in $\widehat{\Pic}(\mathcal{Y}_0)$.

	Combining this with Proposition~\ref{prop:delta degree} and~\eqref{eqn:omega sharp degree} shows
	\[
\frac{1}{2^{d}} \sum_\Phi h^\Falt_{(E,\Phi)} = \frac{1}{4}\frac{\widehat{\deg}(\widehat{\bm{\omega}}_0)}{ \deg_\C(Y_0)} +\frac{1}{4}\cdot\log|D_F| - \frac{1}{2^{d+1}}\frac{\widehat{\deg}\bigl(\widehat{\det}(\bm{H}^\sharp_{\dR})\bigr)}{\deg_{\C}(Y_0)}.
	\]
	 Therefore, we will be done once we verify the identity
\[
\frac{\widehat{\deg}\bigl(\widehat{\det}(\bm{H}^\sharp_{\dR})\bigr)}{\deg_{\C}(Y_0)} =       2^{d} d \cdot \log(2\pi).
\]
  But this is easily done using Lemma~\ref{lem:de Rham det degree} below.
\end{proof}

\begin{lemma}
\label{lem:de Rham det degree}
Let $E'$ be a number field and let $A$ be an abelian scheme over $\co_{E'}$. Suppose that the top degree cohomology $H^{2d}_{\dR}(A/\co_{E'})$ of $A$ is a free module of rank $1$ over $\co_{E'}$. Fix an embedding $E'\hookrightarrow \C$, and an $\co_{E'}$-module generator $e\in H^{2d}_{\dR}(A/\co_{E'})$, and let $\eta(e)$ be a degree $2d$ $C^\infty$ form on $A(\C)$ that represents this generator over $\C$. We then have:
\[
\left\vert\int_{A(\C) } \eta(e)\;\right\vert  = (2\pi)^{-\dim(A)}. 
\]
\end{lemma}
\begin{proof}
As explained in~\cite[Ch. I, \S 1, p. 22]{dmos}, there is a canonical $\co_{E'}$-linear trace map
\[
\mathrm{Tr}_{\dR}: H^{2d}_{\dR}(A/\co_{E'}) \to E',
\]
which, over $\C$, corresponds to the linear functional
\[
\eta\mapsto \frac{1}{(2\pi i)^{\dim A}}\int_{A(\C)}\eta
\]
on top degree $C^\infty$ forms on $A(\C)$. 

So, to prove the lemma, it is enough to show that $\mathrm{Tr}_{\dR}$ maps isomorphically onto $\co_{E'}\subset E'$. Indeed, this would imply that
\[
\int_{A(\C) } \eta(e) \in (2\pi i)^{-\dim(A)}\co_{E'}^\times.
\]

For this, note that $\mathrm{Tr}_{\dR}$ is equal to the composition:
\[
H^{2d}_{\dR}(A/\co_{E'}) \xrightarrow{\simeq} H^d(A,\Omega^d_{A/\co_{E'}}) \xrightarrow[\simeq]{\mathrm{Tr}}\co_{E'},
\]
where the first isomorphism arises from the degeneration of the Hodge-to-de Rham spectral sequence for $A$, and the second is the trace isomorphism from Grothendieck-Serre duality.
\end{proof}

We now begin our preparations for the proof of Proposition~\ref{prop:delta degree}. Suppose that we have inclusions of complete discrete valuation rings $A\subset B\subset C$ with perfect residue fields, with $\mathrm{Frac}(B)$ finite over $\mathrm{Frac}(A)$. Suppose that the set $\Hom(B,C)$ of local $A$-algebra homomorphisms has the maximum possible size $[\mathrm{Frac}(B):\mathrm{Frac}(A)]$.\footnote{In other words, the \'etale $\mathrm{Frac}(A)$-algebra $\mathrm{Frac}(B)$ splits over $\mathrm{Frac}(C)$.}

Fix a subset $\Upsilon\subset \Hom(B,C)$, and consider the map of $C$-algebras:
\begin{align*}
\varphi_{\Upsilon}:C\otimes_AB &\to \prod_{\sigma\in \Upsilon}C\\
c\otimes b &\mapsto (c\cdot\sigma(b))_{\sigma}.
\end{align*}

Set $\mathcal{K}(\Upsilon) = \ker \varphi_{\Upsilon}$. If $\Upsilon^c = \Hom(B,C)\backslash \Upsilon$, then the inclusion
\[
\mathcal{K}(\Upsilon) + \mathcal{K}(\Upsilon^c)\hookrightarrow C\otimes_AB
\]
of $C$-modules is an isomorphism after tensoring with $\mathrm{Frac}(C)$. Therefore, its cokernel has finite length as a $C$-module. Denote this cokernel by $\mathcal{C}(\Upsilon)$. Fix a uniformizer $\pi_B\in B$. Let $B_0\subset B$ be the maximal \'etale $A$-subalgebra. Let $\mathfrak{d}_{B/A}\subset B$ be the different, and let $\mathfrak{D}_{B/A} = \mathrm{Nm}_{B/A}(\mathfrak{d}_{B/A})\subset A$ be the discriminant ideal for $B$ over $A$.

\begin{lemma}
We have:
\label{lem:length C Upsilon}
\begin{align*}
\mathrm{length}_C(\mathcal{C}(\Upsilon)) &=  \frac{1}{2}\cdot \sum_{\stackrel{\tau,\tau'\in\Upsilon}{\tau\neq \tau'}}\mathrm{length}(C/(\tau(\pi_B) - \tau'(\pi_B))) \\
& + \frac{1}{2} \cdot \sum_{\stackrel{\sigma,\sigma'\in\Upsilon^c}{\sigma\neq \sigma'}}\mathrm{length}(C/(\sigma(\pi_B) - \sigma'(\pi_B))) - \frac{1}{2}\cdot \mathrm{length}(C/\mathfrak{D}_{B/A}C).\nonumber
\end{align*}
\end{lemma}
\begin{proof}
By a standard reduction, we can assume that $A=B_0$, so that $B$ is totally ramified over $A$. First consider the cokernel $\mathcal{C}_1(\Upsilon)$ of the natural embedding
\[
\mathcal{K}(\Upsilon) \xrightarrow{x\mapsto (\sigma(x))_{\sigma}} \prod_{\sigma\in\Upsilon^c}C.
\]
We claim that
\begin{equation}
\label{eqn:length C1}
n_1(\Upsilon) \define \mathrm{length}_C(\mathcal{C}_1(\Upsilon)) = \frac{1}{2}\cdot\sum_{\stackrel{\tau,\tau'\in\Upsilon}{\tau\neq \tau'}}\mathrm{length}(C/(\tau(\pi_B) - \tau'(\pi_B))).
\end{equation}
This can be verified using induction on the size of $\Upsilon$, after proving (via a separate inductive argument) that $\mathcal{K}(\Upsilon) \subset C\otimes_BA$ is the principal ideal generated by the element
\[
f_{\Upsilon} = \prod_{\tau\in\Upsilon}(1\otimes \pi_B - \tau(\pi_B)\otimes 1) \in C\otimes_BA.
\]

We also claim that the inclusion
\[
C\otimes_AB \hookrightarrow \prod_{\sigma\in\Hom(B,C)}C
\]
has cokernel of length $\frac{1}{2}\cdot\mathrm{length}(C/\mathfrak{D}_{B/A}C)$. This follows by observing that $\mathfrak{d}_{B/A}^{-1}$ is the dual lattice to $B$ under the canonical non-degenerate trace pairing $(x,y)\mapsto \mathrm{Tr}_{B/A}(x,y)$ on $\mathrm{Frac}(B)$, and that $\prod_{\sigma\in\Hom(B,C)}C$ is a self-dual lattice in $\mathrm{Frac}(C)\otimes_AB$ under the induced $C$-bilinear pairing. 

The lemma now follows by noting that
\[
\mathrm{length}_C(\mathcal{C}(\Upsilon)) = n_1(\Upsilon) + n_1(\Upsilon^c) - \mathrm{length}_C\left(\frac{\prod_{\sigma\in\Hom(B,C)}C}{C\otimes_AB}\right).
\]
\end{proof}

Let $K$ be a finite \'etale $\Q_p$-algebra, and let $P\subset \mathrm{Frac}(W)^\alg$ be a finite Galois extension of $\mathrm{Frac}(W)$ that receives all maps $\eta:K\hookrightarrow \Q_p^\alg$. Let $\mathcal{C}(\Gamma_{\Q_p},\C)$ (resp. $\mathcal{C}^0(\Gamma_{\Q_p},\C)$) be the space of continuous (resp. continuous, conjugation-invariant) $\C$-valued functions on $\Gamma_{\Q_p}$. 

$\mathcal{C}^0(\Gamma_{\Q_p},\C)$ has a basis given by characters of irreducible finite dimensional complex representations of $\Gamma_{\Q_p}$. There is a unique linear functional
\[
\mu_p:\; \mathcal{C}^0(\Gamma_{\Q_p},\C)\to \C,
\]
which associates with every finite dimensional irreducible character $\chi$ the integer $\mu_p(\chi) = \log_pf_p(\chi)$, where $f_p(\chi)$ is the Artin conductor of $\chi$ and $\log_p$ is the base-$p$ logarithm. 

Since $\Gamma_{\Q_p}$ is compact, averaging with respect to the Haar measure of measure $1$ gives us a canonical section $f\mapsto f^0$ of the inclusion
\[
\mathcal{C}^0(\Gamma_{\Q_p},\C)\hookrightarrow \mathcal{C}(\Gamma_{\Q_p},\C),
\]
and so permits us to lift $\mu_p$ to a measure on $\Gamma_{\Q_p}$: $\mu_p(f) \define \mu_p(f^0)$.

With any subset $\Upsilon\subset \Hom(K,P)$, we can associate the function
\begin{align*}
a_{(K,\Upsilon)}:\Gamma_{\Q_p} &\to \Z\\
\sigma&\mapsto | \Upsilon \cap \sigma\circ\Upsilon |.
\end{align*}

% as well as the conjugation invariant average
% \[
% a^0_{(K,\Upsilon)} = \frac{1}{[\Gamma_{\Q_p}:\mathrm{Stab}(\Upsilon)]}\sum_{\tau\in \Gamma_\Q/\mathrm{Stab}(\Upsilon)} a_{(K,\tau\circ\Upsilon)}\in \mathcal{C}^0(\Gamma_{\Q_p},\C).
% \] 

Let $M$ be a finite free $\co_P\otimes_{\Z_p}\co_K$-module of rank $1$. For $\Upsilon\subset \Hom(K,P)$, set 
\[
\mathcal{K}(\Upsilon) = \ker\left(\co_P\otimes_{\Z_p}\co_K \xrightarrow{x\otimes y\mapsto (x\eta(y))_{\eta}}\to \prod_{\eta\in\Upsilon}\co_P\right),
\]
and set $\mathcal{K}(M,\Upsilon) = \mathcal{K}(\Upsilon)\cdot M$. Let $\mathcal{C}(M,\Upsilon)$ be the cokernel of the inclusion
\[
\mathcal{K}(M,\Upsilon) + \mathcal{K}(M,\Upsilon^c)\hookrightarrow M
\]
of $\co_P$-modules. 

It will be useful later to have another description of this cokernel. Set
\[
\mathcal{Q}(M,\Upsilon) = \mathrm{coker}(\mathcal{K}(M,\Upsilon^c)\hookrightarrow M).
\]
Then $\mathcal{C}(M,\Upsilon)$ is also the cokernel of the natural inclusion
\[
\mathcal{K}(M,\Upsilon)\hookrightarrow \mathcal{Q}(M,\Upsilon).
\]

\begin{proposition}
\label{prop:colmez calc}
Let $e_P$ be the absolute ramification index of $P$. Then
\[
\mathrm{length}_{\co_P}\mathcal{C}(M,{\Upsilon}) = - \frac{1}{2}\cdot e_P\cdot\left( \mu_p(a_{(K,\Upsilon)}) + \mu_p(a_{(K,\Upsilon^c)}) \right).
\]
\end{proposition}
\begin{proof}
If $K = \prod_iK_i$ is the decomposition of $K$ into a product of field extensions of $\Q_p$, and 
\[
\Upsilon_i = \Upsilon \cap \Hom(K_i,\Q_p^\alg),
\]
for each $i$, then we have $a^0_{(K,\Upsilon)} = \sum_i a^0_{(K_i,\Upsilon_i)}$. Moreover, if $M_i = M\otimes_{\co_K}\co_{K_i}$, then we have
\[
\mathcal{C}(M,\Upsilon) = \bigoplus_i \mathcal{C}(M_i,\Upsilon_i).
\]

Therefore, without loss of generality, we can assume that $K$ is a field. To compute the right hand side of the asserted identity, for each pair $\eta,\eta'\in \Hom(K,P)$, consider the function $a_{\eta,\eta'}\in \Gamma_{\Q_p}$ given by
\begin{align*}
a_{\eta,\eta'}(\sigma)&= \begin{cases}
1,&\text{ if $\sigma(\eta) = \eta'$};\\
0,&\text{ otherwise.}
\end{cases}
\end{align*}

Fix a uniformizer $\pi_K$ for $K$. Let $K_0\subset K$ be the maximal unramified subextension. By Lemme I.2.4 of~\cite{Colmez} and the remark following Prop. I.2.6 of \emph{loc. cit.}, we have
\begin{equation}\label{eqn:colmez calc}
\mu_p(a_{\eta,\eta'}) = \begin{cases}
\frac{1}{e_P}\ord_{\co_P}(\eta(\mathfrak{d}_{K/\Q_p})\co_P),&\text{ if $\eta = \eta'$};\\
-\frac{1}{e_P}\ord_{\co_P}(\eta(\pi_K) - \eta'(\pi_K)),&\text{ if $\eta\vert_{K_0} = \eta'\vert_{K_0}$ and $\eta\neq \eta'$};\\
0,&\text{ otherwise.}
\end{cases}
\end{equation}
Moreover, the following identity is easily verified:
\begin{equation}\label{eqn:colmez rhs}
a_{(K,\Upsilon)} + a_{(K,\Upsilon^c)}= \sum_{\stackrel{\eta,\eta'\in \Upsilon}{\eta\neq \eta'}}a_{\eta,\eta'} + \sum_{\stackrel{\eta,\eta'\in \Upsilon^c}{\eta\neq \eta'}}a_{\eta,\eta'} + \sum_{\eta:K\to P}a_{\eta,\eta}.
\end{equation}

Now, observe that we are in the situation of Lemma~\ref{lem:length C Upsilon}, with $A = \Z_p$, $B = \co_K$ and $C = \co_P$, and the computation there gives us an explicit formula for the left hand side of the desired identity. Comparing this with~\eqref{eqn:colmez rhs} and~\eqref{eqn:colmez calc} completes the proof of the Proposition.
\end{proof}

\begin{proof}[Proof of Proposition~\ref{prop:delta degree}]
Fix a prime $\mathfrak{q}\subset \co_E$ above a rational prime $p$, and also a point $y\in \mathcal{Y}_0(\F_{\mathfrak{q}}^\alg)$. Let $\co_y$ be the completed \'etale local ring of $\mathcal{Y}_0$ at $y$. Set $W = W(\F_{\mathfrak{q}}^\alg)$. Fix an algebraic closure $\mathrm{Frac}(W)^\alg$ of $\mathrm{Frac}(W)$, and an embedding $\Q^\alg\hookrightarrow\mathrm{Frac}(W)^\alg$ inducing the place $\mathfrak{q}$ on $E\subset \Q^\alg$, embedded via $\iota_0$. This identifies $\co_y$ with the ring of integers in the extension of $\mathrm{Frac}(W)$ generated by the image of $E$.

Restricting the line bundle $\mathscr{L}$ over $\Spec~\co_y$ gives us a free $\co_y$-module $\mathscr{L}_y$ of rank $1$, equipped with a canonical section $s_y:\co_y\to \mathscr{L}_y$. We claim that we have
\begin{equation}
\label{eqn:main length comp}
\mathrm{length}(\mathscr{L}_y/\mathrm{im}(s_y)) = 2^{d-2}\cdot\ord_{\mathfrak{q}}(\mathfrak{d}_{F/\Q}).
\end{equation}

Assuming this for all $\mathfrak{q}$ and $y$, we find
\begin{align*}
\widehat{\deg}(\widehat{\bm{\Delta}}) & = \sum_{\mathfrak{q}\subset\co_E}\log N(\mathfrak{q})\sum_{y\in \mathcal{Y}_0(\F_{\mathfrak{q}}^{\alg})}\frac{\mathrm{length}(\mathscr{L}_y/\mathrm{im}(s_y))}{|\Aut(y)|}\\
& = 2^{d-2}\sum_{\mathfrak{q}\subset\co_E}\left[\log N(\mathfrak{q})\cdot\ord_{\mathfrak{q}}(\mathfrak{d}_{F/\Q})\cdot\sum_{y\in \mathcal{Y}_0(\F_{\mathfrak{q}}^\alg)}\frac{1}{|\Aut(y)|}\right]\\
& = 2^{d-2}\cdot\left(\sum_{y\in \mathcal{Y}_0(\C)}\frac{1}{|\Aut(y)|}\right)\cdot\left(\sum_{\mathfrak{q}\subset\co_E}\log N(\mathfrak{q})\cdot\ord_{\mathfrak{q}}(\mathfrak{d}_{F/\Q})\right)\\
& = 2^{d-1}\cdot\deg_{\C}(Y_0)\cdot\log|D_F|.
\end{align*}
Here, in the third identity, as in the proof of Lemma~\ref{lem:twisting isogeny}, we have used the finite \'etaleness of $\mathcal{Y}_0$ over $\co_E$. 

It remains to show~\eqref{eqn:main length comp}. Note that complex conjugation induces an involution $c$ on the set $\Gamma_{\Q_p}$-set $\Hom(E_{\mathfrak{p}},\Q_p^\alg)$. Set
\[
\mathrm{CM}(E_{\mathfrak{p}}) = \{\Phi_{\mathfrak{p}}\subset \Hom(E_{\mathfrak{p}},\Q_p^\alg):\;\Phi_{\mathfrak{p}}\sqcup c(\Phi_{\mathfrak{p}}) = \Hom(E_{\mathfrak{p}},\Q_p^\alg)\}.
\]
Let $E^\sharp_{\mathfrak{p}}$ be the \'etale $\Q_p$-algebra associated with the $\Gamma_{\Q_p}$-set $\mathrm{CM}(E_{\mathfrak{p}})$. There is an obvious surjection of $\Gamma_{\Q_p}$-sets
\[
\mathrm{CM}(E) \to \mathrm{CM}(E_{\mathfrak{p}})
\]
inducing an inclusion $E^\sharp_{\mathfrak{p}}\hookrightarrow E^\sharp_p = E^\sharp\otimes_\Q\Q_p$ of \'etale $\Q_p$-algebras. Associated with $\iota_0:E_{\mathfrak{q}}\hookrightarrow \Q_p^\alg$ are the subsets
\[
\Phi^\sharp_{\mathfrak{p}} = \{\Phi_{\mathfrak{p}}\in \mathrm{CM}(E_{\mathfrak{p}}):\; \iota_0\in\Phi_{\mathfrak{p}}\}\;;\; \overline{\Phi}^\sharp_{\mathfrak{p}} = \{\Phi_{\mathfrak{p}}\in \mathrm{CM}(E_{\mathfrak{p}}):\; \overline{\iota}_0\in\Phi_{\mathfrak{p}}\},
\]
and we have
\[
\Phi^\sharp = \{\iota^\sharp:E_p^\sharp\to\Q_p^\alg:\;\iota^\sharp\vert_{E^\sharp_{\mathfrak{p}}}\in \Phi^\sharp_{\mathfrak{p}} \}.
\]

Now, let $T_{\mathfrak{q}}\subset T_{\Q_p}$ be as in Remark~\ref{rem:Eq times reps}. Viewed as a representation of $T_{\mathfrak{q}}$, $H^\sharp_p = H^\sharp\otimes_\Q\Q_p$ admits the $T_{\mathfrak{q}}$-stable subspace $H^\sharp_{\mathfrak{p}}$ corresponding to the subspace $E^\sharp_{\mathfrak{p}}\subset E^\sharp_p$. Moreover, we have a canonical lattice $H^\sharp_{\mathfrak{p},\Z_p}\subset H^\sharp_{\mathfrak{p}}$ corresponding to $\co_{E^\sharp_{\mathfrak{p}}}\subset E^\sharp_{\mathfrak{p}}$. This is stable under $K_{0,\mathfrak{q}} = K_0\cap T(\Q_p)$, and we have a natural $K_{0,\mathfrak{q}}$-equivariant isomorphism of $\co_{E^\sharp_p}$-modules:
\[
H^\sharp_{\mathfrak{p},\Z_p}\otimes_{\co_{E_{\mathfrak{p}}^\sharp}}\co_{E^\sharp_p}\xrightarrow{\simeq}H^\sharp_{\Z_p}.
\]

If $\bm{H}^\sharp_{\mathfrak{p},\dR,\co_y}$ is the de Rham realization of $H^\sharp_{\mathfrak{p},\Z_p}$ obtained from Corollary~\ref{cor:realizations y}, then we obtain an isomorphism
\[
\bm{H}^\sharp_{\mathfrak{p},\dR,\co_y}\otimes_{\co_{E_{\mathfrak{p}}^\sharp}}\co_{E^\sharp_p}\xrightarrow{\simeq}\bm{H}^\sharp_{\dR,\co_y}
\]
of filtered $\co_y\otimes_{\Z_p}\co_{E^\sharp_p}$-modules.

Fix a $\co_y$-module generator $\bm{f}_0\in \Fil^1\bm{V}_{0,\dR,\co_y}$, and view it as a map
\[
\mathrm{gr}^{-1}_{\Fil}\bm{H}^\sharp_{\dR,\co_y}\to \Fil^0\bm{H}^\sharp_{\dR,\co_y}.
\]
We find from the construction that this arises via a change of scalars from $\co_{E^\sharp_{\mathfrak{p}}}$ to $\co_{E^\sharp_p}$ of a map
\[
\bm{f}_{0,\mathfrak{p}}:\mathrm{gr}^{-1}_{\Fil}\bm{H}^\sharp_{\mathfrak{p},\dR,\co_y}\to \Fil^0\bm{H}^\sharp_{\mathfrak{p},\dR,\co_y}
\] 

Let $P\subset \mathrm{Frac}(W)^\alg$ be a Galois extension of $\mathrm{Frac}(W)$ containing $\co_y$, which receives all maps $E^\sharp_{\mathfrak{p}}\to \Q_p^\alg$. Then
\[
M \define \bm{H}_{\mathfrak{p},\dR,\co_y}\otimes_{\co_y}\co_P
\]
is a finite free $\co_P\otimes_{\Z_p}\co_{E^\sharp_{\mathfrak{p}}}$-module of rank $1$. One can now check that, in the notation preceding Proposition~\ref{prop:colmez calc}, we have
\[
\Fil^0\bm{H}_{\mathfrak{p},\dR,\co_y}\otimes_{\co_y}\co_P = \mathcal{K}(M,\Phi^\sharp_{\mathfrak{p}})\;;\; \mathrm{gr}^{-1}_{\Fil}\bm{H}_{\mathfrak{p},\dR,\co_y}\otimes_{\co_y}\co_P = \mathcal{Q}(M,\overline{\Phi}^\sharp_{\mathfrak{p}}).
\]

Therefore, we have
\begin{align}
\label{eqn:length decomp}
\mathrm{length}(\mathcal{L}_{y}/\mathrm{im}(s_y)) &= \frac{1}{e(P/E_{\mathfrak{q}})}\cdot 2^{d-d_{\mathfrak{p}}}\cdot \ord_{\co_P}(\det(\bm{f}_{0,\mathfrak{p}})) \\
&= \frac{2^{d-d_{\mathfrak{p}}}}{e(P/E_{\mathfrak{q}})}\cdot\left[ \mathrm{length}_{\co_P}\left(\frac{\mathcal{Q}(M,\Phi_{\mathfrak{p}}^\sharp)}{\bm{f}_{0,\mathfrak{p}}(\mathcal{Q}(M,\overline{\Phi}_{\mathfrak{p}}^\sharp))}\right) - \mathrm{length}_{\co_P}\mathcal{C}(M,\Phi^\sharp_{\mathfrak{p}})\right]\nonumber.
\end{align}
Here, $e(P/E_{\mathfrak{q}})$ is the ramification index of $P$ over $E_{\mathfrak{q}}$, and $d_{\mathfrak{p}} = [F_{\mathfrak{p}}:\Q_p]$.

Now, set $N = \bm{V}_{0,\dR,\co_y}\otimes_{\co_y}\co_P$: this is a free module of rank $1$ over $\co_P\otimes_{\Z_p}\co_{E_{\mathfrak{p}}}$. Set
\[
\Upsilon_{\iota_0} = \Hom(E_{\mathfrak{p}},\Q_p^\alg)\backslash\{\iota_0\}.
\]
Then we have
\[
\Fil^1\bm{V}_{0,\dR,\co_y}\otimes_{\co_y}\co_P = \mathcal{K}(N,\Upsilon_{\iota_0}).
\]
Moreover, the action of any generator of $\mathcal{Q}(N,\Upsilon_{\iota_0})$ induces an isomorphism
\[
\mathcal{Q}(M,\overline{\Phi}^\sharp_{\mathfrak{p}})\xrightarrow{\simeq}\mathcal{Q}(M,\Phi^\sharp_{\mathfrak{p}})
\] 
of $\co_P$-modules. Therefore, we have 
\begin{equation}
\label{eqn:length comp 1}
 \mathrm{length}_{\co_P}\left(\frac{\mathcal{Q}(M,\Phi_{\mathfrak{p}}^\sharp)}{\bm{f}_{0,\mathfrak{p}}(\mathcal{Q}(M,\overline{\Phi}_{\mathfrak{p}}^\sharp))}\right) = 2^{d_{\mathfrak{p}}-1}\cdot \mathrm{length}_{\co_P}\mathcal{C}(N,\Upsilon_{\iota_0}).
\end{equation}

Arguing as in Propositions~\ref{prop:average reflex} and~\ref{prop:colmez height average}, we see that
\[
a^0_{(E^\sharp_{\mathfrak{p}},\Phi^\sharp_{\mathfrak{p}})} = 2^{d_{\mathfrak{p}}-2} \left(  \bm{1} + \frac{1}{d_{\mathfrak{p}}} \mathrm{Ind}_{\Gamma_{F_{\mathfrak{p}}}}^{\Gamma_{\Q_p}}( \chi_{\mathfrak{p}} ) \right),
\]
where $\chi_{\mathfrak{p}}$ is the (possibly trivial) quadratic character of $F_{\mathfrak{p}}$ associated with $E_{\mathfrak{p}}/F_{\mathfrak{p}}$. From this and Proposition~\ref{prop:colmez calc}, one easily deduces that we have
\begin{equation}\label{eqn:length comp 2}
 \mathrm{length}_{\co_P}\mathcal{C}(M,\Phi^\sharp_{\mathfrak{p}}) = e(P/E_{\mathfrak{q}})\cdot 2^{d_{\mathfrak{p}}-2}\cdot (2\cdot\ord_{\mathfrak{q}}(\mathfrak{d}_{E/\Q}) - \ord_{\mathfrak{q}}(\mathfrak{d}_{F/\Q})).
\end{equation}

A similar, but much easier computation shows
\begin{equation}\label{eqn:length comp 3}
\mathrm{length}_{\co_P}\mathcal{C}(N,\Upsilon_{\iota_0}) = -e_P\cdot\mu_p(a(E_{\mathfrak{p}},\Upsilon_{\iota_0})) = e(P/E_{\mathfrak{q}})\cdot \ord_{\mathfrak{q}}(\mathfrak{d}_{E/\Q}).
\end{equation}

Combining~\eqref{eqn:length decomp},~\eqref{eqn:length comp 1},~\eqref{eqn:length comp 2} and~\eqref{eqn:length comp 3} now yields~\eqref{eqn:main length comp} and hence the proposition.

% Note that the map $\bm{f}_0$ from Lemma~\ref{lem:mult by f0} induces an injective map
% \[
% \mathrm{gr}^{-1}_{\Fil}\bm{H}^\sharp_{\dR,\co_y} \to \Fil^0\bm{H}^\sharp_{\dR,\co_y},
% \]
% of finite free $\co_y$-modules of the same rank, and we will again call this $\bm{f}_0$. 

% We now have:
% \begin{align*}
%  & = \mathrm{length}\left(\frac{\det(\Fil^0\bm{H}^\sharp_{\dR,\co_y})}{\mathrm{im}(\det(\bm{f}_0))}\right) \\
% & = \mathrm{length}\left(\frac{\Fil^0\bm{H}^\sharp_{\dR,\co_y}}{\mathfrak{d}_{E/F}\otimes_{\co_E}\Fil^0\bm{H}^\sharp_{\dR,\co_y}}\right) -  \mathrm{length}\left(\frac{\mathrm{gr}^{-1}_{\Fil}\bm{H}^\sharp_{\dR,\co_y}}{\overline{\Fil}^0\bm{H}^\sharp_{\dR,\co_y}}\right)  \\
% & = 2^{d-1}\cdot\ord_{\mathfrak{q}}(\mathfrak{d}_{E/\Q}) - 2^{d-2}\cdot[\ord_{\mathfrak{q}}(\mathfrak{d}_{E/\Q}) - \ord_{\mathfrak{q}}(\mathfrak{d}_{F/\Q})]\\
% & = 
% \end{align*}
% In the second and third identities, we have used Lemma~\ref{lem:mult by f0} and~\ref{lem:length fil0 fil0bar}, respectively.

% Now, we have:

\end{proof}

%%%%%%%%%%%%%%%%%%%%%%%%%%%%%%%%%%%

\subsection{The averaged Colmez conjecture}

%%%%%%%%%%%%%%%%%%%%%%%%%%%%%%%%%%%

As in Remark \ref{rem:hermitian construction}, choose any $\xi\in F^\times$ negative at $\iota_0$
and positive at $\iota_1,\ldots, \iota_{d-1}$.  This defines  a rank two quadratic space  
\[
(\mathscr{V},\mathscr{Q}) = (E,\xi\cdot \mathrm{Nm}_{E/F})
\] 
over $F$, and we set 
\[
(V,Q) = (\mathscr{V} , \mathrm{Tr}_{F/\Q}  \circ \mathscr{Q})
\]  
as in (\ref{isometric equality}).  Fix any maximal lattice $L\subset V$, and  and let $D_{bad,L}$ be the product of all the bad primes with respect to $L$ (see Definition~\ref{defn:D bad}).

Recall the integral model $\mathcal{M}\to\Spec(\Z)$ of the GSpin Shimura variety associated with $L$, as well as the finite cover $\mathcal{Y}\to \mathcal{Y}_0$ associated with the level subgroup $K_{L,0}$ and equipped with a map $\mathcal{Y}\to\mathcal{M}$. We also had the metrized line bundle
$\widehat{\bm{\omega}}$ on $\mathcal{M}$ from~\S\ref{ss:line bundles}. Over $\mathcal{Y}$, this line bundle arises from the Hodge filtration on the vector bundle $\bm{V}_{\dR}$ obtained as the de Rham realization of the pair $(V,\widehat{L})$.

Let $(V^\beef, Q^\beef)$ be a quadratic space of signature $(n^\beef,2)$ with $n\geq 3$, and suppose that we have an isometric embedding 
\[
(V,Q)  \hookrightarrow (V^\beef, Q^\beef)
\]
and a maximal lattice $L^\beef\subset V^\beef$ with $L\subset L^\beef$. This corresponds to a map $\mathcal{M}\to \mathcal{M}^\beef$ of integral models over $\Z$ for the associated GSpin Shimura varieties.   

Suppose that $f \in M^!_{  1- n^\beef/2   } (\omega_{L^\beef})$ has integral Fourier coefficients and  nonzero constant term $c_f(0,0)$. Let $\mathcal{Z}^\beef(f)$ be the corresponding divisor  on $\mathcal{M}^\beef$, and  assume that Hypothesis \ref{hyp:proper} is satisfied. After replacing $f$ by a multiple if necessary, we obtain the vertical metrized line bundle $\widehat{\mathcal{E}}^\beef(f) = (\mathcal{E}^\beef(f),0)$ on $\mathcal{M}^\beef$ as in Theorem~\ref{thm:taut degree}.

As before, we will write $a\approx_L b$ for two real numbers $a,b$ if $a-b$ is a rational linear combination of $\log(p)$ with $p\mid D_{bad,L}$.

\begin{proposition}\label{prop:colmez prelim bound}
We have
\[
\frac{1}{2^d}\sum_\Phi h^\Falt_{(E,\Phi) } - \frac{1}{2^d}\sum_\Phi h^\Colmez_{(E,\Phi) }  \approx_L \frac{1}{4c_f(0,0)}\frac{[\widehat{\mathcal{E}}^\beef(f):\mathcal{Y}]}{\mathrm{deg}_{\C}({Y})}.
\]
\end{proposition}

\begin{proof}
Given Theorems~\ref{thm:taut degree} and~\ref{thm:Faltings height line bundles}, and Proposition~\ref{prop:colmez height average}, we only have to show:
\begin{equation}\label{eqn:degree difference}
[\widehat{\bm{\omega}}:\mathcal{Y}] - \widehat{\deg}_{\mathcal{Y}}(\widehat{\bm{\omega}}_0) - \log|D_F| \approx_L 0.
\end{equation}

For this, note that, via the construction in Proposition~\ref{prop:realizations integral model}, the sheaves $\bm{V}_{\dR}$ and $\bm{V}_{0,\dR}$ are both associated with the standard $T$-representation $V = V_0$, 
but correspond to different $K_{0,L}$-stable lattices in $V_{\A_f}$. 
The first is associated with the lattice $\widehat{L}$, and the second with the lattice $\widehat{L}_0$. 
In particular, since the restrictions of these bundles to the generic fiber does not depend on the $K_{0,L}$-stable lattice, there is a canonical isomorphism
\begin{equation}\label{eqn:omega ident}
\bm{\omega}\vert_{{Y}}\xrightarrow{\simeq}\bm{\omega}_0\vert_{{Y}} 
\end{equation}
of line bundles over ${Y}$. At each point $y\in \mathcal{Y}_{\infty}(\C)$ lying above a place $\iota:F\to \R$, this isomorphism carries the metric $||\cdot||_y$ on $\bm{\omega}_y$ to $|\iota(\xi)|$-times the metric $||\cdot||_{0,y}$. 

Therefore, it is enough to show that~\eqref{eqn:omega ident} induces an isomorphism
\[
\bm{\omega}\vert_{\mathcal{Y}[D_{bad,L}^{-1}]}\xrightarrow{\simeq}\xi\mathfrak{d}^{-1}_{F/\Q}\otimes_{\co_F}\bm{\omega}_0\vert_{\mathcal{Y}[D_{bad,L}^{-1}]}
\]
of line bundles over $\mathcal{Y}[D_{bad,L}^{-1}]$.

This is a statement that can be checked over the complete \'etale local rings of $\mathcal{Y}[D_{bad,L}^{-1}]$. So let $\mathfrak{q}\subset\co_E$ be a prime lying above a prime $p\nmid D_{bad,L}$, and suppose that we have $y\in \mathcal{Y}(\F_{\mathfrak{q}}^{\alg})$. Let $\mathfrak{p}\subset \co_F$ be the prime induced from $\mathfrak{q}$. By the definition of $D_{bad,L}$, $L_{\mathfrak{p}} = L_p\cap V_{\mathfrak{p}}$ contains a maximal $\co_{E,\mathfrak{p}}$-stable quadratic lattice $\Lambda_{\mathfrak{p}}$. We then must have
\begin{equation}\label{eqn:Lambda L 0}
\Lambda_{\mathfrak{p}} = \xi\mathfrak{d}_{F_{\mathfrak{p}}/\Q_p}^{-1}L_{0,\mathfrak{p}}.
\end{equation}

Let $\co_y$ be the complete local ring of $\mathcal{Y}$ at $y$. As explained in Remark~\ref{rem:Eq times reps}, from $\Lambda_{\mathfrak{p}}$ and $L_{\mathfrak{p}}$, we obtain de Rham realization $\bm{\Lambda}_{\mathfrak{p},\dR,\co_y}$ and $\bm{V}_{\mathfrak{p},\dR,\co_y}$ over $\co_y$; these are filtered vector bundles over $\co_y$. 

Choose an isometric embedding $L\hookrightarrow L^\beef$ with $L^\beef$ of signature $(n^\beef,2)$ and self-dual over $\Z_p$. The inclusions $\Lambda_{\mathfrak{p}}\hookrightarrow L_p\hookrightarrow L^\beef_p$ give embeddings
\[
\bm{\Lambda}_{\mathfrak{p},\dR,\co_y}\hookrightarrow\bm{V}^\beef_{\dR,\co_y}
\]
of free $\co_y$-modules.

It now follows from Lemma~\ref{lem:filt beef lambda} that the inclusion
\[
\bm{\omega}_{\co_y}\cap \bm{\Lambda}_{\mathfrak{p},\dR,\co_y} \hookrightarrow \bm{\omega}_{\co_y} = \Fil^1\bm{V}^\beef_{\dR,\co_y}
\]
is an isomorphism. 

Therefore,~\eqref{eqn:Lambda L 0} shows that the isomorphism~\eqref{eqn:omega ident} induces an isomorphism
\[
\bm{\omega}_{\co_y}\xrightarrow{\simeq}\xi(\mathfrak{d}_{F_{\mathfrak{p}}/\Q_p}^{-1}\otimes_{\co_F}\bm{\omega}_{0,\co_y})
\]
of line bundles over $\co_y$, finishing the proof of the Proposition.

% Using this, we find that the contribution at $p$ to the difference on the left hand side of~\eqref{eqn:degree difference} is
% \begin{align*}
% \bigl([\widehat{\bm{\omega}}:\mathcal{Y}] - \widehat{\deg}_{\mathcal{Y}}(\widehat{\bm{\omega}}_0) - \log|D_F| \bigr)_p & = \sum_{\mathfrak{q}\mid p}\log N(\mathfrak{q})\sum_{y\in \mathcal{Y}(\F_{\mathfrak{q}}^\alg)}\frac{\ord_{\mathfrak{q}}(\xi)}{\vert \Aut(y)\vert}.
% \end{align*}

% Since $\mathcal{Y}_{p}$ is finite \'etale over $\co_{E,p}$, as in the proof of Proposition~\ref{prop:point count}, we find
% \[
% \sum_{y\in \mathcal{Y}(\F_{\mathfrak{q}}^{\alg})}\frac{1}{|\Aut(y)|} = \deg_\C({Y}).
% \]

% Therefore, we have
% \begin{align*}
% \bigl([\widehat{\bm{\omega}}:\mathcal{Y}] - \widehat{\deg}_{\mathcal{Y}}(\widehat{\bm{\omega}}_0) - \log|D_F|\bigr)_p & = \deg_\C({Y})\sum_{\mathfrak{q}\mid p}\log N(\mathfrak{q})\ord_{\mathfrak{q}}(\xi)\\
% & = \deg_\C({Y})\log|\mathrm{Nm}_{(E\otimes\Q_p)/\Q_p}(\xi)|_p.
% \end{align*}

% Combining this with~\eqref{eqn:infinite difference} and the product formula yields~\eqref{eqn:degree difference}, and hence the proposition.
\end{proof}

%Recall from~\eqref{ the torus $T=T_E/T_F^1$

%Recall that from the above data we have constructed a morphism 
%$\mathcal{Y} \to \mathcal{M}$ of $\Z$-stacks, and  

%The arithmetic degree $[ \widehat{\bm{\omega}} : \mathcal{Y}]$ was computed in Corollary \ref{cor:taut degree};
%now we compute the the arithmetic degree of the other two line bundles along $\mathcal{Y}$.

%\begin{proposition}\label{prop:lie degree}
%The equality
%\[
%\frac{  [  \widehat{   \bm{\Omega}    }  :  \mathcal{Y} ] }{     \deg_\C (Y) }
%=     2^{d} \sum_\Phi    h^\Falt_{( E ,\Phi )} 
%\]
%holds up to a $\Q$-linear combination of $\{ \log(p) : p \mid D_{bad} \}$.
%\end{proposition}

%\begin{proposition}\label{prop:de Rham degree}
%We have
%\[
%\frac{  [  \widehat{\det}(\bm{H}_{\dR})  :  \mathcal{Y} ] }{    \deg_\C (Y) } =       2^{2d} d \cdot \log(2\pi).
%\]
%\end{proposition}

%\begin{corollary}\label{Cor:faltings calculation approx}
%We have:
%\[
%\frac{2d}{2^d}\cdot h^\Falt_{(E^\sharp,\Phi^\sharp)} \approx_L 
%- \frac{ \Lambda'(0,\chi )  }{ 2 \Lambda(0,\chi )}    
%- \frac{d}{4}\log(16\pi^3 e^\gamma).
%\]
%\end{corollary}
%\begin{proof}

%\end{proof}

\begin{proposition}\label{prop:choosing_good_lattices}
We can find another choice of auxiliary data
\[
(\mathscr{V}',\mathscr{Q}') = (E,\xi'\cdot \mathrm{Nm}_{E/F})
\] 	
and a maximal lattice $L'\subset V'$ such that $\mathrm{gcd}(D_{bad,L},D_{bad,L'}) = 1$.
\end{proposition}
\begin{proof}
It is sufficient to show that, given any finite set of rational primes $S$, we can find $\xi'$ and $L'$ such that no prime in $S$ divides $D_{bad,L'}$. 

To make this more concrete, suppose that we given an ideal $\mathfrak{a}\subset\co_E$ and $\xi'\in F$ satisfying $\iota_0(\xi')<0$ and $\iota_j(\xi')>0$ for $j>0$. For a prime $p$, we will declare the pair to be \emph{good at $p$} if
\[
\Lambda_p \define (\mathfrak{a},\mathrm{Tr}_{E/\Q}(\xi'\mathrm{Nm}_{E/F}))\otimes_{\Z}\Z_p
\]
is an $\co_E$-stable quadratic $\Z_p$-lattice in $(E,\mathrm{Tr}_{E/\Q}(\xi'\mathrm{Nm}_{E/F}))\otimes_{\Q}\Q_p$, which is self-dual over all primes $\mathfrak{p}\mid p$ that are unramified in $E$, and which satisfies
\[
\Lambda_{\mathfrak{p}}^\vee\subset \mathfrak{d}^{-1}_{E_{\mathfrak{q}}/F_{\mathfrak{p}}}\Lambda_{\mathfrak{p}}
\]
when $\mathfrak{p}$ is ramfified in $E$ and $\mathfrak{q}\subset\co_E$ is the unique prime above it. Here, we have set $\Lambda_{\mathfrak{p}} = \Lambda_p\otimes_{\co_{F,p}}\co_{F,\mathfrak{p}}$.

\begin{lemma}
Suppose that $(\mathfrak{a},\xi')$ is good at all $p\in S$. Then there exists a maximal lattice 
\[
L'\subset V' = (E, \mathrm{Tr}_{F/Q}(\xi'\cdot \mathrm{Nm}_{E/F})) 
\]
that is good at all primes $p\in S$. 
\end{lemma}
\begin{proof}
For any prime $p\in S$, and any prime $\mathfrak{p}\subset\co_F$ lying above such $p$, write $\mathscr{V}'_{\mathfrak{p}}\subset V'_{\Q_p}$ for the $\mathfrak{p}$-isotypic part of $V'_{\Q_p}$, and fix a maximal lattice $L'_{\mathfrak{p}} \subset \mathscr{V}'_{\mathfrak{p}}$ containing $\Lambda_{\mathfrak{p}}$. Now, take $L'\subset V'$ to be any maximal lattice such that, for every $p\in S$, $L'_{\Z_p}$ contains $\bigoplus_{\mathfrak{p}\mid p}L'_{\mathfrak{p}}$.
\end{proof}

It now remains to find a pair $(\mathfrak{a},\xi')$ that is good at all primes in $S$. 

Given a pair $(\mathfrak{a},\xi')$ as above, one can check that the pair is good at $p$ if and only if for all primes $\mathfrak{p}\subset \co_F$ lying above $p$, $\mathfrak{p}$ is relatively prime to $\xi'\mathrm{Nm}_{E/F}(\mathfrak{a})\mathfrak{d}_{F/\Q}$.

Write $\mathrm{Cl}^+(F)$ for the narrow class group of $F$ and $\mathrm{Cl}(E)$ for the class group of $E$. The norm map induces a map
\begin{equation}\label{class_group_map}
\mathrm{Cl}(E)\to\mathrm{Cl}^+(F)
\end{equation}
This map is surjective if and only if $E/F$ is ramified at some finite prime. Indeed, via class field theory, the surjectivity of~\eqref{class_group_map} is equivalent to the assertion that the narrow class field of $F$ does not contain $E$. 

Suppose that $E/F$ is unramified at all finite places. In this case, the quadratic character $\chi_{E/F}$ can be viewed as a character
\[
\chi_{E/F} : \mathrm{Cl}^+(F)\to\{\pm 1\},
\]
whose kernel is exactly the image of~\eqref{class_group_map}. We now interrupt the proof for:
\begin{lemma}\label{lem:unramified}
When $E/F$ is unramified at all finite places, $d\equiv 0\pmod{2}$. Moreover, we have
\[
\chi_{E/F}(\mathfrak{d}_{F/\Q}) = (-1)^{d/2}.
\]
\end{lemma}
\begin{proof}
Treating $\chi_{E/F}$ as an id\'ele class character, consider its infinite part $\chi_{E/F,\infty}$. Since $E$ is a totally imaginary extension of $F$, $\chi_{E/F,\infty}$ is the product of the sign characters over all infinite places of $F$. Since $\chi_{E/F}$ is unramified at all finite places, for any unit $\alpha\in\co_F^\times$, we have
\[
\chi_{E/F,\infty}(\alpha) = \chi_{E/F}(\alpha)\chi_{E/F,f}(\alpha)^{-1} = 1.
\]
Applying this to the case $\alpha=-1$ shows that $(-1)^d = 1$, and so $d$ must be even.

The final assertion is an improvement by Armitage   \cite[Theorem 3]{armitage} of a classical result of Hecke.
\end{proof}

We return to the proof of Proposition~\ref{prop:choosing_good_lattices}. Choose an arbitrary $\xi_0\in\co_F$ with $\iota_0(\xi_0)<0$ and $\iota_j(\xi_0)>0$ for $j>0$. Consider the ideal 
\[
\mathfrak{b}=\xi_0\mathfrak{d}_{F/\Q}\subset \co_F.
\]

Assume either that $E/F$ is ramified at some finite prime, \emph{or} that $E/F$ is unramified and $d\equiv 2\pmod{4}$. Under either assumption, we claim that there exists an ideal $\mathfrak{a}\subset\co_E$ and a totally positive element $\eta\in F^\times$ such that 
\[
\eta \mathrm{Nm}_{E/F}\mathfrak{a} = \mathfrak{b}^{-1}.
\]
In other words, the class of $\mathfrak{b}$ in $\mathrm{Cl}^+(F)$ is in the image of~\eqref{class_group_map}.

If $E/F$ is ramified, then this is immediate from the surjectivity of~\eqref{class_group_map}. 

If $E/F$ is unramified, then $\mathfrak{b}=\xi_0\mathfrak{d}_{F/\Q}$, and by~\eqref{lem:unramified}, we have:
\[
\chi_{E/F}(\mathfrak{b}) = (-1)^{\frac{d}{2}+1}.
\]
Therefore, when $d\equiv 2\pmod{4}$, $\mathfrak{b}$ is in the image of~\eqref{class_group_map}, and the claim follows.

Now, it is easily checked that, with $\xi'=\eta\xi_0$, $(\mathfrak{a},\xi')$ is good at \emph{all} primes $p$.

It remains to consider the case where $E/F$ is unramified and $d\equiv 0\pmod{4}$. 
In this case, $\chi_{E/F}(\mathfrak{d}_{F/\Q})=1$ by Lemma \ref{lem:unramified}. 
Therefore, we can find an $\mathfrak{a}\subset\co_E$ and totally positive $\eta\in F^\times$ such that
\[
\eta  \cdot \mathrm{Nm}_{E/F}( \mathfrak{a} ) = \mathfrak{d}_{F/\Q}^{-1}.
\]

Now, given a totally positive $\beta\in F$, the pair $(\mathfrak{a},\beta\eta\xi_0)$ is good at a all primes in $S$ if and only if $\beta\xi_0$ is not divisible by any $p\in S$. Such a $\beta$ can always be found by weak approximation.
\end{proof}

\begin{theorem}\label{thm:average colmez}
We have
\[
\frac{1}{2^{d}} \sum_\Phi h^\Falt_{(E,\Phi)} =  \frac{1}{2^{d}} \sum_\Phi h^\Colmez_{(E,\Phi)}.
\]
\end{theorem}
\begin{proof}
Combining Propositions~\ref{prop:colmez prelim bound} and~\ref{prop:choosing_good_lattices}, we find that we have
\[
\frac{1}{2^{d}} \sum_\Phi h^\Falt_{(E,\Phi)} - \frac{1}{2^{d}} \sum_\Phi h^\Colmez_{(E,\Phi)} = \sum_pb_E(p)\log(p),
\]
where we can compute $b_E(p)$ as follows: Choose auxiliary data $(\mathscr{V},\mathscr{Q})$ and a maximal lattice $L\subset V$ such that $p\nmid D_{bad,L}$. Also choose an auxiliary quadratic space $(V^\beef,Q^\beef)$ of signature $(n^\beef,2)$ with $n^\beef\geq 3$, as well as a maximal lattice $L^\beef\subset V^\beef$ containing $L$. Choose a weakly holomorphic form
\[
f (\tau) = \sum_{ m \gg -\infty} c_f(m) \cdot q^m  \in M_{1-\frac{n}{2}}^!(\omega_{L^\beef})
\] 
with integral principal part and   $c_f(0,0)\not=0$, and satisfying Hypothesis~\ref{hyp:proper}. Then, after replacing $f$ by a suitable multiple, we have:
\begin{equation}\label{eqn:b2_borcherds_bE}
\sum_pb_E(p)\log(p) = \frac{1}{4c_f(0,0)}\frac{[\widehat{\mathcal{E}}^\beef(f):\mathcal{Y}]}{\deg_{\C}({Y})}.
\end{equation}

Therefore, it is enough to show that, for each prime $p$, we can choose $L^\beef$ and $f$ such that $\mathcal{E}^\beef(f)$ does not intersect $\mathcal{M}^\beef_{\F_p}$.

It is an easy exercise, given the classification of quadratic forms over $\Q$, to find $L^\beef$ such that $n^\beef = 2d$, and such that $L^\beef_{(p)}$ is self-dual, and such that $L$ embeds isometricaly in $L^\beef$. Now, the orthogonal complement
\[
\Lambda = L^\perp \subset L^\beef
\]
is a rank $2$ positive definite lattice over $\Z$. Any rational prime not split in the discriminant field of $\Lambda$ will fail to be represented by $\Lambda_\Q$. Therefore, by Theorem~\ref{thm:borcherds_support} below, we can find a weakly modular form $f$ as above such that $c_f(m)\neq 0$ only if $m$ is not represented by $\Lambda_\Q$. In particular, Hypothesis~\ref{hyp:proper} is satisfied, and, since $L^\beef_{(p)}$ is self-dual, by Theorem~\ref{thm:borcherds}, $\mathcal{E}^\beef(f)$ does not intersect $\mathcal{M}^\beef_{\F_p}$, as desired.

We note again that the proof only used knowledge of the divisor of the Borcherds lift of $f$ at primes where $L^\beef$ is self-dual, which is contained in~\cite{Hormann}, and not the full strength of Theorem~\ref{thm:borcherds}.
\end{proof}

The proof above used the following consequence of a result of Bruinier~\cite[Theorem 1.1]{Bru:prescribed}, which we state here for the reader's convenience.

\begin{theorem}[Bruinier]\label{thm:borcherds_support}
Let $L$ be a quadratic lattice of signature $(n,2)$ with $n\geq 2$. If $S$ is any infinite subset of square-free positive elements of $D_{L}^{-1} \Z$ represented by $L^{\vee}$, there is a weakly holomorphic form $f\in M^!_{ 1 - n/2   } ( \omega_{L} )$ such that 
\begin{enumerate}
\item
$c^+_f( m,\mu) \in \Z$ for all $m$ and $\mu$, 
\item
$c^+_f(0,0) \not=0$, 
\item
 if $m>0$ and $m\nin S$, then $c^+_f( -m ,\mu ) =0$ for all $\mu\in L^{\vee} / L$. 
 \end{enumerate}
\end{theorem}

%\nocite{*}
\bibliographystyle{amsalpha}

 \providecommand{\bysame}{\leavevmode\hbox to3em{\hrulefill}\thinspace}
\providecommand{\MR}{\relax\ifhmode\unskip\space\fi MR }
\providecommand{\MRhref}[2]{%
  \href{http://www.ams.org/mathscinet-getitem?mr=#1}{#2}
}
\providecommand{\href}[2]{#2}

%\bibliography{bigCM}

\end{document}